\documentclass[reqno]{amsart}

\pdfoutput=1

\usepackage[english]{babel}
\usepackage{amsmath}
\usepackage{amsthm}
\usepackage{amsfonts}
\usepackage{amssymb}
\usepackage{anysize}
\usepackage{hyperref}
\usepackage{appendix}
\usepackage{graphicx}
\usepackage{appendix}

\newtheorem{defi}{Definition}[section]
\newtheorem{lem}[defi]{Lemma}
\newtheorem{theo}[defi]{Theorem}
\newtheorem{cor}[defi]{Corollary}
\newtheorem{pro}[defi]{Proposition}

\newtheorem{rem}[defi]{Remark}

\DeclareMathOperator{\divop}{div}
\DeclareMathOperator{\curl}{curl}
\DeclareMathOperator{\jac}{Jac}

\DeclareMathOperator{\supp}{supp}
\DeclareMathOperator{\Ker}{Ker}
\DeclareMathOperator{\Img}{Im}
\DeclareMathOperator{\NN}{\mathbb{N}}
\DeclareMathOperator{\RR}{\mathbb{R}}
\DeclareMathOperator{\CC}{\mathbb{C}}

\allowdisplaybreaks

\begin{document}

\title[Generalized Beltrami fields and  vortex structures in the Euler equations]{Stability results, almost global generalized Beltrami fields and applications to vortex structures in the Euler equations}{\thanks{This work has been partially supported by the MINECO-Feder (Spain) research grant, the project MTM2014-53406-R (D.P.\ and J.S.), the Junta de Andaluc\'ia (Spain) Project FQM 954 (D.P.\ and J.S.), the MECD (Spain) research grant FPU2014/06304 (D.P.) and by the ERC Starting Grant~633152 and the ICMAT--Severo Ochoa grant SEV-2011-0087 (A.E.).
}}

\author{Alberto Enciso}
\address{Instituto de Ciencias Matem\'aticas, Consejo Superior de Investigaciones Cient\'\i ficas, 28049 Madrid, Spain}
\email{aenciso@icmat.es}

\author{David Poyato}
\address{Departamento de Matem\'atica Aplicada, Universidad de Granada, 18071 Granada, Spain}
\email{davidpoyato@ugr.es}

\author{Juan Soler}
\address{Departamento de Matem\'atica Aplicada, Universidad de Granada, 18071 Granada, Spain}
\email{jsoler@ugr.es}

\begin{abstract}
Strong Beltrami fields, that is, vector fields in three dimensions whose curl is the product of the field itself by a constant factor, have long played a key role in fluid mechanics and magnetohydrodynamics. In particular, they are the kind of stationary solutions of the Euler equations where one has been able to show the existence of vortex structures (vortex tubes and vortex lines) of arbitrarily complicated topology. On the contrary, there are very few results about the existence of generalized Beltrami fields, that is, divergence-free fields whose curl is the field times a non-constant function. In fact, generalized Beltrami fields (which are also stationary solutions to the Euler equations) have been recently shown to be rare, in the sense that for ``most'' proportionality factors there are no nontrivial Beltrami fields of high enough regularity (e.g., of class $C^{6,\alpha}$), not even locally.

Our objective in this work is to show that, nevertheless, there are ``many'' Beltrami fields with non-constant factor, even realizing arbitrarily complicated vortex structures. This fact is relevant in the study of turbulent configurations. The core results are an ``almost global'' stability theorem for strong Beltrami fields, which ensures that a global strong Beltrami field with suitable decay at infinity can be perturbed to get ``many'' Beltrami fields with non-constant factor of arbitrarily high regularity and defined in the exterior of an arbitrarily small ball, and a ``local'' stability theorem for generalized Beltrami fields, which is an analogous perturbative result which is valid for any kind of Beltrami field (not just with a constant factor) but only applies to small enough domains. 

The proof relies on an iterative scheme of Grad--Rubin type.  For this purpose, we study the Neumann problem for the inhomogeneous Beltrami equation in exterior domains via a boundary integral equation method and we obtain H\"{o}lder estimates, a sharp decay at infinity and some compactness properties for these sequences of approximate solutions. Some of the parts of the proof are of independent interest.
\end{abstract}

\maketitle

\tableofcontents

\section{Introduction}

{\em Beltrami fields}\/, that is, three dimensional vector fields whose curl is proportional to the field, are a particularly important class of smooth stationary solutions of the three-dimensional incompressible Euler equations:
\[
\partial_t u+ (u\cdot \nabla) u=-\nabla p\,,\qquad \divop u=0\,.
\]
In a way, what makes them so special is the celebrated structure theorem of Arnold~\cite{Arnold}, which asserts that, under suitable technical hypotheses, the velocity field of a smooth stationary solution to the Euler equations is either a Beltrami field or ``laminar'', in the sense that it admits a regular first integral whose smooth level sets provide ``layers'' to which the fluid flow is tangent. In fluid mechanics, a Beltrami field is interpreted as a fluid whose velocity is parallel to its vorticity.

Understanding the knot and link type of stream lines and tubes in stationary fluids has also attracted the attention of many researchers, both from the theoretical and the experimental points of view \cite{Enciso2,Enciso3,KlecknerIrvine,WiegelmannSakurai}, because knotted stationary vortex structures turned out to play a key role in the so called Lagrangian theory of turbulence. From a numerical point of view, the description of the flows in the literature that allow for arbitrary vortex structures is mainly based on an active vector formulation of Euler's equations (see \cite{Constantin} and the references therein). The existence of knotted and linked vortex lines and tubes in stationary solutions to the Euler equations has been recently established in \cite{Enciso2,Enciso3} 
using {\em strong Beltrami fields}, that is, Beltrami fields with a constant proportionaly factor:
\begin{equation}\label{stBF}
\curl u=\lambda u\,, \qquad \lambda\in\RR\backslash\{0\}\,.
\end{equation}
Notice that the Beltrami fields in~\cite{Enciso2,Enciso3} can be assumed to fall off as $1/|x|$ at infinity, and that this decay rate is optimal (see the global obstructions in the form of a Liouville type theorem in \cite{CC,Nadirashvili}). Concrete examples of Beltrami fields with constant proportionality factor are the ABC flows, whose analysis has yielded considerable insight into the aforementioned phenomenon of Lagrangian turbulence~\cite{Dombre}. 

The main objective of this paper is to study the existence, regularity and stability results of {\em generalized Beltrami fields} (i.e., Beltrami fields with nonconstant proportionality factor). This vector fields play a fundamental role in the understanding of turbulence. The idea that turbulent  flows can be understood as a superposition of Beltrami flows has already been proposed in \cite{CM,Orz}. They are also relevant in magnetohydrodynamics in the context of vanishing Lorentz force (\textit{force-free fields}) and they can be used to model magnetic relaxation, which is relevant in some astrophysical applications \cite{Kaiser,Kress,Mo,Mo2}. Indeed, to the best of our knowledge there are just a handful of explicit examples, all of which have Euclidean symmetries, and the analysis of Beltrami fields with nonconstant factor has proved to be extremely hard. The heart of the matter is that, as it was recently proved in~\cite{Enciso1}, the equation for a generalized Beltrami field,
\begin{equation}\label{campobeltramigen.domgeneral.defi}
\curl u=fu\,,\qquad \divop u=0\,,
\end{equation}
does not admit any nontrivial solution, even locally, for a ``generic'' nonconstant function~$f$. In a very precise sense, it shows that Beltrami fields with a nonconstant factor are rare and such obstruction is of a purely local nature. These results have been carefully stated in Appendix \ref{Appendix.B} for the reader's convenience.

One of the aims of this paper is to show that, although generalized Beltrami fields are indeed rare, one can still prove some kind of partial stability result. Specifically, we will show that for each nontrivial Beltrami field, there are ``many'' close enough nonconstant proportionality factor that enjoy close nontrivial generalized Beltrami fields. The stabilility result is ``partial'' in the sense that a ``full'' stability result cannot be expected since the space of factors that enjoy nontrivial generalized Beltrami fields does not contain any ball in the $C^{k,\alpha}$ norm by the above-mentioned obstructions.The analysis of stability can be crucial to shed some light on the interactions between the different scales in the study of relevant configurations in a fully turbulent state.

More concretely, we will prove two stability results for generalized Beltrami fields. The first one (Theorem~\ref{paso.limite.teo}) is an ``almost global'' perturbation result for strong Beltrami fields defined on $\RR^3$. Roughly speaking, it asserts that given any nontrivial solution of~\eqref{stBF} on $\RR^3$ with optimal fall-off at infinity (i.e., $1/|x|$) and any arbitrarily small ball~$G$, there are infinitely many nonconstant factors $f$, as close to the constant~$\lambda$ as one wishes in $C^{k,\alpha}(\RR^3)$, such that the corresponding equation~\eqref{campobeltramigen.domgeneral.defi} admits nontrivial solutions on the complement $\RR^3\backslash \overline{G}$. This can be combined with the results in~\cite{Enciso2,Enciso3} to construct almost global Beltrami fields with a nonconstant factor that feature vortex lines and vortex tubes of arbitrarily complicated topology (Theorem~\ref{BeltramiGeneralizado.TubosAunudados.teo}). The second stability result (Theorem~\ref{perturbacion.local.teo}) states an analogue for perturbations of nontrivial Beltrami fields with constant or nonconstant factor defined in a small enough open set where the field does not to vanish. The point of these stability results is that the perturbation of the initial proportionality factor is defined by recursively propagating a two-variable function along the integral curves of a velocity vector field, so that is the flexibility in choosing the proportionality factor that is granted by the method of proof. Notice that the idea of constructing the proportionality factor by dragging along the integral curves of a field is somehow inherent to the problem, as the incompressibility condition $\divop u=0$ implies that, if it is nonconstant, the factor must be a first integral of the generalized Beltrami field, i.e.,
\[
u\cdot \nabla f=0\,.
\]

Let us outline the key aspects of the proofs. For concreteness, since all the ideas involved in the proof of the local partial stability result are essentially present in that of the almost global theorem, we shall only discuss the latter result in this Introduction. As we have already mentioned, the point of the partial stability result is to develop a perturbation technique allowing us to deform the initial factor~$f$, which for the purpose of this discussion can be taken to be a nonzero constant $\lambda$. This requires analyzing a related boundary value problem, namely, the \textit{Neumann boundary value problem for the inhomogeneous Beltrami equation} with constant proportionality factor $\lambda$ in exterior domains. To our best knowledge, this problem has not been directly studied in the literature. Our  analysis is based on a boundary integral equation method for complex-valued solutions which requires some potential theory estimates for generalized volume and single layer potentials and an analysis of the decay properties and radiation conditions of the solutions. They will be determined through the natural connections between the complex-valued solutions of the Beltrami, Helmholtz and Maxwell systems. 

In~\cite{Kaiser}, the authors show that one can perturb a harmonic field (i.e., a Beltrami field with $\lambda=0$) defined in an exterior domain to construct a generalized Beltrami field with a nonconstant factor. However, the perturbed fields and factors are of low regularity (of class $C^{1,\alpha}$ and $C^{0,\alpha}$, respectively). In view of the relevance and important applications of Beltrami fields with nonzero $\lambda$, we have striven to extend the result for harmonic fields to general Beltrami fields, and also to show the existence of perturbations of arbitrarily high regularity (the field will be in $C^{k+1,\alpha}$ and the factor in $C^{k,\alpha}$ for any fixed integer~$k$). It should be stressed that the passing from $\lambda=0$ to nonzero $\lambda$ is not a trivial matter, since the behavior of the equations at infinity is completely different (oversimplifying a little, for $\lambda=0$ the behavior of the fields at infinity is that of a harmonic function, so one gets uniqueness simply from a decay condition, while for nonzero $\lambda$, Beltrami fields solve Helmhotz's equation, so radiation conditions must be specified to obtain uniqueness.) We will present a detailed treatment of these topics (Section~\ref{Beltrami.NoHomogenea.Seccion} and Appendix~\ref{Teoria.Potencial.Tecnicas.Seccion}), since we consider that they are of independent interest.

The gist of the proof of the almost global partial stability result for strong Beltrami fields is to study the convergence in $C^{k,\alpha}$ of an {iterative scheme} that takes the form
\begin{equation*}
\left\{\begin{array}{ll}
\nabla \varphi_n\cdot u_n=0, & x\in \Omega,\\
\varphi_n=\varphi^0, & x\in \Sigma,
\end{array}\right.\hspace{1cm}
\left\{\begin{array}{l}
\curl u_{n+1}-\lambda u_{n+1}=\varphi_n u_n, \hspace{0.5cm} x\in \Omega,\\
u_{n+1}\cdot \eta= u_0\cdot \eta, \hspace{2.08cm} x\in S.
\end{array}\right.
\end{equation*}
Here, $\Omega$ stands for an exterior domain with smooth boundary $S$, $\eta$ is its outward unit normal vector field and $\Sigma$ is some open subset of the boundary. This is a modified Grad--Rubin method (see \cite{Amari2,Bineau} for the original Grad--Rubin method in the setting of force-free fields perturbations of harmonic fields), which we will start up with a strong Beltrami field $u_0$ of constant proportionality factor $\lambda$ (which can be assumed to exhibit knotted and linked vortex structures) and prescribes the value $\varphi^0$ of the perturbation of the proportionality factor $\lambda$ over $\Sigma$. Notice that $\{\varphi_n\}_{n\in\NN}$ and $\{u_n\}_{n\in\NN}$ are taken in a consistent way so that whenever they have limits $\varphi$ and $u$ in some sense, then $\varphi$ is a global first integral of $u$ and such vector field verifies the Beltrami equation (\ref{campobeltramigen.domgeneral.defi}) with $f=\lambda+\varphi$. 

Our approach will be based again on the analysis of \textit{stationary transport equations} along stream tubes and a sequence of inhomogeneous problems of div-curl type that we will call \textit{inhomogeneous Beltrami equations} and which are intimately linked to the Helmhotz equation. In fact, we will start with the complex-valued fundamental solution of the Helmholtz equation in $\RR^3$
$$\Gamma_\lambda(x)=\frac{e^{i\lambda \vert x\vert}}{4\pi\vert x\vert},\ x\in \RR^3\setminus\{0\},$$
and will arrive at a representation formula of Helmholtz--Hodge type for its complex-valued solutions. Then, it is necessary to specify the optimal decay and radiation conditions that allow dealing with generalized volume and single layer potentials, namely,
\begin{align}
\int_{\partial B_R(0)}\vert u(x)\vert\,d_xS & =o(R^2),\hspace{0.37cm}R\rightarrow +\infty,\label{CondCaidaBeltrami.Intro.form}\\
\int_{\partial B_R(0)} \left\vert i \frac{x}{R}\times u(x)- u(x)\right\vert\,d_xS &=o(R),\hspace{0.5cm}R\rightarrow +\infty.\label{CondicionRadiacion.Beltrami.L1SilverMullerBeltrami.Intro.form}
\end{align}
Here, (\ref{CondCaidaBeltrami.Intro.form}) is nothing but a weak decay condition of the velocity field $u$ in $L^1$ and (\ref{CondicionRadiacion.Beltrami.L1SilverMullerBeltrami.Intro.form}) will be called the $L^1$ \textit{Silver--M\"{u}ller--Beltrami radiation condition} ($L^1$ SMB)  and will be deduced from both the classical Sommerfeld and Silver--M\"{u}ller radiation conditions, whose connections with the Helmholtz equation and the Maxwell system are classical. 

Summing up, we will be interested in analyzing the existence and uniqueness of complex-valued smooth solution with high order H\"{o}lder-type regularity  of the general \textit{Neumann boundary value problem for the inhomogeneous Beltrami equation} (NIB)
\begin{equation}\label{BeltramiNoHomog.Intro.eq}
\left\{\begin{array}{l}
\curl u-\lambda u=w, \hspace{0.5cm} x\in\Omega,\\
u\cdot\eta=g, \hspace{1.6cm} x\in \Omega,\\
+\ L^1\mbox{ decay property }(\ref{CondCaidaBeltrami.Intro.form}),\\
+\ L^1\mbox{ SMB radiation condition }(\ref{CondicionRadiacion.Beltrami.L1SilverMullerBeltrami.Intro.form}).
\end{array}\right.
\end{equation}
Notice that although we were originally interested in real-valued Beltrami fields, we will be concerned with complex-valued solutions to (\ref{BeltramiNoHomog.Intro.eq}) and we will then take real parts to obtain the real-valued ones. The reason to do it is twofold. Firstly, this will allow us to employ a representation formula for complex-valued \textit{radiating} fields. Secondly, this presents no problems related to the application to knotted structures as one can realize the fields in~\cite{Enciso2,Enciso3} as the real parts of complex-valued radiating Beltrami fields. Problem (\ref{BeltramiNoHomog.Intro.eq}) was previously studied in \cite{Kress}, who proved $C^1$ regularity results in bounded domains. We introduce some potential theory estimates of high order for generalized potentials associated with inhomogeneous kernels in exterior domains and adapt the boundary integral method  to the unbounded setting. We will also improve regularity from $C^1$ to $C^{k+1,\alpha}$. 

Consequently, we will rely on the complex-valued counterpart of the  modified Grad--Rubin method:
\begin{equation}\label{paso.limite.esquemaiterativo.Intro.form}
\left\{\begin{array}{ll}
\nabla \varphi_n\cdot u_n=0, & x\in \Omega,\\
\varphi_n=\varphi^0, & x\in \Sigma,
\end{array}\right.\hspace{1cm}
\left\{\begin{array}{l}
\curl v_{n+1}-\lambda v_{n+1}=\varphi_n u_n, \hspace{0.5cm} x\in \Omega,\\
v_{n+1}\cdot \eta= u_0\cdot \eta, \hspace{2.08cm} x\in S,\\
+\ L^1\mbox{ Decay property }(\ref{CondCaidaBeltrami.Intro.form}),\\
+\ L^1\mbox{ SBM radiation condition }(\ref{CondicionRadiacion.Beltrami.L1SilverMullerBeltrami.Intro.form}),
\end{array}\right.
\end{equation}
where $u_n=\Re v_n$ are the real parts of the complex-valued solutions $v_n$. The compactness of $\{u_n\}_{n\in \NN}$ in $C^{k+1,\alpha}(\overline{\Omega},\RR^3)$ follows from some Schauder estimates of Equation (\ref{BeltramiNoHomog.Intro.eq}) in H\"{o}lder spaces. Similarly, $\{\varphi_n\}_{n\in \NN}$ will be shown to be compact in $C^{k,\alpha}(\overline{\Omega})$ too. Concerning the application to solutions $u_0$ with knotted vortex structures of the kind constructed in~\cite{Enciso2,Enciso3}, we will see that the solution $u$ inherits the knotted vortex structures from $u_0$ (up to a small deformation) by virtue of structural stability. This is a straightforward consequence of the fact that $u$ can be chosen close to $u_0$ as long as the prescribed value $\varphi^0$ is small enough.

The paper is organized as follows. Section \ref{Esquema.Iterativo.Seccion} is devoted to study the iterative scheme (\ref{paso.limite.esquemaiterativo.Intro.form}). First, we analize the linear transport equations in the right hand side and the convergence of the iterative scheme will then follow from the analysis of NIB (\ref{BeltramiNoHomog.Intro.eq}). Such problem will be studied in Section \ref{Beltrami.NoHomogenea.Seccion} by extending the results in \cite{Kress,Neudert,vonWahl}. By comparison with the \textit{vector-valued divergence-free Helmhotz equation}, the \textit{reduced Maxwell system} and the Beltrami equation, we will deduce the appopriate radiation and decay conditions. The SBM radiation condition (\ref{CondicionRadiacion.Beltrami.L1SilverMullerBeltrami.Intro.form}) will then be connected with the classical Silver--M\"{u}ller and Sommerfeld radiation conditions and we will then present a representation formula of Helmholtz--Hodge type which involves these radiation conditions and that will be extremely useful to obtain our existence, uniqueness and regularity results. In Section \ref{Estructuras.Anudadas.Seccion} we combine the above results to construct small perturbations of the constant proportionality factor $\lambda$ leading to nontrivial generalized Beltrami fields that exhibit the same kind of knots and links and so to construct stationary solutions to the Euler equations. In order to support the above regularity results, Section \ref{Teoria.Potencial.Tecnicas.Seccion} will focus on obtaining  H\"{o}lder estimates of high order for volume and single layer potentials associated with the kernel $\Gamma_\lambda(x)$. The underlying ideas can be adapted to many other general inhomogeneous kernels with a controlled decay at infinity. The local partial stability result for generalized Beltrami fields will be discussed in Section~\ref{Ch.local}. Finally, Appendix \ref{Appendix.A} summarizes some geometric results that will be used throughout the paper and Appendix \ref{Appendix.B} recalls, for the benefit of the reader, the results on the generic non-existence of generalized Beltrami fields proved in~\cite{Enciso1}.\bigskip

\subsubsection*{Notation} 

Let us conclude this Introduction by summing up some notation that will be used throughout the paper without further notice. The notation regarding the domains can be stated as follows:
\begin{equation}\label{GSigmaMu.hipot}
\left\{\hspace{-0.6cm}\text{\parbox{0.90\textwidth}{\begin{itemize}
\item $G$ is a $C^{k+5}$ bounded domain homeomorphic to an Euclidean ball and containing the origin, i.e., $0\in G$.
\item $\Omega:=\RR^3 \backslash \overline G$ is its exterior domain and $S:=\partial \Omega=\partial G$ is the boundary surface.
\item $\eta$ denotes the outward unit normal vector field of $S$.
\end{itemize}}}\right.
\end{equation}
Although most of our results hold under weaker assumption on the boundary regularity (specifically $C^{k+1,\alpha}$ boundaries), there are certain results concerning a singular boundary integral equations which need $S$ to be at least $C^{k+5}$ because higher order derivatives of the normal vector field $\eta$ are involved (see for instance  Theorem \ref{Potencial.capasimple.regularidad.frontera.teo}).

Concerning functional spaces, we will essentially use the same notation as in~\cite{Gilbarg}. Let us agree to say that $C^k(\Omega)$ is the space of functions of class $C^k$ on~$\Omega$ with finite $C^k$~norm (meaning that all their derivatives up to order~$k$ are bounded). We will replace $\Omega$ by $\overline{\Omega}$ when the function and all its derivatives up to order~$k$ can be continuously extended to the closure of $\Omega$. The space $C^{k,\alpha}(\Omega)$ is the inhomogeneous H\"older space with exponent $\alpha\in(0,1)$ and $k$-th order regularity. We will use similar notation $C^k(S)$, $C^{k,\alpha}(S)$ for functions defined on~$S$. Vector-valued analogues of these spaces are denoted in the usual fashion, e.g.\ $C^{k,\alpha}(\Omega,\RR^3)$.

\section{Neumann problem for the inhomogeneous Beltrami equation and radiation conditions}\label{Beltrami.NoHomogenea.Seccion}
In this section we analyze the existence and uniqueness of solutions in $C^{k+1,\alpha}$ of the NIB problems (\ref{BeltramiNoHomog.Intro.eq}) arising in the modified Grad--Rubin iterative method (\ref{paso.limite.esquemaiterativo.Intro.form}). The key tool is a representation formula of Helmholtz--Hodge type for its solutions, which we will combine with the well-posedness of the underlying boundary integral equation for the tangential components in the space of $C^{k+1,\alpha}$ tangent vector fields to the boundary. For this we will need to improve some regularity results for high order derivatives of generalized volume and single layer potentials arising in the classical potential theory, which will require some potential-theoretic estimates for inhomogeneous singular integral kernels that are relegated to Section \ref{Teoria.Potencial.Tecnicas.Seccion} for simplicity of exposition. Regarding the representation formula, we will introduce and discuss in detail the weakest decay and radiation conditions under which this formula holds (namely, (\ref{CondCaidaBeltrami.Intro.form}) and (\ref{CondicionRadiacion.Beltrami.L1SilverMullerBeltrami.Intro.form})), as this topic is of independent interest.  Notice that many other radiation conditions have been used in the literature for related models: the natural one for the scalar complex-valued Helmholtz equation is the \textit{Sommerfeld radiation condition} and those of the reduced Maxwell system are called the \textit{Silver--M\"{u}ller radiation conditions} (SM) (see e.g.\ \cite{ColtonKress,ColtonKress2,Nedelec,Wilcox}).

Let us first recall some previous results in the literature on the exterior NIB boundary value problem (\ref{BeltramiNoHomog.Intro.eq}). Although the same problem is studied in \cite{Kress} for bounded domains and $C^1$ vector fields by means of a related approach \cite{Kress} (which also establishes a Helmholtz--Hodge like representation formula for such fields and employs boundary integral equations), the technique that we present in this section has not been studied in the case of exterior domains and $C^{k,\alpha}$-regularity.  We recall that in \cite{Kress} it was essential to assume that $\lambda$ is ``regular'' with respect to the interior problem. This is the case when $\lambda$ is not a Dirichlet eigenvalue of the Laplacian in the interior domain, or if it is a simple eigenvalue whose eigenfunction has non-zero mean, so this condition holds generically (as it can be seen e.g.\ by considering arbitrarily small rescalings of the domain).

Related results for exterior domains are proved in \cite{Neudert}. Indeed, the technique used in bounded domain by~\cite{vonWahl} and~\cite{Kress} (for $\lambda=0$ and $\lambda\neq 0$, respectively) goes through to the case of $\lambda=0$ and exterior domains via sharp estimates of harmonic volume and single layer potentials in $C^{1,\alpha}$. Roughly speaking, the main technical difference that we will encounter here is that in our case $\lambda$ is a nonzero constant, which leads to inhomogeneous kernels where the above estimates in unbounded domains are much harder to obtain. In fact, while these estimates are standard for $\lambda=0$, only estimates for the first order derivatives have been derived in the case of nonzero $\lambda$ (see \cite{ColtonKress}). In fact, \cite{Kress} only considers $C^1$ estimates even for the (easier) interior problem.

There is some literature regarding Laplace's equation in less regular settings (e.g. $L^p$ data and Lipschitz domains). For $C^1$ domains, \cite{Dahlberg1,Dahlberg2} solved it via the analysis of harmonic measures and \cite{FJR} introduced a method of layer potentials. The latter looks like the method that we propose and is supported by Fredholm's theory: some boundary singular integral operator is shown to be compact and one to one in the $C^1$ setting, leading to biyectivity and an useful lower estimate that entails the well posedness. For purely Lipschitz domains, compactness does no longer hold \cite{FJL} whilst biyectivity is preserved \cite{DahlbergKenig}. Regarding non-symmetric elliptic operators $L=-\divop A(x)\nabla$ in the half-space $(x,t)\in\RR^n\times \RR^+$, the well posedness of the Dirichlet problem with $L^p$ data \cite{HKMP} follows from the method of ``$\varepsilon$-approximability'' and the absolute continuity of the $L$-harmonic measure with respect to the surface measure.

This section is organized as follows. In the first part, we analyze the representation formula, the radiation conditions and some existence and uniqueness results for the scalar complex-valued Helmholtz equation. We will introduce there some classical notation and powerful tools like the \textit{far field pattern} of a \textit{radiating} solution not only in the homogeneous setting but also in the inhomogeneous one. In the second part we move to the Beltrami problem and try to carry out the same program as with Helmholtz equation. We will introduce the natural SM radiation conditions of the reduced Maxwell system and will link them with the natural radiation conditions both for the inhomogeneous Beltrami equation and an intimately related model: the divergence-free Helmholtz equation. Then, we prove the aforementioned representation formula and our existence and uniqueness results, which follow from the generalized potential theory estimates in Section \ref{Teoria.Potencial.Tecnicas.Seccion}
 along with the analysis of the well-posedness for the boundary integral equation for the tangential components. This will be studied in the last paragaph of this section.

\subsection{Inhomogeneous Helmholtz equation in the exterior domain}
The Helmholtz equation with wave number $\lambda\in\RR$ in the exterior domain $\Omega$ stands for the elliptic PDE
$$\Delta a+\lambda^2 a=0,\ \ x\in \Omega,$$
where the unknown is  a possibly complex-valued scalar function $a\in C^2(\Omega,\CC)$. This equation arises in acoustic and electromagnetic mathematics \cite{ColtonKress2,Nedelec} and in the study of high energy eigenvalue asymptotics. 
The Helmholtz equation also appears in the study of Beltrami fields arising either from the incompressible Euler equation or from the force-free field system of magnetohydrodynamics
$$\curl u=\lambda u,\ \ x\in \Omega.$$
Taking $\curl$ for $\lambda\neq 0$ one arrives at the following vector-valued equation
$$\nabla(\divop u)-\Delta u=\lambda^2u,\ \ x\in\Omega.$$
Since Beltrami fields are divergence-free, then one recovers the vector-valued Helmholtz equation in the domain $\Omega$. This relation with the Beltrami equation suggests to study the representation formulas, radiation conditions and uniqueness lemmas for the Helmholtz equation in the literature.

First of all, let us define the next hierarchy of radiation conditions for a complex-valued scalar function $a\in C^1(\Omega,\CC)$.
\begin{defi}\label{CondicionesRadiacion.defi}
$\,$
\begin{enumerate}
\item $L^1$ Sommerfeld radiation condition
\begin{equation}\label{CondicionRadiacion.Helmholtz.L1Sommerfeld.form}
\int_{\partial B_R(0)} \left\vert\nabla a(y)\cdot\frac{y}{R}-i\lambda a(y)\right\vert\,d_yS=o\left(R\right),\ \ R\rightarrow +\infty.
\end{equation}
\item $L^2$ Sommerfeld radiation condition
\begin{equation}\label{CondicionRadiacion.Helmholtz.L2Sommerfeld.form}
\int_{\partial B_R(0)} \left\vert\nabla a(y)\cdot\frac{y}{R}-i\lambda a(y)\right\vert^2\,d_yS=o(1),\ \ R\rightarrow +\infty.
\end{equation}
\item ($L^\infty$) Sommerfeld radiation condition
\begin{equation}\label{CondicionRadiacion.Helmholtz.Sommerfeld.form}
\sup_{y\in \partial B_R(0)}\left\vert\nabla a(y)\cdot \frac{y}{R}-i\lambda a(y)\right\vert=o\left(\frac{1}{R}\right),\ \ R\rightarrow +\infty.
\end{equation}
\end{enumerate}
\end{defi}

The following chain of implications is obvious:
$$L^\infty\mbox{ Sommerfeld }\Longrightarrow L^2\mbox{ Sommerfeld }\Longrightarrow L^1\mbox{ Sommerfeld}.$$
Originally, only the strongest one (\ref{CondicionRadiacion.Helmholtz.Sommerfeld.form}) was considered. However, several authors \cite{ColtonKress2,Nedelec} came to the conclusion that a weaker radiation condition (\ref{CondicionRadiacion.Helmholtz.L2Sommerfeld.form}) may be assumed to obtain representation formulas and certain uniqueness results.  Although we follow the same approach, we weaken the radiation condition to an even weaker one (\ref{CondicionRadiacion.Helmholtz.L1Sommerfeld.form}) by assuming some kind of decay at infinity that will be much weaker than $\vert x\vert^{-1}$ though. As it will be shown later, both decay and radiation conditions can be recovered from the $L^2$ Sommerfeld radiation condition for solution to the Helmholtz equation. Before showing that this radiation conditions leads to a representation formula of Stokes type, let us analyze them in the case of the fundamental solution to the $3$-D Helmholtz equation, 
\begin{equation}\label{SolucionFundamental.form}
\Gamma_\lambda(x)=\frac{e^{i\lambda \vert x\vert}}{4\pi\vert x\vert}=\left(\frac{\cos(\lambda\vert x\vert)}{4\pi\vert x\vert}+i\frac{\sin(\lambda \vert x\vert)}{4\pi\vert x\vert}\right).
\end{equation}
Since
\begin{equation}\label{DerivadaSolucionFundamental.form}
\nabla\Gamma_\lambda(x)=\left(i\lambda-\frac{1}{\vert x\vert}\right)\Gamma_\lambda(x)\frac{x}{\vert x\vert},
\end{equation}
a straightforward inductive argument shows that all the partial derivatives of $\Gamma_\lambda(x)$ up to second order verify an even stronger version of the Sommerfeld radiation condition (\ref{CondicionRadiacion.Helmholtz.Sommerfeld.form}). Hence we easily infer:

\begin{pro}\label{CondicionSommerfeld.SolucionFundamentalHelmholtz.pro}
The fundamental solution of the Helmholtz equation, together with its partial derivatives up to order $2$ satisfy the identities
\begin{align*}
\nabla \Gamma_\lambda(x)\cdot\frac{x}{\vert x\vert}-i\lambda \Gamma_\lambda(x)&=-\frac{\Gamma_\lambda(x)}{\vert x\vert},\\
\nabla \left(\frac{\partial \Gamma_\lambda}{\partial x_i}\right)(x)\cdot \frac{x}{\vert x\vert}-i\lambda \frac{\partial \Gamma_\lambda}{\partial x_i}(x)&=\left(\frac{2}{\vert x\vert}-i\lambda\right)\Gamma_\lambda(x)\frac{x_i}{\vert x\vert^2},\\
\nabla\left(\frac{\partial^2 \Gamma_\lambda}{\partial x_i\partial x_j}\right)(x)\cdot \frac{x}{\vert x\vert}-i\lambda \frac{\partial^2 \Gamma_\lambda}{\partial x_i\partial x_j}(x)&=-\nabla\left(\frac{\partial \Gamma_\lambda}{\partial x_i}\right)(x)\cdot \frac{\partial}{\partial x_j}\left(\frac{x}{\vert x\vert}\right)+\frac{\partial}{\partial x_j}\left(\left(\frac{2}{\vert x\vert}-i\lambda\right)\Gamma_\lambda(x)\frac{x_i}{\vert x\vert^2}\right),
\end{align*}
for every $i,j\in\{1,2,3\}$. Consequently,
$$\sup_{x\in\partial B_R(0)}\left\vert\nabla(D^\gamma \Gamma_\lambda)(x)\cdot \frac{x}{R}-i\lambda D^\gamma\Gamma_\lambda(x)\right\vert=O\left(\frac{1}{R^2}\right),\mbox{ for } R\rightarrow +\infty,$$
for every multi-index with $\vert \gamma\vert\leq 2$.
\end{pro}

In particular, $\Gamma_\lambda(x)$ together with its partial derivatives up to order two verify the Sommerfeld radiation condition (\ref{CondicionRadiacion.Helmholtz.Sommerfeld.form}). It is then an easy task to obtain new complex-valued solutions to the homogeneous Helmholtz equation enjoying such radiation condition through the definition of the generalized volume and single layer potentials associated with the kernel $\Gamma_\lambda(x)$.

\begin{pro}\label{PotencialesCapaSimple.HelmholtzRadiantes.pro}
Let $a$ be the generalized single layer potential with density $\zeta\in C(S)$ associated with the Helmholtz equation, i.e.,
$$a(x):=(\mathcal{S}_\lambda\zeta)(x)=\int_S \Gamma_\lambda(x-y)\zeta(y)\,d_yS,$$
for every $x\in \Omega$. Then, $a$ and all its partial derivatives up to second order are solutions to the Helmholtz equation which verify the Sommerfeld radiation condition (\ref{CondicionRadiacion.Helmholtz.Sommerfeld.form}).
\end{pro}

\begin{proof}
Taking derivatives under the integral sign, one checks that $a$ solves the complex-valued Helmholtz equation in $\Omega$. In order to check Sommerfeld radiation condition (\ref{CondicionRadiacion.Helmholtz.Sommerfeld.form}), let us use the preceding properties in Proposition \ref{CondicionSommerfeld.SolucionFundamentalHelmholtz.pro}. Fixing  $z\in\RR^3$ and taking derivatives under the integral sign, we have
\begin{align*}
\nabla a(x)\cdot \frac{x-z}{\vert x-z\vert}-i\lambda a(x)&= \int_S \left(\nabla_x \Gamma_\lambda(x-y)\frac{x-z}{\vert x-z\vert}-i\lambda \Gamma_\lambda(x-y)\right)\zeta(y)\,d_yS\\
&=\int_S \left(\nabla_x \Gamma_\lambda(x-y)\frac{x-y}{\vert x-y\vert}-i\lambda \Gamma_\lambda(x-y)\right)\zeta(y)\,d_yS\\
&\qquad \qquad \qquad \qquad \qquad+\int_S \nabla_x \Gamma_\lambda(x-y)\left(\frac{x-z}{\vert x-z\vert}-\frac{x-y}{\vert x-y\vert}\right)\zeta(y)\,d_yS.
\end{align*}
Multiplying the first term by $\vert x-z\vert$ and assuming that $\vert x-z\vert$ is big enough (Proposition \ref{CondicionSommerfeld.SolucionFundamentalHelmholtz.pro}), we find
$$
\vert x-z\vert\left\vert\int_S \left(\nabla_x \Gamma_\lambda(x-y)\frac{x-y}{\vert x-y\vert}-i\lambda \Gamma_\lambda(x-y)\right)\zeta(y)\,d_yS\right\vert\leq \frac{\vert x-z\vert}{(\vert x-z\vert-d)^2}\Vert \zeta\Vert_{L^1(S)},
$$
where $d$ stands for $\max\{\vert y-z\vert:\,y\in S\}$. Therefore, this term  vanishes  for $\vert x-z\vert\rightarrow +\infty$. Regarding the second term, it is easily checked that $\nabla \Gamma_\lambda(x)=O(\vert x\vert^{-1})$, when $\vert x\vert\rightarrow +\infty$. To conclude the proof of this result, let us obtain some extra decay from the difference in the middle, which can be upper bounded through the next straightforward reasonings involving the mean value theorem
\begin{align}
\left\vert\frac{x_i-z_i}{\vert x-z\vert}\right.&-\left.\frac{x_i-y_i}{\vert x-y\vert}\right\vert
=\left\vert\int_0^1\frac{d}{d\theta}\frac{x_i-(\theta y_i+(1-\theta)z_i)}{\vert x-(\theta y+(1-\theta)z)\vert}\,d\theta\right\vert\nonumber\\
&=\left\vert\int_0^1\frac{\vert x-(\theta y+(1-\theta)z)\vert (z_i-y_i)-(x_i-(\theta y_i+(1-\theta)z_i))\frac{x-(\theta y+(1-\theta)z)}{\vert x-(\theta y+(1-\theta)z)\vert}\cdot (z-y)}{\vert x-(\theta y+(1-\theta)z)\vert^2}\,d\theta\right\vert\nonumber\\
&\leq 2\vert y-z\vert\int_0^1\frac{1}{\vert x-(\theta y+(1-\theta)z)\vert}\,d\theta \leq \frac{2d}{\vert x-z\vert-d}.\label{PotencialesCapaSimple.HelmholtzRadiantes.ValorMedio.ineq}
\end{align}
Therefore,
$$\vert x-z\vert\left\vert\int_S \nabla_x \Gamma_\lambda(x-y)\left(\frac{x-z}{\vert x-z\vert}-\frac{x-y}{\vert x-y\vert}\right)\zeta(y)\,d_yS\right\vert\leq C\frac{2d\vert x-z\vert}{(\vert x-z\vert -d)^2}\Vert \zeta\Vert_{L^1(S)},$$
whose limit also vanishes as $\vert x-z\vert\rightarrow +\infty$. Consequently, $a$ verifies the Sommerfeld radiation condition centered at any $z\in\RR^3$. In particular, the above assertion also holds  for $z=0$. A similar reasoning with the partial derivatives of $a$ up to second order also holds according to Proposition \ref{CondicionSommerfeld.SolucionFundamentalHelmholtz.pro}.
\end{proof}

The same result remains true for generalized volume potential with compactly supported densities. In this case, radiating solutions for the inhomogeneous complex-valued Helmholtz equation can be obtained. The proof is identical, with the only distinction that we must change the constant $d$ in the lower bounds of the denominators from $d=\max\{\vert y-z\vert:\,y\in S\}$ to $d=\max\{\vert y-z\vert:\,y\in \supp \zeta\}$.

\begin{pro}\label{PotencialesVolumen.HelmholtzRadiantes.pro}
Let $a$ be the generalized volume potential with density $\zeta\in C_c(\overline{\Omega})$ associated with the Helmholtz equation, i.e.,
$$a(x):=(\mathcal{N}_\lambda\zeta)(x)=\int_\Omega \Gamma_\lambda(x-y)\zeta(y)\,d_yS,$$
for every $x\in \Omega$. Then, $a$ solves the inhomogeneous Helmholtz equation
$$-(\Delta a+\lambda^2 a)=-\zeta,$$
in the exterior domain $\Omega$. Moreover, $a$ and all its partial derivatives up to second order verify the Sommerfeld radiation condition (\ref{CondicionRadiacion.Helmholtz.Sommerfeld.form}).
\end{pro} 

To establish the representation formula for the inhomogeneous Helmholtz equation,  we study the radiation conditions for the volume and single layer potentials, as well as its decay properties at infinity. We will need the \textit{Hardly--Littlewood--Sobolev estimates of fractional integrals} \cite[Theorem 1.2.1]{Stein}, which we state not in terms of $L^p$~integrability conditions but in terms of pointwise decay at infinity. For the convenience of the reader, we include a simple derivation of this form of the estimates:

\begin{theo}\label{DecPotencialRiesz}
Consider any dimension $N$ and exponent $0<\alpha<N$. Define the associated Riesz potential by
$$R_\alpha(x):=\frac{1}{\vert x\vert^\alpha},\ x\in \RR^N.$$
For any measurable function $f:\RR^N\longrightarrow\RR$, we have that
\begin{enumerate}
\item the decay property 
$$\vert (R_\alpha*f)(x)\vert\leq C\frac{\Vert \vert x\vert^{\rho}f\Vert_{L^\infty(\RR^N)}}{\vert x\vert^{\alpha-(N-\rho)}},$$
holds for every $x\in\RR^N$ as long as $f=O(\vert x\vert^{-\rho})$ for $\vert x\vert\rightarrow +\infty$ and $\rho$ is any nonnegative exponent such that 
$$N-\alpha<\rho<N.$$ 
Here, $C$ stands for a positive constant that depends on $N$, $\alpha$ and $\rho$ but do not depend on $f$.
\item the optimal decay $\vert x\vert^{-\alpha}$ is obtained in the compactly supported case, i.e.,
$$\vert (R_\alpha*f)(x)\Vert\leq C\frac{\Vert f\Vert_{L^\infty(\RR^N)}}{\vert x\vert^\alpha},$$
for every $x\in\RR^N$, as long as $f\in L^\infty(\RR^N)$ has compact support inside some ball $B_{R_0}(0)$. Now, not only does $C$ depend on $N$ and $\alpha$ but also on the size $R_0>0$ of the support. 
\end{enumerate}
\end{theo}
\begin{proof}
Let us begin with the first item. Fix any constant $0<R<1$ (e.g., $R=1/2$) and split the integral we are interested in into the next two parts
$$\int_{\RR^3}\frac{1}{\vert x-y\vert^\alpha}f(y)\,dy=I_1+I_2,$$
where
$$I_1=\int_{B_{R\vert x\vert}(0)}\frac{1}{\vert x-y\vert^\alpha}f(y)\,dy,\hspace{1cm}I_2=\int_{B_{R\vert x\vert}^c}\frac{1}{\vert x-y\vert^\alpha}f(y)\,dy.$$
In order to estimate $I_1$, notice that
$$y\in B_{R\vert x\vert}(0)\Longrightarrow \vert x-y\vert\geq (1-R)\vert x\vert.$$
Therefore, $I_1$ is bounded by
\begin{align*}
\int_{B_{R\vert x\vert}(0)}\frac{1}{\vert x-y\vert^\alpha}\vert f(y)\vert\,dy&\leq \frac{K}{(1-R)^\alpha}\frac{1}{\vert x\vert^\alpha}\int_{B_{R\vert x\vert}(0)}\frac{1}{\vert y\vert^\rho}\,dy\\
&=\frac{K\omega_N}{(1-R)^\alpha}\frac{1}{\vert x\vert^\alpha}\int_0^{R\vert x\vert}r^{N-1}\frac{1}{r^\rho}\,dr =\frac{K\omega_N}{N-\rho}\frac{R^{N-\rho}}{(1-R)^\alpha}\frac{1}{\vert x\vert^{\alpha-(N-\rho)}}.
\end{align*}
Here $K:=\Vert \vert x\vert^{\rho} f\Vert_{L^\infty(\RR^3)}$ and $\omega_N$ stands for the $(N-1)$-dimensional area of the unit sphere in $\RR^N$. It is worth remarking that we are dealing with finite integrals as a consequence of the hypothesis $\rho<N$. Similarly, the second integral, $I_2$, can also be split as follows
\begin{align*}
\int_{B_{R\vert x\vert}(0)^c}&\frac{1}{\vert x-y\vert^\alpha}\vert f(y)\vert\,dy\\
&=\int_{B_{R\vert x\vert}(x)\setminus B_{R\vert x\vert}(0)}\frac{1}{\vert x-y\vert^\alpha}\vert f(y)\vert\,dy+\int_{(B_{R\vert x\vert}(0)\cup B_{R\vert x\vert}(x))^c}\frac{1}{\vert x-y\vert^\alpha}\vert f(y)\vert\,dy.
\end{align*}
An analogous argument can be used to obtain the next upper bound of the first term
\begin{align*}
\int_{B_{R\vert x\vert}(x)\setminus B_{R\vert x\vert}(0)}\frac{1}{\vert x-y\vert^\alpha}\vert f(y)\vert\,dy&\leq \int_{B_{R\vert x\vert}(x)}\frac{1}{\vert x-y\vert^\alpha}\vert f(y)\vert\,dy \leq K\int_{B_{R \vert x\vert}(x)}\frac{1}{\vert x-y\vert^\alpha}\frac{1}{\vert y\vert^\rho}\,dy\\
&=K \int_{B_{R\vert x\vert}(0)}\frac{1}{\vert x-y\vert^\rho}\frac{1}{\vert y\vert^\alpha}\,dy =\frac{K\omega_N}{N-\alpha}\frac{R^{N-\alpha}}{(1-R)^\rho}\frac{1}{\vert x\vert^{\alpha-(N-\rho)}}.
\end{align*}
This time, finite integrals are involved due to the hypothesis $N-\alpha<\rho$. Regarding the second term, let us decompose the integral into two parts once more. The appropriate subdomains to be considered are
\begin{align*}
A&=\{y\in (B_{R\vert x\vert}(0)\cup B_{R\vert x\vert}(x))^c:\,\vert x-y\vert \leq \vert y\vert\},\\
B&=\{y\in (B_{R\vert x\vert}(0)\cup B_{R\vert x\vert}(x))^c:\,\vert x-y\vert > \vert y\vert\}.
\end{align*}
Let us complete the proof of the first inequality with the following estimates for the integrals over $A$ and $B$, which follow from the same reasoning involving the hypothesis $N-\alpha<\rho$:
\begin{align*}
\int_A\frac{1}{\vert x-y\vert^\alpha}\vert f(y)\vert\,dy&\leq K\int_A \frac{1}{\vert x-y\vert^\alpha}\frac{1}{\vert y\vert^\rho}\,dy\leq K\int_A \frac{1}{\vert x-y\vert^{\alpha+\rho}}\,dy \leq K\int_{B_{R\vert x\vert}(x)^c}\frac{1}{\vert x-y\vert^{\alpha+\rho}}\,dy\\
&=K\omega_N\int_{R\vert x\vert}^{+\infty}r^{N-1}\frac{1}{r^{\alpha+\rho}}\,dr =\frac{K\omega_N}{\alpha-(N-\rho)}\frac{1}{R^{\alpha-(N-\rho)}}\frac{1}{\vert x\vert^{\alpha-(N-\rho)}},\\
& \\
\int_B\frac{1}{\vert x-y\vert^\alpha}\vert f(y)\vert\,dy&\leq K\int_B \frac{1}{\vert x-y\vert^\alpha}\frac{1}{\vert y\vert^\rho}\,dy\leq K\int_B\frac{1}{\vert y\vert^{\alpha+\rho}}\,dy \leq K\int_{B_{R\vert x\vert}(0)^c}\frac{1}{\vert y\vert^{\alpha+\rho}}\,dy\\
&=K\omega_N\int_{R\vert x\vert}^{+\infty}r^{N-1}\frac{1}{r^{\alpha+\rho}}\,dr =\frac{K\omega_N}{\alpha-(N-\rho)}\frac{1}{R^{\alpha-(N-\rho)}}\frac{1}{\vert x\vert^{\alpha-(N-\rho)}}.
\end{align*}

Let us now pass to the second item. Let us start with $\vert x\vert>2R_0$, so that
$$\vert (R_\alpha*f)(x)\vert\leq \int_{B_{R_0}(0)}\frac{1}{\vert x-y\vert^\alpha}\vert f(y)\vert\,dy.$$
Notice that whenever $y\in B_{R_0}(0)$, then one has
$$\vert x-y\vert\geq \vert x\vert-\vert y\vert\geq \vert x\vert -R_0=\left(1-\frac{R_0}{\vert x\vert}\right)\vert x\vert\geq \frac{1}{2}\vert x\vert.$$
Therefore
$$\vert (R_\alpha*f)(x)\vert\leq \frac{2^\alpha}{\vert x\vert^\alpha}\Vert f\Vert_{L^1(\RR^3)}\leq 2^\alpha\vert B_{R_0}(0)\vert\frac{\Vert f\Vert_{L^\infty(\RR^N)}}{\vert x\vert^\alpha}.$$
The case $\vert x\vert\leq 2R_0$ is easier since
$$y\in B_{R_0}(0)\Longrightarrow \vert x-y\vert\leq\vert x\vert+\vert y\vert<3R_0,$$
and consenquently, Young inequality for the convolution of $L^p$ functions leads to
\begin{align*}
\vert (R_\alpha*f)(x)\vert&\leq \int_{B_{3R_0}(x)}\frac{1}{\vert x-y\vert^\alpha}\vert f(y)\vert\,dy =\int_{B_{3R_0}(0)}\vert f(x-y)\vert\frac{1}{\vert y\vert^\alpha}\,dy=\vert f\vert*\left(\chi_{B_{3R_0}(0)}R_\alpha\right)(x)\\
&\leq \Vert R_\alpha\Vert_{L^1(B_{3R_0}(0))}\Vert f\Vert_{L^\infty(\RR^N)}\leq (2R_0)^\alpha\Vert R_\alpha\Vert_{L^1(B_{3R_0}(0))}\frac{\Vert f\Vert_{L^\infty(\RR^N)}}{\vert x\vert^\alpha},
\end{align*}
where $1\leq \frac{2R_0}{\vert x\vert}$ has been used in the last inequality.
\qedhere
\end{proof}

The above results permit obtaining a Stokes-type formula to represent the solutions to the inhomogeneous Helmholtz equation. Now, we deal with the weakest radiation condition, namely, the $L^1$ Sommerfeld radiation condition and some property of weak decay at infinity in $L^1$. Since the proof is completely analogous to the more important result for complex-valued solutions of the inhomogeneous Beltrami equation that we present in the next subsection (Theorem~\ref{HelmholtzHodgeBeltrami.teo}), we will skip the proof. A detailed proof with the more restrictive $L^2$ Sommerfeld radiation condition (\ref{CondicionRadiacion.Helmholtz.L2Sommerfeld.form}) can be found in \cite[Theorem 2.4]{ColtonKress2} and \cite[Theorem 3.1.1]{Nedelec}. 

\begin{theo}\label{FormulaGreenHelmholtz.teo}
Let $a\in C^2(\Omega,\CC)\cap C^1(\overline{\Omega},\CC)$ be any function which verifies the $L^1$ Sommerferld radiation condition (\ref{CondicionRadiacion.Helmholtz.L1Sommerfeld.form}) and the following decay property at infinity
\begin{equation}\label{HipotesisFormulaGreenHelmholtzL1.form}
\int_{\partial B_R(0)}\vert a(y)\vert\,d_yS=o(R^2), \mbox{ when }R\rightarrow +\infty.
\end{equation}
Assume that $\Delta a+\lambda^2 a=O(\vert x\vert^{-\rho})$ when $\vert x\vert\rightarrow +\infty$, for some exponent $2<\rho<3$. Then,
\begin{align}\label{FormulaGreenHelmholtz.form}
a(x)=&-\int_{\Omega} \Gamma_\lambda(x-y)(\Delta a(y)+\lambda^2 a(y))\,dy\\
&+\int_{S}\frac{\partial \Gamma_\lambda(x-y)}{\partial \eta(y)}a(y)\,d_yS-\int_S\Gamma_\lambda(x-y)\frac{\partial a}{\partial \eta}(y)\,d_yS,\nonumber
\end{align}
for every $x\in\Omega$ and, as a consequence,
$$a=O(\vert x\vert^{-(\rho-2)}), \mbox{ when }\vert x\vert\rightarrow +\infty.$$
Indeed, when $\Delta a+\lambda^2 a$ has compact support, one obtains the optimal decay at infinity, namely, 
$$a=O(\vert x\vert^{-1}),  \mbox{ when }\vert x\vert\rightarrow +\infty.$$
\end{theo} 

The properties  follow from Theorem \ref{DecPotencialRiesz} and they may also be found in  \cite{ColtonKress2,Nedelec}. Notice that the decay rates $\vert x\vert^{-(\rho-2)}$ (for the inhomogeneous equation) and $\vert x\vert^{-1}$ (for the homogeneous one) are straighforward consequences of the representation formula. 

Let us now show the link between our $L^1$ version an the $L^2$ version (\cite{ColtonKress2,Nedelec}). To this end, we shall next see that  $(\ref{CondicionRadiacion.Helmholtz.L2Sommerfeld.form}) \Rightarrow(\ref{CondicionRadiacion.Helmholtz.L1Sommerfeld.form})+(\ref{HipotesisFormulaGreenHelmholtzL1.form})$ in the homogeneous case. Indeed, let $a\in C^2(\Omega,\CC)\cap C^1(\overline{\Omega},\CC)$ be any solution to the complex-valued homogeneous Helmholtz equation in the exterior domain fulfilling the $L^2$ Sommerfeld radiation condition (\ref{CondicionRadiacion.Helmholtz.L2Sommerfeld.form}). Computing the square in such radiation condition, we arrive at
\begin{align}\label{HipotesisRepresentacionL1SommerfeldL2.eq1.form}
\left\vert\nabla a(x)\cdot \frac{x}{R}-i\lambda a(x)\right\vert^2&=\left\vert\nabla a(x)\cdot \frac{x}{R}\right\vert^2+\lambda^2\left\vert a(x)\right\vert^2-2\Re\left(i\lambda a(x)\,\nabla \overline{a}(x)\cdot \frac{x}{R}\right)\nonumber\\
&=\left\vert\nabla a(x)\cdot \frac{x}{R}\right\vert^2+\lambda^2\left\vert a(x)\right\vert^2+2\lambda \Im\left(a(x)\,\nabla \overline{a}(x)\cdot \frac{x}{R}\right),
\end{align}
for any $x\in\partial B_R(0)$, where $\Re$ and $\Im$ mean the real and imaginary parts of the corresponding complex numbers. Consider any positive radius $R_0$ such that $\overline{G}\subseteq B_{R_0}(0)$ and define the subdomains $\Omega_R:=B_R(0)\setminus\overline{G}$, for each $R>R_0$. Therefore, the homogeneous Helmholtz equation and Green's formula lead to
$$-\lambda^2\int_{\Omega_R}\vert a(x)\vert^2\,dx=\int_{\Omega_R} a(x)\Delta\overline{a}(x)\,dx=\int_{\partial \Omega_R}a(x)\,\nabla \overline{a}(x)\cdot \nu(x)\,d_xS-\int_{\Omega_R}\nabla a(x)\cdot \nabla\overline{a}(x)\,dx.$$
Let us split the boundary integral into the boundary's connected components
\begin{gather*}
\int_{\partial B_R(0)}a(x)\,\nabla \overline{a}(x)\cdot \frac{x}{R}\,d_yS-\int_S a(x)\,\nabla \overline{a}(x)\cdot \eta(x)\,d_xS\\
=\int_{\Omega_R}\vert\nabla a(x)\vert^2\,dx-\lambda^2\int_{\Omega_R}\vert a(x)\vert^2\,dx.
\end{gather*}
and take imaginary parts in the preceding equation to arrive at
\begin{equation}\label{HipotesisRepresentacionL1SommerfeldL2.eq2.form}
\Im\left(\int_{\partial B_R(0)}a(x)\,\nabla \overline{a}(x)\cdot \frac{x}{R}\,d_xS\right)=\Im\left(\int_{S}a(x)\,\nabla \overline{a}(x)\cdot \eta(x)\,d_xS\right).
\end{equation}
Combining equations (\ref{HipotesisRepresentacionL1SommerfeldL2.eq1.form}) and (\ref{HipotesisRepresentacionL1SommerfeldL2.eq2.form}) along with the $L^2$ Sommerfeld radiation condition (\ref{CondicionRadiacion.Helmholtz.L2Sommerfeld.form}), one obtains
\begin{equation}\label{HipotesisRepresentacionL1SommerfeldL2.eq3.form}
\lim_{R\rightarrow +\infty}\int_{\partial B_R(0)}\left(\left\vert\nabla a(x)\cdot \frac{x}{R}\right\vert^2+\lambda^2\left\vert a(x)\right\vert^2\right)\,d_xS=-2\lambda \Im\left(\int_{S}a(x)\,\nabla \overline{a}(x)\cdot \eta(x)\,d_xS\right).
\end{equation}
Consequently,
\begin{equation}\label{HipotesisRepresentacionL1SommerfeldL2.eq4.form}
\int_{\partial B_R(0)} \vert a(x)\vert^2\,d_xS=O(1)\mbox{ when }R\rightarrow +\infty,
\end{equation}
and Cauchy-Schwarz inequality ensures that
$$\int_{\partial B_R(0)}\vert a(x)\vert\,d_xS=O(R),\ \mbox{ when }R\rightarrow+\infty.$$
In particular, the weak decay property (\ref{HipotesisFormulaGreenHelmholtzL1.form}) holds.

Notice that the $L^2$ version in \cite{ColtonKress2,Nedelec} of the representation formula for the homogeneous equation is a direct consequence of our $L^1$ version in Theorem \ref{FormulaGreenHelmholtz.teo} and the preceding discussion:

\begin{cor}\label{FormulaGreenHelmholtz.cor}
Let $a\in C^2(\Omega,\CC)\cap C^1(\overline{\Omega},\CC)$ be any solution to the complex-valued homogeneous Helmholtz equation in the exterior domain which verifies the $L^2$ Sommerferld radiation condition (\ref{CondicionRadiacion.Helmholtz.L2Sommerfeld.form}). Then,
\begin{equation*}
a(x)=\int_{S}\frac{\partial \Gamma_\lambda(x-y)}{\partial \eta(y)}a(y)\,d_yS-\int_S\Gamma_\lambda(x-y)\frac{\partial a}{\partial \eta}(y)\,d_yS,
\end{equation*}
for every $x\in\Omega$. As a consequence,
$$a=O(\vert x\vert^{-1}), \mbox{ when }\vert x\vert\rightarrow +\infty.$$
\end{cor}

An immediate consequence of the representation formulas in Theorem \ref{FormulaGreenHelmholtz.teo} and Corollary \ref{FormulaGreenHelmholtz.cor} is that a \textit{far field pattern} at infinity exists for each solution to the Helmholtz equation (see \cite{ColtonKress2} for details). The far field pattern of a solution to the Helmholtz equation is a very powerful tool since it provides a description of the asymptotic behavior at infinity. It gives, for instance, easy uniqueness criteria for radiating solutions. A related inverse problem has also been widely studied, as it is interesting to know whether a fixed function over the unit sphere is the far field pattern of some radiating solution to the Helmholtz equation. 

Although most of the literature is only devoted to far field patterns of complex-valued radiating solutions to the homogeneous Helmholtz equation, our problem clearly concerns the inhomogeneous setting. Thus, we revisit the theory of far field patterns and its relation to the general inhomogeneous Helmholtz equation in the particular case of compactly supported inhomogeneities (it would not be hard to extend it to more general inhomogeneous terms suitable decay at infinity). For this, consider any solution $a\in C^2(\Omega,\CC)\cap C^1(\overline{\Omega},\CC)$ to the inhomogeneous Helmholtz equation
$$-(\Delta a+\lambda^2 a)=f,\hspace{0.7cm}x\in \Omega,$$
where $f$ is compactly supported in $\overline{\Omega}$ and $a$ verifies both  the decay condition (\ref{HipotesisFormulaGreenHelmholtzL1.form}) and the $L^1$ Sommerfeld radiation condition (\ref{CondicionRadiacion.Helmholtz.L1Sommerfeld.form}). Then, Theorem \ref{FormulaGreenHelmholtz.teo} leads to
$$a(x)=\int_\Omega\Gamma_\lambda(x-y)f(y)\,dy+\int_S\frac{\partial \Gamma_\lambda(x-y)}{\partial \eta(y)}a(y)\,d_yS-\int_S\Gamma_\lambda(x-y)\frac{\partial a}{\partial \eta}(y)\,d_yS.$$
Consider the compact subset $K:=\mbox{supp}f$ and notice the asymptotic behavior
\begin{align*}
\Gamma_\lambda(x-y)&=\Gamma_\lambda(x)\left\{e^{-i\lambda\frac{x}{\vert x\vert}\cdot y}+O\left(\frac{1}{\vert x\vert}\right)\right\}, \hspace{0.7cm}\mbox{ when }\vert x\vert\rightarrow +\infty,\\
\frac{\partial \Gamma_\lambda(x-y)}{\partial \eta(y)}&=\Gamma_\lambda(x)\left\{\frac{\partial e^{-i\lambda\frac{x}{\vert x\vert}\cdot y}}{\partial\eta(y)}+O\left(\frac{1}{\vert x\vert}\right)\right\}, \hspace{0.4cm}\mbox{ when }\vert x\vert\rightarrow +\infty, 
\end{align*}
where $O\left(\vert x\vert^{-1}\right)$ is uniform in $y\in K\cup S$ in the first formula and uniform in $y\in S$ in the second one. From here we deduce  the asymptotic behavior 
\begin{equation}\label{FarFieldPattern.decomposicion.form}
a(x)=\Gamma_\lambda(x)\left\{a_\infty\left(\frac{x}{\vert x\vert}\right)+O\left(\frac{1}{\vert x\vert}\right)\right\}\mbox{ when }\vert x\vert\rightarrow +\infty,
\end{equation}
where $a_\infty$ is called the \textit{far field pattern} of $a$, and reads as 
$$
a_\infty(\sigma)=\int_\Omega e^{-i\lambda\sigma\cdot y}f(y)\,dy+\int_S \frac{\partial e^{-i\lambda\sigma\cdot y}}{\partial \eta(y)}a(y)\,d_yS-\int_S e^{-i\lambda \sigma\cdot y}\frac{\partial a}{\partial\eta}(y)\,d_yS,
$$
for each point $\sigma\in \partial B_1(0)$. 

It is apparent that $a_\infty$ is uniquely determined from formula (\ref{FarFieldPattern.decomposicion.form}). Hence, we can define the following well-defined linear and one to one map
\begin{equation}\label{FarFieldPattern.aplicacion.form}
\begin{array}{ccc}
 \mathcal{D}_{\infty} & \longrightarrow & C^\infty(\partial B_1(0))\\
 a & \longmapsto & a_\infty,
\end{array}
\end{equation}
where the domain of the \textit{far field pattern mapping} is 
$$\mathcal{D}_{\infty}:=\{a\in C^2(\Omega,\CC)\cap C^1(\overline{\Omega},\CC)\,:\,\Delta a+\lambda^2a\mbox{ has compact support and }(\ref{CondicionRadiacion.Helmholtz.L1Sommerfeld.form})\mbox{ and }(\ref{HipotesisFormulaGreenHelmholtzL1.form})\mbox{ hold}\}.$$
A similar reasoning leads to an explicit formula for the far field pattern of the derivatives of $a$, namely,
\begin{equation}\label{FarFieldPatternDerivadas.form}
(\nabla a)_\infty(\sigma)=i\lambda a_\infty(\sigma)\sigma,\ \forall\sigma\in \partial B_1(0).
\end{equation}
The splitting in (\ref{FarFieldPattern.decomposicion.form}) ensures that
\begin{equation}\label{FarFieldPatternNormaL2.form}
\lim_{R\rightarrow +\infty}\int_{\partial B_R(0)}\vert a(x)\vert^2\,dx=\frac{1}{4\pi}\int_{\partial B_1(0)}\vert a_\infty(\sigma)\vert^2\,d_\sigma S.
\end{equation}
The celebrated \textit{Rellich lemma} \cite[Lemma 2.11]{ColtonKress2} states that the only complex-valued solution $a\in C^2(\Omega,\CC)$ to the exterior homogeneous Helmholtz equation  such that the limit in the left hand side of the preceding formula becomes zero is the zero function identically. Therefore, whenever a solution to the homogeneous Helmholtz equation has a well-defined far field pattern and it vanishes (i.e., $a_\infty\equiv 0$), then $a$ vanishes everywhere. 

The following uniqueness result  is of great interest to deal with Dirichlet and Neumann boundary value problems in the exterior domain. It is an immediate consequence of the Rellich lemma and the discussion leading to Corollary~\ref{FormulaGreenHelmholtz.cor}, and it can be found in \cite[Theorem 2.12]{ColtonKress2}:

\begin{lem}\label{UnicidadHelmholtzNoAcotados.lem}
Consider any solution $a\in C^2(\Omega,\CC)\cap C^1(\overline{\Omega},\CC)$ to the complex-valued homogeneous Helmholtz equation in the exterior domain $\Omega$ fulfilling the $L^2$ Sommerfeld radiation condition (\ref{CondicionRadiacion.Helmholtz.L2Sommerfeld.form}). Then, $a$ verifies the inequality
$$\lambda \Im\left(\int_S a(x)\frac{\partial \overline{a}}{\partial \eta}(x)\,d_xS\right)\leq 0.$$
If the equality holds, then $a$ vanishes everywhere in $\Omega$.
\end{lem}

Before moving to the Beltrami problem, notice that the preceding results for scalar solutions to Helmholtz equation also work for vector-valued solutions. In this case, the decay property and radiation conditions can be considered componentwise. For instance, given any vector-valued solution $u\in C^2(\Omega,\CC^3)\cap C^1(\overline{\Omega},\CC^3)$ to the Helmholtz equation,
$$-(\Delta u+\lambda^2 u)=F,\hspace{0.7cm}x\in\Omega,$$
where $F$ is compactly supported, then, the decay property and the $L^1$ Sommerfeld radiation conditions read as 
\begin{align*}
\int_{\partial B_R(0)}\left\vert u(x)\right\vert\,dx&=o(R^2),\hspace{0.4cm}\mbox{ when }R\rightarrow +\infty,\\
\int_{\partial B_R(0)}\left\vert\jac u(x)\frac{x}{R}-i\lambda u(x)\right\vert\,dx&=o(R),\hspace{0.5cm} \mbox{ when }R\rightarrow +\infty.
\end{align*}
One can wonder whether there are more natural radiation conditions for vector-valued solutions to Helmholtz equation. The general answer is given in \cite[Theorem 4.13]{ColtonKress} and \cite[Section 5, Theorem 2]{Wilcox}, although we will be mostly interested in the divergence-free case:

\begin{rem}\label{SilverMullerHelmholtz.deSommerfeld.obs}
If the vector-valued solution to the preceding Helmholtz equation verifies the above conditions, then $u$ and its first order partial derivatives enjoy the strong Sommerfeld radiation condition and have well defined far field patterns. Let us write the Helmholtz equation in the following equivalent way
$$\curl(\curl u)-\nabla(\divop u)-\lambda^2u=F\ \mbox{ in }\Omega,$$
and take $\curl$ to obtain
$$\curl(\curl(\curl u))-\lambda^2\curl u=\curl F\ \mbox{ in }\Omega.$$
Now,  far field patterns in both equations give
\begin{align*}
0=F_\infty&=(\curl(\curl u)-\nabla(\divop u)-\lambda^2u)_\infty=i\lambda \sigma\times (\curl u)_\infty-i\lambda(\divop u)_\infty \sigma-\lambda^2 u_\infty,\\
0=(\curl F)_\infty&=\left(\curl(\curl(\curl u))-\lambda^2\curl u\right)_\infty\\
&=-i\lambda^3\sigma\times(\sigma\times (\sigma\times u_\infty))-\lambda^2(\curl u)_\infty=i\lambda^3 \sigma\times u_\infty-\lambda^2(\curl u)_\infty.
\end{align*}
Notice that compactly supported functions, such as $F$ and $\curl F$, have vanishing far fields patterns thanks to (\ref{FarFieldPatternNormaL2.form}). Through the definition of far field patterns we arrive at the following two decompositions
\begin{align*}
\frac{x}{\vert x\vert}\times& \curl u(x)-\divop u(x)\frac{x}{\vert x\vert}+i\lambda u(x)\\
&=\Gamma_\lambda(x)\left\{\left(\frac{x}{\vert x\vert}\times (\curl u)_\infty\left(\frac{x}{\vert x\vert}\right)-(\divop u)_\infty\left(\frac{x}{\vert x\vert}\right)+i\lambda u_\infty\left(\frac{x}{\vert x\vert}\right)\right)+O\left(\frac{1}{\vert x\vert}\right)\right\},\\
\lambda\frac{x}{\vert x\vert}\times & u(x)+i\curl u(x)\\
&=\Gamma_\lambda(x)\left\{\left(\lambda\frac{x}{\vert x\vert}\times u_\infty\left(\frac{x}{\vert x\vert}\right)+i(\curl u)_\infty\right)+O\left(\frac{1}{\vert x\vert}\right)\right\},
\end{align*}
when $\vert x\vert\rightarrow +\infty$. Consequently, the terms associated to the far field patterns vanish and we obtain the radiation conditions
\begin{align}
\sup_{x\in B_R(0)}\left\vert\frac{x}{R}\times\curl u(x)-\divop u(x)\frac{x}{R}+i\lambda u(x)\right\vert&=o\left(\frac{1}{R}\right),\ \mbox{ when }R\rightarrow +\infty,\label{SilverMullerHelmholtz1.divergencianocero.form}\\
\sup_{x\in B_R(0)}\left\vert\lambda\frac{x}{R}\times u(x)+i\curl u(x)\right\vert&=o\left(\frac{1}{R}\right),\ \mbox{ when }R\rightarrow +\infty.\label{SilverMullerHelmholtz2.divergencianocero.form}
\end{align}
When $u$ is a divergence-free solution to the Helmholtz equation (as in our case), the radiation condition are simpler and read
\begin{align}
\sup_{x\in B_R(0)}\left\vert\frac{x}{R}\times\curl u(x)+i\lambda u(x)\right\vert&=o\left(\frac{1}{R}\right),\ \mbox{ when }R\rightarrow +\infty,\label{SilverMullerHemholtz1.form}\\
\sup_{x\in B_R(0)}\left\vert\lambda\frac{x}{R}\times u(x)+i\curl u(x)\right\vert&=o\left(\frac{1}{R}\right),\ \mbox{ when }R\rightarrow +\infty.\label{SilverMullerHemholtz2.form} 
\end{align}
\end{rem}

\subsection{Inhomogeneous Beltrami equation in the exterior domain}
Now, we move to the complex-valued inhomogeneous Beltrami equation. In order to understand where the natural radiation condition (\ref{CondicionRadiacion.Beltrami.L1SilverMullerBeltrami.Intro.form}) comes form, let us extend the arguments in Remark \ref{SilverMullerHelmholtz.deSommerfeld.obs} in the homogeneous case. To this end, we will connect three different systems that will provide an appropriate terminology. The heuristic idea is summarized in Figure \ref{fig:Fig5}.
\begin{figure}[t]
\centering
\includegraphics[scale=0.85]{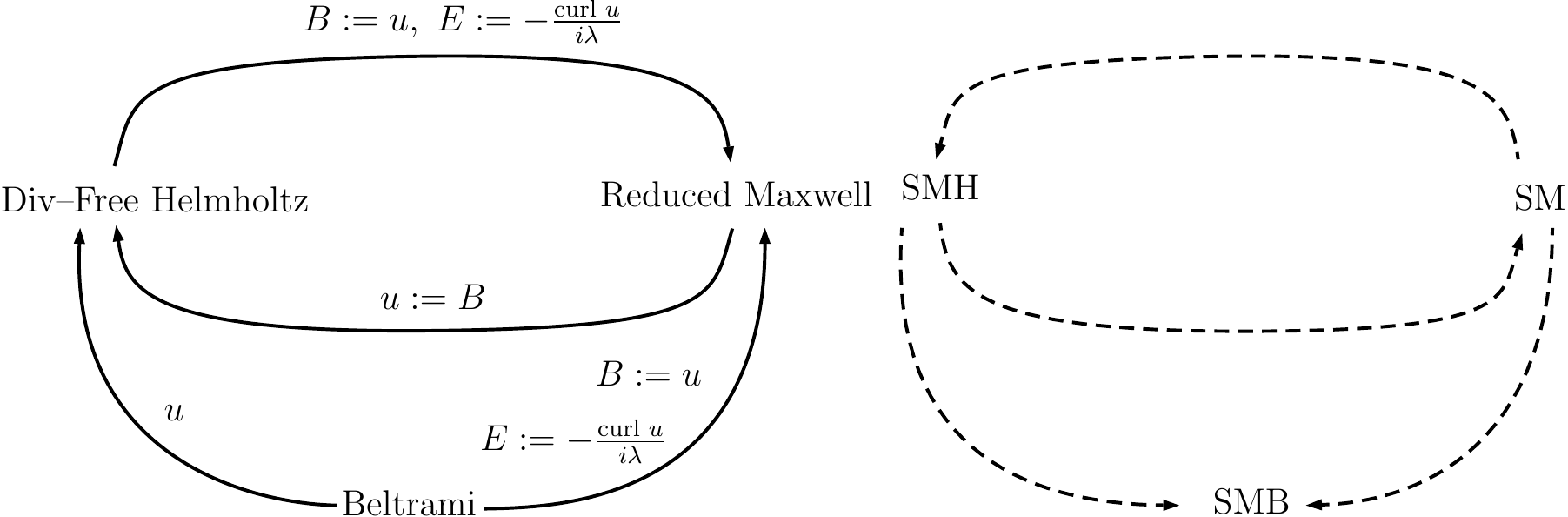}
\caption{Sketch of the connections between the three related models: \textit{divergence-free Helmholtz equation}, \textit{reduced Maxwell system} and \textit{Beltrami equation}. The picture in the left shows the bonds between such models whilst the picture in the right exhibits the associated relations between its natural radiation conditions.}
\label{fig:Fig5}
\end{figure}
Through the relations between the vector fields $u$ and $B$ in the left hand side of such pictures, we find (see \cite[Theorem 6.4]{ColtonKress2} and \cite{Wilcox}) that the divergence-free Helmholtz equation and the reduced Maxwell system \cite[Definition 6.5]{ColtonKress2} are completely equivalent, i.e., 
$$\left\{\begin{array}{ll}
\Delta u+\lambda^2 u=0, & x\in \Omega,\\
\divop u=0, & x\in \Omega.
\end{array}\right.\Longleftrightarrow\left\{\begin{array}{ll}
\curl E-i\lambda B=0, & x\in \Omega,\\
\curl B+i\lambda E=0, & x\in \Omega.
\end{array}\right.$$
In order that the solutions to this system could be represented through the classical \textit{Stratton--Chu formulas} \cite[Theorem 6.6]{ColtonKress2}, the \textit{Silver--M\"{u}ller radiation conditions} (SM)  have to be considered:
\begin{align*}
\sup_{x\in B_R(0)}\left\vert B(x)\times \frac{x}{R}-E(x)\right\vert=o\left(\frac{1}{R}\right), & \ \mbox{ when }R\rightarrow+\infty,\\
\sup_{x\in B_R(0)}\left\vert E(x)\times \frac{x}{R}+B(x)\right\vert=o\left(\frac{1}{R}\right), & \ \mbox{ when }R\rightarrow+\infty.
\end{align*}
Due to our choice of $B$ and $E$, the SM radiation conditions leads to (\ref{SilverMullerHemholtz1.form})--(\ref{SilverMullerHemholtz2.form}) again. Thus, the natural radiation conditions for the divergence-free vector-valued Helmholtz equation are actually a consequence of the SM radiation conditions for the reduced Maxwell system. Therefore, we will call them the \textit{Silver--M\"{u}ller--Helmholtz radiation conditions} (SMH). 

Let us now consider the case of the Beltrami equation
$$\curl u-\lambda u=0, \hspace{0.3cm}x\in\Omega.$$
When $\lambda\neq 0$, then $u$ is a solution to the divergence-free Helmholtz equation, and consequently it also solves the reduced Maxwell system. Therefore, one may want to transfer the SMH or the original SM radiation condition to the Beltrami framework. An easy substitution in (\ref{SilverMullerHemholtz1.form}) and (\ref{SilverMullerHemholtz2.form}) leads to a single radiation condition for Beltrami fields, which we will call it the \textit{Silver--M\"{u}ller--Beltrami  radiation condition} (SMB):
$$
\sup_{x\in B_R(0)}\left\vert i\frac{x}{R}\times u(x)-u(x)\right\vert=o\left(\frac{1}{R}\right),\ \mbox{ when }R\rightarrow+\infty.
$$

It might seem that the only connection between the Beltrami equation and the divergence-free vector-valued Helmholtz equation is the first implication sketched in Figure \ref{fig:Fig5}, but the connection is actually much stronger. The reason is the following. Given any solution $u$ to the Beltrami equation, it is obviously a solution to the divergence-free Helmholtz equation. The point is that, conversely, given any solution $\widehat{u}$ to the divergence-free Helmholtz equation, 
\begin{equation}\label{HelmholtzBeltrami.relacion.eq}
u:=\frac{\curl \widehat{u}+\lambda \widehat{u}}{2\lambda}.
\end{equation}
is a solution to the Beltrami equation, and all the solutions can be constructed this way.

In view of this converse relation, it is natural to wonder about the radiation conditions that one should assume on $\widehat{u}$ in order for $u$ to verify the SMB radiation condition. For this, notice that
$$i\frac{x}{R}\times u(x)-u(x)=\frac{i}{2\lambda}\left(\frac{x}{R}\times \curl \widehat{u}(x)+i\lambda \widehat{u}(x)\right)+\frac{i}{2\lambda}\left(\lambda\frac{x}{R}\times \widehat{u}(x)+i\curl \widehat{u}(x)\right),$$
for every $x\in \partial B_R(0)$. Therefore, the SMB radiation condition on $u$ is recovered form the SMH radiation conditions on $\widehat{u}$, so all the possible links between the three models and its corresponding radiation conditions in Figure \ref{fig:Fig5} follow.

\begin{rem}\label{HelmholtzBeltrami.relacion.obs}
The complex-valued Beltrami fields $u$ satisfying the SMB radiation condition take the form (\ref{HelmholtzBeltrami.relacion.eq}) for some solution $\widehat{u}$ of the divergence-free Helmholtz equation satisfying the SMH radiation conditions.
\end{rem}

As in the Sommerfeld radiation condition, let us consider the following hierarchy of SMB radiation conditions. 

\begin{defi}\label{CondicionesRadiacionBeltrami.defi}
$\,$
\begin{enumerate}
\item $L^1$ Silver--M\"{u}ller--Beltrami 
\begin{equation}\label{CondicionRadiacion.Beltrami.L1SilverMullerBeltrami.form}
\int_{\partial B_R(0)} \left\vert i \frac{x}{R}\times u(x)- u(x)\right\vert\,d_xS=o(R),\ R\rightarrow +\infty.
\end{equation}
\item $L^2$ Silver--M\"{u}ller--Beltrami 
\begin{equation}\label{CondicionRadiacion.Beltrami.L2SilverMullerBeltrami.form}
\int_{\partial B_R(0)} \left\vert i \frac{x}{R}\times u(x)-u(x)\right\vert^2\,d_xS=o(1),\ R\rightarrow +\infty.
\end{equation}
\item ($L^\infty$) Silver--M\"{u}ller--Beltrami 
\begin{equation}\label{CondicionRadiacion.Beltrami.SilverMullerBeltrami.form}
\sup_{x\in \partial B_R(0)}\left\vert i \frac{x}{R}\times u(x)- u(x)\right\vert=o\left(\frac{1}{R}\right),\ R\rightarrow +\infty.
\end{equation}
\end{enumerate}
\end{defi}

The next theorem shows the desired decomposition theorem of Helmholtz--Hodge type under the above $L^1$ decay and radiation hypotheses, that were already mentioned in the Introduction (see (\ref{CondCaidaBeltrami.Intro.form}) and (\ref{CondicionRadiacion.Beltrami.L1SilverMullerBeltrami.Intro.form})):

\begin{theo}\label{HelmholtzHodgeBeltrami.teo}
Let $u\in C^1(\overline{\Omega},\CC^3)$ be any vector field which verifies the $L^1$ SMB condition (\ref{CondicionRadiacion.Beltrami.L1SilverMullerBeltrami.form}) and the following decay property at infinity
\begin{equation}
\int_{\partial B_R(0)}\vert u(x)\vert\,d_xS=o(R^2),\ \mbox{ when }R\rightarrow +\infty\label{CondCaidaBeltrami.form}\\
\end{equation}
Assume that $\divop u,\,\curl u-\lambda u=O(\vert x\vert^{-\rho})$ when $\vert x\vert\rightarrow+\infty$ for some exponent $2<\rho<3$. Then, $u$ can be decomposed as
$$u(x)=-\nabla\phi(x) +\curl A(x)+\lambda A(x),$$
for every $x\in \Omega$, where $\phi$ and $A$ are the scalar and vector fields
\begin{align*}
\phi(x)&=\int_{\Omega}\Gamma_\lambda(x-y)\divop u(y)\,dy+\int_S \Gamma_\lambda(x-y)\eta(y)\cdot u(y)\,d_y S,\\
A(x)&=\int_{\Omega}\Gamma_\lambda(x-y)(\curl u(y)-\lambda u(y))\,dy+\int_S \Gamma_\lambda(x-y)\eta(y)\times u(y)\,d_y S.
\end{align*}
As a consequence,
$$u=O(\vert x\vert^{-(\rho-2)}), \mbox{ when }\vert x\vert\rightarrow +\infty.$$
Indeed, when both $\divop u$ and $\curl u-\lambda u$ are compactly supported, one obtains the optimal decay at infinity, namely,
$$u=O(\vert x\vert^{-1}), \mbox{ when }\vert x\vert\rightarrow +\infty,$$
and $u$ satisfies the Sommerfeld radiation condition (\ref{CondicionRadiacion.Helmholtz.Sommerfeld.form}) componentwise.
\end{theo}

\begin{proof}
Consider any $x\in \Omega$ and fix any couple of radii $\varepsilon_0,R_0>0$ such that
$$\overline{B}_{\varepsilon_0}(x)\subseteq \Omega\ \mbox{ and }\ \overline{B}_{\varepsilon_0}(x)\cup\overline{G}\subseteq B_{R_0}(0).$$
Define the subdomain $\Omega(x,\varepsilon,R):=\Omega\cap (B_R(0)\setminus\overline{B}_{\varepsilon}(x))$ for $R>R_0$ and $\varepsilon>\varepsilon_0$, as in Figure \ref{fig:Fig1}.

\begin{figure}[t]
\centering
\includegraphics[scale=0.8]{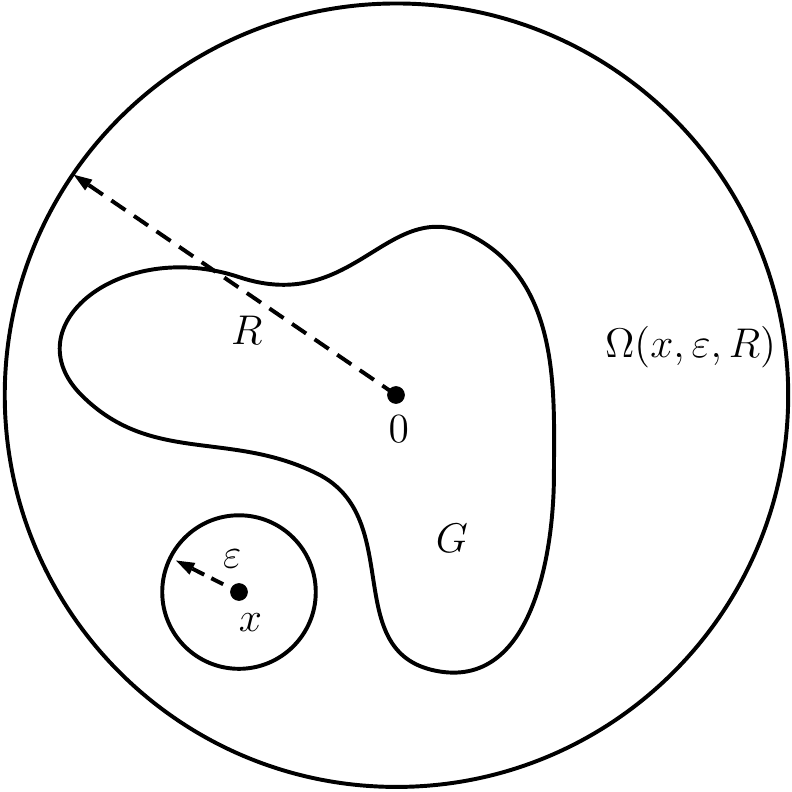}
\caption{Domain $\Omega(x,\varepsilon,R)$.}
\label{fig:Fig1}
\end{figure}

The main difference between the formula of Stokes type for the scalar Helmholtz equation in the preceding paragaph and the formula of Helmholtz--Hodge type for the Beltrami equation here is that the former holds true by virtue of the scalar Green's first formula while the later needs some sort of vector formula of Green type to be derived. 

Let us fix any vector $e\in \CC^3$. Since $\Gamma_\lambda$ solves the scalar homogeneous Helmholtz equation outside the origin, then $\Gamma_\lambda e$ is a solution to the vector-valued homogeneous Helmholtz equation too. Therefore, the following identity holds
\begin{equation*}\label{HHBeltramiEq1}
0=-\int_{\Omega(x,\varepsilon,R)}(\Delta (\Gamma_\lambda(x-y)e)+\lambda^2(\Gamma_\lambda(x-y)e))\cdot u(y)\,dy.
\end{equation*}
As in the classical Helmholtz--Hodge theorem, it is essential to bear the next formula in mind 
$$\curl(\curl)=\nabla(\divop)-\Delta,$$
which allows writing the above identity in the following way
\begin{align}
0=&-\int_{\Omega(x,\varepsilon,R)}\nabla_y\left(\divop_x\left(\Gamma_\lambda(x-y)e\right)\right)\cdot u(y)\,dy +\int_{\Omega(x,\varepsilon,R)}\curl_y\left(\curl_x\left(\Gamma_\lambda(x-y)e\right)\right)\cdot u(y)\,dy \nonumber \\ &+\lambda^2\int_{\Omega(x,\varepsilon,R)}(\Gamma_\lambda(x-y)e)\cdot u(y)\,dy.\label{HHBeltramiEq2}
\end{align}
Hence Equation (\ref{HHBeltramiEq2}) can be written as
\begin{align}
0=&-\int_{\Omega(x,\varepsilon,R)}\divop_y\left(\divop_x\left(\Gamma_\lambda(x-y)e\right)u(y)\right)\,dy +\int_{\Omega(x,\varepsilon,R)}\divop_x\left(\Gamma_\lambda(x-y)e\right)\divop u(y)\,dy\nonumber\\
&+\int_{\Omega(x,\varepsilon,R)}\divop_y\left(\curl_x\left(\Gamma_\lambda(x-y)e\right)\times u(y)\right)\,dy+\int_{\Omega(x,\varepsilon,R)}\curl_x\left(\Gamma_\lambda(x-y)e\right)\cdot \curl u(y)\,dy\nonumber\\
&+\lambda^2\int_{\Omega(x,\varepsilon,R)}\left(\Gamma_\lambda(x-y)e\right)\cdot u(y)\,dy. \label{HHBeltramiEq3}
\end{align}
Now, we can apply the divergence theorem to the first and third terms in (\ref{HHBeltramiEq3}) and standard vector calculus identities to find
\begin{align}
0=&-\int_{\partial \Omega(x,\varepsilon,R)}\left(\nabla_x\Gamma_\lambda(x-y)\nu(y)\cdot u(y)\right)\cdot e\,d_yS +\int_{\Omega(x,\varepsilon,R)}\left(\nabla_x \Gamma_\lambda(x-y)\divop u(y)\right)\cdot e\,dy\nonumber\\
&+\int_{\partial \Omega(x,\varepsilon,R)}\left(\nabla_x \Gamma_\lambda(x-y)\times (\nu(y)\times u(y))\right)\cdot e\,d_yS-\int_{\Omega(x,\varepsilon,R)}\left(\nabla_x \Gamma_\lambda(x-y)\times \curl u(y)\right)\cdot e\,dy\nonumber\\
&+\lambda^2\int_{\Omega(x,\varepsilon,R)}\left(\Gamma_\lambda(x-y)u(y)\right)\cdot e\,dy. \label{HHBeltramiEq5}
\end{align}

Let us now remove the dot product by $e$ (notice that (\ref{HHBeltramiEq5}) holds for any constant vector $e\in \CC^3$) and subtract and add the appropriate terms to obtain the following formula
\begin{align*}
0=&-\int_{\partial \Omega(x,\varepsilon,R)}\nabla_x\Gamma_\lambda(x-y)\nu(y)\cdot u(y)\,d_yS +\int_{\Omega(x,\varepsilon,R)}\nabla_x \Gamma_\lambda(x-y)\divop u(y)\,dy\nonumber\\
&+\int_{\partial \Omega(x,\varepsilon,R)}\nabla_x \Gamma_\lambda(x-y)\times (\nu(y)\times u(y))\,d_yS-\int_{\Omega(x,\varepsilon,R)}\nabla_x \Gamma_\lambda(x-y)\times (\curl u(y)-\lambda u(y))\,dy\nonumber\\
&+\lambda\left(-\int_{\Omega(x,\varepsilon,R)}\nabla_x \Gamma_\lambda(x-y)\times u(y)\,dy+\lambda\int_{\Omega(x,\varepsilon,R)}\Gamma_\lambda(x-y)u(y)\,dy\right).
\end{align*}
We can write the last term in terms of $\curl u-\lambda u$ and $\nu\times u$ using integration by parts:
\begin{align}
0=&-\int_{\partial \Omega(x,\varepsilon,R)}\nabla_x\Gamma_\lambda(x-y)\nu(y)\cdot u(y)\,d_yS +\int_{\Omega(x,\varepsilon,R)}\nabla_x \Gamma_\lambda(x-y)\divop u(y)\,dy \nonumber \\
&+\int_{\partial \Omega(x,\varepsilon,R)}\nabla_x \Gamma_\lambda(x-y)\times (\nu(y)\times u(y))\,d_yS -\int_{\Omega(x,\varepsilon,R)}\nabla_x \Gamma_\lambda(x-y)\times (\curl u(y)-\lambda u(y))\,dy\nonumber\\
&+\lambda\left(-\int_{\Omega(x,\varepsilon,R)} \Gamma_\lambda(x-y)(\curl u(y)-\lambda u(y))\,dy+\int_{\partial\Omega(x,\varepsilon,R)}\Gamma_\lambda(x-y)\nu(y)\times u(y)\,d_yS\right).\label{HHBeltramiEq6}
\end{align}
Let us finally take limits when $\varepsilon\rightarrow 0$ and $R\rightarrow +\infty$ in the preceding identities. We start with the volume integrals, that obviously converges to the integral over the whole exterior domain due to the dominated convergence theorem, the \textit{Hardy--Littewood--Sobolev theorem of fractional integration} (Theorem~\ref{DecPotencialRiesz}) and the hypotheses on $\divop u$ and $\curl u-\lambda u$:
\begin{align*}
\int_{\Omega(x,\varepsilon,R)}\nabla_x \Gamma_\lambda(x-y)\divop u(y)\,dy&\longrightarrow\int_{\Omega}\nabla_x \Gamma_\lambda(x-y)\divop u(y)\,dy,\\
\int_{\Omega(x,\varepsilon,R)}\nabla_x \Gamma_\lambda(x-y)\times (\curl u(y)-\lambda u(y))\,dy & \longrightarrow \int_{\Omega}\nabla_x \Gamma_\lambda(x-y)\times (\curl u(y)-\lambda u(y))\,dy,\\
\int_{\Omega(x,\varepsilon,R)} \Gamma_\lambda(x-y)(\curl u(y)-\lambda u(y))\,dy& \longrightarrow \int_{\Omega} \Gamma_\lambda(x-y)(\curl u(y)-\lambda u(y))\,dy,
\end{align*}
when $\varepsilon\rightarrow 0$ and $R\rightarrow +\infty$. 

Regarding the boundary integrals, it is worth splitting them into the three connected components of the boundary surface of $\Omega(x,\varepsilon,R)$:
$$\partial \Omega(x,\varepsilon,R)=S\cup \partial B_\varepsilon(x)\cup \partial B_R(0).$$
Since the integrals over $S$ are not relevant in the limit $\varepsilon\rightarrow 0$ and $R\rightarrow +\infty$, we focus on the two remaining terms. On the one hand, the boundary terms over the sphere $\partial B_\varepsilon(x)$ can be written as 
\begin{align*}
I_\varepsilon:=&\int_{\partial B_\varepsilon(x)}\nabla_x \Gamma_\lambda(x-y)\frac{y-x}{\varepsilon}\cdot u(y)\,d_yS-\int_{\partial B_\varepsilon(x)}\nabla_x\Gamma_\lambda(x-y)\times\left(\frac{y-x}{\varepsilon}\times u(y)\right)\,d_yS\\
&-\lambda\int_{\partial B_\varepsilon(x)}\Gamma_\lambda(x-y)\frac{y-x}{\varepsilon}\times u(y)\,d_yS.
\end{align*}
Notice that the derivative formula (\ref{DerivadaSolucionFundamental.form}) for $\Gamma_\lambda(x)$ now reads
$$
\nabla_x\Gamma_\lambda(x-y)=\left(i\lambda-\frac{1}{\vert x-y\vert}\right)\frac{e^{i\lambda\vert x-y\vert}}{4\pi\vert x-y\vert}\frac{x-y}{\vert x-y\vert}.
$$

This identity and \textit{Lagrange's formula}
$$
v=(e\cdot v)\,e-e\times(e\times v),
$$
for any unit vector $e$ and any general vector $v$ show that
\begin{align*}
I_\varepsilon:=&-\left(i\lambda-\frac{1}{\varepsilon}\right)\frac{e^{i\lambda\varepsilon}}{4\pi\varepsilon}\int_{\partial B_\varepsilon(x)}u(y)\,d_yS-\lambda\frac{e^{i\lambda\varepsilon}}{4\pi \varepsilon}\int_{\partial B_\varepsilon(x)}\frac{y-x}{\varepsilon}\times u(y)\,d_yS\\
=&i\lambda\frac{e^{i\lambda\varepsilon}}{4\pi\varepsilon}\int_{\partial B_\varepsilon(x)}\left(i\frac{y-x}{\varepsilon}\times u(y)-u(y)\right)\,d_yS+\frac{e^{i\lambda\varepsilon}}{4\pi\varepsilon^2}\int_{\partial B_\varepsilon(x)}u(y)\,d_yS.
\end{align*}
Consequently, the first term converges to zero as $\varepsilon\rightarrow 0$ while the second term converges to $u(x)$ due to the properties of the mean value over spheres of continuous functions.

In addition, the boundary terms over $\partial B_R(0)$ may also be written in a similar way
\begin{align*}
I_R:=&-\int_{\partial B_R(0)}\nabla_x \Gamma_\lambda(x-y)\frac{y}{R}\cdot u(y)\,d_yS+\int_{\partial B_R(0)}\nabla_x\Gamma_\lambda(x-y)\times\left(\frac{y}{R}\times u(y)\right)\,d_yS\\
&+\lambda\int_{\partial B_R(0)}\Gamma_\lambda(x-y)\frac{y}{R}\times u(y)\,d_yS\\
=&\int_{\partial B_R(0)}\left(i\lambda-\frac{1}{\vert x-y\vert}\right)\frac{e^{i\lambda\vert x-y\vert}}{4\pi\vert x-y\vert}\frac{y-x}{\vert y-x\vert}\frac{y}{R}\cdot u(y)\,d_yS\\
&-\int_{\partial B_R(0)}\left(i\lambda-\frac{1}{\vert x-y\vert}\right)\frac{e^{i\lambda \vert x-y\vert}}{4\pi \vert x-y\vert}\frac{y-x}{\vert y-x\vert}\times\left(\frac{y}{R}\times u(y)\right)\,d_yS\\
&+\lambda\int_{\partial B_R(0)}\frac{e^{i\lambda \vert x-y\vert}}{4\pi\vert x-y\vert}\frac{y}{R}\times u(y)\,d_yS.
\end{align*}
This time, the reasoning is slightly different because Lagrange's formula for the triple vector product cannot be directly applied since $B_R(0)$ is not centered at $x$. See Remark \ref{HelmholtzHodgeBeltrami.obs} below for the behavior of this boundary integrals if we had defined $\Omega(x,\varepsilon,R)=\Omega\cap B_R(x)\cap (\RR^3\setminus \overline{B}_\varepsilon(x))$ instead of $\Omega(x,\varepsilon,R)=\Omega\cap B_R(0)\cap (\RR^3\setminus \overline{B}_\varepsilon(x))$. Let us add and subtract the appropriate terms in order to obtain a more suggestive equation where Lagrange's formula for the triple vector product can be applied
\begin{align*}
I_R:=&-i\lambda\int_{\partial B_R(0)}\frac{e^{i\lambda\vert x-y\vert}}{4\pi\vert x-y\vert}\left(i\frac{y}{R}\times u(y)-u(y)\right)\,d_yS-\int_{\partial B_R(0)}\frac{e^{i\lambda \vert x-y\vert}}{4\pi \vert x-y\vert^2} u(y)\,d_yS\\
&+\int_{\partial B_R(0)}\left(i\lambda-\frac{1}{\vert x-y\vert}\right)\frac{e^{i\lambda\vert x-y\vert}}{4\pi\vert x-y\vert}\left(\frac{y-x}{\vert y-x\vert}-\frac{y}{R}\right)\frac{y}{R}\cdot u(y)\,d_yS\\
&-\int_{\partial B_R(0)}\left(i\lambda-\frac{1}{\vert x-y\vert}\right)\frac{e^{i\lambda \vert x-y\vert}}{4\pi \vert x-y\vert}\left(\frac{y-x}{\vert y-x\vert}-\frac{y}{R}\right)\times\left(\frac{y}{R}\times u(y)\right)\,d_yS.
\end{align*}
Then, the same argument as in (\ref{PotencialesCapaSimple.HelmholtzRadiantes.ValorMedio.ineq}) leads to the following bound of the norm of $I_R$ for $R>\vert x\vert$
\begin{align}
\vert I_R\vert\leq&\frac{\vert\lambda\vert}{4\pi(R-\vert x\vert)}\int_{\partial B_R(0)}\left\vert i\frac{y}{R}\times u(y)-u(y)\right\vert\,d_yS\label{HHBeltramiEq7}\\
&+\frac{1}{4\pi (R-\vert x\vert)^2}\int_{\partial B_R(0)}\vert u(y)\vert\,d_yS\nonumber+\frac{2C\vert x\vert}{4\pi(R-\vert x\vert)^2}\int_{\partial B_R(0)}\vert u(y)\vert\,d_yS.\nonumber
\end{align}
Thereby, $I_R\rightarrow 0$ when $R\rightarrow +\infty$, thanks to the $L^1$ SMB radiation condition (\ref{CondicionRadiacion.Beltrami.L1SilverMullerBeltrami.form}) and the $L^1$ decay property (\ref{CondCaidaBeltrami.form}).

Now that we then have the representation formula in the statement of the theorem, the asymptotic behavior at infinity follows from Theorem \ref{DecPotencialRiesz} and the componentwise Sommerfeld radiation condition in the compactly supported case is a direct consequence of Propositions \ref{PotencialesCapaSimple.HelmholtzRadiantes.pro} and \ref{PotencialesVolumen.HelmholtzRadiantes.pro}.
\end{proof}

\begin{rem}\label{HelmholtzHodgeBeltrami.obs}
Consider $\Omega(x,\varepsilon,R)=\Omega\cap B_R(x)\cap (\RR^3\setminus \overline{B}_\varepsilon(x))$ instead of $\Omega(x,\varepsilon,R)=\Omega\cap B_R(0)\cap (\RR^3\setminus \overline{B}_\varepsilon(x))$ in Eq. (\ref{HHBeltramiEq6}). We can argue in the same way both for the boundary terms over $\partial B_\varepsilon(x)$ and for those over $\partial B_R(x)$. Then, the former has already been studied in the above proof and the later reads
\begin{align}
I_R:=&\left(i\lambda-\frac{1}{R}\right)\frac{e^{i\lambda R}}{4\pi R}\int_{\partial B_R(x)}u(y)\,d_yS+\lambda\frac{e^{i\lambda R}}{4\pi R}\int_{\partial B_R(x)}\frac{y-x}{R}\times u(y)\,d_yS\label{HHBeltramiEq8}\\\
=&-i\lambda\frac{e^{i\lambda R}}{4\pi R}\int_{\partial B_R(x)}\left(i\frac{y-x}{\varepsilon}\times u(y)-u(y)\right)\,d_yS-\frac{e^{i\lambda R}}{4\pi R^2}\int_{\partial B_R(x)}u(y)\,d_yS.\nonumber
\end{align}
Therefore, the same representation theorem might have been obtained from the following radiation and decay conditions
\begin{align*}
\int_{\partial B_R(x)}\left(i\frac{y-x}{\varepsilon}\times u(y)-u(y)\right)\,d_yS&=o(R),\hspace{0.5cm}\mbox{ when }R\rightarrow +\infty,\\
\int_{\partial B_R(x)}u(y)\,d_yS&=o(R^2), \hspace{0.35cm}\mbox{ when }R\rightarrow +\infty,
\end{align*}
for every $x\in \Omega$. The hypotheses are stronger than (\ref{CondicionRadiacion.Beltrami.L1SilverMullerBeltrami.form}) and (\ref{CondCaidaBeltrami.form}) in the sense that they have to be assumed on every $x\in\Omega$. However, they are weaker in the sense that norms can be removed here. Therefore, one might take advantage of certain geometric cancellations of our vector fields  to ensure these conditions.

An obvious but interesting feature of the above boundary terms is that in both cases, when $\Omega(x,\varepsilon,R)=\Omega\cap B_R(0)\cap (\RR^3\setminus \overline{B}_\varepsilon(x))$ (\ref{HHBeltramiEq7}) and $\Omega(x,\varepsilon,R)=\Omega\cap B_R(x)\cap (\RR^3\setminus\overline{B}_\varepsilon(x))$ (\ref{HHBeltramiEq8}),  the harmonic case $\lambda=0$ does not need to prescribe any radiation condition at infinity, as it is the case in the classical Helmholtz--Hodge theorem and in \cite{Neudert,vonWahl}.
\end{rem}

\begin{rem}\label{HipotesisRepresentacionL1SilverMullerBeltramiL2.obs}
Analogously to the case of the Helmholtz equation, the $L^2$ SMB radiation condition imply both the $L^1$ SMB radiation condition and the weak decay property in $L^1$: $(\ref{CondicionRadiacion.Beltrami.L2SilverMullerBeltrami.form})\Rightarrow (\ref{CondicionRadiacion.Beltrami.L1SilverMullerBeltrami.form})+(\ref{CondCaidaBeltrami.form})$. Indeed, let $u\in C^1(\overline{\Omega},\CC^3)$  be any solution to the complex-valued homogeneous Beltrami equation in the exterior domain which satisfies the $L^2$ SMB radiation condition (\ref{CondicionRadiacion.Beltrami.L2SilverMullerBeltrami.form}). Let us compute the square in the radiation condition as follows
\begin{align}\
\left\vert i \frac{x}{R}\times u(x)-u(x)\right\vert^2&=\left\vert \frac{x}{R}\times u(x)\right\vert^2+\left\vert u(x)\right\vert^2-2\Re\left(i u(x)\cdot\left(\frac{x}{R}\times\overline{u}(x)\right)\right)\nonumber\\
&=\left\vert \frac{x}{R}\times u(x)\right\vert^2+\left\vert u(x)\right\vert^2+2\Im\left(u(x)\cdot\left(\frac{x}{R}\times\overline{u}(x)\right)\right)\label{HipotesisRepresentacionL1SilverMullerBeltramiL2.eq1.form},
\end{align}
for any $x\in\partial B_R(0)$. For any positive radius such that $\overline{G}\subseteq B_{R_0}(0)$, we define the subdomains $\Omega_R:=B_R(0)\setminus\overline{G}$, for each $R>R_0$. Elementary computations involving the Beltrami equation leads to
$$
-\lambda\int_{\Omega_R}\vert u(x)\vert^2\,dx=\int_{\Omega_R} u(x)\cdot\curl\overline{u}(x)\,dx=\int_{\Omega_R}\divop(\overline{u}\times u)\,dx+\int_{\Omega_R}\overline{u}(x)\cdot\curl u(x)\,dy.
$$
Let us compute the mean value of the two preceding equalities and obtain, thanks to the divergence theorem,
\begin{multline*}
-\lambda\int_{\Omega_R}\vert u(x)\vert^2\,dx=\Re\left(\int_{\Omega_R} \overline{u}(x)\cdot \curl u(x)\,dx\right)\\
+\frac{1}{2}\int_{\partial B_R(0)}(\overline{u}(x)\times u(x))\cdot \frac{x}{R}\,d_xS-\frac{1}{2}\int_S (\overline{u}(x)\times u(x))\cdot \eta(y)\,d_yS.
\end{multline*}
Taking imaginary parts in the above equation leads to
\begin{equation}\label{HipotesisRepresentacionL1SilverMullerBeltramiL2.eq2.form}
\Im\left(\int_{\partial B_R(0)}(\overline{u}(x)\times u(x))\cdot \frac{x}{R}\,d_xS\right)=\Im\left(\int_S (\overline{u}(x)\times u(x))\cdot \eta(x)\,d_xS\right),
\end{equation}
for each $R>R_0$. Finally, (\ref{HipotesisRepresentacionL1SilverMullerBeltramiL2.eq1.form}),  (\ref{HipotesisRepresentacionL1SilverMullerBeltramiL2.eq2.form}) along with the $L^2$ SMB radiation condition (\ref{CondicionRadiacion.Beltrami.L2SilverMullerBeltrami.form}) lead to
\begin{equation}\label{HipotesisRepresentacionL1SilverMullerBeltramiL2.eq3.form}
\lim_{R\rightarrow +\infty}\int_{\partial B_R(0)}\left(\left\vert \frac{x}{R}\times u(x)\right\vert^2+\left\vert u(x)\right\vert^2\right)\,dx=2\Im\left(\int_S \overline{u}(x)\cdot\left(\eta(x)\times u(x)\right)\,d_xS\right).
\end{equation}
As a consequence,
$$\int_{\partial B_R(0)}\vert u(x)\vert^2\,dx=O(1), \mbox{ when }R\rightarrow +\infty,$$
and Cauchy-Schwarz inequality leads to the $L^1$ decay property
$$\int_{\partial B_R(0)}\vert u(x)\vert\,dx=O(R), \mbox{ when }R\rightarrow +\infty.$$
In particular, one gets (\ref{CondCaidaBeltrami.form}).
\end{rem}

This remark is useful because in order to check the hypotheses in Theorem \ref{HelmholtzHodgeBeltrami.teo} it is sometimes simpler to check that the $L^2$ SMB radiation condition holds. Furthermore, it can be combined with the Rellich lemma \cite[Lemma 2.11]{ColtonKress2} to obtain a uniqueness result, which is similar to that for the reduced Maxwell system in \cite[Theorem 6.10]{ColtonKress2}:

\begin{lem}\label{UnicidadBeltramiNoAcotados.lem}
Consider any solution $u\in C^1(\overline{\Omega},\CC^3)$ to the complex-valued homogeneous Beltrami equation in the exterior domain satisfying the $L^2$ SMB radiation condition (\ref{CondicionRadiacion.Beltrami.L2SilverMullerBeltrami.form}). Then, $u$ verifies the inequality
$$\Im\left(\int_S \overline{u}(x)\cdot (\eta(x)\times u(x))\,d_xS\right)\geq 0.$$
If the equality holds, then $u$ vanishes everywhere in $\Omega$.
\end{lem}

To conclude, let us state the existence result for the complex-valued homogeneous Beltrami equation that will be needed in the modified Grad-Rubin iterative scheme in Section \ref{Esquema.Iterativo.Seccion}. Since this iterative method only involves compactly supported inhomogeneities, we will focus on this case although it is easy to extend it to general inhomogeneous terms with an appropriate fall off at infinity. Hereafter we will denote by $\mathfrak{X}^{k,\alpha}(S)\equiv \mathfrak{X}^{k,\alpha}(S,\RR^3)$ the real vector space of all tangent vector fields on $S$ of regularity $C^{k,\alpha}$, i.e., 
$$\mathfrak{X}^{k,\alpha}(S):=\{\xi\in C^{k,\alpha}(S,\RR^3):\,\xi\cdot \eta=0\mbox{ on }S\}.$$
Its complex counterpart will be denoted by $\mathfrak{X}^{k,\alpha}(S,\CC^3)$.

\begin{theo}\label{BeltramiNoHomog.teo}
Let  $0\neq\lambda\in\RR$ be any constant that is not a Dirichlet eigenvalue of the Laplace operator in the interior domain, $w\in C^{k,\alpha}_c(\overline{\Omega},\CC^3)$ and $g\in C^{k+1,\alpha}(S,\CC)$ such that $\divop w\in C^{k,\alpha}(\overline{\Omega},\CC)$ and the following compatibility condition 
\begin{equation}\label{BeltramiNoHomog.Condiciongyw.form}
\int_S \left(\lambda g+w\cdot \eta\right)\,dS=0
\end{equation}
is satisfied. Consider any solution $\xi\in \mathfrak{X}^{k+1,\alpha}(S,\CC^3)$ to the boundary integral equation
\begin{equation}\label{EcIntegral.DatoFrontera.Beltrami.form}
\left(\frac{1}{2}I-T_\lambda\right)\xi=\mu,\ x\in S,
\end{equation}
where $T_\lambda\xi$ and $\mu$ are defined by
\begin{align}
(T_\lambda\xi)(x)&=\int_S \eta(x)\times\left(\nabla_x \Gamma_\lambda(x-y)\times\xi(y)\right)\,d_yS+\lambda \int_S \Gamma_\lambda(x-y)\eta(x)\times\xi(y)\,d_yS,\label{EcIntegral.DatoFrontera.Beltrami.T.form}\\
\mu(x)&=\frac{1}{\lambda}\int_{\Omega}\eta(x)\times\nabla_x \Gamma_\lambda(x-y)\divop w(y)\,dy-\int_S \eta(x)\times\nabla_x \Gamma_\lambda(x-y)g(y)\,d_yS\nonumber\\
&+\int_{\Omega}\eta(x)\times(\nabla_x \Gamma_\lambda(x-y)\times w(y))\,dy+\lambda\int_{S}\Gamma_\lambda(x-y)\eta(x)\times w(y)\,d_yS.\label{EcIntegral.DatoFrontera.Beltrami.mu.form}
\end{align}
Define the complex-valued vector field
\begin{equation}\label{BeltramiNoHomog.DescomposicionSolucion.form}
u(x):=-\nabla\phi(x) +\curl A(x)+\lambda A(x),\ x\in \Omega,
\end{equation}
where $\phi$ and $A$ stand for the scalar and vector fields
\begin{align}
\phi(x)&=-\frac{1}{\lambda}\int_{\Omega}\Gamma_\lambda(x-y)\divop w(y)\,dy+\int_S \Gamma_\lambda(x-y)g(y)\,d_y S,\label{BeltramiNoHomog.Potencialphi.form}\\
A(x)&=\int_{\Omega}\Gamma_\lambda(x-y)w(y)\,dy+\int_S \Gamma_\lambda(x-y)\xi(y)\,d_y.\label{BeltramiNoHomog.PotencialA.form}
\end{align}
Then, $u$ is a complex-valued solution to the exterior NIB problem
\begin{equation}\label{BeltramiNoHomog.eq}
\left\{\begin{array}{l}
\curl u-\lambda u=w, \hspace{0.5cm} x\in\Omega,\\
u\cdot\eta=g, \hspace{1.6cm} x\in \Omega,\\
+\ L^1\mbox{ SMB radiation condition }(\ref{CondicionRadiacion.Beltrami.L1SilverMullerBeltrami.form}),\\
+\ L^1\mbox{ decay property }(\ref{CondCaidaBeltrami.form}).
\end{array}\right.
\end{equation}
Furthermore, the decay and radiation conditions are stronger since $u$ actually behaves as $O\left(\vert x\vert^{-1}\right)$ at infinity and verifies the Sommerfeld radiation condition (\ref{CondicionRadiacion.Helmholtz.Sommerfeld.form}) componentwise. 
\end{theo}
\begin{proof}
Since the divergence of any solution $u$ can be recovered from the equation through the identity $\divop u=-\frac{1}{\lambda}\divop w$, then one arrives at the next expression for the candidate to be a solution to (\ref{BeltramiNoHomog.eq})
$$u(x)=-\nabla\phi(x) +\curl A(x)+\lambda A(x),$$
where $\phi$ and $A$ are defined as follows
\begin{align*}
\phi(x)&=-\frac{1}{\lambda}\int_{\Omega}\Gamma_\lambda(x-y)\divop w(y)\,dy+\int_S \Gamma_\lambda(x-y)g(y)\,d_y S,\\
A(x)&=\int_{\Omega}\Gamma_\lambda(x-y)w(y)\,dy+\int_S \Gamma_\lambda(x-y)\eta(y)\times u_+(y)\,d_y S.
\end{align*}
Consider $\xi:=\eta\times u_+$, where $u_\pm$ denotes the limits of $u$ at $S$ from $\Omega$ and $G$ respectively. In order to obtain a more manageable formula for $\xi$, one can use the well known \textit{jump relations} for the derivatives of a single layer potential associated with the fundamental solution to the Helmholtz equation, $\Gamma_\lambda(x)$ (see e.g.~\cite{ColtonKress}). This formulas lead to the following identity
\begin{align}
u_\pm(x)&=\frac{1}{\lambda}\int_{\Omega}\nabla_x \Gamma_\lambda(x-y)\divop w(y)\,dy-\mbox{PV}\int_S \nabla_x \Gamma_\lambda(x-y)g(y)\,d_yS\label{Limites.u.form}\\
&+\int_{\Omega}\nabla_x \Gamma_\lambda(x-y)\times w(y)\,dy+\mbox{PV}\int_S \nabla_x \Gamma_\lambda(x-y)\times\xi(y)\,d_yS\nonumber\\
&+\lambda\int_{\Omega}\Gamma_\lambda(x-y)w(y)\,dy+\lambda \int_S \Gamma_\lambda(x-y)\xi(y)\,d_yS\nonumber \pm\frac{1}{2}\eta(x)g(x)\mp\frac{1}{2}\eta(x)\times \xi(x),\nonumber
\end{align}
where $\mbox{PV}$ stands for the Cauchy principal value integral. It is clear that the terms in the last line are actually $\pm\frac{1}{2}u_\pm(x)$. Consequently, one can take cross products by $\eta(x)$ and arrive at the boundary integral equation in (\ref{EcIntegral.DatoFrontera.Beltrami.form}) for the tangential component $\xi$. There, we have intentionally avoided the $\mbox{PV}$ signs because the $\eta(x)$ factor in such integrals provides certain geometrical cancellations (see Section \ref{Teoria.Potencial.Tecnicas.Seccion}) leading to absolutely convergent integrals.

Now, let us show that the field $u$ thus defined is a solution to (\ref{BeltramiNoHomog.eq}) as long as $\xi$ solves the boundary integral equation (\ref{EcIntegral.DatoFrontera.Beltrami.form}). We will prove later that $\xi$ is unique and, consequently, (\ref{BeltramiNoHomog.eq}) is uniquely solvable. First, let us obtain some PDEs for the potentials $\phi$ and $A$ both in the interior and the exterior domain. Since volume and single layer potentials are indeed  complex-valued solutions to such PDEs, we have
\begin{equation}\label{BeltramiNoHomog.EcuacionesAyPhi.eq}
\Delta \phi+\lambda^2\phi=\left\{\begin{array}{ll}\frac{1}{\lambda}\divop w, & x\in \Omega\\
0, & x\in G\end{array}\right.\hspace{1cm}\Delta A+\lambda^2A=\left\{\begin{array}{ll}-w, & x\in \Omega\\
0, & x\in G\end{array}\right.
\end{equation}
Therefore,
\begin{align*}
\curl u-\lambda u&=\nabla(\divop A)-\Delta A+\lambda\curl A+\lambda \nabla \phi-\lambda\curl A-\lambda^2 A\\
&=-(\Delta A+\lambda^2A)+\nabla\underbrace{\left(\divop A+\lambda\phi\right)}_{a}.
\end{align*}
A direct substitution of (\ref{BeltramiNoHomog.EcuacionesAyPhi.eq}) into the previous formula leads to the following PDE for $u$ at any side of the boundary surface $S$:
\begin{equation}\label{BeltramiNoHomog.Ecuacion.u.eq}
\curl u-\lambda u=\left\{\begin{array}{ll} w+\nabla a, & x\in \Omega ,\\
\nabla a, & x\in G.
\end{array}\right.
\end{equation}
In order to show that $u$ solves (\ref{BeltramiNoHomog.eq}), it remains to check that $\nabla a$ is  zero in the exterior domain and $u$ satisfies the boundary condition $u_+\cdot\eta=g$ (the decay and radiation conditions will be studied later). To this end, it might be useful to find first a PDE for $a$. The same reasoning as above shows that $a$ solves the homogeneous Helmholtz equation, specifically
\begin{align}
\Delta a+\lambda^2 a&=\divop(\Delta A)+\lambda\Delta\phi+\lambda^2 \divop A+\lambda^3\phi\nonumber\\
&=\divop(\Delta A+\lambda^2 A)+\lambda(\Delta\phi+\lambda^2\phi)=0.\label{BeltramiNoHomog.Ecuacion.a.eq}
\end{align}

Let us show first the jump relations for the scalar potential $a$. Straightforward computations on the explicit formulas for $\phi$ and $A$ involving the divergence theorem lead to
\begin{align*}
a(x)&=\divop A(x)+\lambda\phi(x)\\
&=\int_{\Omega}\nabla_x \Gamma_\lambda(x-y)\cdot w(y)\,dy+\int_S\nabla_x \Gamma_\lambda(x-y)\cdot \xi(y)\,d_yS\\
&\hspace{0.5cm}-\int_{\Omega}\Gamma_\lambda(x-y)\divop w(y)\,dy+\lambda\int_S \Gamma_\lambda(x-y)g(y)\,d_y S\\
&=-\int_{\Omega}\divop_y(\Gamma_\lambda(x-y)w(y))\,dy+\int_S\nabla_x \Gamma_\lambda(x-y)\cdot \xi(y)\,d_yS+\lambda\int_S \Gamma_\lambda(x-y)g(y)\,d_y S\\
&=\int_{S}\Gamma_\lambda(x-y)(\lambda g(y)+w(y)\cdot \eta(y))\,d_y S+\int_S \nabla_x \Gamma_\lambda(x-y)\cdot \xi(y)\,d_yS.
\end{align*}
Finally, notice that $\nabla_x\Gamma_\lambda(x-y)\cdot \xi(y)=-(\nabla_S)_y\left[\Gamma_\lambda(x-y)\right]\cdot \xi(y)$ for every $y\in S$ because of $\xi$ being a tangent vector field along $S$. Hence, the integration by parts formula over $S$ (see Appendix \ref{Appendix.A}) yields the next simpler expression for $a$:
$$a(x)=\int_{S}\Gamma_\lambda(x-y)\left(\lambda g(y)+w(y)\cdot \eta(y)+\divop_S\xi(y)\right)\,d_yS,$$
i.e., $a$ is just a new single layer potential. As such, the first and second jumps relations read
\begin{equation}\label{salto.a.form}
a_+-a_-\equiv 0,\hspace{0.5cm}\left(\frac{\partial a}{\partial\eta}\right)_+-\left(\frac{\partial a}{\partial\eta}\right)_-\equiv-\left(\lambda g+w\cdot \eta+\divop_S\xi\right),
\end{equation}
on the surface $S$. In particular, $a$ is continuous across $S$ but its normal derivative exhibits a jump discontinuity with height $\lambda g+w\cdot \eta+\divop_S \xi$. The same kind of reasoning yields the jump relation for $u$
\begin{equation}\label{salto.u.form}
u_+-u_-=g\,\eta-\eta\times \xi,\ x\in S.
\end{equation}
Consequently, the boundary integral equation (\ref{EcIntegral.DatoFrontera.Beltrami.form}) along with the jump relation (\ref{salto.u.form}) ensure that
\begin{equation}\label{componentes.tangenciales.u.form}
\eta\times u_+=\xi,\hspace{0.25cm}\eta\times u_-=0,
\end{equation}
on $S$. Regarding $a$, let us show that it is indeed constant on $S$ and to this end, define the next vector field in the interior domain $G$:
$$v:=\lambda u+\nabla a, \hspace{0.25cm} x\in G.$$
Notice that $v$ is a strong Beltrami field with factor $\lambda$ by virtue of (\ref{BeltramiNoHomog.Ecuacion.u.eq}). Then, one can repeat the same kind of uniqueness criterion as in Lemma \ref{UnicidadBeltramiNoAcotados.lem} in the simpler bounded setting, specifically
$$\lambda\int_{G}\vert v\vert^2\,dx=\int_G \overline{v}\cdot \curl v\,dx=\int_G\divop(v\times \overline{v})\,dx=\int_S(\eta\times v)\cdot \overline{v}\,dS.$$
Now, notice that we can substitute both $v$ and $\overline{v}$ in the above formula with its tangential parts thanks to the presence of a cross product by the unit normal vector field $\eta$ and
$$-\eta\times(\eta\times v)=-\lambda\eta\times(\eta\times u_-)+\nabla_S a=\nabla_S a,$$
by virtue of (\ref{componentes.tangenciales.u.form}). Thereby, the integration by parts formula in Appendix \ref{Appendix.A} leads again to
$$\lambda\int_G\vert v\vert^2\,dx=\int_S\left(\eta\times \nabla_S a\right)\cdot \nabla\overline{a}\,dS=-\int_S a\,\overline{\curl_S\left(\nabla_S a\right)}\,dS=0,$$
where the well know formula $\curl_S\nabla_S=0$ (see Proposition \ref{GradienteDivergenciaRotacional.S.Propiedades.pro}) has been used in the last step. Consequently, $v$ vanishes everywhere in $G$ and, in particular, $\nabla_S a\equiv 0$, i.e., $a_{\pm}\equiv a_0=\mbox{const}$ on $S$.

We will next prove that $a$ vanishes everywhere in the exterior domain $\Omega$ using the uniqueness result in Lemma \ref{UnicidadHelmholtzNoAcotados.lem}. Notice that since $a$ can be written as a sum of volume and single layer potentials with compactly supported densities together with its first order partial derivatives, then $a$ satisfies the Sommerfeld radiation condition (\ref{CondicionRadiacion.Helmholtz.Sommerfeld.form}) due to Propositions \ref{PotencialesCapaSimple.HelmholtzRadiantes.pro} and \ref{PotencialesVolumen.HelmholtzRadiantes.pro}. Consequently, this lemma can be applied.  We therefore want to show that
\begin{equation}\label{eqnew}
\Im\left(\int_S a_+\left(\frac{\partial \overline{a}}{\partial\eta}\right)_+\,dS\right)=0\,.
\end{equation}
To derive~\eqref{eqnew}, let us first pass from the exterior trace values to the interior ones thanks to the jump relations (\ref{salto.a.form})
$$\int_S a_+\left(\frac{\partial \overline{a}}{\partial\eta}\right)_+\,dS=I+II,$$
where both terms read
\begin{align*}
I&:=-a_0\int_S\left(\lambda g+w\cdot \eta+\divop_S\xi\right)\,dS,\\
II&:=\int_S a_-\left(\frac{\partial \overline{a}}{\partial\eta}\right)_-\,dS. 
\end{align*}
On the one hand, $I$ becomes zero because of the divergence theorem over surfaces and the compatibility condition (\ref{BeltramiNoHomog.Condiciongyw.form}) in the hypothesis. On the other hand, integrate by parts in $II$ to arrive at
$$II:=\int_S\divop\left(a\nabla\overline{a}\right)\,dS=\int_G\vert\nabla a\vert^2\,dx+\int a\Delta\overline{a}\,dx=\int_G\vert \nabla a\vert^2\,dx-\lambda^2\int_G\vert a\vert^2\,dx,$$
where the Helmholtz equation (\ref{BeltramiNoHomog.Ecuacion.a.eq}) has being used. Therefore, one arrives at
$$\Im\left(\int_S a_+\left(\frac{\partial \overline{a}}{\partial \eta}\right)_+\,dS\right)=\Im\left(\int_G\vert \nabla a\vert^2\,dx-\lambda^2\int_G\vert a\vert^2\,dx\right)=0,$$
and consequently $a=0$ in $\Omega$ and $u$ solves the inhomogeneous Beltrami equation. 

Before proving the boundary condition and the decay and radiation properties, let us show that $a$ also vanishes in the interior domain. On the one hand, $a$ solves the homogeneous Helmholtz equation in such domain and it also satisfies the interior homogeneous Dirichlet conditions in $S$ since $a_-=a_+$ on $S$ and $a=0$ in $\Omega$. Moreover, $\lambda$ is prevented from being a Dirichlet eigenvalue of the Laplacian in the interior domain, so $a$ also vanishes in $G$. In particular, the jumps relations (\ref{salto.a.form}) yields
\begin{equation}\label{Ligadura.g.w.xi.form1}
\lambda g+w\cdot \eta+\divop_S\xi\equiv 0.
\end{equation}
Furthermore, since $u$ is now a solution to the next inhomogeneous Beltrami equation,
$$\curl u-\lambda u=w, \hspace{0.25cm}x\in\Omega,$$
taking trace values at $S$ one gets
$$\eta\cdot(\curl u)_+-\lambda\eta\cdot u_+=w\cdot \eta.$$
Now, one can write the first term in an intrinsic way through the properties in Proposition \ref{GradienteDivergenciaRotacional.S.Propiedades.pro}, specifically
$$\eta\cdot (\curl u)_+=-\divop_S(\eta\times u_+)=-\divop_S\xi,$$
and, consequently, the above formula can be restated as
\begin{equation}\label{Ligadura.g.w.xi.form2}
\eta\times u_++w\cdot \eta+\divop_S\xi\equiv 0.
\end{equation}
Then, comparing (\ref{Ligadura.g.w.xi.form1}) and (\ref{Ligadura.g.w.xi.form2}) entails the boundary condition $\eta\times u_+=g$.

Finally, let us show the decay and radiation conditions on $u$. First, since
$$\Gamma_\lambda(x),\nabla\Gamma_\lambda(x)=O\left(\vert x\vert^{-1}\right), \ \mbox{ when }\vert x\vert\rightarrow +\infty,$$
and $w$ has compact support, then $u$ enjoys the optimal decay $u=O\left(\vert x\vert^{-1}\right)$ when $\vert x\vert\rightarrow +\infty$ according to Theorem \ref{DecPotencialRiesz}. Second, as $u$ is again a sum of single and volume layer potential associated with the Helmholtz equation along with some partial derivatives, then $u$ satisfies Sommerfeld radiation condition componentwise thanks to Propositions \ref{PotencialesCapaSimple.HelmholtzRadiantes.pro} and \ref{PotencialesVolumen.HelmholtzRadiantes.pro}. Therefore, one can show that $u$ verifies SMH conditions (\ref{SilverMullerHemholtz1.form}) and (\ref{SilverMullerHemholtz2.form}) thanks to Remark \ref{SilverMullerHelmholtz.deSommerfeld.obs}. Since $\curl u-\lambda u=w$ and $w$ is compactly supported, then $u$ actually satisfies SMB radiation condition (\ref{CondicionRadiacion.Beltrami.SilverMullerBeltrami.form}) and this finishes the proof.
\end{proof}

\subsection{Well-posedness of the boundary integral equation}

One should also notice that, in addition to the uniqueness result proved in Theorem~\ref{BeltramiNoHomog.teo}, we will also need a study of the regularity of the solution, which is obviously in $C^1(\overline{\Omega},\CC^3)$ by the decomposition (\ref{BeltramiNoHomog.DescomposicionSolucion.form}). We will prove in this next subsection that the regularity assumptions on the data $w$ and $g$ actually leads to $C^{k+1,\alpha}(\Omega,\CC^3)$ regularity on $u$, estimating its $C^{k+1,\alpha}(\Omega,\CC^3)$ norm in terms of the natural norms of the data $w$ and $g$. Some necessary potential theoretic estimates have been relegated to Section \ref{Teoria.Potencial.Tecnicas.Seccion} to streamline the exposition. 

Let us start by studying the well-posedness of (\ref{EcIntegral.DatoFrontera.Beltrami.form}) using the Riesz--Fredholm theory for compact operators, which follows easily from our previous estimates:

\begin{pro}\label{BeltramiNoHomog.Compacidad.EcuacionIntegral.pro}
The linear operator
$
T_\lambda:\mathfrak{X}^{k+1,\alpha}(S)\longrightarrow \mathfrak{X}^{k+1,\alpha}(S)
$
 is compact.
 \end{pro}
\begin{proof}
The gain of regularity proved in Theorem \ref{Potencial.capasimple.regularidad.frontera.teo} implies that $T_\lambda$ defines a continuous linear operator 
$$
T_\lambda:\mathfrak{X}^{k,\alpha}(S)\longrightarrow \mathfrak{X}^{k+1,\alpha}(S).
$$
Since the embedding $\mathfrak{X}^{k+1,\alpha}(S)\hookrightarrow \mathfrak{X}^{k,\alpha}(S)$ is compact by the Ascoli--Arzel\`a theorem, the proposition follows.
\end{proof}

The proposition ensures that it is possible to apply Riesz--Fredholm theory to the operator $\frac{1}{2}I-T_\lambda$. In particular, $\frac{1}{2}I-T_\lambda$ is one to one if, and only if, it is onto, i.e.,
$$
\Ker\left(\frac{1}{2}I-T_\lambda\right)=0\Longleftrightarrow \Img\left(\frac{1}{2}I-T_\lambda\right)=\mathfrak{X}^{k+1,\alpha}(S).
$$
As it is hard to show explicitly that such operator is onto, let us equivalently show that it is one to one. This will be easier thanks to the uniqueness Lemma \ref{UnicidadBeltramiNoAcotados.lem} for the Beltrami equation and the existence Theorem \ref{BeltramiNoHomog.teo}.
\begin{pro}\label{BeltramiNoHomog.Unicidad.EcuacionIntegral.pro}
The bounded linear operator $\frac{1}{2}I-T_\lambda$ on $\mathfrak{X}^{k+1,\alpha}(S)$ is one to one and onto. Consequently, the boundary integral equation (\ref{EcIntegral.DatoFrontera.Beltrami.form}) has a unique solution  $\xi\in \mathfrak{X}^{k+1,\alpha}(S)$ for any $\mu\in \mathfrak{X}^{k+1,\alpha}(S)$.
\end{pro}
\begin{proof}
According to the preceding argument, we only have to show that $\Ker\left(\frac{1}{2}I-T_\lambda\right)=0$. To this end, let us consider an arbitrary $\xi\in \Ker\left(\frac{1}{2}I-T_\lambda\right)$ and show that $\xi\equiv 0$. By definition, $\xi\in \mathfrak{X}^{k+1,\alpha}(S)$ solves the boundary integral equation
$$\frac{1}{2}\xi-T_\lambda\xi=0,\ \mbox{ on }S.$$
Define $u(x):=\curl A(x)+\lambda A(x)$, where $A$ is the vector potential
$$A(x):=\int_S \Gamma_\lambda(x-y)\xi(y)\,d_yS.$$
Thus, Theorem \ref{BeltramiNoHomog.teo} for $w\equiv 0$ and $g\equiv 0$ leads to a solution $u\in C^1(\overline{\Omega},\CC^3)$ to the homogeneous Beltrami equation in $\Omega$
$$\left\{\begin{array}{ll}
\curl u=\lambda u, & x\in \Omega,\\
\eta\cdot u_+=g=0, & x\in S,
\end{array}\right.$$
that satisfies the Dirichlet boundary condition $\eta\times u_+=\xi$ on $S$ and the SMB radiation condition. 

We would like to show that this boundary value problem has a unique solution, but this does not follow directly from Lemma \ref{UnicidadBeltramiNoAcotados.lem}. However, since $\eta \cdot u=0$ on $S$, then  $u_+=-\eta\times(\eta\times u_+)$ on $S$ and we have the following relation between the $\curl$ operator on $S$,  $\curl_S$, and the $\curl$ operator on $\RR^3$ (see Proposition \ref{GradienteDivergenciaRotacional.S.Propiedades.pro} in Appendix \ref{Appendix.A}):
\begin{align*}
\curl_S u_+&=\curl_S\left(-\eta\times(\eta\times u_+)\right)=\eta\cdot \curl u_+ =\lambda\,\eta\cdot u_+=0.
\end{align*}
As $S$ is homeomorphic to a sphere, Poincar\'e's lemma shows that $u$ has an associated potential $\psi\in C^2(S)$ on the surface, i.e., $u=\nabla_S \psi$ on $S$, where $\nabla_S$ stands for the Riemannian connection on the surface $S$ (see Appendix \ref{Appendix.A}). Consequently, 
\begin{align*}
\Im\left(\int_S\overline{u}\cdot\left(\eta\times u\right)\,dS\right)&=\Im\left(\int_S\overline{\nabla_S \psi}\cdot\left(\eta\times \nabla_S \psi\right)\,dS\right)=-\Im\left(\int_S \overline{\curl_S\left(\nabla_S \psi\right)}\psi\,dS\right)=0.
\end{align*}
The identity follows from an integration by parts on $S$ and the classical property $\curl_S(\nabla_S \psi)=0$. Therefore, Lemma \ref{UnicidadBeltramiNoAcotados.lem} yields the desired result.
\end{proof}

\begin{rem}\label{ExitenciaUnicidadEcIntegral.obs}
The importance of the above result lies on the following facts.
\begin{enumerate}
\item First, the existence part of the above result ensures that it is possible to choose some $\xi$ solving (\ref{EcIntegral.DatoFrontera.Beltrami.form}). Obviously, it is essential to rigurously establish the existence Theorem \ref{BeltramiNoHomog.teo}.
\item Second, the uniqueness result shows that since $\xi$ can uniquely be chosen, then (\ref{BeltramiNoHomog.eq}) has a unique solution too.
\item Finally, it provides a very useful estimate for the subsequent result. Since $\frac{1}{2}I-T_\lambda$ is linear, continuous and bijective, then $\left(\frac{1}{2}I-T_\lambda\right)^{-1}$ is continuous by virtue of the Banach isomorphism theorem. Consequently, there exists a positive constant $c$ (which depends on $G$ and $\lambda$) such that
\begin{equation}\label{AplicacionAbierta.ineq}
c\Vert \xi\Vert_{C^{k+1,\alpha}(S)}\leq \left\Vert\left(\frac{1}{2}I-T_\lambda\right)\xi\right\Vert_{C^{k+1,\alpha}(S)},
\end{equation}
for any $\xi\in \mathfrak{X}^{k+1,\alpha}(S)$.
\end{enumerate}
\end{rem}
We conclude by proving the following regularity result for the solution $u$ of (\ref{BeltramiNoHomog.eq}) according to Theorem \ref{BeltramiNoHomog.teo}. It is an immediate consequence of the decomposition (\ref{BeltramiNoHomog.DescomposicionSolucion.form}), the estimates for the volume and single layer potentials in Section \ref{Teoria.Potencial.Tecnicas.Seccion} (Lemmas \ref{Potencial.volumetrico.regularidad.teo} and \ref{Potencial.capasimple.regularidad.teo}) and the estimate (\ref{AplicacionAbierta.ineq}).
\begin{cor}\label{BeltramiNoHomog.EstimacionSchauder.cor}
Assume that the hypothesis in Theorem \ref{BeltramiNoHomog.teo} are satisfied, fix any $R>0$ such that $\overline{G}\subseteq B_R(0)$ and assume that the closure of $\Omega_R:=B_R(0)\setminus \overline{G}$ contains the support of $w$. Then, there exists some nonnegative constant $C_0=C_0(k,\alpha,G,R,\lambda)$ such that the next estimate 
\begin{equation}\label{EstimacionSchauder.u.ineq}
\Vert u\Vert_{C^{k+1,\alpha}(\Omega)}\leq C_0\left\{\Vert w\Vert_{C^{k,\alpha}(\Omega)}+\Vert\divop w\Vert_{C^{k,\alpha}(\Omega)}+\Vert g\Vert_{C^{k+1,\alpha}(S)}\right\}.
\end{equation}
holds. In particular, not only does $u$ belong to $C^1(\overline{\Omega},\CC^3)$, but also to $C^{k+1,\alpha}(\overline{\Omega},\CC^3)$.
\end{cor}

\subsection{Optimal fall-off in exterior domains}
Before passing to the next section, it is worth discussing the differences between the optimal fall-off $\vert x\vert^{-1}$ of the solutions to inhomogeneous Beltrami equation and that of the solutions of the div-curl problem. First, it is well known that the exterior Neumann boundary value problem associated with the div-curl system
\begin{equation}\label{DivCurl.Neumann.Problem.eq}
\left\{\begin{array}{ll}
\curl u=w, & x\in \Omega,\\
\divop u=f, & x\in \Omega,\\
u\cdot \eta=g, & x\in S,\\
u=O(\vert x\vert^{1-\rho}), & x\in \Omega,
\end{array}\right.
\end{equation}
where $w,f=O(\vert x\vert^{-\rho})$ and $\rho\in (1,3)$, is uniquely solvable when appropriate regularity spaces are considered (see \cite{Kaiser,Neudert}) and $w$ has zero flux in the exterior domain. Moreover, the solution inherits the optimal fall-off $\vert x\vert^{-2}$ when $w$ and $f$ are assumed to have compact support. In particular, any harmonic field ($w=0,f=0$) so obtained decays at infinity as $\vert x\vert^{-2}$. Such result is an easy consequence of the Helmholtz--Hodge representation formula in \cite[Theorem 4.1]{Neudert} and the natural fall-off of the fundamental solution of the Laplace equation, $\Gamma_0(x)$.

In our case, the exterior Neumann boundary value problem associated with the inhomogeneous Beltrami equation (\ref{BeltramiNoHomog.eq}) has an associated representation formula of Helmholtz--Hodge type (\ref{BeltramiNoHomog.DescomposicionSolucion.form}) that transfers the ``optimal fall-off'' $\vert x\vert^{-1}$ to the solution in Theorem \ref{BeltramiNoHomog.teo} when $w$ is assumed to have compact support. Let us show that it is indeed the optimal decay rate. To this end, assume that $u$ solves the equation
$$\curl u-\lambda u=w, \hspace{0.25cm}x\in \Omega,$$
(not necessarily fulfilling neither (\ref{CondCaidaBeltrami.form}) nor (\ref{CondicionRadiacion.Beltrami.L1SilverMullerBeltrami.form})) for some divergence-free vector field $w$. Then, the solution $u$ is divergence-free too. Hence, taking $\curl$ in the inhomogeneous Beltrami equation, we are led to the vector-valued Helmholtz equation
$$-(\Delta u+\lambda^2 u)=\lambda w+\curl w, \hspace{0.25cm} x\in \Omega.$$
Consider $K:=\supp w\subseteq \overline{\Omega}$ and notice that $\lambda w+\curl w$ is also compactly supported in $K$. Imagine that $u$ decayed as $\vert x\vert^{-(1+\varepsilon)}$ for some small $\varepsilon>0$. Hence, a straightforward computation leads to
$$\lim_{R\rightarrow +\infty}\int_{\partial B_R(0)}\vert u(x)\vert^2=0.$$
Consequently, Rellich's Lemma  \cite[Lemma 2.11]{ColtonKress2} would show that $u$ vanishes outside some sufficiently large ball centered at the origin and containing $K$. Then, the unique continuation principle of the Helmholtz equation allow proving that $u$ is also compactly supported in $K$ (see \cite{Littman} for the study of such property in many other linear PDEs with constants coefficients). In particular, $g$ would vanish outside $K\cap S$. In an equivalent way, the next result holds.

\begin{cor}\label{BeltramiNoHomog.CaidaOptima.cor}
Let $u\in C^{k+1,\alpha}(\overline{\Omega},\RR^3)$ be a solution to
$$\curl u-\lambda u=w,\hspace{0.25cm}x\in \Omega,$$
for a divergence-free compactly supported $w$ and some $\lambda\in\RR\setminus\{0\}$. If $u$ is transverse to $S$ at some point outside the support of $w$, then $u$ cannot decay faster than $\vert x\vert^{-1}$ at infinity, i.e., there exists no $\varepsilon>0$ such that
$$u=O(\vert x\vert^{-(1+\varepsilon)}), \hspace{0.25cm}\mbox{ when }\vert x\vert\rightarrow +\infty.$$
\end{cor}

The above Corollary can be interpreted in two different ways. First, it establishes the optimal fall-off of a ``transverse'' strong Beltrami field ($w=0$). Second, it also deals with some kind of ``transverse'' generalized Beltrami fields in exterior domains ($w=\varphi u$) that will be of a great interest in our work. We restrict to the second result since it contains the first one as a particular case.

\begin{cor}\label{BeltramiNoHomog.CaidaOptima.GeneralizedBeltramis.cor}
Let $u\in C^{k+1,\alpha}(\overline{\Omega},\RR^3)$ be a generalized Beltrami field, i.e.,
$$\left\{\begin{array}{ll}
\curl u-f u=0,\ & x\in \Omega,\\
\divop u=0, & x\in\Omega,
\end{array}\right.$$
whose proportionality factor is a compactly supported perturbation of a constant proportionality factor $\lambda\in\RR\setminus\{0\}$, i.e., $f=\lambda+\varphi$ for some $\varphi\in C^{k,\alpha}_c(\overline{\Omega})$. If $u$ is transverse to $S$ at some point outside the support of the perturbation $\varphi$, then $u$ cannot decay faster than $\vert x\vert^{-1}$ at infinity, i.e., there exists no $\varepsilon>0$ such that
$$u=O(\vert x\vert^{-(1+\varepsilon)}),\hspace{0.25cm}\mbox{ when }\vert x\vert\rightarrow +\infty.$$
\end{cor}

\begin{rem}\label{BeltramiNoHomog.CaidaOptima.RelacionNadirashvili.rem}
In particular, the above result leads to the natural counterpart for exterior domain of the Liouville theorem in \cite{CC,Nadirashvili} about the fall-off of entire generalized Beltrami fields. Such theorem states that there is no globally defined generalized Beltrami field decaying faster than $\vert x\vert^{-1}$ at infinity. As many others Liouville type results, it strongly depends on the solution being defined in the whole $\RR^3$. In our case we remove this hypothesis but, in return, we need to argue with generalized Beltrami fields with constant proportionality factor outside a compact set enjoying some trasversality condition on the boundary surface of the exterior domain.
\end{rem}

\section{An iterative scheme for strong Beltrami fields}\label{Esquema.Iterativo.Seccion}

Our objective in this section is to set the iterative scheme that we will use to establish the partial stability of strong Beltrami fields that will yield the existence of almost global Beltrami fields with a non-constant factor and complex vortex structures.

\subsection{Further notation and preliminaries} We devote a few lines to introduce some notation that will be in continuous use in the rest of the paper. Although most of the results are classical  \cite{Gilbarg}, others are inspired in \cite{Kaiser}, where they have been used in the electromagnetic framework.  

On the differentiable surface~$S$, we will consider local charts of the same regularity as $S$ (that is, maps $\mu$ covering open subsets $\Sigma\subseteq S$ of the form
$$\begin{array}{cccc}\mu: & D & \longrightarrow & \RR^3\\
 & s & \longmapsto & \mu(s),
\end{array}$$ 
where $\mu(D)=\Sigma$ and $D$ is a disk in the plane). Without any loss of generality, we will assume $\mu$ to be a local parametrization up to the boundary so that $\mu$ can be homeomorphically extended to the closure $\overline{D}$, $\overline{\Sigma}=\mu(\overline{D})$. 

We will also consider the corresponding $C^k$ and $C^{k,\alpha}$ spaces of functions defined on a coordinate neighborhood $\Sigma$ of $S$ provided with a local chart $\mu$. Up to the degree of smoothness of the surface, by compactness they are known to be independent of the choice of the chart, so one can write
\begin{align*}
 C^k(\Sigma)&:=\{f\in C^k(\Sigma):\,f\circ \mu\in  C^k(D)\},\\
C^{k,\alpha}(\Sigma)&:=\{f\in  C^k(\Sigma):f\circ \mu\in C^{k,\alpha}(D)\}
\end{align*}
and similarly for spaces on $\overline{\Sigma}$.
These spaces can be respectively endowed with the complete norms
$$
\Vert f\Vert_{ C^k(\Sigma)}:=\Vert f\circ\mu\Vert_{ C^k(D)},\qquad \Vert f\Vert_{C^{k,\alpha}(\Sigma)}:=\Vert f\circ\mu\Vert_{C^{k,\alpha}(D)}.
$$

Let us consider a $C^{k,\alpha}$ surface. An useful result is \textit{Calder\'on's extension theorem} for $C^{k,\alpha}$ functions, see e.g.\ \cite[Lemma 6.37]{Gilbarg}:

\begin{pro}\label{ExtensionHolder.pro}
Let $O\subseteq\RR^3$ be a $C^{k,\alpha}$ domain with bounded boundary $\partial O$, and let $O'$  be any open subset such that $\overline{O}\subseteq O'$. Then, there exists a linear operator
$$
\begin{array}{cccc}
\mathcal{P}: & C^{k,\alpha}(\overline{O}) & \longrightarrow & C^{k,\alpha}(\overline{O'})\\
 & f & \longmapsto & P(f)\equiv\overline{f},
\end{array}
$$
such that
\begin{enumerate}
\item $\mathcal{P}$ is an extension operator, i.e.,
$\left.P(f)\right\vert_{O}=f,\ \forall f\in C^{k,\alpha}(\overline{O}).$
\item The support of $\mathcal{P}(f)$ is contained in the open subset $O'$ for evey $f\in C^{k,\alpha}(\overline{\Omega})$.
\item $\mathcal{P}$ is continuous in the $C^{k,\alpha}$ topology, i.e.,
$$
\Vert \mathcal{P}(f)\Vert_{C^{k,\alpha}(O')}\leq C_\mathcal{P}\Vert f\Vert_{C^{k,\alpha}(O)},\ \forall f\in C^{k,\lambda}(\overline{O}).
$$
\item $\mathcal{P}$ is also continuous in the $ C^{m}$ topology for any $0\leq m\leq k$, i.e.,
$$
\Vert \mathcal{P}(f)\Vert_{ C^{m}(O')}\leq C_\mathcal{P}\Vert f\Vert_{ C^{m}(O)},\ \forall f\in C^{k,\alpha}(\overline{O}).
$$
\end{enumerate}
In the above inequalities, $C_\mathcal{P}$ stands for a constant which depends on $k,O$ and $O'$.
\end{pro}

To describe the stream lines and tubes associated with a velocity field $u\in C^{k+1,\alpha}(\overline{\Omega},\RR^3)$ in presence of a boundary surface which $u$ is not tangent to, it is convenient to consider an extension of the field to trivially obtain the following characterization from the classical Picard--Lindel\"{o}f theorem for ODEs on H\"older spaces:

\begin{pro}\label{FlujoHolder.pro}
Let $O\subseteq \RR^3$ be a $C^{k+1,\alpha}$ bounded domain, where $k\geq 0$ and $0<\alpha\leq 1$. Consider any vector field $u\in C^{k+1,\alpha}(\overline{O},\RR^3)$, its associated extension $\overline{u}=\mathcal{P}(u)\in C^{k+1,\alpha}(\RR^3,\RR^3)$ according to Proposition \ref{ExtensionHolder.pro}, any point $x_0\in \RR^3$ and an initial time $t_0\in\RR$. Consider the associated characteristic system
\begin{equation}\label{PVIFlujo.eq}
\left\{\begin{array}{ll}
\displaystyle\frac{dX}{dt}=\overline{u}(X), & t\in \RR,\\
\displaystyle X(t_0)=x_0. & 
\end{array}\right.
\end{equation}
Then, such problem is uniquely and globally (in time) solvable, its solution will be denoted $X(t;t_0,x_0)$, $X(t;t_0,\cdot)$ is a $C^{k+1}$ global diffeomorphism of the Euclidean space for every $t,t_0\in\RR$ and its inverse is $X(t_0;t,\cdot)$. The solutions to these problems are the stream lines of the extended velocity field $\overline{u}$.

Consider any $x_0\in \overline{O}$ and let $T(x_0)\geq 0$ be the greatest time for which the stream line $X(t;0,x_0),\ t>0$ remains inside the open subset $O$, i.e.,
$$T(x_0):=\sup\{T>0:\,X(t;0,x_0)\in O\ \forall\,t\in(0,T)\}.$$
Then, $X(t;0,x_0),\ 0<t<T(x_0)$ is a stream line of $u$, or equivalently, it solves the ODE
\begin{equation*}
\left\{\begin{array}{ll}
\displaystyle\frac{dX}{dt}=u(X), & 0<t<T(x_0),\\
\displaystyle X(0)=x_0. & 
\end{array}\right.
\end{equation*}
Notice that when $X(t;0,x_0)\notin \overline{O},\ \forall\,t\in(0,T)$ for some $T>0$, then $T(x_0)=0$, i.e., the corresponding stream line of $\overline{u}$ does not originally enter the region $O$.
\end{pro}

We will also consider stream tubes of a velocity field which emanate from an open subset of the surface $S$. Consider any vector field $u\in C^{k+1,\alpha}(\overline{\Omega},\RR^3)$, $\overline{u}=\mathcal{P}(u)\in C^{k+1,\alpha}(\RR^3,\RR^3)$ its extension according to Calder\'on's extension theorem, $X(t;t_0,x_0)$ its associated flux mapping through Proposition \ref{FlujoHolder.pro} and an open subset $\Sigma\subseteq S$ together with a local chart $\mu:D\longrightarrow S$. The stream tube of $u$ which emanates from $\Sigma$ is the collection of all stream lines of $u$ radiating from the points in the open subset $\Sigma$, i.e.,
$$
\mathcal{T}(\Sigma,u):=\{X(t;0,\mu(s)):\,s\in D,\ 0<t<T(\mu(s))\}.
$$
It is also useful to consider bounded stream lines with ``height'' $T>0$
$$
\mathcal{T}(\Sigma,u,T):=\{X(t;0,\mu(s)):\,s\in D,\ 0<t<\min\{T,T(\mu(s))\}\}.
$$
\begin{figure}[t]
\centering
\includegraphics[scale=0.8]{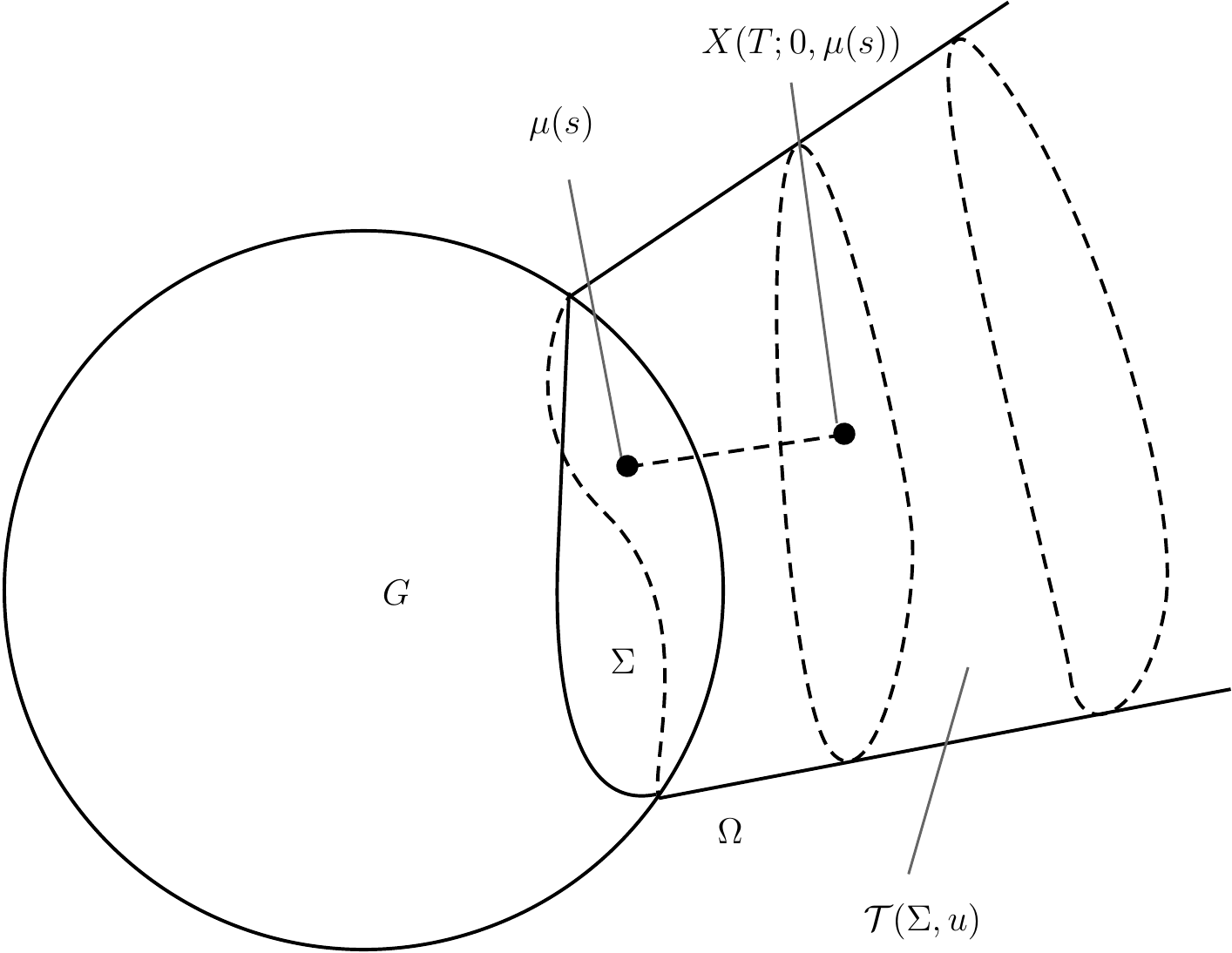}
\caption{Stream lines and tubes of the velocity field $u$.}
\label{fig:Fig01}
\end{figure}
Notice that in order for a stream line of $u$ to be well defined, it is necessary that the velocity field points towards the exterior domain. The same condition leads to well defined stream tubes emanating from $\Sigma$. The regularity in the preceding result follows from Peano's differentiability theorem. The same regularity result may be used in order to derive the regularity in the stream tubes parametrization.

\begin{pro}\label{TuboFlujo.parametr.pro}
Consider $G,\Sigma,$ and $\mu$ verying the hypothesis (\ref{GSigmaMu.hipot}), $u\in C^{k+1,\alpha}(\overline{\Omega},\RR^3)$ be a velocity field in the exterior domain, and assume that the vector field $u$ points towards the exterior domain at any point of $\Sigma$, i.e., there exits a positive $\rho_0>0$ such that $u\cdot \eta\geq \rho_0$ on $\Sigma$. Then, a well defined stream line of $u$ emanates from each point of $\Sigma$ and they smoothly foliate the whole stream tube $\mathcal{T}(\Sigma,u)$. To make this statement more precise, let us define
$$\mathcal{D}(\Sigma,u):=\{(t,s):\,s\in D,\ 0<t<T(\mu(s))\},$$
and the mapping
$$\begin{array}{cccl}
\phi: & \mathcal{D}(\Sigma,u) & \longrightarrow & \mathcal{T}(\Sigma,u)\\
 & (t,s) & \longmapsto & \phi(t,s):=X(t;0,\mu(s)).
\end{array}$$
Then,
\begin{enumerate}
\item $T(\mu(s))>0$, for each $s\in D$.
\item $\phi$ is bijective.
\item $\phi$ is a $C^{k+1}$ diffeomorphism.
\item $\jac(\phi)$ and $\jac(\phi)^{-1}$ belongs to $C^{k,\alpha}$ locally in $t$, i.e., there exists a function $\kappa:\RR_0^+\times\RR_0^+\longrightarrow\RR_0^+$ which is increasing with respect to each variable, such that if one defines
$$\mathcal{D}(\Sigma,u,T):=\left\{(t,s):\,s\in D,\ 0<t<\min\{T,T(\mu(s))\}\right\}$$
and the mapping
$$\phi_{\vert \mathcal{D}(\Sigma,u,T)}:\mathcal{D}(\Sigma,u,T)\longrightarrow \mathcal{T}(\Sigma,u,T),$$
then, 
$$\Vert \jac(\phi)\Vert_{C^{k,\alpha}(\overline{\mathcal{D}(\Sigma,\mu,T))}},\Vert \jac(\phi)^{-1}\Vert_{C^{k,\alpha}(\overline{\mathcal{T}(\Sigma,\mu,T))}}\leq \kappa\left(\Vert u\Vert_{C^{k+1,\alpha}(\overline{\Omega})},T\right),$$
for every positive number $T$.
\end{enumerate}
\end{pro}

\begin{proof}
We sketch the proof of this result for the reader's convenience (see \cite[Lemma 5.1]{Kaiser} for $k=0$ and in~\cite[Proposici\'on 2.1.7]{Poyato} for arbitrary $k$).
The first assertion is apparent: since $u$ points outwards at any point in $\Sigma$, then the stream line of $u$ arising from $\mu(s)$ points towards $\Omega$ at $t=0$. Hence, a small piece of such stream line must stay in $\Omega$. Regarding the second assertion, $\phi$ is clearly onto by virtue of the definition of $\mathcal{T}(\Sigma,u)$. To check that $\phi$ is one to one, note that different stream lines cannot touch because of the uniqueness part in Proposition \ref{FlujoHolder.pro}, and that the streamlines of $u$ emerging from $\Sigma$ cannot be closed loops because $u$ points outwards at $\Sigma$.

The $C^{k+1}$ regularity of $\phi$ is clear because so is $X(t;t_0,x_0)$ by \textit{Peano's differentiability theorem} as stated in Proposition \ref{FlujoHolder.pro}. Let us show that its Jacobian matrix is regular at any point in $\mathcal{D}(\Sigma,u)$ to obtain the same regularity of $\phi^{-1}$ through the inverse mapping theorem. This matrix takes the form
$$\jac(\phi)(t,s)=\left(\left.\frac{\partial \phi}{\partial t}(t,s)\right\vert\left.\frac{\partial\phi}{\partial s_1}(t,s)\right\vert\left.\frac{\partial\phi}{\partial s_2}(t,s)\right.\right).$$
For fixed $s\in D$, each column vector is a solution to the linear ODE
$$\dot{x}(t)=\jac(u)(\phi(t,s))x(t).$$
Thus, $\jac(\phi)(\cdot,s)$ is a solution matrix to such linear ODE, whose determinant at $t=0$ equals
\begin{equation}\label{FlujoHolder.cota.determinante}
\det(\jac(\phi)(0,s))=\left\vert\frac{\partial \mu}{\partial s_1}(s)\times \frac{\partial\mu}{\partial s_2}(s)\right\vert u(\mu(s))\cdot\eta(\mu(s))\geq \rho_1\rho_0>0.
\end{equation}
Here $\rho_1$ stands for any positive uniform lower bound of the first factor. Thus, $\jac(\phi)(t,s)$ is regular for all $t$ by the \textit{Jacobi--Liouville formula}. In particular, the derivatives of $\jac(\phi)$ and $\jac(\phi)^{-1}$ up to order $k$ can be continuously extended to $\overline{\mathcal{D}(\Sigma,u,T)}$ by the analogous properties of $u$ and $\mu$. 

Let us finally recursively show that all of them are bounded and the $k$-th order ones are $\alpha$-H\"{o}lder continuous indeed. First, notice that
\begin{align*}
\left\vert\frac{\partial \phi}{\partial t}(t,s)\right\vert&\leq \Vert u\Vert_{C^0(\Omega)},\\
\left\vert \frac{\partial \phi}{\partial s_i}(t,s)\right\vert&\leq \left\vert\frac{\partial \mu}{\partial s_i}(s)\right\vert+\int_0^t\Vert \jac(u)\Vert_{C^0(\Omega)}\left\vert\frac{\partial \phi}{\partial s_i}(\tau,s)\right\vert\,d\tau,
\end{align*}
for every $(t,s)\in\mathcal{D}(\Sigma,u,T)$. As a consequence, Gronwall's lemma amounts to the upper bound
$$\Vert \jac(\phi)\Vert_{C^0(\mathcal{D}(\Sigma,u,T))}\leq \Vert u\Vert_{C^0(\Omega)}+\Vert \mu\Vert_{C^1(D,\RR^3)}e^{T\Vert \jac(u)\Vert_{C^0(\Omega)}}\leq \kappa(\Vert u\Vert_{C^{k+1,\alpha}(\Omega)},T),$$
for some function $\kappa$ which is separately increasing. 

Assume now that the analogous estimate
\begin{equation}\label{FlujoHolder.estim.Bm}
\Vert \jac(\phi)\Vert_{C^m(\mathcal{D}(\Sigma,u,T))}\leq \kappa(\Vert u\Vert_{C^{k+1,\alpha}(\Omega)},T),
\end{equation}
holds true for some $n$ such that $0<n\leq k$ and all $m$ with $0\leq m<n$ and let us prove it for $m=n$. Fix any multi-index $\gamma$ such that $\vert\gamma\vert=n$ and take derivatives on the characteristic system (\ref{PVIFlujo.eq}) to arrive at
$$D^\gamma\left(\frac{\partial\phi_i}{\partial t}\right)=D^\gamma(u_i(\phi(t,s)))=\gamma!\sum_{(l,\beta,\delta)\in \mathcal{D}(\gamma)}(D^\delta u_i)(\phi(t,s))\prod_{r=1}^l\frac{1}{\delta_r!}\left(\frac{1}{\beta_r!}D^{\beta_r}\phi(t,s)\right)^{\delta_r}.$$
The above formula is nothing but a chain rule for high order partial derivatives of a composition function. Here, $\mathcal{D}(\gamma)$ stands for the set of all the possible decompositions of $\gamma$
$$\gamma=\sum_{r=1}^l\vert \delta_r\vert \beta_r,$$
where $\delta_r,\beta_r$ are multi-indices, $\delta:=\sum_{r=1}^l\delta_r$ and for every $r=1,\ldots,l-1$ there exists some $i_r\in\{1,2,3\}$ such that
$(\beta_r)_i=(\beta_{r+1})_i$ for every $i\neq i_r$ and $(\beta_r)_{i_r}<(\beta_{r+1})_{i_r}.$
Similarly
$$\frac{\partial}{\partial t}D^\gamma\left(\frac{\partial \phi_i}{\partial s_j}\right)=\sum_{q=1}^l\sum_{\rho\leq \gamma}\sum_{(l,\beta,\delta)\in\mathcal{D}(\rho)}\binom{\gamma}{\rho}\rho!\left(D^\delta\frac{\partial u_i}{\partial x_q}\right)(\phi(t,s))\prod_{r=1}^l\frac{1}{\delta_r!}\left(\frac{1}{\beta_r!}D^{\beta_r}\phi(t,s)\right)^{\delta_r}D^{\gamma-\rho}\frac{\partial\phi_k}{\partial s_j}(t,s).$$
Notice that the first derivative formula only involves derivatives of $\phi(t,s)$ and $u$ up to order $n$. Regarding the second formula, the only term involving a derivative of $\phi(t,s)$ of order $n+1$ is the associated with the multi-index $\rho=0$. Hence, the next estimates hold true by virtue of (\ref{FlujoHolder.estim.Bm})
\begin{align*}
\left\vert D^\gamma\left(\frac{\partial\phi_i}{\partial t}\right)(t,s)\right\vert&\leq \kappa(\Vert u\Vert_{C^{k+1,\alpha}(\Omega)},T),\\
\left\vert D^\gamma\left(\frac{\partial\phi_i}{\partial s_j}\right)(t,s)\right\vert&\leq \kappa(\Vert u\Vert_{C^{k+1,\alpha}(\Omega)},T)\sum_{q=1}^3\left(1+\int_0^t\left\vert D^\gamma\left(\frac{\partial \phi_q}{\partial s_j}\right)(\tau,s)\right\vert\,d\tau\right),
\end{align*}
for every $(t,s)\in \mathcal{D}(\Sigma,u,T)$. Again, Gronwall's lemma shows that (\ref{FlujoHolder.estim.Bm}) holds true when $m=n$.

Finally, let us obtain the aforementioned $\alpha$-H\"{o}lder estimate of the higher order derivatives of $\jac(\phi)$. To this end, take any column vector $x^j(t,s)$ of the Jacobian matrix $\jac(\phi)(t,s)$ and note that when $\gamma=(\gamma_1,\gamma_2,\gamma_3)$ is a multi-index of the highest order $k$, then all the preceding derivative formulas can be added up to obtain the PDE
\begin{align}\label{FlujoHolder.derivadak}
\begin{split}
\frac{\partial}{\partial t}&D^\gamma x_i^j(t,s)=\sum_{q=1}^3A_{q}^{i,j}(\gamma)\frac{\partial u_i}{\partial x_q}(\phi(t,s))D^\gamma x_q^j(t,s)+F_i(t,s)\\
&+\sum_{\beta\in \Gamma_\gamma}\sum_{(j_1,\ldots,j_{k+1})\in J_\gamma}\sum_{(i_1,\ldots,i_{k+1})\in I_\gamma}B_{i_1,\ldots,i_{k+1}}^{j_1,\ldots,j_{k+1}}(\beta)(D^\beta u_i)(\phi(t,s))x_{i_1}^{j_1}(\phi(t,s))\cdots x_{i_{k+1}}^{j_{k+1}}(\phi(t,s)).
\end{split}
\end{align}
Here $A_q^{i,j}(\gamma)$ and $B_{i_1,\ldots,i_{k+1}}^{j_1,\ldots,j_{k+1}}(\beta)$ denote nonnegative constant coefficients and $F_i(t,s)$ consists of finitely many sums and products of both derivatives of $u$ up to order $k$ and derivatives of $\phi$ up to order $k$. Furthermore, $\Gamma_\gamma$ is a set of $3$-multi-indices with order $k+1$ depending on $\gamma$ and $I_\gamma,J_\gamma$ are sets of $(k+1)$-multi-indices also depending on $\gamma$. 

Let us first prove the $\alpha$-H\"{o}lder continuity in the variable $s$ using the integral version of the above equation. Specifically, take $s_1,s_2\in D$, $t\in (0,T)$ and notice that
$$D^\gamma x_i^j(t,s_1)-D^\gamma x_i^j(t,s_2)=I+II+III+IV,$$
where
\begin{align*}
I&:=D^\gamma x_i^j(0,s_1)-D^\gamma x_i^j(0,s_2),\\
II&:=\int_0^t (F_i(\tau,s_1)-F_i(\tau,s_2))\,d\tau,\\
III&:=\sum_{q=1}^3A_q^{i,j}(\gamma)\int_0^t\left(\frac{\partial u_i}{\partial x_q}(\phi(\tau,s_1))D^\gamma x_q^j(\tau,s_1)-\frac{\partial u_i}{\partial x_q}(\phi(\tau,s_2))D^\gamma x^j_q(\tau,s_2)\right)\,d\tau,\\
IV&:=\sum\limits_{\substack{\beta\in \Gamma_\gamma\\(i_1,\ldots,i_{k+1})\in I_\gamma\\(j_1,\ldots,j_{k+1})\in J_\gamma}}B_{i_1,\ldots,i_{k+1}}^{j_1,\ldots,j_{k+1}}(\beta)\int_0^t\left.(D^\beta u_i)(\phi(\tau,s))x_{i_1}^{j_1}(\phi(\tau,s))\cdots x_{i_{k+1}}^{j_{k+1}}(\phi(\tau,s))\right\vert_{s_1}^{s_2}\,d\tau.
\end{align*}

Regarding the terms $I,II$, one can easily see that
$$I\leq \kappa\left(\Vert u\Vert_{C^{k+1,\alpha}(\Omega)},T \right)\vert s_1-s_2\vert^\alpha,\hspace{0.5cm} II\leq T\kappa(\Vert u\Vert_{C^{k+1,\alpha}(\Omega)},T)\vert s_1-s_2\vert^\alpha.$$
In the first case, the estimate obviously follows from the regularity of $\mu$ in the particular case when $D^\gamma x_i^j$ involves no derivative of $\phi(t,s)$ with respect to $t$. A straightforward recursive argument on the order of the derivatives with respect to $t$ yields the general assertion. The second case is obvious by the definition of $F_i(t,s)$ and the mean value theorem. 
Furthermore, adding and subtracting crossed terms in $III$, it is clear that it can be bounded by the mean value theorem as
$$III\leq \kappa(\Vert u\Vert_{C^{k+1,\alpha}(\Omega)},T)\vert s_1-s_2\vert^\alpha+\kappa\left(\Vert u\Vert_{C^{k+1,\alpha}(\Omega)},T\right)\sum_{q=1}^3\int_0^t\vert D^\gamma x_q^j(\tau,s_1)-D^\gamma x_q^j(\tau,s_2)\vert\,d\tau.$$

So far, only low order derivatives of $u$ have being involved, and therefore the mean value theorem has sufficed to obtain Lipschitz conditions of such derivatives (terms $I,II$ and $III$). In contrast, $IV$ contains the derivatives of $u$ of the highest order, $k+1$. Since they cannot be handled again by the mean value theorem, then the $\alpha$-H\"{o}lder continuity of $D^{k+1}u$ must be used. By appropriately adding and subtracting crossed terms, using the above-mentioned H\"{o}lder continuity of $D^\beta u_i$ on the first factor and the mean value theorem on the second one, one easily obtains the upper bound
\begin{multline*}
IV\leq \sum_{\beta\in \Gamma_\gamma}\sum_{(j_1,\ldots,j_{k+1})\in J_\gamma}\sum_{(i_1,\ldots,i_{k+1})\in I_\gamma}B_{i_1,\ldots,i_{k+1}}^{j_1,\ldots,j_{k+1}}(\beta)\\
\times T\left(\kappa(\Vert u\Vert_{C^{k+1,\alpha}(\Omega},T)^{\alpha+k+1}[D^\beta u_i]_{\alpha,\Omega}\vert s_1-s_2\vert^\alpha+\kappa(\Vert u\Vert_{C^{k+1,\alpha}(\Omega},T)^{k+2}\Vert D^\beta u_i\Vert_{C^0(\Omega)}\vert s_1-s_2\vert\right).
\end{multline*}

To conclude, let us combine all the above estimates and use Gronwall's lemma to arrive at
$$\vert D^\gamma x_i^j(t,s_1)-D^\gamma x_i^j(t,s_2)\vert\leq \kappa(\Vert u\Vert_{C^{k+1,\alpha}(\Omega},T)\vert s_1-s_2\vert^\alpha,$$
for an appropriately function $\kappa$. Regarding the $\alpha$-H\"{o}lder condition in the variable $t$, one only needs to note that $\partial_t D^\gamma x_i^j(t,s)$ is uniformly bounded by virtue of (\ref{FlujoHolder.derivadak}).

Finally, note that
$$\jac(\phi)^{-1}=\frac{1}{\det(\jac(\phi))}\left(\left.\frac{\partial\phi}{\partial s_1}\times\frac{\partial\phi}{\partial s_2}\,\right\vert\left.\frac{\partial\phi}{\partial s_2}\times\frac{\partial\phi}{\partial t}\,\right\vert\left.\frac{\partial\phi}{\partial t}\times\frac{\partial\phi}{\partial s_1}\right.\right)^\intercal,$$
and that the Jacobi--Liouville formula along with the lower bound in (\ref{FlujoHolder.cota.determinante}) yield a uniform lower bound for the Jacobian determinant:
$$\det(\jac(\phi)(t,s))=\det(\jac(\phi)(0,s))\exp\left(\int_0^t\mbox{Tr}(\jac(u)(\phi(\tau,s)))\,d\tau\right)\geq \rho_0\rho_1\exp\left(-3T\Vert \jac(u)\Vert_{C^0(\Omega)}\right).$$
Hence, the $C^{k+1,\alpha}(\mathcal{D}(\Sigma,u,T))$ estimate for $\jac(\phi)^{-1}$ easily follows from that of $\jac(\phi)$.
\end{proof}

The analysis in  the next sections requires stream tubes of $u$ that are bounded and have both ends on~$S$. These structures were considered (although its existence was not proved) in \cite{Kaiser}. In our setting, we will say that the stream tube of $u$ arising from $\Sigma$ is a $(\rho_0,T,\delta)$-stream tube of $u$ when the previous two conditions hold, i.e.,
\begin{itemize}
\item $u\cdot\eta\geq \rho_0$ on $\Sigma$.
\item For every $s\in D$ there exists two associated positive numbers $0<T_0(s),T_\delta(s)<\frac{T}{2}$ such that $X(T_0(s);0,\mu(s))\in S$ and $X(T_\delta(s);0,\mu(s))\in S_\delta$.
\end{itemize}
Here $\rho_0,T,\delta$ are positive constants which measure the initial angle of the streams lines over $\Sigma$, the time at which the whole tube has returned to the surface and the depth that the stream lines achieve into the interior domain $G$, while $S_\delta$ stands for the boundary of the subdomain of $G$ made of the points in $G$ at distance at least $\delta$ from $S$, i.e., $G_\delta:=\{x\in G:\,\mbox{dist}(x,S)>\delta\}$ (see Figure \ref{fig:Fig02}.)
\begin{figure}[t]
\centering
\includegraphics[scale=0.6]{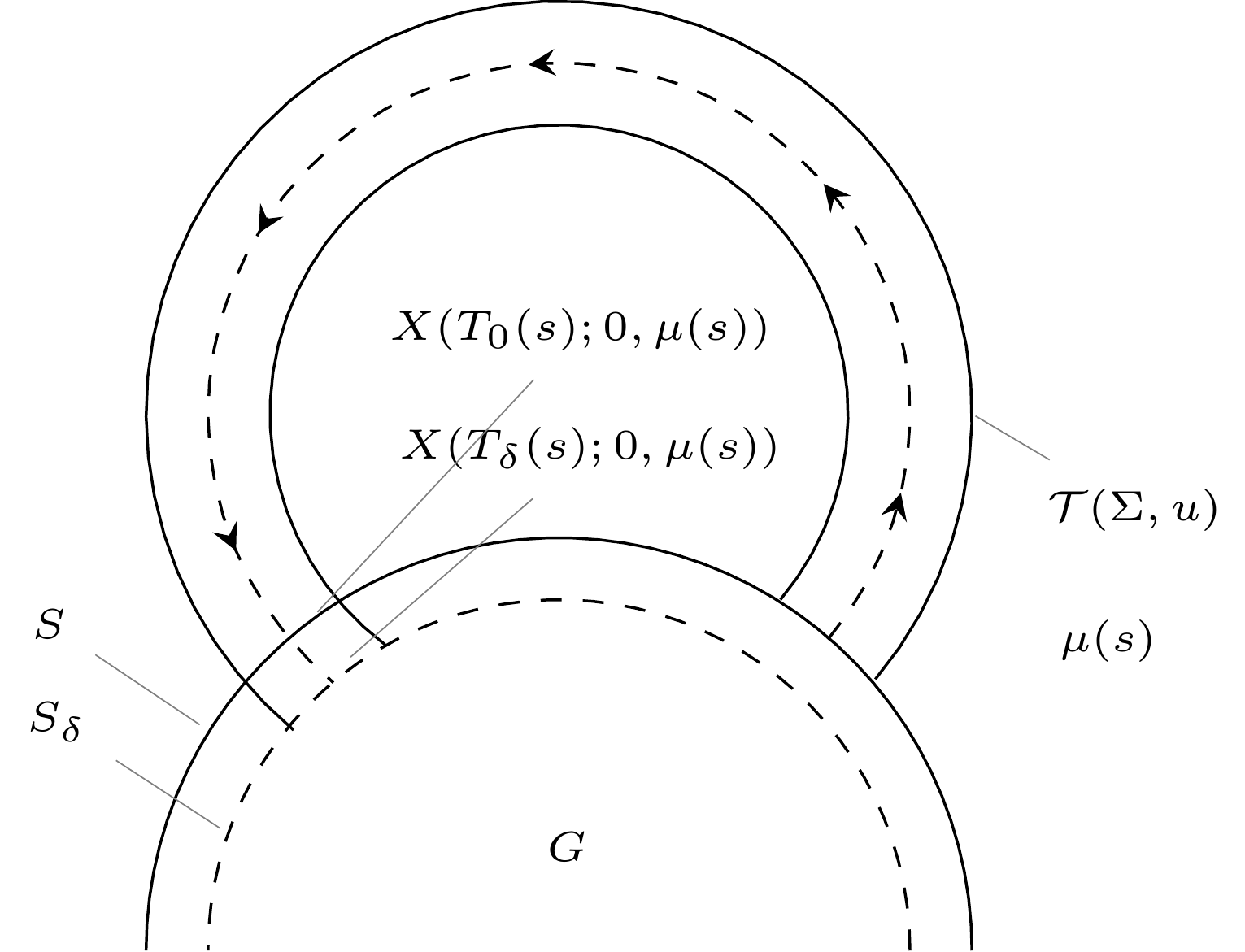}
\caption{$(\rho_0,T,\delta)$-stream tube of $u$.}
\label{fig:Fig02}
\end{figure}

Since a stream tube consists of integral curves, it is elementary that the diameter of a $(\rho_0,T,\delta)$-stream tube is bounded in terms of the sup norm of the vector field, the flow time~$T$ and the diameter at time~0 as
\begin{equation}\label{diametro.tuboflujo.form}
\mbox{diam}(\mathcal{T}(\Sigma,u))\leq T\Vert u\Vert_{ C^0(\Omega)}+\mbox{diam}(\Sigma).
\end{equation}
(A detailed proof of this can be found in \cite[Lemma 4.6]{Kaiser}).
In a similar way, \cite[Lemma 4.7]{Kaiser} provides a criterion to obtain  ``almost'' $(\rho_0,T,\delta)$-stream tubes for velocity fields which are ``close enough'' to any other given velocity field enjoying this kind of stream tubes. This merely asserts that, as is well known, a $C^0$-small perturbation of the initial vector field will not prevent the integral curves of the perturbed field from intersecting a surface to which the initial flow was transverse. This can be quantitatively written as follows:

\begin{lem}\label{TuboFlujoRecurr.Perturb.lem}
Let $G,\Sigma,\mu$ verify (\ref{GSigmaMu.hipot}) and consider $u_1,u_2\in C^{k+1,\alpha}(\overline{\Omega},\RR^3)$. Define $\mathcal{T}_i:=\mathcal{T}(\Sigma,u_i)$ its stream tubes emanating from $\Sigma$ and assume that $\mathcal{T}_1$ is a $(\rho_0,T,\delta)$-stream tube of $u_1$ and
\begin{enumerate}
\item $u_1\cdot \eta=u_2\cdot \eta$ on $\Sigma$.
\item $u_1$ and $u_2$ are ``close enough'' in $ C^0(\Omega)$ norm. Specifically, assume
$$\Vert u_1-u_2\Vert_{ C^0(\Omega)}< 2\frac{(1-\theta)\delta}{C_\mathcal{P}T}e^{-\frac{1}{2}C_\mathcal{P}T\Vert u_1\Vert_{ C^1(\Omega)}},$$
for some $0<\theta<1$.
\end{enumerate}
Then, $\mathcal{T}_2$ is also a $(\rho_0,T,\theta\delta)$-stream tube of $u_2$.
\end{lem}

\subsection{Iterative scheme} In this section we discuss the iterative scheme that is typically used in the literature to obtain nonlinear force-free fields in the magnetohydrodynamical setting as small perturbations of harmonic fields: the \textit{Grad--Rubin iterative method} (see, the review \cite{WiegelmannSakurai}). It goes back to 1958, when it was originally proposed by H. Grad and H. Rubin in connection with applications to plasma physics.  A numerical implementation in the context of coronal magnetic fields is introduced in \cite{Amari2}, where the Grad--Rubin method was obtained through the decomposition of the Beltrami equation with small proportionality factor $f$ into a hyperbolic part, which transports the proportionality factor $f$ along the magnetic field lines, and an elliptic one, to correct the magnetic field step by step using Ampere's law. 

From a mathematical point of view, this method was used in \cite{Bineau} to obtain small perturbations of harmonic fields in bounded domains, leading to generalized Beltrami fields with small non-constant proportionality factors. 

Let us assume that $G,\Sigma,\mu$ satisfies the hypotheses (\ref{GSigmaMu.hipot}) and consider any initial harmonic vector field $u_0$ in the exterior domain $\Omega$, i.e., a solution to the equation
\begin{equation}\label{CampoArmonico.form}
\left\{\begin{array}{ll}
\curl u_0=0, & x\in \Omega,\\
\divop u_0=0, & x\in \Omega,\\
\vert u_0(x)\vert\leq \frac{C}{\vert x\vert^{\rho-1}}, & x\in\Omega,
\end{array}\right.
\end{equation}
where $\rho\in(1,3)$ is a parameter which controls the decay at infinity of the vector field $u_0$. Dynamically, we will assume that $u_0$ has a tube that starts and ends on the boundary, i.e., that the field $u_0$ has a $(\rho_0,T,\delta)$-stream tube.

We want to known whether there exists a Beltrami field $u$ with ``small'' proportionality factor $f$ such that $u$ (resp.\ $f$) is close to $u_0$ (resp. $0$) and satisfies the exterior Neumann problem
\begin{equation}\label{BeltramiPerturbacionArmonico.form}
\left\{\begin{array}{ll}
\curl u=fu, & x\in \Omega,\\
\divop u=0, & x\in \Omega,\\
u\cdot \eta=u_0\cdot \eta, & x\in S,\\
\vert u(x)\vert\leq \frac{C}{\vert x\vert^{\rho-1}}, & x\in \Omega.
\end{array}\right.
\end{equation}
To solve this problem, the Grad--Rubin method analyzes the following iterative scheme:
\begin{equation}\label{BeltramiPerturbacionArmonico.EsquemaIterativo.form}
\left\{\begin{array}{ll}
\curl u_{n+1}=f_n u_n, & x\in \Omega,\\
\divop u_{n+1}=0, & x\in\Omega,\\
u_{n+1}\cdot \eta=u_0\cdot \eta, & x\in S,\\
\vert u_{n+1}(x)\vert\leq \frac{C}{\vert x\vert^{\rho-1}}, & x\in \Omega,
\end{array}\right.\hspace{1.5cm}\left\{\begin{array}{ll}
\nabla f_n\cdot u_n=0, & x\in \Omega,\\
f_n =f^0, & x\in \Sigma.
\end{array}\right.
\end{equation}

The strong convergence to a solution of (\ref{BeltramiPerturbacionArmonico.form}) was rigorously proved in \cite{Kaiser} for small enough prescriptions $f^0$ of the proportionality factor on $\Sigma$ and under the assumption that the field $u_0$ has a $(\rho_0,T,\delta)$-stream tube by providing the a priori bounds that were missing in \cite{Bineau}. The way to go is the following. First, it is easy to solve the steady transport equation in the right hand side when the stream tube of $u_n$ emanating from the open subset $\Sigma$ is a $(\rho_0,T,\delta)$-stream tube and the prescription $f^0$ of the proportionality factor has compact support inside of $\Sigma$. The transportation of $f^0$ along this stream tube leads to the existence and uniqueness of $f_n$ in $C^{1,\alpha}(\overline{\Omega})$ whenever $u_n$ lies in $C^{1,\alpha}(\overline{\Omega},\RR^3)$. Second, in order to solve the exterior inhomogeneous Neumann problem for the div-curl equation, 
a boundary integral equation method for harmonic fields was studied in \cite{Neudert}. Roughly speaking, the existence of solution to the Neumann problem was solved by means of the Helmholtz--Hodge's decomposition theorem for vector fields with suitable decay at infinity, splitting them into harmonic volume and single layer potentials depending on its divergence, its curl and its normal and tangential component. Although the tangential component $\eta\times u_+$ is not prescribed in the boundary data, it can be obtained from a boundary integral equation which can be deduced from \textit{jump relations} for the derivatives of harmonic single layer potentials (see \cite[Theorem 14.IV]{Miranda1}, \cite[Teorema 2.II]{Miranda2} or \cite[Theorem 2.17]{ColtonKress} for a proof of jump relation and \cite{Neudert,vonWahl} for a detailed study of the associated boundary integral equation). H\"{o}lder estimates for the solution follow from a careful study of the harmonic volume and single layer potentials in the decomposition \cite{ColtonKress,Heinemann,Miranda1,Miranda2}. Although only $C^{0,\alpha}$ proportionality factors and $C^{1,\alpha}$ magnetic fields were considered \cite{Kaiser}, a $C^{k,\alpha}$ theory can be obtained after studying higher order derivates of potentials. The only hypotheses that the data must satisfy to ensure the well-posedness of the inhomogeneous div-curl problems are:
\begin{itemize}
\item The inhomogeneous term $f_nu_n$ in the equation for $\curl u_{n+1}$ must have ``zero flux'' in $\Omega$. Specifically, the flux of $f_n u_n$ across any closed surface in $\Omega$ must vanish or equivalently $f_n u_n$ must be divergence-free in $\Omega$ and its flux across the surface $S$ has to be zero. This hypothesis can be deduced in each step from the corresponding hypothesis in the preceding step. As it is verified in the step $n=0$ because $u_0$ is divergence-free and $f_0$ verifies the steady transport equation, these hypothesis are satisfied for every step $n\in\NN$ in the iteration.
\item The right hand sides of the equations for $\divop u_{n+1}$ and $\curl u_{n+1}$ must decay at infinity as $\vert x\vert^{-\rho}$. This is easily checked since $f_n$ has compact support.
\end{itemize}

A natural question is to ascertain whether these results can be adapted to get perturbations of strong Beltrami fields with any constant proportionality factor $\lambda\neq 0$. We devote now some lines to explain why a direct application of the same Grad--Rubin method cannot work. Assume that $u_0$ is a strong Beltrami field with constant proportionality factor in the exterior domain $\Omega$. We will restrict ourselves to strong Beltrami fields $u_0$ with optimal decay at infinity, say $\vert x\vert^{-1}$ (see \cite{Enciso2,Nadirashvili} and compare with the $\vert x\vert^{-2}$ decay in the harmonic case).
\begin{equation}\label{CampoBeltramiFuerte.form}
\left\{\begin{array}{ll}
\curl u_0=\lambda u_0, & x\in \Omega,\\
\divop u_0=0, & x\in \Omega,\\
\vert u_0(x)\vert\leq \frac{C}{\vert x\vert}, & x\in\Omega.
\end{array}\right.
\end{equation}

Now, we would like to solve the next problem
\begin{equation}\label{BeltramiPerturbacion.form}
\left\{\begin{array}{ll}
\curl u=(\lambda+\varphi)u, & x\in \Omega,\\
\divop u=0, & x\in \Omega,\\
u\cdot \eta=u_0\cdot \eta, & x\in S,\\
\vert u(x)\vert\leq \frac{C}{\vert x\vert}, & x\in \Omega,
\end{array}\right.
\end{equation}
where $\varphi$ is a ``small'' perturbation of the constant proportionality factor $\lambda$. The same ideas as above lead to the next modification of the classical Grad--Rubin iterative method
\begin{equation*}
\left\{\begin{array}{ll}
\curl u_{n+1}=(\lambda+\varphi_n) u_n, & x\in \Omega,\\
\divop u_{n+1}=0, & x\in\Omega,\\
u_{n+1}\cdot \eta=u_0\cdot \eta, & x\in S,\\
\vert u_{n+1}(x)\vert\leq \frac{C}{\vert x\vert}, & x\in \Omega,
\end{array}\right.\hspace{1.5cm}\left\{\begin{array}{ll}
\nabla \varphi_n\cdot u_n=0, & x\in \Omega,\\
\varphi_n =\varphi^0, & x\in \Sigma.
\end{array}\right.
\end{equation*}
Although the steady transport system in the right hand side can be solved in exactly the same way, we cannot ensure now the previous hypothesis leading to the well-posedness of the exterior inhomogeneous div-curl problems. On the one hand, the vanishing flux hypothesis follows from the same argument as above. On the other hand, $(\lambda+\varphi_0)u_0$ should decay as $\vert x\vert^{-2}$ in order for the solution $u_1$ to decay as $\vert x\vert^{-1}$. Unfortunately, it is not possible when $\lambda\neq 0$ because $u_0$ has optimal decay $\vert x\vert^{-1}$.

To solve this problem, we move the term $\lambda u$ in the equation for $\curl u$ from the inhomogeneous side, to the homogeneous one arriving at
\begin{equation*}
\left\{\begin{array}{ll}
\curl u_{n+1}-\lambda u_{n+1}=\varphi_n u_n, & x\in \Omega,\\
\divop u_{n+1}=0, & x\in\Omega,\\
u_{n+1}\cdot \eta=u_0\cdot \eta, & x\in S,\\
\vert u_{n+1}(x)\vert\leq \frac{C}{\vert x\vert}, & x\in \Omega,
\end{array}\right.\hspace{1.5cm}\left\{\begin{array}{ll}
\nabla \varphi_n\cdot u_n=0, & x\in \Omega,\\
\varphi_n =\varphi^0, & x\in \Sigma.
\end{array}\right.
\end{equation*}
However, as it will be shown in the next section, the exterior inhomogeneous Beltrami equation for divergence-free vector fields is an overdetermined system in general. Notice that when one computes the divergence in the first equation and assumes $\lambda\neq 0$, one recovers $\divop u_{n+1}$ from the first equation
$$
\divop u_{n+1}=-\frac{1}{\lambda}\divop(\varphi_nu_n).
$$
Therefore, it is an easy task to check that as soon as $u_0$ is divergence-free and $\varphi_n$ is a fist integral of $u_n$, then $u_{n+1}$ is also divergence-free in each step of the iteration. Consequently, we can simplify the previous overdetermined systems by removing the divergence-free conditions.
\begin{equation}\label{BeltramiPerturbacion.EsquemaIterativo.fom}
\left\{\begin{array}{ll}
\curl u_{n+1}-\lambda u_{n+1}=\varphi_n u_n, & x\in \Omega,\\
u_{n+1}\cdot \eta=u_0\cdot \eta, & x\in S,\\
\vert u_{n+1}(x)\vert\leq \frac{C}{\vert x\vert}, & x\in \Omega,
\end{array}\right.\hspace{1.5cm}\left\{\begin{array}{ll}
\nabla \varphi_n\cdot u_n=0, & x\in \Omega,\\
\varphi_n =\varphi^0, & x\in \Sigma.
\end{array}\right.
\end{equation}
The stationary problem along a $(\rho_0,T,\delta)$-stream tube of $u_n$ in the right hand side of (\ref{BeltramiPerturbacion.EsquemaIterativo.fom}) will be studied in the $C^{k+1,\alpha}$ setting in the next section. The inhomogeneous Beltrami equations in the left hand side was studied in the preceding Section \ref{Beltrami.NoHomogenea.Seccion} through the analysis of complex-valued solutions satisfying both the $L^1$ decay condition (\ref{CondCaidaBeltrami.Intro.form}) and the $L^1$ SMB radiation condition (\ref{CondicionRadiacion.Beltrami.L1SilverMullerBeltrami.Intro.form}). Consequently, we arrive at the \textit{modified Grad--Rubin iterative method} discussed in the Introduction (Equation (\ref{paso.limite.esquemaiterativo.Intro.form})).

\subsection{Linear transport problem}\label{Problema.TransporteLineal.Subseccion}
We begin with the steady transport equations along $(\rho_0,T,\delta)$-stream tubes in the right hand side of (\ref{paso.limite.esquemaiterativo.Intro.form}). The main idea to find a solution is to transport $\varphi^0$ along the foliated stream tube and to check that this definition leads to regular enough factors $f_n$ of $u_n$ due to the regularity of the tube. 

\begin{theo}\label{problematransporte.teo}
Let $G,\Sigma,\mu$ satisfy the hypotheses (\ref{GSigmaMu.hipot}), consider any $u\in C^{k+1,\alpha}(\overline{\Omega},\RR^3)$ such that $\mathcal{T}(\Sigma,u)$ is a $(\rho_0,T,\delta)$-stream tube of such a velocity field and assume that $\varphi^0\in C^{k+1,\alpha}_c(\Sigma)$. Consider the first integral equation associated with $u$
\begin{equation}\label{esquemaiter.transportevarphi.eq}
\left\{\begin{array}{ll}
\displaystyle u\cdot \nabla \varphi=0 & \mbox{ in }\Omega\\
\displaystyle \varphi=\varphi^0 & \mbox{ on }\Sigma.
\end{array}\right.
\end{equation}
Then, there exists an unique solution $\varphi$ along $\mathcal{T}(\Sigma,u)$, its support lies in the closure of $\mathcal{T}(\Sigma,u)$ and it can be extended to a global solution in $\Omega$ with zero value outside $\mathcal{T}(\Sigma,u)$. Moreover, it belongs to $C^{k+1,\alpha}(\overline{\Omega})$ and the estimate 
$$
\Vert \varphi\Vert_{C^{k+1,\alpha}(\Omega)}\leq \Vert \varphi^0\Vert_{C^{k+1,\alpha}(\Sigma)}\,\kappa\left(\Vert u\Vert_{C^{k+1,\alpha}(\Omega)},T\right)
$$
holds, for some continuous function $\kappa:\RR_0^+\times\RR_0^+\longrightarrow\RR_0^+$.
\end{theo}

\begin{proof}
Let us now sketch the proof of this result, which can be found in \cite[Lemmas 4.8, 4.9 and 5.2]{Kaiser} for $k=0$. 
 Define the Calder\'on extension of $u$, $\overline{u}:=\mathcal{P}(u)$, according to Proposition \ref{ExtensionHolder.pro} and denote its flux mapping by $X(t;t_0,x_0)$. First, let us prove the uniqueness part of our assertion. Notice that as long as $\varphi$ is a smooth first integral of $u$, then
$$\frac{d}{dt}\varphi(X(t;0,\mu(s)))=(\overline{u}\cdot \nabla\varphi)(X(t;0,\mu(s)))=(u\cdot \nabla\varphi)(X(t;0,\mu(s)))=0,$$
for every $(t,s)\in\mathcal{D}(\Sigma,u)$. Therefore, $\varphi(x)=\varphi^0(\mu(s(x)))$ for every $x\in \mathcal{T}(\Sigma,u)$, where $(t(x),s(x))=\phi^{-1}(x)$. Second, regarding the existence assertion, the previous formula for $\varphi$ defines a smooth function in $\mathcal{T}(\Sigma,u)$ (by virtue of the bijectivity and regularity of the parametrization $\phi$ in Proposition \ref{FlujoHolder.pro}) which obviously solves (\ref{esquemaiter.transportevarphi.eq}) along the stream tube. Furthermore, with the exception of the endpoints, it is compactly supported in the interior of the tube. The extension of $\varphi$ by zero outside the tube yields a global smooth solution of (\ref{esquemaiter.transportevarphi.eq}) in $\Omega$. 

To show the bound for $\Vert \varphi\Vert_{C^{k+1,\alpha}(\Omega)}$ (equivalently for $\Vert \varphi\Vert_{C^{k+1,\alpha}(\mathcal{T}(\Sigma,u))}$), let us fix any multi-index $\gamma=(\gamma_1,\gamma_2,\gamma_3)$ such that $\vert \gamma\vert\leq k+1$ and note that
$$D^\gamma\varphi(x)=\gamma!\sum_{(l,\beta,\delta)\in \mathcal{D}(\gamma)}(D^\delta(\varphi^0\circ\mu))(s(x))\prod_{r=1}^l\frac{1}{\delta_r!}\left(\frac{1}{\beta_r!}D^{\beta_r}s(x)\right)^{\delta_r}.$$
for every $x\in \mathcal{T}(\Sigma,u)$. First of all, it is necessary to know how to handle $D^{\beta_r}s(x)$. To this end, note that $\jac(\phi^{-1})(x)=\jac(\phi)^{-1}(\phi^{-1}(x))$, so
$$D^\rho(\jac(\phi^{-1})_{i,j})(x)=\sum_{n=1}^{n_\rho}\prod_{\substack{\beta\in \Gamma_n\\
1\leq p,q\leq 3}}A_{n,p,q}^{i,j}(\rho,\beta)(D^\beta(\jac(\phi)^{-1}_{p,q}))(\phi^{-1}(x)),$$
for every multi-index $\rho$ such that $\vert \rho\vert\leq k$. Here, $A_{n,p,q}^{i,j}(\rho,\beta)$ stand for constant coefficients and $\Gamma_n$ is a set of $3$-multi-indices of order at most $\vert \rho\vert\leq k$. Expanding the products of sums by distributivity, each term in $D^\gamma\varphi$ takes the form
$$(D^\delta(\varphi^0\circ\mu))(s(x))\prod_{\substack{\beta\in \Gamma\\
1\leq p,q\leq 3}}B_{p,q}^{i,j}(\gamma,\beta)(D^\beta(\jac(\phi)^{-1}_{p,q}))(\phi^{-1}(x)),$$
where $\Gamma$ is a set of multi-indices with degree at most $k$. The first factor can be bounded by $\Vert \varphi^0\Vert_{C^{k+1,\alpha}(\Sigma)}$ whilst the terms in the second factor are bounded by $\kappa(\Vert u\Vert_{C^{k+1,\alpha}(\Omega)},T)$ as stated in Proposition \ref{FlujoHolder.pro}. Hence, it is clear that
$$\Vert \varphi\Vert_{C^{k+1}(\Omega)}\leq \Vert\varphi^0\Vert_{C^{k+1,\alpha}(\Sigma)}\kappa(\Vert u\Vert_{C^{k+1,\alpha}(\Omega)},T).$$

Finally, for any multi-index with maximum order $k+1$, the $\alpha$-H\"{o}lder seminorm of $D^\gamma \varphi$ can be estimated as follows. Take $x_1,x_2\in \mathcal{T}(\Sigma,u)$ and appropriately add and subtract the crossed terms. Since $D^\delta(\varphi^0\circ\mu)$ is bounded by $\Vert \varphi^0\Vert_{C^{k+1,\alpha}(\Sigma)}$ and $D^\beta(\jac(\phi)^{-1}_{pq})$ is bounded by $\kappa(\Vert u\Vert_{C^{k+1,\alpha}(\Omega)},T)$, then it only remains to obtain estimates for
\begin{align*}
I:&=\left.(D^\delta(\varphi^0\circ\mu))(s(x))\right\vert_{x_1}^{x_2},\\
II:&= \left.(D^\beta(\jac(\phi)^{-1}_{p,q}))(\phi^{-1}(x))\right\vert_{x_1}^{x_2}.
\end{align*}
First, we distinguish the cases $\vert \delta\vert<k+1$ and $\vert\delta\vert=k+1$. In the former case, the mean value theorem, the estimates in Proposition \ref{FlujoHolder.pro} for $\jac(\phi)^{-1}$ and the estimate (\ref{diametro.tuboflujo.form}) of the diameter of the stream tube $\mathcal{T}(\Sigma,u)$ yield the upper bound
\begin{align*}
I&\leq \Vert \varphi^0\Vert_{C^{k+1,\alpha}(\Sigma)}\kappa(\Vert u\Vert_{C^{k+1,\alpha}(\Omega)},T)\vert x_1-x_2\vert\\
&\leq \Vert \varphi^0\Vert_{C^{k+1,\alpha}(\Sigma)}\kappa(\Vert u\Vert_{C^{k+1,\alpha}})(T\Vert u\Vert_{C^0(\Omega)}+\mbox{diam}(\Sigma))^{1-\alpha}\vert x_1-x_2\vert^\alpha.
\end{align*}
In the later case, the $\alpha$-H\"{o}lder continuity of $D^\delta(\varphi^0\circ\mu)$ gives rise to an analogous estimate
$$
I\leq \Vert \varphi^0\Vert_{C^{k+1,\alpha}(\Sigma)}\kappa\left(\Vert u\Vert_{C^{k+1,\alpha}(\Omega)},T\right)^\alpha\vert x_1-x_2\vert^\alpha.$$
Second, note that $D^\beta(\jac(\phi)^{-1}_{p,q})$ is $\alpha$-H\"{o}lder continuous with H\"{o}lder's constant that can be bounded above by $\kappa(\Vert u\Vert_{C^{k+1,\alpha}(\Omega)},T)$ by virtue of Proposition \ref{FlujoHolder.pro}. Thus,
$$II\leq\kappa(\Vert u\Vert_{C^{k+1,\alpha}(\Omega)},T)\vert \phi^{-1}(x_1)-\phi^{-1}(x_2)\vert^\alpha.$$
The mean value theorem then leads to the desired upper estimate
$$\vert D^\gamma \varphi(x_1)-D^\gamma\varphi(x_2)\vert\leq \kappa(\Vert u\Vert_{C^{k+1,\alpha}(\Omega)},T)\vert x_1-x_2\vert^\alpha,$$
appropriately modifying the separately increasing function $\kappa$.
\end{proof}

In addition to the existence and uniqueness results of (\ref{esquemaiter.transportevarphi.eq}), in order to take limits in (\ref{paso.limite.esquemaiterativo.Intro.form}) we will need a compactness result for $\{\varphi_n\}_{n\in\NN}$. Once we know that the sequence $\{u_n\}_{n\in\NN}$ converges in $C^{k+1,\alpha}(\overline{\Omega},\RR^3)$, a result of stability for the problem (\ref{esquemaiter.transportevarphi.eq}) leads to the convergence of the sequence $\{\varphi_n\}_{n\in\NN}$ in $C^{k,\alpha}(\overline{\Omega})$. This stability result was proved in \cite[Lemma 5.3]{Kaiser} in the  $C^{1,\alpha}$ framework and can be easily extended  to $C^{k+1,\alpha}$ using the same lines as in Theorem~\ref{problematransporte.teo} (the details, which are straightforward, can be found in \cite[Corollary 2.4.4]{Poyato}):

\begin{cor}\label{problematransporte.perturb.cor}
Let $G$, $\Sigma$, $\mu$ satisfy the properties (\ref{GSigmaMu.hipot}). Consider any couple of vector fields $u_1,u_2\in C^{k+1,\alpha}(\overline{\Omega},\RR^3)$, and denote as $\mathcal{T}_1:=\mathcal{T}(\Sigma,u_1)$ and $\mathcal{T}_2:=\mathcal{T}(\Sigma,u_2)$ the associated stream tubes which emanate from $\Sigma$. Assume that $\mathcal{T}_i$ is a $(\rho_0,T,\delta_i)$-stream tube of $u_i$. Consider any boundary data $\varphi^0\in C^{k+1,\alpha}_c(\Sigma)$ and the solutions $\varphi_1$ and $\varphi_2$ (according to Theorem \ref{problematransporte.teo}) to each transport problem associated with $u_1$ and $u_2$ respectively:
$$\left\{\begin{array}{ll}
\nabla \varphi_1\cdot u_1=0, & x\in \Omega,\\
\varphi_1=\varphi^0, & x\in \Sigma,
\end{array}\right.\hspace{1cm}\left\{\begin{array}{ll}\nabla \varphi_2\cdot u_2=0, & x\in \Omega,\\
\varphi_2=\varphi^0, & x\in \Sigma.\end{array}\right.$$
Then,
$$\Vert \varphi_1-\varphi_2\Vert_{C^{k,\alpha}(\Omega)}\leq \Vert \varphi^0\Vert_{C^{k+1,\alpha}(\Sigma)}\cdot\kappa\left(\Vert u_1\Vert_{C^{k+1,\alpha}(\Omega)},T\right)\cdot\kappa\left(\Vert u_2\Vert_{C^{k+1,\alpha}(\Omega)},T\right)\Vert u_1-u_2\Vert_{C^{k+1,\alpha}(\Omega)},$$
where $\kappa:\RR_0^+\times\RR_0^+\longrightarrow\RR_0^+$ is a continuous and separately increasing function which does not depend on $u_i,\varphi^0$ or $T$.
\end{cor}

\subsection{Limit of the approximate solutions}
The existence and uniqueness results in Theorems \ref{problematransporte.teo} and \ref{BeltramiNoHomog.teo} together with the stability result for the transport problem in Corollary \ref{problematransporte.perturb.cor} now allow us to take the limit as $n\rightarrow +\infty$ in the modified Grad--Rubin iterative scheme (\ref{paso.limite.esquemaiterativo.Intro.form}). Therefore, we obtain a generalized Beltrami field which is close to the initial strong Beltrami field and whose proportionality factor is a non-constant small enough perturbation of the initial constant proportionality factor $\lambda$:

\begin{theo}\label{paso.limite.teo}
Let $G,\Sigma,\mu$ satisfy the hypotheses (\ref{GSigmaMu.hipot}) and assume that $0\neq \lambda\in\RR$ is not a Dirichlet eigenvalue of Laplace operator in the interior domain $G$. Consider any complex-valued strong Beltrami field $v_0\in C^{k+1,\alpha}(\overline{\Omega},\CC^3)$ which satisfy the $L^1$ SMB radiation condition (\ref{CondicionRadiacion.Beltrami.L1SilverMullerBeltrami.form}) and the $L^1$ decay property (\ref{CondCaidaBeltrami.form}) in the exterior domain. Consider its real part $u_0:=\Re v_0$, and assume that $\mathcal{T}(\Sigma,u_0)$ is a $(\rho_0,T,\delta)$-stream tube of the velocity field $u_0$. 
Let $\varepsilon_0$ be a positive number. Then, there exists a nonnegative constant $\delta_0$ for which  the real parts $u_{n+1}$ of the solutions $v_{n+1}\in C^{k+1,\alpha}(\overline{\Omega},\CC^3)$ together with the solutions $\varphi_n\in C^{k+1,\alpha}(\overline{\Omega})$ of the coupled problems in the modified Grad--Rubin iterative scheme (\ref{paso.limite.esquemaiterativo.Intro.form})
(Theorems \ref{problematransporte.teo} and \ref{BeltramiNoHomog.teo})
have a limit vector field $u\in C^{k+1,\alpha}(\overline{\Omega},\RR^3)$ and a limit perturbation of the proportionality factor $\varphi\in C^{k,\alpha}(\overline{\Omega})$ such that
$$
u_n\rightarrow u\mbox{ in }C^{k+1,\alpha}(\overline{\Omega},\RR^3),\ \ \varphi_n\rightarrow \varphi\mbox{ in }C^{k,\alpha}(\overline{\Omega}),
$$
as $n\rightarrow +\infty$, for any $\varphi^0\in C^{k+1,\alpha}_c(\Sigma)$ with $\Vert \varphi^0\Vert_{C^{k+1,\alpha}(\Sigma)}<\delta_0$. Also,  $(u,\lambda+\varphi)$ solves the following boundary value problem
$$\left\{\begin{array}{ll}
\curl u=(\lambda+\varphi)u, & x\in\Omega,\\
\divop u=0, & x\in \Omega,\\
u\cdot \eta=u_0\cdot \eta, & x\in S,\\
\varphi=\varphi^0, & x\in \Sigma.
\end{array}\right.$$
Furthermore, $u=O(\vert x\vert^{-1})$ as $\vert x\vert\rightarrow +\infty$, $\mathcal{T}(\Sigma,u)$ is a $(\rho_0,T,\delta/2)$-stream tube of $u$, $\varphi$ has compact support inside the closure of such stream tube and $u$ is close enough to $u_0$, specifically
$$
\Vert u-u_0\Vert_{C^{k+1,\alpha}(\Omega)}\leq \varepsilon_0\Vert u\Vert_{C^{k+1,\alpha}(\Omega)}.
$$
\end{theo}
\begin{proof}
For simplicity of notation, we will denote the stream tubes associated with each vector field $u_n$ which emanates from $\Sigma$ by $\mathcal{T}_n:=\mathcal{T}(\Sigma,u_n)$. First of all, it is necessary to check whether the hypothesis of Theorems \ref{problematransporte.teo} and \ref{BeltramiNoHomog.teo} hold and they can be deduced in each step from the corresponding hypotheses in the previous step in the iteration. Let us begin with the step $n=0$:
$$\left\{\begin{array}{ll}
\nabla \varphi_0\cdot u_0=0, & x\in \Omega,\\
\varphi_0=\varphi^0, & x\in \Sigma,
\end{array}\right. \hspace{1.5cm}\left\{\begin{array}{l}
\curl v_{1}-\lambda v_{1}=\varphi_0 u_0, \hspace{1cm} x\in \Omega,\\
v_{1}\cdot \eta= u_0\cdot \eta, \hspace{2.1cm} x\in S,\\
+\ L^1\mbox{ Decay property }(\ref{CondCaidaBeltrami.Intro.form}),\\
+\ L^1\mbox{ SBM radiation condition }(\ref{CondicionRadiacion.Beltrami.L1SilverMullerBeltrami.Intro.form}).
\end{array}\right.$$

The hypotheses imply that $\mathcal{T}_0$ is a $(\rho_0,T,\delta)$-stream tube of $u_0$  and $\varphi^0\in C^{k+1,\alpha}_c(\Sigma)$. Hence, there exists a global solution $\varphi_0$ to the transport equation  (Theorem \ref{problematransporte.teo}). Moreover, $\varphi_0u_0\in C^{k+1,\alpha}_c(\overline{\Omega},\RR^3)\subseteq C^{k,\alpha}_c(\overline{\Omega},\RR^3)$ and its compact support is contained in the stream tube $\overline{\mathcal{T}_0}$. In particular, the estimate (\ref{diametro.tuboflujo.form}) ensures that $\mbox{supp}(\varphi_0u_0)\subseteq \overline{\mathcal{T}_0}\subseteq \overline{\Omega_R}$, where $\Omega_R:=B_R(0)\setminus \overline{G}$ and $R:=2T\Vert u_0\Vert_{C^{k+1,\alpha}(\Omega)}+\mbox{diam}(\Sigma)$. On the other hand, as $S$ is regular enough, so  $\eta$ is  and, consequently, $u_0\cdot \eta\in C^{k+1,\alpha}(S)$. An integration by parts leads to the following expression
\begin{align*}
\int_S (\lambda u_0\cdot \eta&+\varphi_0u_0\cdot \eta)\,dS =\lambda\int_S u_0\cdot\eta\,dS+\int_S\varphi_0 u_0\cdot \eta\,dS\\
&=\lambda\int_S u_0\cdot \eta\,dS+\int_{\partial B_{R'}(0)}\varphi_0 u_0\cdot \eta\,dS-\int_{\Omega_{R'}}\divop(\varphi_0 u_0)\,dx
\end{align*}
For $R'>R$, the second term vanishes as a consequence of the previous estimate for the diameter of the initial stream tube. Regarding the third term, notice that the same argument as above leads to
$$\divop(\varphi_0u_0)=\nabla\varphi_0\cdot u_0+\varphi_0 \divop u_0=0.$$

Finally, in Appendix \ref{Appendix.A} we show that the Beltrami equation for $u_0$ allows us to write 
$$u_0\cdot \eta=-\frac{1}{\lambda}\divop_S(\eta\times u_0).$$
The divergence theorem then concludes that the first term vanishes too. Therefore, the hypotheses of Theorem \ref{BeltramiNoHomog.teo} are satisfied, so there is a unique solution $v_1$ to the corresponding complex-valued inhomogeneous Beltrami equation in the right hand side of the step $n=0$. 

Let us prove an estimate for $u_1-u_0$ that will be useful to prove the Cauchy condition in $C^{k+1,\alpha}(\overline{\Omega},\RR^3)$ for the sequence $\{u_n\}_{n\in\NN}$. This vector field is the real part of $v_1-v_0$, which satisfies the complex-valued exterior Neumann problem
\begin{equation*}
\left\{\begin{array}{ll}
(\curl-\lambda)(v_1-v_0)=\varphi_0 u_0, \hspace{1cm} x\in\Omega,\\
(v_1-v_0)\cdot \eta=0, \hspace{1.72cm} x\in S,\\
+\ L^1\mbox{ decay condition } (\ref{CondCaidaBeltrami.Intro.form}),\\
+\ L^1\mbox{ SMB radiation conditon }(\ref{CondicionRadiacion.Beltrami.L1SilverMullerBeltrami.Intro.form}).
\end{array}\right.
\end{equation*}
Therefore, the uniqueness of the solution to this problem (Proposition \ref{BeltramiNoHomog.Unicidad.EcuacionIntegral.pro}), the $C^{k+1,\alpha}$ estimates of such solutions (Corollary \ref{BeltramiNoHomog.EstimacionSchauder.cor}), and the $C^{k,\alpha}$ estimates for the solution of the steady transport equation (Theorem \ref{problematransporte.teo}) allow us to obtain the following estimate for $v_1-v_0$ and, consequently, for $u_1-u_0$:
\begin{align*}
\Vert u_1-u_0\Vert_{C^{k+1,\alpha}(\Omega)}&=\Vert \Re(v_1-v_0)\Vert_{C^{k+1,\alpha}(\Omega)}\leq \Vert v_1-v_0\Vert_{C^{k+1,\alpha}(\Omega)} \leq C_0\Vert \varphi_0 u_0\Vert_{C^{k,\alpha}(\Omega)}.
\end{align*}
Here $C_0>0$ is some constant, which depends on $k,\alpha,\lambda,G$ and $R$. The \textit{Leibniz rule} for the derivative of a product reads
$$D^\gamma(\varphi_0u_0)=\sum_{\beta\leq \gamma}\binom{\gamma}{\beta}D^\beta \varphi_0 D^{\gamma-\beta}u_0,$$
for any multi-index $\gamma$. 

Therefore, the estimates in Theorem \ref{problematransporte.teo} for the derivatives up to order $k$ of $\varphi_0$ and the combination of the mean value theorem and the Calder\'on's extension theorem (Proposition \ref{ExtensionHolder.pro}) to estimate the $C^{0,\alpha}$-norm of the derivatives of $u_0$ up to order $k$ allow us to arrive at the inequality
\begin{align*}
\Vert D^\gamma(\varphi_0 u_0)\Vert_{ C^{0}(\Omega)}&\leq C_k\Vert \varphi^0\Vert_{C^{k+1,\alpha}(\Sigma)}\kappa\left(\Vert u_0\Vert_{C^{k+1,\alpha}(\Omega)},T\right)\Vert u_0\Vert_{C^{k+1,\alpha}(\Omega)},
\end{align*}
for every multi-index $\gamma$ with $\vert \gamma\vert\leq k$, and
\begin{align*}
\Vert D^\gamma(&\varphi_0 u_0)\Vert_{C^{0,\alpha}(\Omega)} =\Vert D^\gamma(\varphi_0 u_0)\Vert_{C^{0,\alpha}(\mathcal{T}_0)}\\
&\leq C_kC_\mathcal{P}\Vert \varphi^0\Vert_{C^{k+1,\alpha}(\Sigma)}\kappa\left(\Vert u_0\Vert_{C^{k+1,\alpha}},T\right)\Vert u_0\Vert_{C^{k+1,\alpha}(\Omega)}(T\Vert u_0\Vert_{C^{k,\alpha}(\Omega)}+\mbox{diam}\Sigma)^{1-\alpha},
\end{align*}
for every multi-index $\gamma$ so that $\vert \gamma\vert=k$ and a nonnegative constant $C_k$ depending on $k$. To derive the last estimate, we have used that
\begin{align*}
\vert D^{\gamma-\beta}&u_0(x)-D^{\gamma-\beta}u_0(y)\vert \leq \Vert D^{\gamma-\beta}\overline{u_0}\Vert_{C^1(\RR^3)}\vert x-y\vert \leq C_\mathcal{P}\Vert u_0\Vert_{C^{k+1,\alpha}(\Omega)}\vert x-y\vert^\alpha(\mbox{diam}\mathcal{T}_0)^{1-\alpha},
\end{align*}
for every $x,y\in \mathcal{T}_0$ and the estimate (\ref{diametro.tuboflujo.form}) for the diameter of the $(\rho_0,T,\delta)$-stream tube of $\mathcal{T}_0$. Hence the following inequality 
\begin{multline*}
\Vert u_1-u_0\Vert_{C^{k+1,\alpha}(\Omega)}\\
\leq K\left\{1+(T\Vert u_0\Vert_{C^{k+1,\alpha}(\Omega)}+\mbox{diam}\Sigma)^{1-\alpha}\right\}\Vert \varphi^0\Vert_{C^{k+1,\alpha}(\Sigma)}\kappa\left(\Vert u_0\Vert_{C^{k+1,\alpha}(\Omega)},T\right)\Vert u_0\Vert_{C^{k+1,\alpha}(\Omega)}
\end{multline*}
holds, with a constant $K=K(k,\alpha,\lambda,G,R)$.

Now, we can fix the small parameter $\delta_0$ such that it satisfies
\begin{equation}\label{paso.limite.estimac.delta0}
\left\{\begin{array}{l}
\displaystyle K\left\{1+(4T\Vert u_0\Vert_{C^{k+1,\alpha}(\Omega)}+\mbox{diam}\Sigma)^{1-\alpha}\right\}\times\\
\displaystyle\times\left\{\kappa\left(2\Vert u_0\Vert_{C^{k+1,\alpha}(\Omega)},T\right)+\Vert u_0\Vert_{C^{k+1,\alpha}(\Omega)}\kappa\left(2\Vert u_0\Vert_{C^{k+1,\alpha}(\Omega)},T\right)^2\right\}\delta_0<\frac{1}{2}\min\{\varepsilon_0,1\}.\\
\,\\
\displaystyle K\left\{1+(4T\Vert u_0\Vert_{C^{k+1,\alpha}(\Omega)}+\mbox{diam}\Sigma)^{1-\alpha}\right\}\times\\
\displaystyle\times \left\{\kappa\left(2\Vert u_0\Vert_{C^{k+1,\alpha}(\Omega)},T\right)+\Vert u_0\Vert_{C^{k+1,\alpha}(\Omega)}\kappa\left(2\Vert u_0\Vert_{C^{k+1,\alpha}(\Omega)},T\right)^2\right\}\Vert u_0\Vert_{C^{k+1,\alpha}(\Omega)}\,\delta_0 \\
\displaystyle <\frac{1}{4}\frac{2\delta}{C_\mathcal{P}T}e^{-\frac{1}{2}C_\mathcal{P}\Vert u_0\Vert_{C^{k+1,\alpha}(\Omega)}T}.
\end{array}\right.
\end{equation}
Then we infer
\begin{equation}\label{paso.limite.induccion.0}
\left\{\begin{array}{rcl}
\displaystyle\Vert u_1-u_0\Vert_{C^{k+1,\alpha}(\Omega)}&<&\displaystyle\min\{\varepsilon_0,1\}\frac{1}{2}\Vert u_0\Vert_{C^{k+1,\alpha}(\Omega)},\\
\displaystyle \Vert u_1-u_0\Vert_{C^{k+1,\alpha}(\Omega)}&<&\displaystyle\frac{1}{4}\frac{2\delta}{C_\mathcal{P}T}e^{-\frac{1}{2}C_\mathcal{P}T\Vert u_0\Vert_{C^{k+1,\alpha}(\Omega)}},\\
\displaystyle \Vert u_1\Vert_{C^{k+1,\alpha}(\Omega)}&\leq& \displaystyle\frac{3}{2}\Vert u_0\Vert_{C^{k+1,\alpha}(\Omega)}.
\end{array}\right.
\end{equation}

To obtain  similar estimates for the remaining terms of the iterative scheme we will use induction to show that
\begin{equation}\label{paso.limite.induccion.n}
\left\{\begin{array}{rcl}
\displaystyle\Vert u_{n+1}-u_n\Vert_{C^{k+1,\alpha}(\Omega)}&\leq&\displaystyle\frac{1}{2^{n}}\Vert u_1-u_0\Vert_{C^{k+1,\alpha}(\Omega)}<\min\{\varepsilon_0,1\}\frac{1}{2^{n+1}}\Vert u_0\Vert_{C^{k+1,\alpha}(\Omega)},\\
\displaystyle \Vert u_{n+1}-u_{n}\Vert_{C^{k+1,\alpha}(\Omega)}&<&\displaystyle\frac{1}{2}\frac{1}{2^{n+1}}\frac{2\delta}{C_\mathcal{P}T}e^{-\frac{1}{2}C_\mathcal{P}T\Vert u_0\Vert_{C^{k+1,\alpha}(\Omega)}},\\
\displaystyle \Vert u_{n+1}-u_0\Vert_{C^{k+1,\alpha}(\Omega)}&<& \displaystyle\min\{\varepsilon_0,1\}\sum_{i=1}^{n+1}\frac{1}{2^i}\Vert u_0\Vert_{C^{k+1,\alpha}(\Omega)},\\
\displaystyle\Vert u_{n+1}-u_0\Vert_{C^{k+1,\alpha}(\Omega)}&<& \displaystyle\frac{1}{2}\sum_{i=1}^{n+1}\frac{1}{2^i}\frac{2\delta}{C_\mathcal{P}T}e^{-\frac{1}{2}C_\mathcal{P}\Vert u_0\Vert_{C^{k+1,\alpha}(\Omega)}},\\
\displaystyle \Vert u_{n+1}\Vert_{C^{k+1,\alpha}(\Omega)}&<& \displaystyle\sum_{i=0}^{n+1}\frac{1}{2^i} \Vert u_0\Vert_{C^{k+1,\alpha}(\Omega)}.
\end{array}\right.
\end{equation}
This is true for $n=0$ due to (\ref{paso.limite.induccion.0}), so we can assume that the inductive hypotheses  holds for all indices less than $n$. Specifically, we assume that $\varphi_m$, $v_{m+1}$ are well defined, i.e., the corresponding problems have a unique solution, that $u_{m+1}$ are divergence-free and (\ref{paso.limite.induccion.n}) hold for indices $m<n$. 

Let us now prove that the result is verified for the index $m=n$. The inductive hypotheses imply the existence of a vector field $v_n\in C^{k+1,\alpha}(\overline{\Omega},\CC^3)$ and $\varphi_{n-1}\in C^{k,\alpha}(\overline{\Omega})$. Moreover, $\mathcal{T}_n$ is a $\left(\rho_0,T,\left(1-\frac{1}{2}\sum_{i=1}^n\frac{1}{2^i}\right)\delta\right))$-stream tube of the real part $u_n=\Re v_n$ because of the third inequality in  (\ref{paso.limite.induccion.n}). Consequently, there exists a unique solution $\varphi_n\in C^{k,\alpha}(\overline{\Omega})$ to the transport problem in the left hand side of (\ref{paso.limite.esquemaiterativo.Intro.form}) according to Theorem \ref{problematransporte.teo}. The last estimate in (\ref{paso.limite.induccion.n}) along with (\ref{diametro.tuboflujo.form}) lead to $\mathcal{T}_n\subseteq \overline{\Omega_R}$. Therefore, $\varphi_n$ is compactly supported in $\overline{\Omega_R}\subseteq \overline{\Omega}$ and the same argument as in the step $n=0$ ensures the existence and uniqueness of a solution $v_{n+1}\in C^{k+1,\alpha}(\overline{\Omega},\CC^3)$ to the complex-valued exterior Neumann problem for the inhomogeneous Beltrami equation in the right hand side of (\ref{paso.limite.esquemaiterativo.Intro.form}). 

Notice that the vanishing flux hypothesis in  Theorem \ref{BeltramiNoHomog.teo} is satisfied. To check it we repeat the previous argument to get
$$\int_S (\lambda u_0\cdot\eta+\varphi_nu_n\cdot \eta)\,dS=\lambda\int_S u_0\cdot \eta\,dS+\int_{\partial B_{R'}(0)}\varphi_nu_n\cdot \eta\,dS-\int_{\Omega_{R'}}\divop(\varphi_nu_n)\,dx.$$
The first term is zero as before, the second one also vanishes for a choice $R'>R$ and the last one is zero too because $\varphi_n$ is a first integral of $u_n$ and $u_n$ is divergence-free according to the induction hypothesis. Consequently, it is easy that $u_{n+1}$ is also divergence-free. 

To conclude, let us prove the inductive hypothesis (\ref{paso.limite.induccion.n}) for $u_{n+1}-u_n$. Taking the difference of the corresponding complex-valued exterior boundary value problems we have that $v_{n+1}-v_n$ solves
\begin{equation*}
\left\{\begin{array}{l}
(\curl-\lambda)(v_{n+1}-v_n)=\varphi_n u_n-\varphi_{n-1}u_{n-1}, \hspace{1cm} x\in\Omega,\\
(v_{n+1}-v_n)\cdot \eta=0, \hspace{3.72cm} x\in S,\\
+\ L^1 \mbox{ decay conditions } (\ref{CondCaidaBeltrami.Intro.form}),\\
+\ L^1\mbox{ SMB radiation condition }(\ref{CondicionRadiacion.Beltrami.L1SilverMullerBeltrami.Intro.form}).
\end{array}\right.
\end{equation*}
Again, thanks to the uniqueness property (Proposition \ref{BeltramiNoHomog.Unicidad.EcuacionIntegral.pro}), the $C^{k+1,\alpha}$ estimates for these solutions (Corollary \ref{BeltramiNoHomog.EstimacionSchauder.cor}) and the $C^{k,\alpha}$ estimates for the solution of the steady transport equation (Theorem \ref{problematransporte.teo}), we obtain the following estimate for $v_{n+1}-v_n$ and, consequently, for $u_{n+1}-u_n$
\begin{align*}
\Vert u_{n+1}-u_n\Vert_{C^{k+1,\alpha}(\Omega)}&=\Vert \Re(v_{n+1}-v_n)\Vert_{C^{k+1,\alpha}(\Omega)}\leq \Vert v_{n+1}-v_n\Vert_{C^{k+1,\alpha}(\Omega)}\\
&\leq C_0\Vert \varphi_n u_n-\varphi_{n-1}u_{n-1}\Vert_{C^{k,\alpha}(\Omega)}.
\end{align*}

Now,  $\varphi_nu_n-\varphi_{n-1}u_{n-1}$ has compact support inside $\overline{\mathcal{T}_n}\cup\overline{\mathcal{T}_{n-1}}\subseteq \overline{\Omega_R}$ (see estimate (\ref{diametro.tuboflujo.form}) and the last inequalities for the $C^{k+1,\alpha}$ norms of $u_n$ and $u_{n-1}$ in the inductive hypothesis). Thus, Theorem \ref{Potencial.volumetrico.regularidad.teo} asserts that the constant $C_0=C_0(k,\alpha,\lambda,G,R)$ is the same as in the basic step because all the supports of the inhomogeneous terms in the complex-valued exterior Neumann problems are contained in the same bounded subset $\overline{\Omega_R}$ of the exterior domain. This is a crucial fact because it prevents those constants from depending on the iteration number $n$ and avoids the blowup when $n\rightarrow +\infty$. Notice that
$$
\Vert \varphi_n u_n-\varphi_{n-1}u_{n-1}\Vert_{C^{k,\alpha}(\Omega)}\leq \Vert(\varphi_n-\varphi_{n-1})u_n\Vert_{C^{k,\alpha}(\Omega)}+\Vert \varphi_{n-1}(u_n-u_{n-1})\Vert_{C^{k,\alpha}(\Omega)}.
$$
Since $\mathcal{T}_n$ is a $\left(\rho_0,T,\left(1-\frac{1}{2}\sum_{i=1}^{n}\frac{1}{2^i}\right)\delta\right)$-stream tube of $u_{n}$, $\mathcal{T}_{n-1}$ is a $\left(\rho_0,T,\left(1-\frac{1}{2}\sum_{i=1}^{n-1}\frac{1}{2^i}\right)\delta\right)$-stream tube of $u_{n-1}$ and $u_{n-1}\cdot \eta=u_0\cdot \eta=u_{n}\cdot \eta$ on $S$, we can apply both estimates in Theorem \ref{problematransporte.teo} and Corollary \ref{problematransporte.perturb.cor} to obtain the inequality
\begin{align*}
\Vert \varphi_n u_n-\varphi_{n-1}&u_{n-1}\Vert_{C^{k+1,\alpha}(\Omega)}
\leq K\Vert \varphi^0\Vert_{C^{k+1,\alpha}(\Sigma)}\left\{1+(4T\Vert u_0\Vert_{C^{k+1,\alpha}(\Omega)}+\mbox{diam}\Sigma)^{1-\alpha}\right\}\\
&\times\left\{\kappa(2\Vert u_0\Vert_{C^{k+1,\alpha}(\Omega)},T)+\Vert u_0\Vert_{C^{k+1,\alpha}(\Omega)}\kappa\left(2\Vert u_0\Vert_{C^{k+1,\alpha}(\Omega)},T\right)^2\right\}\Vert u_n-u_{n-1}\Vert_{C^{k+1,\alpha}(\Omega)}.
\end{align*}
Consequently, the estimate
\begin{align*}
\Vert u_{n+1}&-u_n\Vert_{C^{k+1,\alpha}(\Omega)}
\leq K\Vert \varphi^0\Vert_{C^{k+1,\alpha}(\Sigma)}\left\{1+(4T\Vert u_0\Vert_{C^{k+1,\alpha}(\Omega)}+\mbox{diam}\Sigma)^{1-\alpha}\right\}\\
&\times\left\{\kappa(2\Vert u_0\Vert_{C^{k+1,\alpha}(\Omega)},T)+\Vert u_0\Vert_{C^{k+1,\alpha}(\Omega)}\kappa\left(2\Vert u_0\Vert_{C^{k+1,\alpha}(\Omega)},T\right)^2\right\}\Vert u_n-u_{n-1}\Vert_{C^{k+1,\alpha}(\Omega)}
\end{align*}
holds, with a constant $K$ independent of $n$. Since $\Vert \varphi^0\Vert_{C^{k+1,\alpha}(\Sigma)}<\delta_0$ and $\delta_0$ is small enough to ensure (\ref{paso.limite.estimac.delta0}), one has
$$\Vert u_{n+1}-u_n\Vert_{C^{k+1,\alpha}(\Omega)}<\frac{1}{2}\Vert u_n-u_{n-1}\Vert_{C^{k+1,\alpha}(\Omega)},$$
and the inductive hypothesis for indices less than $n$ leads to the first two inequalities in (\ref{paso.limite.induccion.n}). 

The last three estimates can be obtained as follows. Firstly, the preceding two estimates together with the induction hypotheses lead to
\begin{align*}
\Vert u_{n+1}-u_0\Vert_{C^{k+1,\alpha}(\Omega)}&\leq \sum_{i=0}^n\Vert u_{i+1}-u_i\Vert_{C^{k+1,\alpha}(\Omega)}\\
&\leq \min\{\varepsilon_0,1\}\sum_{i=0}^n\frac{1}{2^{i+1}}\Vert u_0\Vert_{C^{k+1,\alpha}(\Omega)}=\min\{\varepsilon_0,1\}\sum_{i=1}^{n+1}\frac{1}{2^i}\Vert u_0\Vert_{C^{k+1,\alpha}(\Omega)}.
\end{align*}
Similarly, we have
\begin{align*}
\Vert u_{n+1}-u_0\Vert_{C^{k+1,\alpha}(\Omega)}&\leq \sum_{i=0}^n\Vert u_{i+1}-u_i\Vert_{C^{k+1,\alpha}(\Omega)} \leq \frac{1}{2}\sum_{i=1}^{n+1}\frac{1}{2^i}\frac{2\delta}{C_\mathcal{P}T}e^{-\frac{1}{2}C_\mathcal{P}T\Vert u_0\Vert_{C^{k+1,\alpha}(\Omega)}}.
\end{align*}
The last inequality in (\ref{paso.limite.induccion.n}) is obvious by the triangle inequality:
\begin{align*}
\Vert u_{n+1}\Vert_{C^{k+1,\alpha}(\Omega)}&\leq \Vert u_0\Vert_{C^{k+1,\alpha}(\Omega)}+\Vert u_{n+1}-u_0\Vert_{C^{k+1,\alpha}(\Omega)}\\
&\leq\left(1+\sum_{i=1}^{n+1}\frac{1}{2^i}\right)\Vert u_0\Vert_{C^{k+1,\alpha}(\Omega)}=\sum_{i=0}^{n+1}\frac{1}{2^i}\Vert u_0\Vert_{C^{k+1,\alpha}(\Omega)}.
\end{align*}

Using the above inequalities in (\ref{paso.limite.induccion.n}) one can show that $\{u_n\}_{n\in \NN}$ and $\{\varphi_n\}_{n\in\NN}$ are Cauchy sequences in $C^{k+1,\alpha}(\overline{\Omega},$ $\RR^3)$ and $C^{k,\alpha}(\overline{\Omega})$, respectively. On the one hand, we find
\begin{align*}
\Vert u_{n+m}-u_n\Vert_{C^{k+1,\alpha}(\Omega)}&\leq \sum_{i=n}^{n+m-1}\Vert u_{i+1}-u_i\Vert_{C^{k+1,\alpha}(\Omega)}\\
&<\sum_{i=n}^{n+m-1}\frac{1}{2^{i+1}}\Vert u_0\Vert_{C^{k+1,\alpha}(\Omega)} \leq \sum_{i=n+1}^{+\infty}\frac{1}{2^i}\Vert u_0\Vert_{C^{k+1,\alpha}(\Omega)}= \frac{1}{2^n} \|u_0\|_{C^{k+1,\alpha}(\Omega)}.
\end{align*}
Likewise, the third inequality in (\ref{paso.limite.induccion.n}) along with the property  
$
u_{n}\cdot \eta=u_0\cdot \eta\mbox{ on }S,
$
shows that $\mathcal{T}_n$ are $\left(\rho_0,T,\left(1-\frac{1}{2}\sum_{i=0}^n \frac{1}{2^i}\right)\delta\right)$-stream tubes of $u_n$. Therefore, $\{\varphi_n\}_{n\in\NN}$ also satisfies the Cauchy condition in $C^{k,\alpha}(\overline{\Omega})$ due to Corollary \ref{problematransporte.perturb.cor}. Thus, it converges in $C^{k,\alpha}$ to some $\varphi\in C^{k,\alpha}(\overline{\Omega})$.

Let us now take the limit as $n\rightarrow +\infty$ in the iterative scheme to deduce
$$\begin{array}{ccccccccccc}
\divop u_{n+1} & = & 0 & \hspace{1cm} & \curl u_{n+1} -\lambda u_{n+1}& = & \varphi_n u_n & \hspace{1cm} & u_{n+1}\cdot \eta & = & u_0\cdot \eta\\
\downarrow & & \downarrow & \hspace{1cm} &  \downarrow & & \downarrow & \hspace{1cm} & \downarrow & & \downarrow\\
\divop u & = & 0 & \hspace{1cm} & \curl u -\lambda u& = & \varphi u & \hspace{1cm} & u\cdot \eta & = & u_0\cdot \eta.
\end{array}$$

Moreover, the $L^1$ SMB radiation condition (\ref{CondicionRadiacion.Beltrami.L1SilverMullerBeltrami.Intro.form}) and the decay property (\ref{CondCaidaBeltrami.Intro.form}) lead to complex-valued solutions $v_n$ to the exterior Neumann problem for the inhomogeneous Beltrami equations in the iterative scheme with the asymptotic behavior 
$$\vert v_n(x)\vert\leq \frac{C}{\vert x\vert},\ x\in\Omega,$$
for every $n$ and a constant $C$ independent of $n$. To check it, notice that Theorem \ref{BeltramiNoHomog.teo} provides a decomposition of $v_{n+1}$ into generalized volume and single layer potentials whose densities are $u_0\cdot \eta$, $\varphi_n u_n$ and the sequence $\xi_n$ of solutions to the boundary integral equations (\ref{EcIntegral.DatoFrontera.Beltrami.form}). The single layer potentials and its first order partial derivarives are dominated by the corresponding integral kernels $\Gamma_\lambda$ and $\nabla\Gamma_\lambda$ for $x$ far enough from the surface $S$. This leads to an upper bound $C\vert x\vert^{-1}$ where $C$ depends on the $ C^0$ norm of the densities $u_0\cdot \eta$ and $\xi_n$. Both quantities can be bounded above by $\Vert u_0\cdot\eta\Vert_{C^{k,\alpha}(S)}$ and $\Vert\varphi_nu_n\Vert_{C^{k,\alpha}(\Omega)}$, which are uniformly bounded with respect to $n$. Furthermore, the volume layer potentials and its first order partial derivatives can be bounded by $C\vert x\vert^{-1}$ for an  $n$-independent constant thanks to Theorem \ref{DecPotencialRiesz} and the above argument. Consequently, we get the same asymptotic behavior at infinity for the limit vector field $u$.

Let us show now that $\mathcal{T}(\Sigma,u)$ is a $\left(\rho_0,T,\delta/2\right)$-stream tube of $u$ and that the support of $\varphi$ lies in it. Since, by taking limits in the fourth inequality in (\ref{paso.limite.induccion.n}),
$$
\Vert u-u_0\Vert_{C^{k+1,\alpha}(\Omega)}\leq \frac{1}{2}\frac{2\delta}{C_\mathcal{P}T}e^{-\frac{1}{2}C_\mathcal{P}T\Vert u_0\Vert_{C^{k+1,\alpha}(\Omega)}},
$$
Corollary \ref{problematransporte.perturb.cor} yields the first assertion. The second one is clear by taking into account that $\mbox{supp}\,\varphi_n\subseteq \overline{\mathcal{T}_n}$, for every $n\in\NN$. Finally, to check that the limit solution is close to the initial strong Beltrami field $u_0$, it suffices to take limits in the third inequality in (\ref{paso.limite.induccion.n}) to get
\[
\Vert u-u_0\Vert_{C^{k+1,\alpha}(\Omega)}\leq \min\{\varepsilon_0,1\}\sum_{i=1}^{+\infty}\frac{1}{2^i}\Vert u_0\Vert_{C^{k+1,\alpha}(\Omega)}\leq \varepsilon_0\Vert u_0\Vert_{C^{k+1,\alpha}(\Omega)}.  
\qedhere \]
\end{proof}

\begin{rem}\label{paso.limite.caidaoptima.rem}
Notice that the generalized Beltrami field $u\in C^{k+1,\alpha}(\overline{\Omega},\RR^3)$ obtained by means of the preceding theorem solves the equation
$$\curl u-f u=0, \hspace{0.25cm}x\in \Omega,$$
with proportionality factor $f=\lambda+\varphi$ and a compactly supported perturbation $\varphi\in C^{k,\alpha}(\overline{\Omega})$. Moreover, by construction it decays as $\vert x\vert^{-1}$ at infinity. Let us show now why this is indeed the optimal decay. First, recall that
$$\divop(\varphi u)=\nabla \varphi\cdot u+\varphi\divop u=0,$$
since $u$ is divergence-free and $\varphi$ is a first integral of $u$. Second, $\varphi$ is compactly supported in $\overline{\mathcal{T}(\Sigma,u)}$, which is a $(\rho_0,T,\delta/2)$-stream tube of $u$. Indeed, consider any open subset $\Sigma'\subseteq S$ such that 
$$\supp\varphi^0\subseteq \Sigma'\subseteq \overline{\Sigma'}\subseteq \Sigma.$$
Then, the preceding proof actually shows that $\varphi$ is compactly supported in $\overline{\mathcal{T}(\Sigma',u)}$. Take any $x\in \Sigma\setminus \overline{\Sigma'}$ and note that $u(x)\cdot \eta(x)\geq \rho_0>0$. Hence, the transversality condition and Corollary \ref{BeltramiNoHomog.CaidaOptima.cor} entail the optimal decay $\vert x\vert^{-1}$ of the vector field $u$.
\end{rem}

A related remark in the harmonic case ($\lambda=0$) is in order now.

\begin{rem}\label{paso.limite.caidaoptima.armonicos.rem}
Recall that a similar result to that in Theorem \ref{paso.limite.teo} was previously proved in \cite{Kaiser} to obtain generalized Beltrami fields $u\in C^{1,\alpha}(\overline{\Omega},\RR^3)$ (nonlinear force-free fields), i.e., solutions to
$$\curl u=fu, \hspace{0.25cm}x\in \Omega,$$
with compactly supported small proportionality factors $f\in C^{0,\alpha}(\overline{\Omega})$.

On the one hand, the low regularity $C^{1,\alpha}$ and $C^{0,\alpha}$ is not a weakness in such result since despite not being directly considered in \cite{Kaiser}, our results in Section \ref{Teoria.Potencial.Tecnicas.Seccion} provide the necessary background to promote the existence theorem in \cite{Kaiser} to a high regularity setting.  On the other hand, such generalized Beltrami fields decay as $\vert x\vert^{-2}$ at infinity. There is no contradiction neither with Corollary \ref{BeltramiNoHomog.CaidaOptima.GeneralizedBeltramis.cor} (since it holds under the assumption $\lambda\neq 0$) nor with the Liouville theorem in \cite{CC,Nadirashvili} (since it just holds for globally defined generalized Beltrami fields). 

On the contrary, the latter can be used to show an interesting property of such generalized Beltrami fields obtained as perturbations of harmonic fields. Specifically: they cannot be globally extended to the whole space by virtue of the fall-off obstructions in \cite{CC,Nadirashvili}. Nevertheless, the same cannot be directly said for generalized Beltrami fields obtained as perturbations of strong Beltrami fields.
\end{rem}

\section{Knotted and linked stream lines and tubes in generalized Beltrami fields}\label{Estructuras.Anudadas.Seccion}

Our objective in this section is to apply the convergence result for the modified Grad--Rubin method (\ref{paso.limite.esquemaiterativo.Intro.form}) that we established in the previous section (Theorem~\ref{paso.limite.teo}) to show the existence of almost global Beltrami fields of class $C^{k+1,\alpha}$ with a nonconstant factor that realize any given configuration of vortex tubes and vortex lines, modulo a small diffeomorphism. Here $k$ is an arbitrary integer.

\subsection{Knots and links in strong Beltrami fields}
The main result in \cite{Enciso2} ensures the existence of strong Beltrami field with the sharp decay at infinity and exhibiting any finite collection of (possibly knotted and liked) vortex lines and thin vortex tubes. The vortex tubes can be though as (small deformations of) metric neighborhoods of a smooth knotted loop in $\RR^3$. Specifically, for a closed curve $\Gamma\subseteq \RR^3$ of any knot type the associated tube of thickness $\varepsilon>0$ will be denoted by
$$
\mathcal{T}_\varepsilon(\Gamma):=\{x\in \RR^3\,:\,\mbox{dist}(x,\Gamma)<\varepsilon\}.
$$

With this notation, the main result of~\cite{Enciso2} reads as follows. In the statement, let us agree to say that a vortex tube $\mathcal T$ of a field~$u$ is \textit{structurally stable} if any divergence-free field $u'$ that is close enough to~$u$ in $C^{3,\alpha}(V)$ has an ``invariant tube'' (i.e., invariant torus) of the form $\Psi(\mathcal T)$, where $V$ is any fixed neighborhood of $\mathcal T$ and $\Psi$ is a diffeomorphism of $\RR^3$ that is close to the identity in $C^\alpha$.

\begin{theo}\label{BeltramiFuerte.TubosAunudados.Enciso.teo}
Let $\Gamma_1,\ldots,\Gamma_n$ be $n$ pairwise disjoint (possibly knotted and linked) closed curves in $\RR^3$. For any small enough $\varepsilon$, we can transform the collection of pairwise disjoint thin tubes $\mathcal{T}_\varepsilon(\Gamma_1),$ $\ldots,$ $\mathcal{T}_\varepsilon(\Gamma_n)$ by a diffeomorphism $\Phi$ of $\RR^3$, arbitrarily close to the identity in any $C^m$ norm, so that $\Phi(\mathcal{T}_\varepsilon(\Gamma_1)),$ $\dots,$ $\Phi(\mathcal{T}_\varepsilon(\Gamma_n))$ are vortex tubes of a strong Beltrami field $u$, which satisfies the equation $\curl u=\lambda u$ in $\RR^3$ for some non-zero constant $\lambda$ and decays at infinity as $\vert x\vert^{-1}$. Furthermore, these vortex tubes are structurally stable.
\end{theo}

For the benefit of the reader, let us briefly discuss the main ideas of the proof. The gist is to construct a ``local'' Beltrami field realizing the desired configuration of vortex tubes so that they are structurally stable and they approximate the local Beltrami field by a Beltrami field with the same constant that is global, that is, defined everywhere in $\RR^3$.

Hence, the first step is to show the existence of Beltrami field that possess the desired collection of vortex tubes and which is defined in a neigborhood of the boundary tori. Since the tubes can be chosen analytic without loss of generality (by slightly deformating the initial vortex tube configuration), one could try to do that using a kind of Cauchy--Kowalewski theorem that was proved in \cite{Enciso1}. However, this would lead to a local Beltrami field defined in a region whose complement is not connected, and this should prevent us from applying any kind of approximation theorem (this restriction is already present in the classical theorem of Runge, and appears in all the approximations theorems known to date). This leads to move from the above-mentioned Cauchy problem to the following boundary value problem of Neumann type in the interior of the above tori
$$
\left\{\begin{array}{ll}
\displaystyle\curl \widetilde{u}_0=\lambda \widetilde{u}_0, & x\in \cup_{i=1}^n\mathcal{T}_\varepsilon(\Gamma_i),\\
\displaystyle \widetilde{u}_0\cdot \nu=0, & x\in \cup_{i=1}^n \partial\mathcal{T}_\varepsilon(\Gamma_i),
\end{array}\right.
$$
where $\nu$ stands for the unit outward normal vector to the boundary of the tube. The tangency boundary condition ensures that any solution $\widetilde{u}_0$ has the invariant tori $\partial \mathcal{T}_\varepsilon(\Gamma_1),\ldots,\partial \mathcal{T}_\varepsilon(\Gamma_n)$. To get nontrivial solutions, one can prescribe the $L^2$ projection of $\widetilde{u}_0$ into the space of tangential harmonic fields (called the \textit{harmonic part} of $\widetilde{u}_0$). Therefore, the above Neumann boundary value problem for a vector field $\widetilde{u}_0$ with fixed harmonic part is uniquely solvable by means of a variational approach as long as $\lambda$ does not belong to the spectrum of certain operator with a compact inverse (in particular, it works if $|\lambda|< C/\varepsilon$). 

The above argument ensures the existence of local Beltrami fields tangent to the tubes but it does not say anything about the structural stability of the tubes. This is obtained by applying a KAM theorem. What makes the proof subtle is that the applicability of the KAM theorem involves a combination of delicate PDE and dynamical systems estimates. Indeed, the above existence theorem applies to most domains and $\lambda$'s, while the structural stability hinges on the smallness of~$\lambda$. The point is to derive fine asymptotics for the field $\widetilde{u}_0$ for small $\varepsilon$ and use this information to show that one can effectively use KAM theory on an associated Poincar\'e map. The reason for which the estimates are subtle is that, from the point of view of KAM theory, the situation is very degenerate because the twist condition is barely satisfied, and in turn this has a bearing on the power of the PDE estimates that are needed to make the KAM argument go through.

What is of a greater direct interest for our purpose here is the way that we pass from the local solution $\widetilde{u}_0$ to a global Beltrami field $u_0$ solving the same Beltrami equation in the whole space and exhibiting the desired decay properties at infinity. Especially, in the next section we will need to use that the global Beltrami field $u_0$ is of the form
\begin{equation}\label{equ0}
u_0=\frac{\curl(\curl+\lambda)}{2\lambda^2} \widetilde u_0\,,
\end{equation}
where $\widetilde u_0$ is a finite Fourier--Bessel series of the form
\begin{equation}\label{eqtu0}
\widetilde{u}_0(x)=\sum_{l=0}^{L}\sum_{m=-l}^{l}c_l^mj_l(\lambda \vert x\vert)Y_{l}^m\left(\frac{x}{\vert x\vert}\right),
\end{equation}
where $c_l^m$ are constant vectors in $\CC^3$, $j_l$ stands for the {spherical Bessel function of first kind and $l$-th order} and $Y_l^m$ is the {$m$-th spherical harmonic of $l$-th order}. Obviously $\widetilde u_0$ is real-valued.

\subsection{Knots and links in almost global generalized Beltrami fields}

Our goal in this section is to show that the partial stability result for almost global Beltrami fields allows us to conclude the existence of Beltrami fields with a non-constant proportionality factor that are defined in all of $\RR^3$ but, say, in the complement of an arbitrarily small ball, and which have a collection of vortex tubes and vortex lines of arbitrary topology. More precisely, our objective is to prove the following result. Let us recall that in the Introduction we defined that a vortex tube (invariant torus) of a divergence-free field $u$ is structurally stable if any divergence-free field that is close enough to $u$ in $C^{3,\alpha}$ has an invariant torus given by a $C^{0,\alpha}$-small diffeomorphism of the initial tube. Although we shall not state these properties explicitly, just as in~\cite{Enciso2} the vortex tubes that we construct are accumulated on by a positive-measure set of invariant tori on which the vortex lines are ergodic.

\begin{theo}\label{BeltramiGeneralizado.TubosAunudados.teo}
Let $G$ be an exterior domain satisfying the hypotheses (\ref{GSigmaMu.hipot}) and consider any collection of disjoint knotted and linked thin tubes $\mathcal{T}_\varepsilon(\Gamma_1),\ldots,\mathcal{T}_\varepsilon(\Gamma_n)$ whose closure is contained in the exterior domain~$\Omega$. Then, for $\varepsilon$ small enough and any $k,\alpha$ there exists a nonzero constant $\lambda$, an open subset $\Sigma\subseteq S$ and some $\delta_0>0$ with the following property: for any function $\varphi^0\in C^{k+1,\alpha}_c(\Sigma)$ with $\Vert \varphi^0\Vert_{C^{k+1,\alpha}(\Sigma)}<\delta_0$ there is a Beltrami field~$u\in C^{k+1,\alpha}(\overline{\Omega},\RR^3)$ with factor $\lambda+\varphi$, where $\varphi$ is a function in $C^{k,\alpha}(\overline{\Omega})$ satisfying $\varphi|_\Sigma=\varphi^0$:
$$\left\{\begin{array}{ll}
\curl u=(\lambda+\varphi)u, & x\in\Omega,\\
\divop u=0, & x\in\Omega.
\end{array}\right.$$
Furthermore, $u=O\left(\vert x\vert^{-1}\right)$ as $\vert x\vert\rightarrow +\infty$, the support of $\varphi$ is compact and lies in the $(\rho_0,T,\delta)$-stream tube $\mathcal{T}(\Sigma,u)$ of $u$ radiating from $\Sigma$ (with the exception of the endpoints) and $\mathcal{T}_\varepsilon(\Gamma_1),\ldots,\mathcal{T}_\varepsilon(\Gamma_n)$ can be modified by a diffeomorphism $\Phi$ close enough to the identity in any $C^m$ norm into a collection of structurally stable vortex tubes of $u$, $\Phi(\mathcal{T}_\varepsilon(\Gamma_1)),\ldots,\Phi(\mathcal{T}_\varepsilon(\Gamma_n))$, (possibly) knotted and linked with $\mathcal{T}(\Sigma,u)$.
\end{theo}

\begin{proof}
Take a curve $\Gamma_0$ intersecting $S$ transversally and such that $\mathcal{T}_\varepsilon(\Gamma_0)\cap \Omega$ has only a connected component. We also assume that $\Gamma_0$ does not intersect any of the other curves $\Gamma_j$, so that the setup is then as depicted in Figure \ref{fig:Fig6}). For $\varepsilon>0$ small enough, Theorem \ref{BeltramiFuerte.TubosAunudados.Enciso.teo} asserts the existence of some diffeomorphism $\Phi'$ arbitrarily close to the identity map in any $C^m$ norm such that $\Phi'(\mathcal{T}_\varepsilon(\Gamma_0)),\ldots,\Phi'(\mathcal{T}_\varepsilon(\Gamma_n))$ are vortex tubes of a strong Beltrami field $u_0$ which satisfies the equation $\curl u_0=\lambda u_0$ in $\RR^3$ for some non-zero constant $\lambda$ (of order $\varepsilon^3$). By construction, these tubes are structurally stable and $\Phi'$ can be assumed to be arbitrarily close to the identity in any $C^m$ norm, so the new thin tubes enjoy the same geometric features as we had assumed on the initial ones. Let us then take the point $x_0\in S\cap \Phi'(\Gamma_0)$ where $u_0$ points outwards and consider any open and connected neighborhood $\Sigma$ of $x_0$ in $S$ such that $\Sigma\subseteq S\cap \Phi'(\mathcal{T}_\varepsilon(\Gamma_0))$.

\begin{figure}[t]
\centering

\includegraphics[scale=0.7]{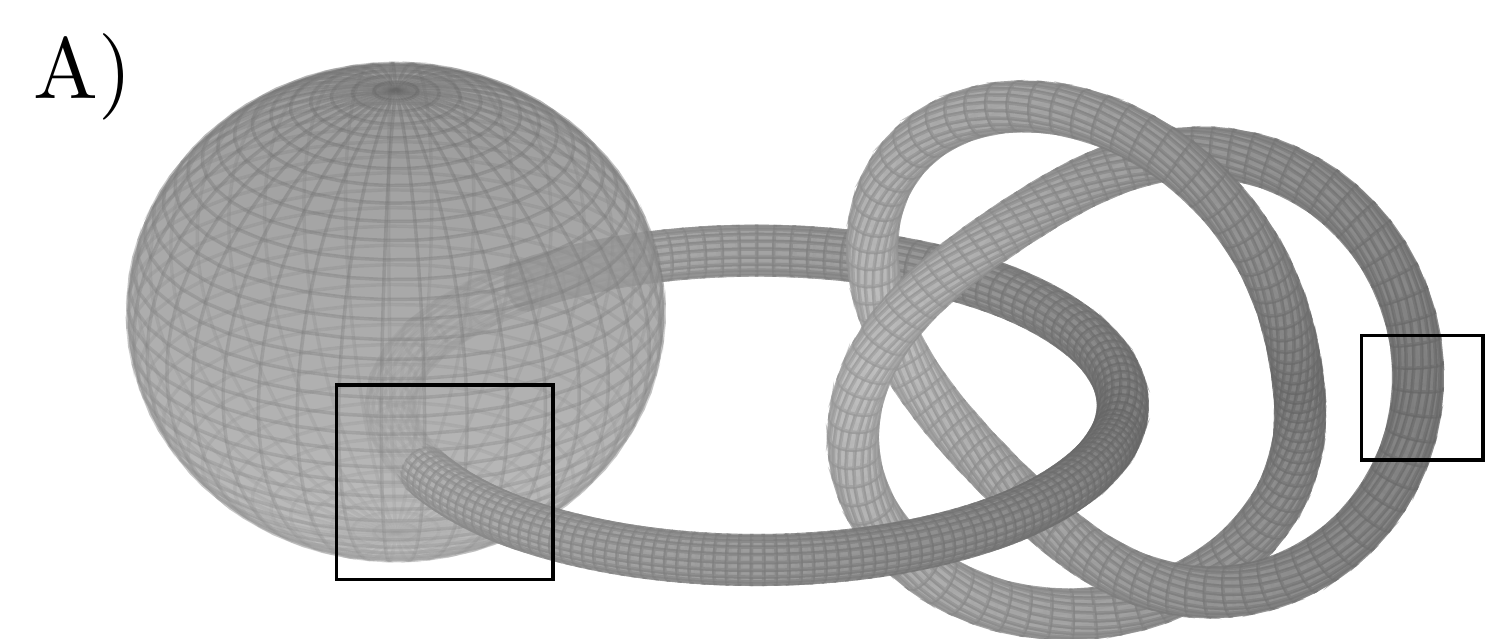}

\centering

\includegraphics[scale=0.7]{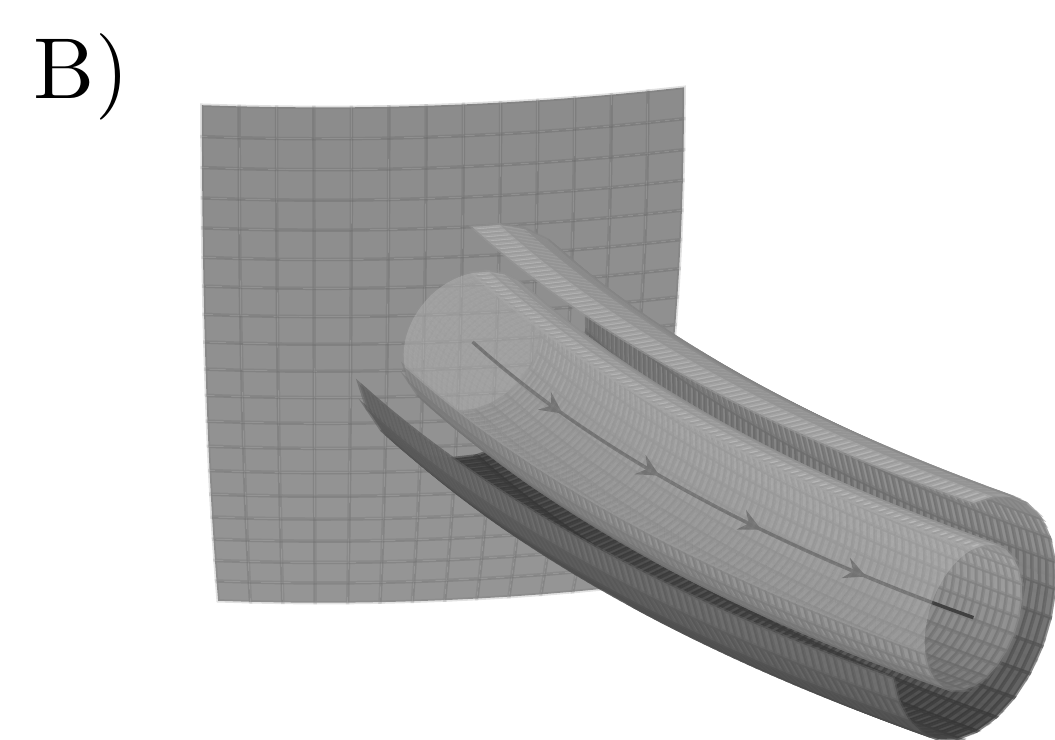}
\hspace{1cm}
\includegraphics[scale=0.7]{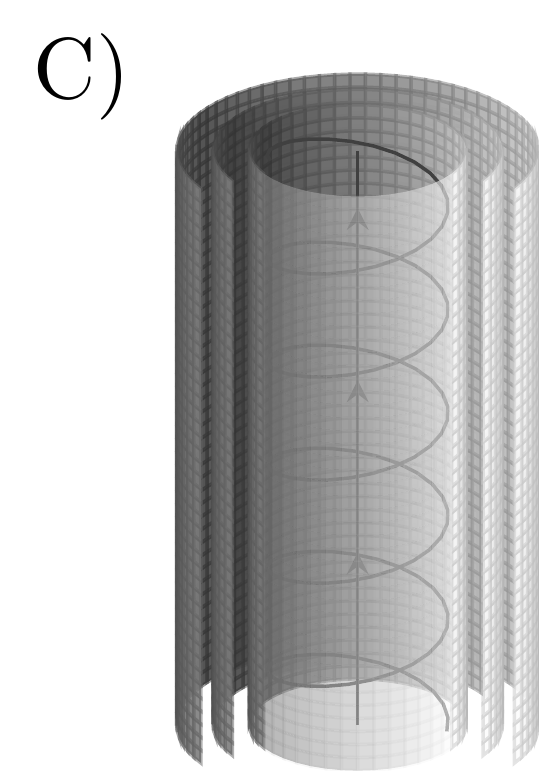}

\caption{A) Collection of knotted and linked vortex tubes of the strong Beltrami field $u_0$, $\{\Phi'(\mathcal{T}_\varepsilon(\Gamma_0)),\Phi'(\mathcal{T}_\varepsilon(\Gamma_1))\}$, respectively homeomorphic to the unknot and to the trefoil. B) Transverse intersection of the vortex tube $\Phi'(\mathcal{T}_\varepsilon(\Gamma_0))$ and the interior domain $G$. Here we have zoomed in the squared region on the left side of the above figure, showing the smaller outward pointing $(\rho_0,T,\delta)$-stream tube of $u_0$ that emerges from $\Sigma$. The perturbation $\varphi$ of the constant proportionality factor $\lambda$ will be supported there. C) Zoom of the vortex tube $\Phi'(\mathcal{T}_\varepsilon(\Gamma_1))$ with trefoil knot. It shows the internal structure of such vortex tube of $u_0$, which contains uncountably many nested tori and knotted vortex lines.}
\label{fig:Fig6}
\end{figure}

We recorded in Equations~\eqref{equ0}-\eqref{eqtu0} that $u_0$ is of the form
\[
u_0=\frac{\curl(\curl+\lambda)}{2\lambda^2}\sum_{l=0}^{L}\sum_{m=-l}^{l}c_l^mj_l(\lambda \vert x\vert)Y_{l}^m\left(\frac{x}{\vert x\vert}\right).
\]
Since $u_0$ is obviously real-valued, it is the real part of the vector field
\[
v_0=\frac{\curl(\curl+\lambda)}{2\lambda^2}\sum_{l=0}^{L}\sum_{m=-l}^{l}c_l^mh_l^{(1)}(\lambda \vert x\vert)Y_{l}^m\left(\frac{x}{\vert x\vert}\right),
\]
where $h_l^{(1)}:=j_l+ iy_l$ is the spherical Hankel function of $l$-th order and $y_l$ denotes the spherical Bessel function of the second kind and $l$-th order. By construction, $v_0$ satisfies the Beltrami equation (and in particular is smooth) in $\RR^3\backslash\{0\}$, while it diverges at the origin due to the presence of a Bessel function of the second kind. In particular, it is a Beltrami field in~$\Omega$.

The advantage of $v_0$ is that, as the Hankel function $h_l^{(1)}$ has been chosen to satisfy the scalar radiation condition
\[
(\partial_r -i\lambda)h_l^{(1)}(\lambda r)=o(r^{-1}),
\]
it is straightforward to check that $v_0\in C^{k+1,\alpha}(\overline{\Omega},\CC^3)$ is a complex-valued solution to the Beltrami equation in the exterior domain $\Omega$ which satisfies the $L^1$ SMB radiation condition (\ref{CondicionRadiacion.Beltrami.L1SilverMullerBeltrami.form}) and the weak $L^1$ decay property (\ref{CondCaidaBeltrami.form}) (see \cite[Equation 2.41]{ColtonKress} along with Remark \ref{HelmholtzBeltrami.relacion.obs} and Figure \ref{fig:Fig5}). It is also apparent that $\mathcal{T}(\Sigma,u_0)\subseteq \Phi'(\mathcal{T}_\varepsilon(\Gamma_0))$ is a $(\rho_0,T,\delta)$-stream tube of $u_0$ by construction (see Figure \ref{fig:Fig6}), and that $\lambda\sim\varepsilon^3$ can be prevented from being a Dirichlet eigenvalue of the Laplace operator in the interior domain $G$ as long as $\varepsilon$ is taken small enough. Then, we are ready to apply the convergence Theorem \ref{paso.limite.teo} for the modified Grad--Rubin method starting up with the strong Beltrami field $u_0$. This result ensures the existence of $\delta_0>0$ so that whenever $\Vert \varphi^0\Vert_{C^{k+1,\alpha}(\Sigma)}\leq \delta_0$, then there exists a generalized Beltrami field $u\in C^{k+1,\alpha}(\overline{\Omega},\RR^3)$ and a perturbation $\varphi\in C^{k,\alpha}(\overline{\Omega})$ solving the exterior boundary value problem
$$\left\{\begin{array}{ll}
\curl u=(\lambda+\varphi)u, & x\in\Omega,\\
\divop u=0, & x\in \Omega,\\
u\cdot \eta=u_0\cdot \eta, & x\in S,\\
\varphi=\varphi^0, & x\in \Sigma.
\end{array}\right.$$
Furthermore, $u=O\left(\vert x\vert^{-1}\right)$ as $\vert x\vert\rightarrow +\infty$, $\mathcal{T}(\Sigma,u)$ is a $(\rho_0,T,\delta/2)$ stream tube of $u$, $\varphi$ is compactly supported in the closure of such stream tube and
$\Vert u-u_0\Vert_{C^{k+1,\alpha}(\Omega)}$
can be made arbitrarily small. In view of the structural stability of the vortex tubes of $u_0$, the theorem follows.
\end{proof}

\section{Local stability of generalized Beltrami fields}
\label{Ch.local}

Our objective in this section is to show that, in fact, any generalized Beltrami field possesses a local partial stability property which can be essentially regarded as a local version of Theorem~\ref{paso.limite.teo}. We recall that, in view of the results in~\cite{Enciso1}, one cannot prove a full stability result even in arbitrarily small open sets, so we regard this partial stability (where partial is understood in a very precise sense) as a satisfactory counterpart to the results in this paper.

\subsection{A local stability theorem}

We shall next present the local stability result that constitutes the core of this section. The philosophy of this result is that, as one is able to perturb strong Beltrami fields, one should also be able to perturb generalized Beltrami fields in small domains, since in a small region a $C^{k,\alpha}$~function behaves as a constant plus a small perturbation. Somehow, this reduces our effort to estimates similar to the ones that we have already obtained, so our presentation of the proof of this result will be a little sketchier than before. The gist will be to show that, although the strong convergence of the modified Grad--Rubin iterative scheme cannot be granted in $C^{k+1,\alpha}$ for $u_n$ and $C^{k,\alpha}$ for $f_n$, we can pass to the limit in $C^{1,\alpha}$ and $C^{0,\alpha}$ provided that both the domain and the perturbation of the proportionality factor are small enough. Elliptic regularity will then yield the desired high order regularity by a bootstrap argument. 

In order to support our argument, let us first sketch the effect of the size of the domain on the solutions of the next Neumann boundary value problem associated with the inhomogeneous Beltrami equation in some open ball $B_R(x_0)$
\begin{equation}\label{BeltramiNoHomog.ball.eq}
\left\{
\begin{array}{ll}
\curl u-\lambda u=w, & x\in B_R(x_0),\\
u\cdot \eta=0, & x\in \partial B_R(x_0),
\end{array}
\right.
\end{equation}
where $w\in C^{0,\alpha}(\overline{B}_R(x_0),\RR^3)$ has zero flux. We will be interested in the case where $R$ becomes very small. 

This problem has being carefully analyzed in \cite{vonWahl} for bounded domains and in \cite{Kaiser} for exterior unbounded domains in the harmonic case ($\lambda=0$). The non-harmonic counterpart was studied in \cite{Kress} and Section \ref{Beltrami.NoHomogenea.Seccion} for the inhomogeneous Beltrami equation in bounded and exterior domains respectively. In the bounded setting, $\lambda$ has to be assumed ``regular'' (see \cite{Kress}). To this end, notice that taking $\vert\lambda\vert<c/R$ (for an appropriate universal constant $c>0$) prevents $\lambda$ from being an eigenvalue of the Laplacian in $B_R(x_0)$. Hence, $\vert\lambda\vert<c/R$ is a sufficient condition ensuring the well-posedness of (\ref{BeltramiNoHomog.ball.eq}). All the above results provide an estimate for the unique solution $u$ to (\ref{BeltramiNoHomog.ball.eq}) in terms of $w$ of the form
$$\Vert u\Vert_{C^{1,\alpha}(B_R(x_0))}\leq C_{\lambda,R}\Vert w\Vert_{C^{0,\alpha}(B_R(x_0))},$$
where the dependence of the constant $C_{\lambda.R}$ on $\lambda$ and $R$ is not explicit. The next technical result aims to provide some explicit $R$-dependent estimate for $u$ in some space.

\begin{lem}\label{div-curl.ball.estimate.lem}
Let $u\in C^{1,\alpha}(\overline{B}_R(x_0),\RR^3)$ be the unique solution to the Neumann boundary value problem associated with the Beltrami equation (\ref{BeltramiNoHomog.ball.eq}) for $\vert \lambda\vert<c/R$ and $R\in (0,1)$. Then,
\begin{equation}\label{div-curl.ball.estimates.form}
\Vert u\Vert_{C^{1,\alpha}(B_R(x_0))}\leq C R^{-\alpha}\Vert w\Vert_{C^{0,\alpha}(B_R(x_0))},
\end{equation}
for some positive constant $C$ depending on $\alpha$ but not on $u,w,x_0$ or $R$.
\end{lem}

\begin{proof}
To obtain an explicit $R$-dependent estimate of $u$ in some space, let us perform the next change of variables $y=\frac{x-x_0}{R}$. Then, one obtains the following vector fields in the unit ball centered at the origin:
$$U(y)=u(x),\hspace{0.5cm}W(y)=w(x),$$
solving the next Neumann boundary value problem for the Beltrami equation in $B_1(0)$:
$$
\left\{
\begin{array}{ll}
\curl U-\lambda R\,U=R\,W, & y\in B_1(0),\\
U\cdot \eta=0, & y\in \partial B_1(0).
\end{array}
\right.
$$
Thus, the above-mentioned results yield the following estimate for some $R$-independent positive constant $C$
$$\Vert U\Vert_{C^{1,\alpha}(B_1(0))}\leq C R \Vert W\Vert_{C^{0,\alpha}(B_1(0))},$$
where the assumption $\vert\lambda\vert<c/R$ has been used to avoid the $\lambda$-dependence of the constant $C$. Note that by definition
\begin{align*}
\Vert W\Vert_{C^{0,\alpha}(B_1(0))}&=\Vert w\Vert_{C^0(B_R(x_0))}+R^{\alpha}[w]_{\alpha,B_R(x_0)},\\
\Vert U\Vert_{C^{1,\alpha}(B_1(0))}&=\Vert u\Vert_{C^0(B_R(x_0))}+R\sum_{i=1}^3\Vert \partial_{x_i}u\Vert_{C^0(B_R(x_0))}+R^{1+\alpha}\sum_{i=1}^3 [\partial_{x_i}u]_{\alpha,B_R(x_0)}.
\end{align*}
Since $R\in (0,1)$, then we are led to (\ref{div-curl.ball.estimates.form}).
\end{proof}

Another key ingredient is to show that $C^{1,\alpha}$ vector fields near a non-equilibrium point verify a ``structurally stable'' flow box theorem, to be understood in the next precise sense.

\begin{lem}\label{flow-box.structurally.stable.lem}
Let $u\in C^{1,\alpha}(\Omega,\RR^3)$ be a (nontrivial) vector field and consider some $x_0\in \Omega$ such that $u(x_0)\neq 0$. There exist $R_0>0$ and $\delta_0>$ such that $\overline{B}_{2R_0}(x_0)\subseteq \Omega$, $u$ vanishes nowhere in the ball and for every $0<R<R_0$ there exists some surface $\Sigma_R\subseteq \partial B_{R}(x)$ and a positive function $T_R\in C(\Sigma_R)$ such that for every $v\in C^{1,\alpha}(\overline{B}_{R}(x_0),\RR^3)$ with $\Vert u-v\Vert_{C^{1,\alpha}(B_{R}(x_0))}<\delta_0$, then
$$B_R(x_0)\subseteq \mathcal{T}(\Sigma_R,\overline{v},T_R)\subseteq B_{2R}(x_0).$$
\end{lem}

\begin{figure}[t]
\centering
\includegraphics[scale=0.75]{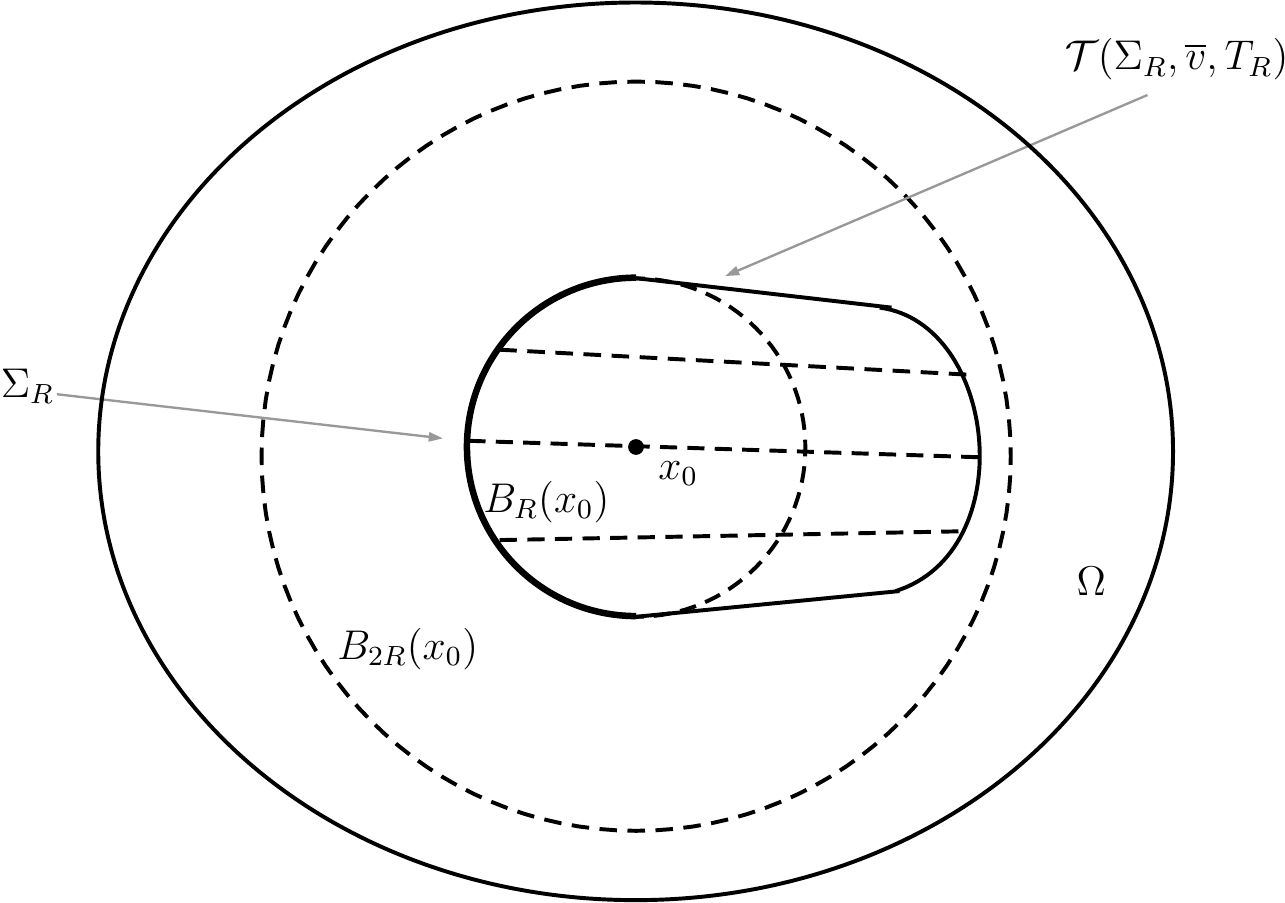}
\caption{Flow box $\mathcal{T}(\Sigma_R,\overline{v},T_R)$ covering the small ball $B_R(x_0)$.}
\label{fig:FigLocal-pics}
\end{figure}

Here, the above stream tube reads
$$\mathcal{T}(\Sigma_R,\overline{v},T_R):=\{X^{\overline{v}}(t;0,x):\,x\in \Sigma_R,\,t\in (0,T_R(x))\},$$
$\overline{v}$ is the Calder\'on extension of $v$ from $B_R(x_0)$ to $\overline{B}_{2R_0}(x_0)$ (Proposition \ref{ExtensionHolder.pro}) and the height $T_R$ of the stream tube is not constant but it continuously depends, stream line by stream line, on the base point $x\in \Sigma_R$ (see Figure \ref{fig:FigLocal-pics}). Furthermore, the parametrizations $\mu_R$ of $\Sigma_R$ can be normalized by choosing 
$$\mu_R(s)=R\mu(s),\ s\in D_R,$$
for some open subset $D_R\subseteq D_1(0)$ of the unit disc centered at $0$, and some local parametrization of the unit sphere $\mu:D_1(0)\longrightarrow \partial B_1(x_0)$. Since the proof follows the same lines as Lemma \ref{TuboFlujoRecurr.Perturb.lem} in Section \ref{Esquema.Iterativo.Seccion}, we skip it and pass to the central result of this section.

\begin{theo}\label{perturbacion.local.teo}
Let $u_0$ be a nontrivial generalized Beltrami field of class $C^{k+1,\alpha}(\Omega,\RR^3)$, where $k\in\NN$ and $\alpha\in (0,1)$, and consider its (nonconstant) proportionality factor $f_0\in C^{k,\alpha}(\Omega)$. Take some nonequilibrium point $x_0\in \Omega$ of $u_0$ and fix some $\varepsilon_0>0$. Then, for each small enough radius $R>0$ there is some surface $\Sigma_R\subseteq \partial B_R(x_0)$ and some constant $\delta_R>0$ so that for every $\varphi^0\in C^{k+1,\alpha}(\Sigma_R)$ with $\Vert \varphi^0\Vert_{C^{k+1,\alpha}(\Sigma_R,\mu_R)}<\delta_R$ there exist $\varphi\in C^{k,\alpha}(\overline{B}_R(x_0))$ and $u\in C^{k+1,\alpha}(\overline{B}_R(x_0),\RR^3)$ such that $\varphi=\varphi^0$ on $\Sigma_R$ and $u$ is a strong Beltrami field with proportionality factor $f_0+\varphi$ enjoying the same normal component as $u_0$ in $\partial B_R(x_0)$, i.e.,
$$\left\{
\begin{array}{ll}
\curl u=(f_0+\varphi)u, & x\in B_R(x_0),\\
\divop u=0, & x\in B_R(x_0),\\
u\cdot \eta=u_0\cdot \eta, & x\in \partial B_R(x_0).
\end{array}
\right.$$
Furthermore,
$$\Vert u-u_0\Vert_{C^{k+1,\alpha}(B_R(x_0))}\leq \varepsilon_0\Vert u_0\Vert_{C^{k+1,\alpha}(B_R(x_0))}.$$
\end{theo}

\begin{proof}
The proof has two steps. First, we will prove the theorem for low H\"{o}lder exponents and regularity (namely, $\alpha\in (0,1/2)$ and $k=0$). Second, we will show a bootstrap argument based on elliptic gain of regularity that will raise the estimates in the first step to its full strength and will conclude the proof of the theorem for general regularity and H\"{o}lder exponents.

Then, let us first assume that $\alpha\in (0,1/2)$, define $\lambda_0:=f_0(x_0)$ and fix some radius $R_0>0$ so that $\overline{B}_{2R_0}(x_0)\subseteq \Omega$, $u_0$ vanishes nowhere in $\overline{B}_{2R_0}(x_0)$ and the assertions in Lemma \ref{flow-box.structurally.stable.lem} fulfil. Without loss of generality, we can assume that $R_0<\min\{1,c/\vert \lambda_0\vert\}$. Moreover, note that the homogeneous generalized Beltrami equation can be restated as an inhomogeneous Beltrami equation with constant proportionality factor and an inhomogeneous term taking the form of a small remainder, i.e.,
\begin{equation}\label{BeltamiNoHomog.restated.remainder.eq}
\curl u_0-\lambda_0u_0=\mathcal{R}(x-x_0)u_0, \hspace{0.25cm} x\in \Omega,
\end{equation}
where $f_0(x)=\lambda_0+\mathcal{R}(x-x_0)$ for every $x\in\overline{B}_{2R_0}(x_0)$, i.e.,
$$\mathcal{R}(z):=\left(\int_0^1 \nabla f_0(x_0+\theta z) \,d\theta\right)\cdot z,\hspace{0.4cm}z\in \overline{B}_{2R_0}(0).$$

Next, consider the following modified iterative scheme of Grad--Rubin type. It consists of a sequence of transport equations
\begin{equation}\label{Grad-Rubin-modied.ball.eq1}
\left\{\begin{array}{ll}
\nabla\varphi_n\cdot u_n=-\nabla f_0\cdot u_n, & x\in B_R(x_0),\\
\varphi_n=\varphi^0, & x\in \Sigma_R.,
\end{array}\right.
\end{equation}
along with a sequence of boundary value problems associated with the inhomogeneous Beltrami equation
\begin{equation}\label{Grad-Rubin-modied.ball.eq2}
\left\{\begin{array}{ll}
\curl u_{n+1}-\lambda_0 u_{n+1}=\mathcal{R}(x-x_0)u_n+\varphi_n u_n, & x\in B_R(x_0),\\
u_{n+1}\cdot \eta=u_0\cdot \eta, & x\in \partial B_R(x_0).
\end{array}\right.
\end{equation}
Note that they have been chosen in a consistent way so that as long as $\{u_n\}_{n\in\NN}$ and $\{\varphi_n\}_{n\in\NN}$ have limits (in some sense), then the limits $u$ and $\varphi$ give rise to a generalized Beltrami field whose proportionality factor is a perturbation $f_0+\varphi$ of the initial factor $f_0$. Without loss of generality, we can assume that $\lambda_0\neq 0$ (in the case $\lambda_0=0$ would need the additional condition $\divop u_{n+1}=0$). 

Let us show that for every $n\in\NN$ both $u_{n+1}\in C^{1,\alpha}(\overline{B}_R(x_0),\RR^3)$ and $f_n\in C^{0,\alpha}(\overline{B}_R(x_0))$ are well defined and that
\begin{equation}\label{paso.limite.paso.n.local}
\left\{\begin{array}{rcl}
\displaystyle\Vert u_{n+1}-u_n\Vert_{C^{1,\alpha}(B_R(x_0))}&\leq&\displaystyle\frac{1}{2^n}\Vert u_1-u_0\Vert_{C^{1,\alpha}(B_R(x_0))}<\min\{\varepsilon_0,1\}\frac{1}{2^{n+1}}\Vert u_0\Vert_{C^{1,\alpha}(B_R(x_0))},\\
\displaystyle\Vert u_{n+1}-u_0\Vert_{C^{1,\alpha}(B_R(x_0))}&\leq&\displaystyle\min\{\varepsilon_0,1\}\sum_{i=1}^{n+1}\frac{1}{2^i}\Vert u_0\Vert_{C^{1,\alpha}(B_R(x_0))},\\
\displaystyle\Vert u_{n+1}\Vert_{C^{1,\alpha}(B_R(x_0))}&\leq&\displaystyle\min\{\varepsilon_0,1\}\sum_{i=0}^{n+1}\Vert u_0\Vert_{C^{1,\alpha}(B_R(x_0))}.
\end{array}\right.
\end{equation}
Let us start with $n=0$. On the one hand, the transport problem (\ref{Grad-Rubin-modied.ball.eq1}) with $n=0$ can be solved in $B_R(x_0)$ as $B_R(x_0)\subseteq \mathcal{T}(\Sigma_R,u_0,T_R)\subseteq B_{2R}(x_0)$ by virtue of Lemma \ref{flow-box.structurally.stable.lem}. Indeed, 
$$\varphi_0(X^{u_0}(t;0,x))=\varphi^0(x)-\int_0^t(\nabla f_0\cdot u_0)(X^{u_0}(\tau,0,x))\,d\tau,\hspace{0.25cm}x\in \Sigma_R,\,t\in (0,T_R(x))$$
defines a solution in $\mathcal{T}(\Sigma_R,u_0,T_R)$ and, in particular, in $B_R(x_0)$. Now, notice that
\begin{align*}
\int_{\partial B_R(x_0)}&\left(\mathcal{R}(\cdot-x_0)u_0 +\varphi_0 u_0\right)\cdot \eta\,dS+\lambda_0\int_{\partial B_R(x_0)}u_0\cdot \eta\,dS\\
&=\int_{B_R(x_0)}\left(\nabla(f_0+\varphi_0)\cdot u_0+(f_0+\varphi_0)\divop u_0\right)\,dx=0,
\end{align*}
and $\lambda$ is regular (see \cite{Kress}) with respect to the inhomogeneous problem (\ref{Grad-Rubin-modied.ball.eq2}) with $n=0$ because $R<R_0<c/\vert \lambda_0\vert$. Hence, (\ref{Grad-Rubin-modied.ball.eq2}) has an unique solution $u_1\in C^{1,\alpha}(\overline{B}_R(x_0),\RR^3)$ by virtue of the existence theorem in \cite{Kress}. Notice that since $\divop u_0=0$ and the first integral equations in (\ref{Grad-Rubin-modied.ball.eq1}) hold, then
$$-\lambda_0\divop u_{1}=(f_0+\varphi_0)\divop u_0+\nabla (f_0+\varphi_0)\cdot u_0=0.$$
Furthermore, $u_1-u_0$ solves the Neumann boundary value problem
$$\left\{\begin{array}{ll}
\displaystyle(\curl-\lambda_0)(u_1-u_0)=\mathcal{R}(x-x_0)u_0+\varphi_0 u_0, & x\in B_R(x_0),\\
\displaystyle(u_1-u_0)\cdot \eta=0, & x\in B_R(x_0).
\end{array}\right.$$
Consequently,
\begin{align*}
\Vert u_1-u_0\Vert_{C^{1,\alpha}(B_R(x_0))}&\leq \frac{C}{R^\alpha}\Vert \mathcal{R}(\cdot-x_0)u_0+\varphi_0u_0\Vert_{C^{0,\alpha}(B_R(x_0))}\\
&\leq \frac{C}{R^\alpha}\left(\Vert \mathcal{R}(\cdot-x_0)\Vert_{C^{0,\alpha}(B_R(x_0))}+\Vert \varphi_0\Vert_{C^{0,\alpha}(B_R(x_0))}\right)\Vert u_0\Vert_{C^{0,\alpha}(B_R(x_0))}.
\end{align*}
A similar result to that in Theorem \ref{problematransporte.teo} yields the estimate
\begin{multline*}
\Vert \varphi_0\Vert_{C^{0,\alpha}(B_R(x_0))}\leq \left(\Vert \varphi^0\Vert_{C^{1,\alpha}(\Sigma_R,\mu_R)}+R^{1-\alpha}+\Vert T_R\Vert_{C^0(\Sigma_R)}\right)\\
\times\kappa\left(\Vert u_0\Vert_{C^{1,\alpha}(B_R(x_0))},\Vert T_R\Vert_{C^0(\Sigma_R)},\Vert \mu_R\Vert_{C^{1,\alpha}(B_R(x_0))}\right),
\end{multline*}
for some separately increasing function $\kappa$.  Regarding the remainder, it is clear that
\begin{equation}\label{paso.limite.local.remainder.ineq}
\Vert \mathcal{R}(\cdot-x_0)\Vert_{C^{0,\alpha}(B_R(x_0))}\leq C R^{1-\alpha},
\end{equation}
which is indeed the reason behind the estimate for $\varphi_0$ that we stated above. Notice that although $\mathcal{R}$ is clearly bounded above by $R$ in $B_R(0)$, the $\alpha$-H\"{o}lder constant is $O(R^{1-\alpha})$. Specifically, take $z_1,z_2\in B_R(0)$ and split $\mathcal{R}$ as follows
$$\mathcal{R}(z_1)-\mathcal{R}(z_2)=I+II,$$
where
\begin{align*}
I&:=\left(\int_0^1\nabla f_0(x_0+\theta z_1)\,d\theta\right)\cdot (z_1-z_2),\\
II&:=\left(\int_0^1(\nabla f_0(x_0+\theta z_1)-\nabla f_0(x_0+\theta z_1))\,d\theta\right)\cdot z_2.
\end{align*}
By virtue of the $\alpha$-H\"{o}lder continuity of $\nabla f_0$, $II$ can be bounded as follows:
$$\vert II\vert\leq \Vert f_0\Vert_{C^{1,\alpha}(B_R(x_0))}\vert z_2\vert\int_0^1 \vert z_1-z_2\vert^\alpha\theta^\alpha\,d\theta\leq \frac{\Vert f_0\Vert_{C^{1,\alpha}(B_R(x_0))}}{\alpha+1}R\vert z_1-z_2\vert^\alpha.$$
The first term enjoys the bound
$$\vert I\vert\leq \Vert \nabla f_0\Vert_{C^0(B_R(x_0))}\vert z_1-z_2\vert\leq 2^{1-\alpha}\Vert \nabla f_0\Vert_{C^0(B_R(x_0))}R^{1-\alpha}\vert z_1-z_2\vert^\alpha,$$
which then leads to the desired estimate (\ref{paso.limite.local.remainder.ineq}). Notice that one could have raised the $R^{1-\alpha}$ power to $R$ if one assumed that $\nabla f_0(x_0)=0$. 

Also, note that $\Vert \mu_R\Vert_{C^{1,\alpha}(D_R)},\Vert T_R\Vert_{C^0(\Sigma_R)}\leq C_0 R$ for some universal constant $C_0>0$. Then, the above estimate for $u_1-u_0$ can be written as
\begin{multline*}
\Vert u_1-u_0\Vert_{C^{0,\alpha}(B_R(x_0))}\\
\leq \frac{C}{R^\alpha}\left(\Vert \varphi^0\Vert_{C^{1,\alpha}(\Sigma_R,\mu_R)}+2R^{1-\alpha}\right)\left\{1+\kappa\left(\Vert u_0\Vert_{C^{1,\alpha}(B_R(x_0))},C_0,\Vert \mu\Vert_{C^{1,\alpha}(D_1(0))}\right)\right\}\Vert u_0\Vert_{C^{0,\alpha}(B_R(x_0))}.
\end{multline*}
Hereafter we will assume that
\begin{align}\label{paso.limite.local.hipotesis.R.delta0.ineq}
\begin{split}
C\left(\frac{\delta_R}{R^\alpha}+2R^{1-2\alpha}\right)&\left\{1+\kappa\left(2\Vert u_0\Vert_{C^{1,\alpha}(B_R(x_0))},C_0,\Vert \mu\Vert_{C^{1,\alpha}(D_1(0))}\right)\right.\\
&\left.+2\kappa\left(2\Vert u_0\Vert_{C^{1,\alpha}(B_R(x_0))},C_0,\Vert \mu\Vert_{C^{1,\alpha}(D_1(0))}\right)^2\Vert u_0\Vert_{C^{1,\alpha}(B_R(x_0))} \right\}<\frac{\varepsilon_0}{2},
\end{split}
\end{align}
with $\varepsilon_0\in (0,1)$ small enough so that $\varepsilon_0\Vert u_0\Vert_{C^{0,\alpha}(B_R(x_0))}<\delta_0$. Since we are considering low H\"{o}lder exponents $\alpha\in (0,1/2)$, then we can ensure the existence of small enough $R\in (0,R_0)$ and $\delta_R>0$ enjoying the above property. 

Let us assume that we have already defined $f_{m}\in C^{0,\alpha}(\overline{B}_R(x_0))$ and $u_{m+1}\in C^{1,\alpha}(\overline{B}_R(x_0),\RR^3)$ for every $m<n$ such that they verify (\ref{Grad-Rubin-modied.ball.eq1})--(\ref{paso.limite.paso.n.local}) and $u_m=0$ is divergence-free for every index $m<n$. To close the inductive argument let us prove the result for $m=n$. First, the transport problem (\ref{Grad-Rubin-modied.ball.eq1}) can be uniquely solved in $B_R(x_0)$ by virtue of Lemma \ref{flow-box.structurally.stable.lem}, the inductive hypothesis (\ref{paso.limite.paso.n.local}) and the assumption on $\varepsilon_0$ since
$$\Vert u_n-u_0\Vert_{C^{1,\alpha}(B_R(x_0))}\leq \varepsilon_0\Vert u_0\Vert_{C^{1,\alpha}(B_R(x_0))}<\delta_0.$$
Second, the boundary value problem (\ref{Grad-Rubin-modied.ball.eq2}) can also be uniquely solved since
\begin{align*}
\int_{\partial B_R(x_0)}&\left(\mathcal{R}(\cdot-x_0)u_n +\varphi_n u_n\right)\cdot \eta\,dS+\lambda_0\int_{\partial B_R(x_0)}u_0\cdot \eta\,dS\\
&=\int_{B_R(x_0)}\left(\nabla(f_0+\varphi_n)\cdot u_n+(f_0+\varphi_n)\divop u_n\right)\,dx=0,
\end{align*}
by the inductive hypothesis and $\lambda$ is assumed to be a regular value. Furthermore, a similar argument to that in the step $n=0$ shows that $u_{n+1}$ is divergence-free again. Let us finally obtain the desired estimates for $u_{n+1}-u_n$. To this end, note that $u_{n+1}-u_n$ solves the boundary value problem
$$\left\{\begin{array}{ll}
\displaystyle(\curl-\lambda_0)(u_{n+1}-u_n)=\mathcal{R}(\cdot-x_0)(u_n-u_{n-1})+(\varphi_n -\varphi_{n-1})u_n+\varphi_{n-1}(u_n-u_{n-1}), & \displaystyle x\in B_R(x_0),\\
\displaystyle(u_{n+1}-u_n)\cdot \eta=0, & \displaystyle x\in \partial B_R(x_0).
\end{array}\right.$$
Hence, we arrive at the following bound
\begin{multline*}
\Vert u_{n+1}-u_n\Vert_{C^{1,\alpha}(B_R(x_0))}\leq \frac{C}{R^\alpha}(\Vert \mathcal{R}(\cdot-x_0)\Vert_{C^{0,\alpha}(B_R(x_0))}\Vert u_n-u_{n-1}\Vert_{C^{0,\alpha}(B_R(x_0))}\\
+\Vert \varphi_n-\varphi_{n-1}\Vert_{C^{0,\alpha}(B_R(x_0))}\Vert u_n\Vert_{C^{0,\alpha}(B_R(x_0))}+\Vert \varphi_{n-1}\Vert_{C^{0,\alpha}(B_R(x_0))}\Vert u_n-u_{n-1}\Vert_{C^{0,\alpha}(B_R(x_0))}).
\end{multline*}
On the one hand, the remainder can be bounded above as in (\ref{paso.limite.local.remainder.ineq}). On the other hand, $\Vert \varphi_n\Vert_{C^{0,\alpha}(B_R(x_0))}$ and $\Vert \varphi_n-\varphi_{n-1}\Vert_{C^{0,\alpha}(B_R(x_0))}$ can be estimated as 
\begin{align*}
\Vert \varphi_{n-1}\Vert_{C^{0,\alpha}(B_R(x_0))}\leq& \left(\Vert \varphi^0\Vert_{C^{1,\alpha}(\Sigma_R,\mu_R)}+R^{1-\alpha}+\Vert T_R\Vert_{C^0(\Sigma_R)}\right)\\
&\times\kappa\left(\Vert u_{n-1}\Vert_{C^{1,\alpha}(B_R(x_0))},\Vert T_R\Vert_{C^0(\Sigma_R)},\Vert \mu_R\Vert_{C^{1,\alpha}(B_R(x_0))}\right),\\
\Vert \varphi_n-\varphi_{n-1}\Vert_{C^{0,\alpha}(B_R(x_0))}\leq& \left(\Vert \varphi^0\Vert_{C^{1,\alpha}(\Sigma_R,\mu_R)}+R^{1-\alpha}+\Vert T_R\Vert_{C^0(\Sigma_R)}\right)\\
&\times\kappa\left(\Vert u_{n}\Vert_{C^{1,\alpha}(B_R(x_0))},\Vert T_R\Vert_{C^0(\Sigma_R)},\Vert \mu_R\Vert_{C^{1,\alpha}(B_R(x_0))}\right)\\
&\times\kappa\left(\Vert u_{n-1}\Vert_{C^{1,\alpha}(B_R(x_0))},\Vert T_R\Vert_{C^0(\Sigma_R)},\Vert \mu_R\Vert_{C^{1,\alpha}(B_R(x_0))}\right)\\
&\times\Vert u_n-u_{n-1}\Vert_{C^{1,\alpha}(B_R(x_0))}.
\end{align*}
Consequently, the inductive hypothesis along with our choice (\ref{paso.limite.local.hipotesis.R.delta0.ineq}) leads to the first inequality in (\ref{paso.limite.paso.n.local}) and the remaining two inequalities obviously follows from the first one by virtue of the triangle inequality.

As in Section \ref{Esquema.Iterativo.Seccion}, the first inequality in (\ref{paso.limite.paso.n.local}) shows that $\{u_n\}_{n\in\NN}$ is a Cauchy sequence in $C^{1,\alpha}(\overline{B}_R(x_0),\RR^3)$. By completeness, consider $u\in C^{1,\alpha}(\overline{B}_R(x_0))$ such that
$$u_n\rightarrow u\hspace{0.25cm}\mbox{ in }C^{1,\alpha}(\overline{B}_R(x_0,\RR^3)).$$
Moreover, the same reasoning as above yields the estimate
\begin{align*}
\Vert \varphi_n-\varphi_{m}\Vert_{C^{0,\alpha}(B_R(x_0))}\leq& \left(\Vert \varphi^0\Vert_{C^{1,\alpha}(\Sigma_R,\mu_R)}+R^{1-\alpha}+\Vert T_R\Vert_{C^0(\Sigma_R)}\right)\\
&\times\kappa\left(\Vert u_{n}\Vert_{C^{1,\alpha}(B_R(x_0))},\Vert T_R\Vert_{C^0(\Sigma_R)},\Vert \mu_R\Vert_{C^{1,\alpha}(B_R(x_0))}\right)\\
&\times\kappa\left(\Vert u_{m}\Vert_{C^{1,\alpha}(B_R(x_0))},\Vert T_R\Vert_{C^0(\Sigma_R)},\Vert \mu_R\Vert_{C^{1,\alpha}(B_R(x_0))}\right)\\
&\times\Vert u_n-u_{m}\Vert_{C^{1,\alpha}(B_R(x_0))},
\end{align*}
for every indices $n,m\in\NN$. Then, there exists some constant $K=K(\delta_R,R,\Vert u_0\Vert_{C^{0,\alpha}})>0$ so that
$$\Vert \varphi_n-\varphi_m\Vert_{C^{0,\alpha}(\overline{B}_R(x_0))}\leq K\Vert u_n-u_m\Vert_{C^{1,\alpha}(B_R(x_0))}.$$
Hence, $\{\varphi_n\}_{n\in\NN}$ is also a Cauchy sequence in $C^{0,\alpha}(\overline{B}_R(x_0))$ and one can consider $\varphi\in C^{0,\alpha}(\overline{B}_R(x_0))$ such that
$$\varphi_n\rightarrow \varphi\hspace{0.25cm}\mbox{ in }C^{0,\alpha}(\overline{B}_R(x_0)).$$
Taking limits in (\ref{Grad-Rubin-modied.ball.eq1})-(\ref{Grad-Rubin-modied.ball.eq2}) we are led to a generalized Beltrami field $u\in C^{1,\alpha}(\overline{B}_R(x_0),\RR^3)$ solving
$$
\left\{
\begin{array}{ll}
\curl u=(f_0+\varphi)u, & x\in B_R(x_0),\\
\divop u=0, & x\in B_R(x_0),\\
u\cdot \eta=u_0\cdot \eta, & x\in \partial B_R(x_0),
\end{array}
\right.
$$
for a perturbation $\varphi\in C^{0,\alpha}(\overline{B}_R(x_0))$ of the factor such that $\varphi=\varphi^0$ on $\Sigma_R$.

Let us finally show that $u\in C^{k+1,\alpha}(\overline{B}_R(x_0),\RR^3)$ and $\varphi\in C^{k,\alpha}(\overline{B}_R(x_0))$ by a bootstrap argument based on the elliptic gain of regularity. Recall that the vector-valued boundary problem associated with the Laplacian with relative boundary conditions,
$$\left\{\begin{array}{ll}
\Delta w=F, & x\in B_R(x_0),\\
u\cdot \eta=G, & x\in \partial B_R(x_0),\\
\curl u\times \eta=H, & x\in \partial B_R(x_0),
\end{array}\right.$$
is well known to satisfy the estimate
$$
\|w\|_{C^{l+1,\alpha}(B_R(x_0))}\leq C\big(\|F\|_{C^{l-1,\alpha}(B_R(x_0))}+\|G\|_{C^{l+1,\alpha}(\partial{B_R(x_0)})}+\|H\|_{C^{l,\alpha}(\partial{B_R(x_0)})}\big)\,.
$$
The key observation now is that, by acting with the curl operator on the equation for $u$, it follows that
$$\left\{\begin{array}{ll}
\Delta u=-\curl((f_0+\varphi) u), & x\in B_R(x_0),\\
u\cdot \eta=u_0\cdot\eta, & x\in \partial B_R(x_0),\\
\curl u\times \eta=(f_0+\varphi)u\times \eta, & x\in \partial B_R(x_0).
\end{array}\right.$$
Then, the next hierarchy of inequalities hold for every $l\geq 0$
\begin{multline*}
\Vert u\Vert_{C^{l+1,\alpha}(B_R(x_0))}\\
\leq C(\Vert (f_0+\varphi_0)u\Vert_{C^{l,\alpha}(B_R(x_0))}+\Vert u_0\cdot \eta\Vert_{C^{l+1,\alpha}(\partial B_R(x_0))}+\Vert (f_0+\varphi)u\times \eta\Vert_{C^{l,\alpha}(B_R(x_0))}).
\end{multline*}
We then get that the fact that $u$ is of class $C^{1,\alpha}$ implies that $\varphi$ is of class $C^{0,\alpha}$. In turns, it ensures that $u$ is in $C^{2,\alpha}$ and, repeating the argument as many times as necessary (up to the regularity on $\varphi^0$ and $u_0$, i.e., $C^{k+1,\alpha}$) we derive the desired gain of regularity. Indeed, the estimate
$$\Vert u-u_0\Vert_{C^{1,\alpha}(B_R(x_0))}\leq \varepsilon_0\Vert u_0\Vert_{C^{1,\alpha}(B_R(x_0))},$$
can be promoted to its $C^{k+1,\alpha}$ version, i.e.,
$$\Vert u-u_0\Vert_{C^{k+1,\alpha}(B_R(x_0))}\leq \varepsilon_0\Vert u_0\Vert_{C^{k+1,\alpha}(B_R(x_0))}.$$

So far, we have only taken low H\"{o}lder exponents $\alpha\in (0,1/2)$. Assume now that $u_0\in C^{k+1,\alpha'}(\Omega,\RR^3)$ and $f_0\in C^{k,\alpha'}(\Omega)$ for some $\alpha'\in (\alpha,1)$. In particular, $u_0\in C^{k+1,\alpha}(\overline{B}_{2R_0}(x_0),\RR^3)$ and $\varphi_0\in C^{k,\alpha}(\overline{B}_{2R}(x_0))$. The above argument, yields a strong Beltrami field $u\in C^{k+1,\alpha}(\overline{B}_R(x_0),\RR^3)$ with proportionality factor $f_0+\varphi$ for some perturbation $\varphi\in C^{k,\alpha}(\overline{B}_R(x_0))$ such that $\varphi=\varphi^0$ on $\Sigma_R$ as long as $R$ is small enough and $\Vert \varphi^0\Vert_{C^{k+1,\alpha'}(\Sigma_R)}<\delta_R$. Since 
$$\Vert \varphi^0\Vert_{C^{k+1,\alpha}(\Sigma_R)}=\Vert \varphi^0\circ\mu_R\Vert_{C^{k+1,\alpha}(D_R)}\leq \Vert \varphi^0\circ\mu_R\Vert_{C^{k+1,\alpha'}(D_R)}=\Vert \varphi^0\Vert_{C^{k+1,\alpha'}(\Sigma_R)},$$
then, the above smallness assumption on the $C^{k+1,\alpha}(\Sigma_R)$ norm $\varphi^0$ follows from the corresponding assumption on the $C^{k+1,\alpha'}(\Sigma_R)$ norm, i.e., 
$$\Vert \varphi^0\Vert_{C^{k+1,\alpha'}(\Sigma_R)}<\delta_R.$$
Since $\varphi$ solves
$$\left\{\begin{array}{ll}
\nabla \varphi\cdot u=-\nabla f_0\cdot u, & x\in B_R(x_0),\\
\varphi=\varphi^0, & x\in \Sigma_R,
\end{array}
\right.$$
then, a similar result to that in Theorem \ref{problematransporte.teo} leads to $\varphi\in C^{1,\alpha}(\overline{B}_R(x_0))$ because so is $u$, $f_0$ and $\varphi^0$. In particular $\varphi\in C^{0,\alpha'}(\overline{B}_R(x_0))$ and $u\in C^{0,\alpha'}(\overline{B}_R(x_0),\RR^3)$. Then, the above bootstrap in the Beltrami equation yields $\varphi\in C^{k,\alpha'}(\overline{B}_R(x_0))$ and $u\in C^{k+1,\alpha'}(\overline{B}_R(x_0))$, thereby concluding the proof of the theorem.
\end{proof}

\section{Potential theory techniques for inhomogeneous integral kernels}\label{Teoria.Potencial.Tecnicas.Seccion}
Our goal in this section is to extend some results of classical potential theory  to inhomogeneous kernels like the fundamental solution of the Helmholtz equation $\Gamma_\lambda(x)$ (see e.g.\
\cite{Coifman,David,Gilbarg,Giraud,Li,Miranda1,Miranda2,Semmes,Stein} in the case of homogeneous kernels). While there are some previous results concerning the inhomogeneous case (see \cite{ColtonKress,ColtonKress2,Nedelec} for a study of $\Gamma_\lambda(x)$ with non-zero $\lambda$), only low order H\"{o}lder estimates have been obtained. Our approach roughly follows the treatment of  \cite{Heinemann,Neudert} for the harmonic case ($\lambda=0$), and we will introduce nontrivial modifications to derive  higher order H\"{o}lder estimates of generalized volume and single layer potentials in the inhomogeneous setting. These results were used in Section~\ref{Beltrami.NoHomogenea.Seccion} and, of course, the main point throughout is to be able to consider exterior (unbounded) domains.

\subsection{Inhomogeneous volume and single layer potentials}

In our context, all the integral kernels that we need to consider come from the fundamental solution of the $3$-dimensional Helmholtz equation (\ref{SolucionFundamental.form})
$$
\Gamma_\lambda(z)=\frac{e^{i\lambda \vert z\vert}}{4\pi \vert z\vert}=\frac{1}{4\pi}\left(\frac{\cos(\lambda\vert z\vert)}{\vert z\vert}+i\frac{\sin(\lambda \vert z\vert)}{\vert z\vert}\right),\ z\in \RR^3\setminus\{0\}.
$$
For $\lambda =0$ we recover the Newtonian potential associated to the Laplace equation in $\RR^3$, \cite{Gilbarg,Giraud,Miranda1,Miranda2}. As it is not longer homogeneous, the classical theory cannot be directly applied. 

Fortunately, this kernel can be though to be ``almost homogeneous'' in the following sense. Let us consider the functions
\begin{equation}\label{PhiPsi.SolucionFundamental.form}
\left\{\begin{array}{ll}
\displaystyle\phi_\lambda(r):=\frac{e^{i\lambda r}}{4\pi r}, & r>0,\\
\displaystyle\psi_\lambda(r):=\phi_\lambda(r)-\frac{1}{4\pi r}\equiv\frac{e^{i\lambda r}-1}{4\pi r}, & r>0.
\end{array}\right.
\end{equation}
From the definition one has the following splitting 
\begin{equation}\label{PhiPsi.SolucionFundamental.Descomposicion.form}
\phi_\lambda(r)=\frac{1}{4\pi r}+\psi_\lambda(r),
\end{equation}
and consequently, the following decomposition of the fundamental solution 
\begin{equation}\label{PhiPsi.SolucionFundamental.Descomposicion.Gamma0+Rlambda}
\Gamma_\lambda(z)=\phi_\lambda(\vert z\vert)=\frac{1}{4\pi \vert z\vert}+\psi_\lambda(\vert z\vert)=:\Gamma_0(z)+R_\lambda(z)
\end{equation}
holds. This amounts to a decomposition of the inhomogeneous kernel $\Gamma_\lambda(z)$ into the homogeneous part $\Gamma_0(z)$ and an inhomogeneous remainder $R_\lambda(z)$ enjoying lower order singularities at the origin. This is the main argument supporting our subsequent results: we do not need our whole kernel to be purely homogeneous, but only the principal (or more singular) part. While higher order derivatives of harmonic potentials can be directly controlled through the harmonic kernel $\Gamma_0(z)$ and the classical results in \cite{Gilbarg,Giraud,Miranda1,Miranda2}, it is also important to control the behavior of the higher order derivatives of $R_\lambda(z)$. 

Specifically, we can compute the first derivative of $\psi_\lambda(r)$ and write it by means of homogeneous functions and $\psi_\lambda(r)$ itself
$$
\psi_\lambda'(r)=i\lambda\frac{1}{4\pi r}+\left(i\lambda-\frac{1}{r}\right)\psi_\lambda(r).
$$
As $\psi_\lambda(r)$ is locally bounded near $r=0$ and decay as $r^{-1}$ at infinity, it is globally bounded. Thus,
$$
\vert\psi_\lambda'(r)\vert\leq C\left(1+\frac{1}{r}\right),\ \ r>0.
$$
A recursive reasoning leads to estimates for higher order derivatives of $\psi_\lambda(r)$ of the type
\begin{equation}\label{PhiPsi.SolucionFundamental.Cotas.Derivadas.form}
\vert \psi_\lambda^{(m)}(r)\vert \leq C\left(1+\frac{1}{r^m}\right),\ \ r>0,
\end{equation}
where $C=C(\lambda,m)$ is a nonnegative constant. Consequently, we have the following bounds for $R_\lambda(z)=\psi_\lambda(\vert z\vert)$ and its higher order derivatives
\begin{equation}\label{PhiPsi.SolucionFundamental.Cotas.Derivadas.Rlambda.form}
\vert D^\gamma R_\lambda(z)\vert\leq C\left(1+\frac{1}{\vert z\vert^{\vert \gamma\vert}}\right),
\end{equation}
for every $z\in\RR^3\setminus\{0\}$ and each multi-index $\gamma$, in contrast with the analogous bounds for $\Gamma_0(z)$:

\begin{equation}\label{PhiPsi.SolucionFundamental.Cotas.Derivadas.Gamma0.form}
\vert D^\gamma \Gamma_0(z)\vert\leq C\frac{1}{\vert z\vert^{\vert \gamma\vert+1}}.
\end{equation}

A basic fact is that the remainder $R_\lambda(z)$,  which is not homogeneous, is one degree less singular than $\Gamma_0(z)$, so we will combine statements about singular integrals (such as $D^2 \Gamma_0(z)$) for which the Calderon--Zygmund theory essentially applies, with a treatment of weakly singular integral kernels (such as $D^2 R_\lambda(z)$) based on the Hardy--Littlewood--Sobolev theorem. See also \cite{Nedelec} for a treatment of pseudo-homogeneous kernels.

For the sake of completeness, we shall next introduce the kind of kernels that we will consider in this section. Let us consider a bounded domain $D\subseteq \RR^N$ and a continuous function $K(x,z),\ x\in \overline{D},\ z\in\RR^N\setminus\{0\}$. $K$ is said to be a \textit{weakly singular kernels of exponent $\beta$} if there exists a nonnegative constant $C$ such that
$$
\vert K(x,z)\vert\leq \frac{C}{\vert z\vert^\beta},\ \forall\,x\in\overline{D},\ \forall\,y\in\RR^N\setminus\{0\},
$$
for a given $0\leq \beta\leq N-1$. The kind of singular integral kernel that arises in this paper are first order partial derivatives of positively homogeneous kernel or degree $-(N-1)$, i.e.,
$$\frac{\partial}{\partial z_i}K(x,z),\ x\in \overline{D},\ z\in\RR^3\setminus\{0\},$$ 
where $K(x,z)$ satisfies
$$K(x,\lambda z)=\frac{1}{\lambda^{N-1}}K(x,z),$$
for all $x\in\overline{D},\ z\in\RR^N\setminus\{0\},\ \lambda>0$ and $K(x,\sigma)$ is continuous for $x\in\overline{D}$ and $\sigma\in \partial B_1(0)$.

A classical results about the boundedness of generalized volume and single layer potencial in H\"{o}lder spaces allows us to bound the single layer potential associated with $\Gamma_\lambda(z)$ both in bounded and unbounded domains (see \cite[Teorema 2.I]{Miranda2}):

\begin{theo}[Generalized single layer potential]\label{Potencial.capasimple.regularidad.teo}
Let $G\subseteq \RR^3$ be a bounded domain with regularity $C^{k+1,\alpha}$, $\Omega:=\RR^3\setminus G$ its outer domain and $S=\partial G$ the boundary surface. Consider the generalized single layer potential associated with the Helmholtz equation and generated by a density $\zeta:S\longrightarrow \RR$ over the boundary,
$$(\mathcal{S}_\lambda\zeta)(x):=\int_S\Gamma_\lambda(x-y)\zeta(y)\,d_yS,\ \ x\in\RR^3\setminus S.$$
Then, $\mathcal{S}_\lambda\zeta$ is well defined both in $G$ and $\Omega$ for each $\zeta\in C^{k,\alpha}(S)$, it belongs to $C^{k+1,\alpha}(\overline{G})$ and $C^{k+1,\alpha}(\overline{\Omega})$ respectively and we have the associated bounded linear operators
$$\begin{array}{cccl}
\mathcal{S}^-_\lambda: & C^{k,\alpha}(S) & \longrightarrow & C^{k+1,\alpha}(\overline{G}),\\
 & \zeta & \longmapsto & \left.(\mathcal{S}_\lambda\zeta)\right\vert_{G},\\
\mathcal{S}^+_\lambda: & C^{k,\alpha}(S) & \longrightarrow & C^{k+1,\alpha}(\overline{\Omega}),\\
 & \zeta & \longmapsto & \left.(\mathcal{S}_\lambda\zeta)\right\vert_{\Omega},
\end{array}$$
i.e., there exists a nonnegative constant $K=K(k,\alpha,\lambda,G)$ so that  $\zeta\in C^{k,\alpha}(S)$ satisfies
\begin{align*}
\Vert \mathcal{S}^-_\lambda\zeta\Vert_{C^{k+1,\alpha}(G)}&\leq K\Vert \zeta\Vert_{C^{k,\alpha}(S)},\\
\Vert \mathcal{S}^+_\lambda\zeta\Vert_{C^{k+1,\alpha}(\Omega)}&\leq K\Vert \zeta\Vert_{C^{k,\alpha}(S)}.
\end{align*}
Differentiation under the integral sign leads to
\begin{align*}
\nabla(\mathcal{S}^-_\lambda\zeta)(x)&=\int_S\nabla_x \Gamma_\lambda(x-y)\zeta(y)\,d_yS,\ \ x\in G,\\
\nabla(\mathcal{S}^+_\lambda\zeta)(x)&=\int_S\nabla_x \Gamma_\lambda(x-y)\zeta(y)\,d_yS,\ \ x\in \Omega.
\end{align*}
\end{theo}

We omit the proof of this theorem since we are interested in a more singular regularity result that generalizes this one. Specifically, we will study the regularity along the boundary surface $S$ of these generalized single layer potentials along with some other related potentials with inhomogeneous kernels that arose in previous sections, where we will use arguments as in \cite[Teorema 2.I]{Miranda2}. In the next results, we show the regularity of generalized volume (or Newtonian) potentials with compactly supported densities both for interior and exterior domains, which conclude with the derivation of the classical \textit{H\"{o}lder--Korn--Lichtenstein--Giraud inequality} for high order estimates of H\"{o}lder type in the inhomogeneous case.

\begin{lem}\label{Potencial.volumetrico.regularidad.lem1}
Let $G\subseteq \RR^3$ be a bounded domain with regularity $C^{k+1,\alpha}$, $\Omega:=\RR^3\setminus G$ its exterior domain and $S=\partial G$ the boundary surface. Define the generalized volume potential on $G$ associated with the Helmholtz equation and generated by a density in $G$, $\zeta:G\longrightarrow\RR$
$$(\mathcal{N}^-_\lambda\zeta)(x)=\int_{G}\Gamma_\lambda(x-y)\zeta(y)\,dy,\ \ x\in G.$$
Then, $\mathcal{N}^-_\lambda\zeta \in C^{k+2,\alpha}(\overline{G})$ is well defined over $G$ for every $\zeta\in C^{k,\alpha}(\overline{G})$,
and 
$$
\begin{array}{cccc}
\mathcal{N}^-_\lambda: & C^{k,\alpha}(\overline{G}) & \longrightarrow & C^{k+2,\alpha}(\overline{G}),\\
 & \zeta & \longmapsto & \mathcal{N}^-_\lambda\zeta,
\end{array}$$
defines a bounded linear operator, i.e., there exists a nonnegative constant $K=K(k,\alpha,\lambda,G)$ so that
$$\Vert \mathcal{N}^-_\lambda\zeta\Vert_{C^{k+2,\alpha}(G)}\leq K\Vert \zeta\Vert_{C^{k,\alpha}(\overline{G})},$$
for every density $\zeta\in C^{k,\alpha}(\overline{G})$.
\end{lem}
\begin{proof}
The proof follows the lines of \cite[Teorema 3.II]{Miranda2} for the harmonic case $\lambda=0$, that we extend to the  inhomogeneous case. 

Let us obtain first a $ C^1$ estimate of $\mathcal{N}^-_\lambda\zeta$. Since
$$\Gamma_\lambda(z)=O(\vert z\vert^{-1})\ \mbox{ and }\nabla \Gamma_\lambda(z)=O(\vert z\vert^{-2})\ \mbox{ as }\vert z\vert\rightarrow 0,$$
$G$ is bounded and $\zeta\in  C^0(G)$, then one can take derivatives under the integral sign, i.e.,
$$\frac{\partial}{\partial x_i}(\mathcal{N}^-_\lambda\zeta)(x)=\int_{G}\frac{\partial}{\partial x_i}\Gamma_\lambda(x-y)\zeta(y)\,dy,\ \ x\in G.$$
Moreover, straightforward computations supported by the local integrability of both $\Gamma_\lambda(z)$ and $\nabla\Gamma_\lambda(z)$ show that
$$\Vert \mathcal{N}^-_\lambda\zeta\Vert_{ C^1(G)}\leq C\Vert \zeta\Vert_{ C^0(G)}\leq C\Vert \zeta\Vert_{C^{k,\alpha}(G)}.$$

Fix any multi-index $\gamma$ with $\vert \gamma\vert\leq k$ and takes derivatives again under the integral sign to get
$$D^\gamma\frac{\partial}{\partial x_i}(\mathcal{N}^-_\lambda\zeta)(x)=\int_G D^\gamma_x\frac{\partial}{\partial x_i}\Gamma_\lambda(x-y)\zeta(y)\,dy.$$
For any index $1\leq l\leq 3$ so that $e_l\leq \gamma$, an integration by parts recasts the above identity as
\begin{align*}
D^\gamma\frac{\partial}{\partial x_i}(\mathcal{N}^-_\lambda\zeta)(x)&=-\int_G D^{\gamma-e_l}_x\frac{\partial}{\partial x_i}\frac{\partial}{\partial y_l}\Gamma_\lambda(x-y)\zeta(y)\,dy\\
&=\int_{G}D^{\gamma-e_l}_x\frac{\partial}{\partial x_i}\Gamma_\lambda(x-y)\frac{\partial \zeta}{\partial y_l}(y)\,dy-\int_S D^{\gamma-e_l}_x\frac{\partial}{\partial x_i}\Gamma_\lambda(x-y)\zeta(y)\eta_l(y)\,d_yS,
\end{align*} 
where $\eta$ stands for the exterior unit normal vector field along $S$. A recursive reasoning leads to
\begin{align*}
D^\gamma\frac{\partial}{\partial x_i}&(\mathcal{N}^-_\lambda\zeta)(x)
=-\sum_{m_1=1}^{\gamma_1}\int_S D^{\gamma-m_1e_1}_x\frac{\partial}{\partial x_i}\Gamma_\lambda(x-y)D^{(m_1-1)e_1}\zeta(y)\eta_1(y)\,d_yS\\
&-\sum_{m_2=1}^{\gamma_2}\int_S D^{\gamma-\gamma_1e_1-\gamma_2e_2}_x\frac{\partial}{\partial x_i}\Gamma_\lambda(x-y)D^{\gamma_1e_1+(m_2-1)e_2}\zeta(y)\eta_2(y)\,d_yS\\
&-\sum_{m_3=1}^{\alpha_3}\int_S D^{\gamma-\gamma_1e_1-\gamma_2 e_2-m_3 e_3}_x\frac{\partial}{\partial x_i}\Gamma_\lambda(x-y)D^{\gamma_1e_1+\gamma_2e_2+(m_3-1)e_3}\zeta(y)\eta_3(y)\,d_yS\\
&+\int_G\frac{\partial}{\partial x_i}\Gamma_\lambda(x-y)D^\gamma \zeta(y)\,dy.
\end{align*}
Therefore, the same argument as above for the last volume integral along with Theorem \ref{Potencial.capasimple.regularidad.teo} for the boundary integrals show the following upper bound for the derivatives up to order $k+1$ of the generalized volume potential:
$$\left\Vert D^\gamma\frac{\partial}{\partial x_i}(\mathcal{N}^-_\lambda\zeta)\right\Vert_{ C^0(G)}\leq K\Vert \zeta\Vert_{C^{k,\alpha}(G)}.
$$
Here $\vert \gamma\vert\leq k$ and $1\leq i\leq 3$. 

To complete the proof, we consider the derivatives of order $k+2$. Let us then consider another index $1\leq j\leq 3$ and take derivatives under the integral sign once more to arrive at
\begin{align*}
D^\gamma\frac{\partial^2}{\partial x_i\partial x_j}&(\mathcal{N}^-_\lambda\zeta)(x)
=-\sum_{m_1=1}^{\gamma_1}\int_S D^{\gamma-m_1e_1+e_j}_x\frac{\partial}{\partial x_i}\Gamma_\lambda(x-y)D^{(m_1-1)e_1}\zeta(y)\eta_1(y)\,d_yS\\
&-\sum_{m_2=1}^{\gamma_2}\int_S D^{\gamma-\gamma_1e_1-\gamma_2e_2+e_j}_x\frac{\partial}{\partial x_i}\Gamma_\lambda(x-y)D^{\gamma_1e_1+(m_2-1)e_2}\zeta(y)\eta_2(y)\,d_yS\\
&-\sum_{m_3=1}^{\alpha_3}\int_S D^{\gamma-\gamma_1e_1-\gamma_2 e_2-m_3 e_3+e_j}_x\frac{\partial}{\partial x_i}\Gamma_\lambda(x-y)D^{\gamma_1e_1+\gamma_2e_2+(m_3-1)e_3}\zeta(y)\eta_3(y)\,d_yS\\
&+\int_G\frac{\partial^2}{\partial x_i\partial x_j}\Gamma_\lambda(x-y)D^\gamma \zeta(y)\,dy.
\end{align*}
Similar estimates for the boundary terms can be obtained in $C^{0,\alpha}(G)$ by virtue of Theorem \ref{Potencial.capasimple.regularidad.teo}, while the volume integral has to be studied separately carrying out an adaptation of the ideas in the harmonic case \cite[Teorema 3.II]{Miranda2}. 

We first split it into two parts and use again integration by parts in the second term
\begin{align*}
\int_{G}\frac{\partial^2}{\partial x_i\partial x_j}&\Gamma_\lambda(x-y)D^\gamma\zeta(y)\,dy\\
=&\int_{G}\frac{\partial^2}{\partial x_i\partial x_j}\Gamma_\lambda(x-y)(D^\gamma\zeta(y)-D^\gamma\zeta(x))\,dy+D^\gamma\zeta(x)\int_G\frac{\partial}{\partial x_j}\frac{\partial}{\partial x_i}\Gamma_\lambda(x-y)\,dy\\
=&\int_G\frac{\partial^2}{\partial x_i\partial x_j}\Gamma_\lambda(x-y)(D^\gamma\zeta(y)-D^\gamma\zeta(x))\,dy-D^\gamma\zeta(x)\int_S\frac{\partial}{\partial x_i}\Gamma_\lambda(x-y)\eta_j(y)\,d_yS\\
=:&F(x)-H(x).
\end{align*}
The idea behind such decomposition is apparent now since the second term, $H(x)$ can be bounded in $C^{0,\alpha}(G)$ according to Theorem \ref{Potencial.capasimple.regularidad.teo}
$$\left\Vert H\right\Vert_{C^{0,\alpha}(G)}\leq K\Vert \eta\Vert_{C^{0,\alpha}(G)}\,\Vert \zeta\Vert_{C^{k,\alpha}(S)}$$
and we have cancelled an $\alpha$ power of the singularity in the first term $F(x)$:
$$\left\vert F(x)\right\vert\leq \left[D^\gamma\zeta\right]_{\alpha,G}\int_G \left\vert\frac{\partial^2}{\partial x_i\partial x_j}\Gamma_\lambda(x-y)\right\vert\,\vert x-y\vert^\alpha\,dy.$$
Bearing (\ref{PhiPsi.SolucionFundamental.Descomposicion.Gamma0+Rlambda}) in mind, we obtain the derivative formulas
\begin{align*}
\frac{\partial}{\partial z_i}\Gamma_\lambda(z)&=\left(-\frac{1}{4\pi\vert z\vert^2}+\psi_\lambda'(\vert z\vert)\right)\frac{z_i}{\vert z\vert},\\
\frac{\partial}{\partial z_i}\frac{\partial}{\partial z_j}\Gamma_\lambda(z)&=\left(\frac{2}{4\pi \vert z\vert^3}+\psi_\lambda''(\vert z\vert)\right)\frac{z_i}{\vert z\vert}\frac{z_j}{\vert z\vert}+\left(-\frac{1}{4\pi\vert z\vert^2}+\psi_\lambda'(\vert z\vert)\right)\frac{\delta_{ij}\vert z\vert-\frac{z_iz_j}{\vert z\vert}}{\vert z\vert^2}.
\end{align*}
Thus, the estimates (\ref{PhiPsi.SolucionFundamental.Cotas.Derivadas.Rlambda.form}) and (\ref{PhiPsi.SolucionFundamental.Cotas.Derivadas.Gamma0.form}) amount to the following $ C^0$ estimate
\begin{align*}
\left\vert F(x)\right\vert&\leq C\Vert \zeta\Vert_{C^{k,\alpha}(G)}\int_G \left(1+\frac{1}{\vert x-y\vert^3}\right)\vert x-y\vert^\alpha\,dy\leq C \Vert \zeta\Vert_{C^{k,\alpha}(G)}\int_G\frac{dy}{\vert x-y\vert^{3-\alpha}}.
\end{align*}
The local integrability of $\vert z\vert^{\alpha-3}$ along with the boundedness of $G$ lead again to the upper bound 
$$\Vert F\Vert_{ C^0(G)}\leq K\Vert \zeta\Vert_{C^{k,\alpha}(G)}.$$

Let us finally show the local $\alpha$-H\"{o}lder property for $F$, i.e.,
$$\vert F(x^1)-F(x^2)\vert\leq C\vert x^1-x^2\vert^\alpha,$$
for every $x^1,x^2\in G$ such that $\vert x^1-x^2\vert<\delta$ and any small $\delta>0$. To this end, consider a neighborhood $U$ of $x^1$ with 
$$B_{2d}(x^1)\subseteq U\subseteq B_{7d}(x^1)$$
so that,
\begin{align*}
F(x_1)-F(x_2)=&\int_G \frac{\partial^2 \Gamma_\lambda(x^1-y)}{\partial x_i\partial x_j}(D^\gamma\zeta(y)-D^\gamma\zeta (x^1))\,dy-\int_G\frac{\partial^2\Gamma_\lambda(x^2-y)}{\partial x_i\partial x_j}(D^\gamma\zeta(y)-D^\gamma\zeta(x^1))\,dy\\
=&\int_{G\cap B_{7d}(x^1)}\frac{\partial^2 \Gamma_\lambda(x^1-y)}{\partial x_i\partial x_j}(D^\gamma\zeta(y)-D^\gamma\zeta (x^1))\,dy\\
&-\int_{G\cap B_{7d}(x^1)}\frac{\partial^2\Gamma_\lambda(x^2-y)}{\partial x_i\partial x_j}(D^\gamma\zeta(y)-D^\gamma\zeta(x^2))\,dy\\
&+\int_{G\setminus B_{7d}(x^1)}\left(\frac{\partial^2 \Gamma_\lambda(x^1-y)}{\partial x_i\partial x_j}-\frac{\partial^2G_\lambda(x^2-y)}{\partial x_i\partial x_j}\right)(D^\gamma\zeta(y)-D^\gamma\zeta(x^1))\,dy\\
&-(D^\gamma\zeta(x^1)-D^\gamma\zeta(x^2))\int_{G\setminus B_{7d}(x^1)}\frac{\partial^2 \Gamma_\lambda(x^2-y)}{\partial x_i\partial x_j}\,dy.
\end{align*}
Taking Euclidean norms, we finally arrive at
\begin{align}
\vert F(x^1)-F(x^2)\vert \leq &\int_{G\cap B_{7d}(x^1)}\left\vert\frac{\partial^2 \Gamma_\lambda(x^1-y)}{\partial x_i\partial x_j}\right\vert \,\vert(D^\gamma\zeta(y)-D^\gamma\zeta (x^1))\vert\,dy\nonumber\\
&+\int_{G\cap B_{8d}(x^2)}\left\vert\frac{\partial^2 \Gamma_\lambda(x^2-y)}{\partial x_i\partial x_j}\right\vert\,\vert(D^\gamma\zeta(y)-D^\gamma\zeta (x^2))\vert\,dy\nonumber\\
&+\int_{G\setminus B_{2d}(x^1)}\left\vert\frac{\partial^2 \Gamma_\lambda(x^1-y)}{\partial x_i\partial x_j}-\frac{\partial^2\Gamma_\lambda(x^2-y)}{\partial x_i\partial x_j}\right\vert\,\vert D^\gamma\zeta(y)-D^\gamma\zeta(x^1)\vert\,dy\nonumber\\
&+\vert D^\gamma\zeta(x^1)-D^\gamma\zeta(x^2)\vert\,\int_{G\setminus U}\left\vert\frac{\partial^2 \Gamma_\lambda(x^2-y)}{\partial x_i\partial x_j}\right\vert\,dy,\label{Potencial.volumetrico.regularidad.lem1.DescomposicionF}
\end{align}
where in the last three terms we have respectively used that $G\cap B_{7d}(x^1)\subseteq G\cap B_{8d}(x^2)$, $G\setminus B_{7d}(x^1)\subseteq G\setminus B_{2d}(x^1)$ and $G\setminus B_{7d}(x^1)\subseteq G\setminus U$.

The first term in (\ref{Potencial.volumetrico.regularidad.lem1.DescomposicionF}) can be bounded by virtue of the $\alpha$-H\"{o}lder property for $D^\gamma \zeta$ and the fact that $D^2 \Gamma_\lambda(z)=O\left(\vert z\vert^{-3}\right)$ near the origin:
\begin{align*}
\int_{G\cap B_{7d}(x^1)}\left\vert\frac{\partial^2 \Gamma_\lambda(x^1-y)}{\partial x_i\partial x_j}\right\vert \,&\vert(D^\gamma\zeta(y)-D^\gamma\zeta (x^1))\vert\,dy
\leq C\int_{G\cap B_{7d}(x^1)}\frac{1}{\vert x^1-y\vert^{3-\alpha}}\,dy\\
&\leq C\int_{B_{7d}(x^1)}\frac{dy}{\vert x^1-y\vert^{3-\alpha}}=4\pi C\int_0^{7d}r^{2-3+\alpha}\,dr=4\pi C\frac{7^\alpha }{\alpha}\vert x^1-x^2\vert^\alpha.
\end{align*}
A similar bound follows for the second term. Regarding the third term in (\ref{Potencial.volumetrico.regularidad.lem1.DescomposicionF}), we find
\begin{align*}
\left\vert\frac{\partial^2 \Gamma_\lambda}{\partial z_i\partial z_j}\right.\left.(x^1-y)-\frac{\partial^2 \Gamma_\lambda}{\partial z_i\partial z_j}(x^2-y)\right\vert
=&\left\vert\int_0^1\frac{d}{d\theta}\left[\frac{\partial^2 \Gamma_\lambda}{\partial z_i\partial z_j}(\theta x^1+(1-\theta)x^2-y)\right]\,d\theta\right\vert\\
\leq &\int_0^1\left\vert\left(\nabla\frac{\partial^2}{\partial z_i\partial z_j}\Gamma_\lambda\right)(\theta x^1+(1-\theta)x^2-y)\cdot (x^1-x^2)\right\vert\,d\theta,\\
\leq & C\int_0^1\frac{d\theta}{\vert \theta x^1+(1-\theta)x^2-y\vert^4}\vert x^1-x^2\vert\,d\theta.
\end{align*}
Since $y\in G\setminus B_{2d}(x^1)$ in the third term of the decomposition (\ref{Potencial.volumetrico.regularidad.lem1.DescomposicionF}) and $0\leq \theta\leq 1$, then
\begin{align*}
\frac{1}{\vert \theta x^1+(1-\theta)x^2-y\vert^4}
=\frac{1}{\vert (1-\theta) (x^2-x^1)+(x^1-y)\vert^4}\leq\frac{1}{(\vert x^1-y\vert-\vert x^1-x^2\vert)^4}\leq \frac{2^4}{\vert x^1-y\vert^4}.
\end{align*}
Therefore,
$$\left\vert\frac{\partial^2 \Gamma_\lambda}{\partial z_i\partial z_j}(x^1-y)-\frac{\partial^2 \Gamma_\lambda}{\partial z_i\partial z_j}(x^2-y)\right\vert\leq C\frac{\vert x^1-x^2\vert}{\vert x^1-y\vert^4},\ \ \forall\,y\in G\setminus B_{2d}(x^1).$$
The above estimate allows obtaining the desired estimate of $\alpha$-H\"{o}lder type for the third term in (\ref{Potencial.volumetrico.regularidad.lem1.DescomposicionF})
\begin{align*}
\int_{G\setminus B_{2d}(x^1)}&\left\vert\frac{\partial^2 \Gamma_\lambda(x^1-y)}{\partial x_i\partial x_j}-\frac{\partial^2\Gamma_\lambda(x^2-y)}{\partial x_i\partial x_j}\right\vert\,\vert D^\gamma\zeta(y)-D^\gamma\zeta(x^1)\vert\,dy\\
\leq & C\Vert \zeta\Vert_{C^{k,\alpha}(G)}\vert x^1-x^2\vert\int_{\RR^3\setminus B_{2d}(x^1)}\frac{dy}{\vert x^1-y\vert^{4-\alpha}}=4\pi C\Vert \zeta\Vert_{C^{k,\alpha}(G)}\vert x^1-x^2\vert\int_{2d}^{+\infty}r^{2-4+\alpha}\,dr\\
=&4\pi C\Vert \zeta\Vert_{C^{k,\alpha}(G)}\vert x^1-x^2\vert\int_{2d}^{+\infty}\frac{1}{r^{1+(1-\alpha)}}\,dr=\frac{4\pi C}{1-\alpha}\frac{1}{2^{1-\alpha}}\Vert\zeta\Vert_{C^{k,\alpha}(G)}\frac{\vert x^1-x^2\vert}{\vert x^1-x^2\vert^{1-\alpha}}\\
=&C\Vert \zeta\Vert_{C^{k,\alpha}(G)}\vert x^1-x^2\vert^\alpha.
\end{align*}

Concerning the last term in (\ref{Potencial.volumetrico.regularidad.lem1.DescomposicionF}), we are done as long as one notices that $D^\gamma\zeta\in C^{0,\alpha}(G)$ and shows
$$\int_{G\setminus U}\left\vert\frac{\partial^2 \Gamma_\lambda(x^2-y)}{\partial x_i\partial x_j}\right\vert\,dy\leq C,$$
for some positive constant $C$ depending on $\delta$ but not on $d=\vert x^1-x^2\vert$. For that, assume first that $2d\leq \mbox{dist}(x^1,S)$ and define $U:=\overline{B}_{2d}(x^1)$. Then $G\setminus U\equiv G\setminus\overline{B}_{2d}(x^1)$.
\begin{figure}[t]
\centering
\includegraphics[scale=0.75]{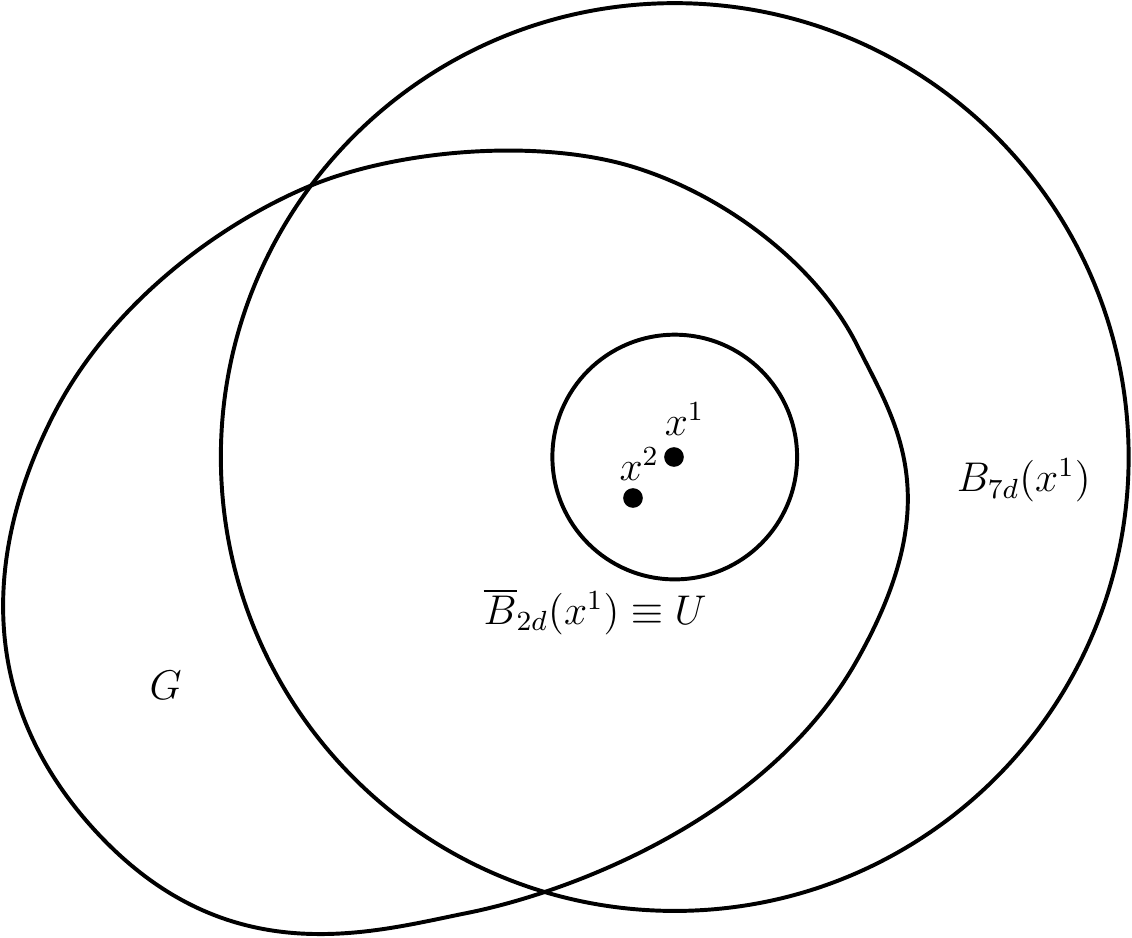}
\caption{Choice of $U=\overline{B}_{2d}(x^1)$ in the first case.}
\label{fig:Fig3}
\end{figure}
Note that $\partial (G\setminus U)=S\cup \partial B_{2d}(x^1)$ and integrate by parts to get
$$\int_{G\setminus U}\frac{\partial^2 \Gamma_\lambda(x^2-y)}{\partial x_i\partial x_j}\,dy=\int_S \frac{\partial \Gamma_\lambda(x^2-y)}{\partial x_i}\eta_j(y)\,d_yS-\int_{\partial B_{2d}(x^1)}\frac{\partial\Gamma_\lambda(x^2-y)}{\partial x_i}\frac{(y-x^1)_j}{\vert y-x^1\vert}\,d_yS.$$
Theorem \ref{Potencial.capasimple.regularidad.teo} provides an upper bound of the first term. On the other hand, by definition
$$\sup_{y\in\partial B_{2d}(x^1)}\frac{1}{\vert x^2-y\vert^2}\leq \frac{1}{\vert x^1-x^2\vert^2}$$
and one has the asymptotic behavior
$$\left\vert \frac{\partial \Gamma_\lambda}{\partial z_i}(z)\right\vert\leq C\left(1+\frac{1}{\vert z\vert^2}\right),\ \ z\in\RR^3\setminus\{0\},$$
so the second term is bounded as
$$
\left\vert\int_{\partial B_{2d}(x^1)}\frac{\partial \Gamma_\lambda(x^2-y)}{\partial x_i}\frac{(y-x^1)_j}{\vert y-x^1\vert}\,d_yS\right\vert\leq \widetilde{C}\frac{1}{\vert x^1-x^2\vert^2}\int_{\partial B_{2d}(x^1)}\,d_yS=4\pi \widetilde{C}\equiv C.
$$

Secondly, let us consider the opposite case $2d> \mbox{dist}(x^1,S)$. Now the configuration is slightly different. Let us fix some $\widetilde{x}^1\in S$ so that $\vert x^1-\widetilde{x}^1\vert=\mbox{dist}(x^1,S)$ and define $U:=\overline{B}_{4d}(\widetilde{x}^1)$.
\begin{figure}[t]
\includegraphics[scale=0.75]{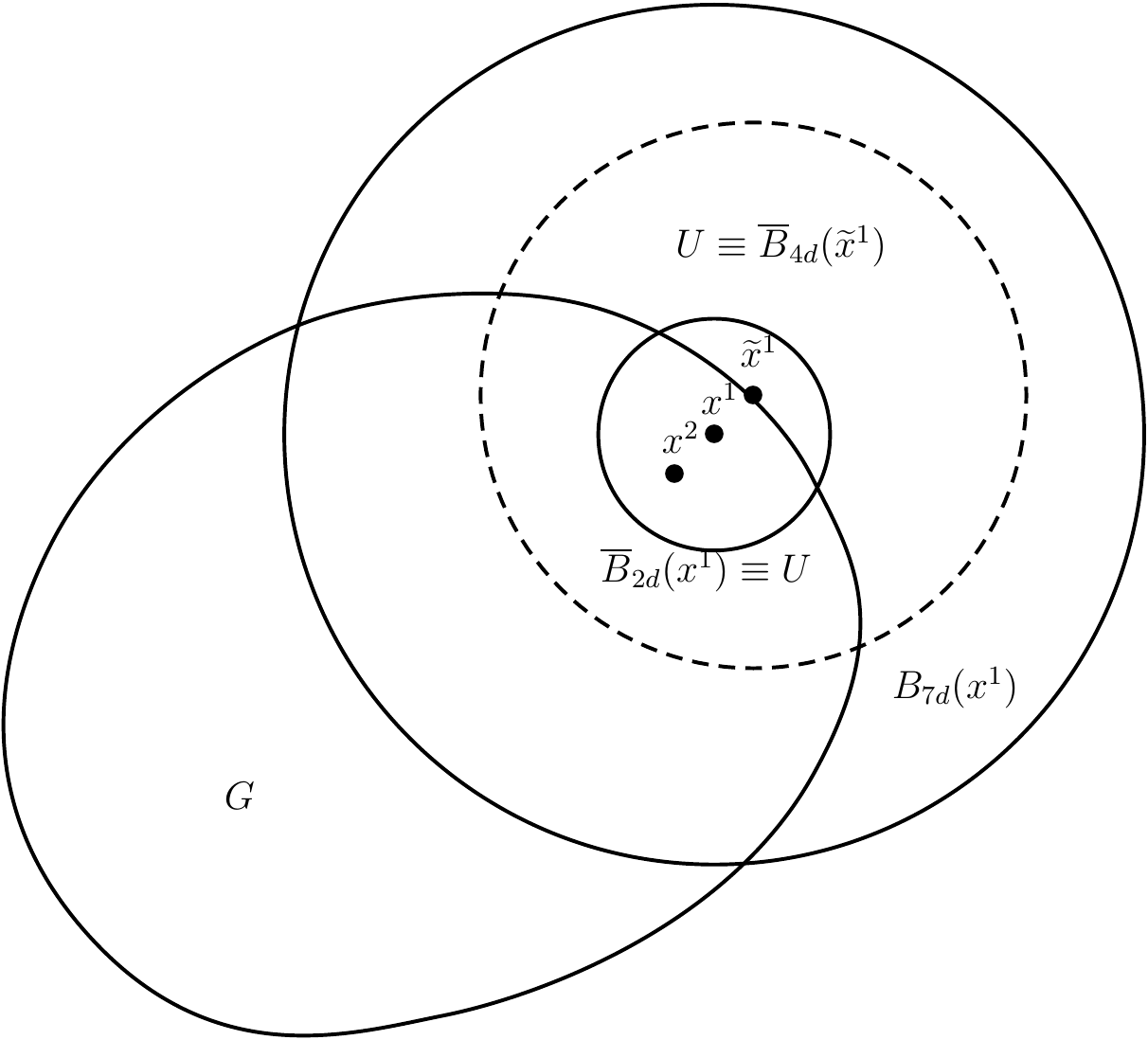}
\caption{Choice of $U=\overline{B}_{4d}(\widetilde{x}^1)$ in the second case.}
\label{fig:Fig4}
\end{figure}
By definition $x^2\in B_{3d}(\widetilde{x}^1)$ and consequently, $B_{2d}(x^1)\subseteq U\subseteq B_{7d}(x^1)$. Thus, $U$ is as in (\ref{Potencial.volumetrico.regularidad.lem1.DescomposicionF}). The last term there takes the form
$$\int_{G\setminus U}\frac{\partial^2 \Gamma_\lambda(x^2-y)}{\partial x_i\partial x_j}\,dy=\int_S \frac{\partial \Gamma_\lambda(x^2-y)}{\partial x_i}\eta_j(y)\,d_yS-\int_{\partial (U\cap G)}\frac{\partial \Gamma_\lambda(x^2-y)}{\partial x_i}\nu_j(y)\,d_yS.$$
Since the first term can be bounded through the same reasonings as above, we focus on the second term. It will be estimated following the idea in \cite[Lemma 2.IV]{Miranda2}. To this end, define some cut-off function $\xi\left(\frac{\vert y-\widetilde{x}^1\vert}{d}\right)$ for $\xi\in C^\infty_c(\RR_0^+)$ such that
\begin{equation*}
\left\{\begin{array}{ll}
\xi(r)=1, & r\in\left[0,\frac{7}{2}\right],\\
\xi(r)\in (0,1), & r\in \left(\frac{7}{2},4\right),\\
\xi(r)=0, & r\geq 4,
\end{array}\right.
\end{equation*}
and consider the splitting
\begin{align}
\int_{\partial(G\cap B_{4d}(\widetilde{x}^1))}\frac{\partial \Gamma_\lambda(x^2-y)}{\partial x_i}\nu_j(y)\,d_yS&=\int_{G\cap \partial B_{4d}(\widetilde{x}^1)}\frac{\partial \Gamma_\lambda(x^2-y)}{\partial x_i}\nu_j(y)\,d_yS\nonumber\\
&+\int_{S\cap B_{4d}(\widetilde{x}^1)}\frac{\partial \Gamma_\lambda(x^2-y)}{\partial x_i}\left[1-\xi\left(\frac{\vert y-\widetilde{x}^1\vert}{d}\right)\right]\nu_j(y)\,d_yS\nonumber\\
&+\int_{S\cap B_{4d}(\widetilde{x}^1)}\frac{\partial \Gamma_\lambda(x^2-y)}{\partial x_i}\xi\left(\frac{\vert y-\widetilde{x}^1\vert}{d}\right)\nu_j(y)\,d_yS.\label{Potencial.volumetrico.regularidad.lem1.DescomposicionMiranda}
\end{align}
Bear in mind again that $x^2\in B_{3d}(\widetilde{x}^1)$, so
$$\vert y-x^2\vert\geq \vert y-\widetilde{x}^1\vert-3d=4d-3d=d,$$
for each $y\in G\cap \partial B_{4d}(\widetilde{x}^1)$ and
$$\left\vert\int_{G\cap\partial B_{4d}(\widetilde{x}^1)}\frac{\partial \Gamma_\lambda(x^2-y)}{\partial x_i}\nu_j(y)\,d_yS\right\vert\leq \frac{\widetilde{C}}{\vert x^1-x^2\vert^2}\vert \partial B_{4d}(\widetilde{x}^1)\vert\leq 4\pi \widetilde{C}.$$
In the second term $y\in S\cap B_{4d}(\widetilde{x}^1)$. Moreover, in order that $y$  belongs to the support of the cut-off function, one has to assume $\vert y-\widetilde{x}^1\vert\geq \frac{7}{2}d$. Thus,
$$\vert y-x^2\vert\geq \vert y-\widetilde{x}^1\vert-3d\geq \frac{7}{2}d-3d=\frac{d}{2},$$
and consequently,
$$\left\vert\int_{S\cap B_{4d}(\widetilde{x}^1)}\frac{\partial \Gamma_\lambda(x^2-y)}{\partial x_i}\nu_j(y)\,d_yS\right\vert\leq \frac{\widetilde{C}}{\vert x^1-x^2\vert^2}\vert S\cap B_{4d}(\widetilde{x}^1)\vert.$$
The upper bound for the second term is done once we note that
\begin{equation}\label{Potencial.volumetrico.regularidad.lem1.AhlforsDavid.Miranda}
\vert S\cap B_{4d}(\widetilde{x}^1)\vert\leq C d^2.
\end{equation}

To prove the corresponding bound for the third term in (\ref{Potencial.volumetrico.regularidad.lem1.DescomposicionMiranda}), we consider the potential$$
\mathcal{S}(x)=\int_{S}\frac{\partial \Gamma_\lambda(x-y)}{\partial x_i}\xi\left(\frac{\vert y-\widetilde{x}^1\vert}{d}\right)\nu_j(y)\,d_yS,\ \ x\in G,$$
whose $C^{0,\alpha}$ estimate follows again from Lemma \ref{Potencial.capasimple.regularidad.teo}:
$$\Vert \mathcal{S}\Vert_{C^{0,\alpha}(\overline{G})}\leq C\left\Vert\xi\left(\frac{\vert \cdot-\widetilde{x}^1\vert}{d}\right)\nu_j\right\Vert_{C^{0,\alpha}(S)}\leq C\left(1+\frac{1}{d^\alpha}\right).$$
Let us now fix $\delta>0$ small enough so that $x-\theta \eta(x)\in G$ for every couple $x\in S$ and $0<\theta<4\delta$. Thus, $\mathcal{S}(\widetilde{x}^1-4d\eta(\widetilde{x}^1))=0$ and consequently
\begin{align*}
\vert \mathcal{S}(x^2)\vert&=\vert \mathcal{S}(x^2)-\mathcal{S}(\widetilde{x}^1-4d\eta(\widetilde{x}^1))\vert\leq C\left(1+\frac{1}{d^\alpha}\right)\vert x^2-\widetilde{x}^1+4d\eta(\widetilde{x}^1)\vert^\alpha\\
&\leq C\left(1+\frac{1}{d^\alpha}\right)(3d+4d)^\alpha\leq \widetilde{C}.
\qedhere
\end{align*}
\end{proof}

\begin{rem}\label{Potencial.volumetrico.regularidad.lem1.SuperficiesAhlforsDavid}
Let us elaborate on the key inequality (\ref{Potencial.volumetrico.regularidad.lem1.AhlforsDavid.Miranda}), which is a key regularity property intimately connected with deep issues in harmonic analysis. Indeed, a surface $S$ is said to be \textit{Alhfolrs--David regular} when
$$\vert S\cap B_R(x)\vert\leq C R^2,$$
for every couple $x\in S$, $R>0$ and some nonnegative constant $C$. These surfaces (originally curves) arise from the study of singular integrals along curves \cite{David}, and had already appeared in \cite{Calderon,Coifman,Li} on $L^2$ estimates for the Cauchy integral along Lipschitz curves. His results were improved in \cite{David} to the more general setting of Alhfors--David curves and it was generalized in \cite{Semmes} to the $N$-dimensional framework. Specifically, Ahlfors--David regularity was shown to control singular integral operators that are much more general than the Cauchy integral. Of course, $C^{k,\alpha}$ surfaces are Ahlfors--David regular.
\end{rem}

\begin{lem}\label{Potencial.volumetrico.regularidad.lem2}
Let $G\subseteq \RR^3$ be a bounded domain with regularity $C^{k+1,\alpha}$, $\Omega:=\RR^3\setminus \overline{G}$ its exterior domain and $S=\partial G$ the boundary surface. Define the generalized volume potential on $\Omega$ associated with the Helmholtz equation and generated by a density in $G$, $\zeta:G\longrightarrow\RR$
$$(\mathcal{N}^+_\lambda\zeta)(x)=\int_{G}\Gamma_\lambda(x-y)\zeta(y)\,dy,\ \ x\in\Omega.$$
Then, $\mathcal{N}^+_\lambda\zeta \in C^{k+2,\alpha}(\overline{\Omega})$ is well defined for every $\zeta\in C^{k,\alpha}(\overline{G})$, 
and 
$$
\begin{array}{cccc}
\mathcal{N}^+_\lambda: & C^{k,\alpha}(\overline{G}) & \longrightarrow & C^{k+2,\alpha}(\overline{\Omega}),\\
 & \zeta & \longmapsto & \mathcal{N}^+_\lambda\zeta,
\end{array}
$$
is a bounded linear operator, i.e., 
$$
\Vert \mathcal{N}^+_\lambda\zeta\Vert_{C^{k+2,\alpha}(\Omega)}\leq K\Vert \zeta\Vert_{C^{k,\alpha}(\overline{G})},$$
for every density $\zeta\in C^{k,\alpha}(\overline{G})$.
\end{lem}

\begin{proof}
Our argument is based on some ideas of \cite[Teorema 3.II]{Miranda2}. Consider $R>0$ large enough for $\overline{G}\subseteq B_R(0)$ and let us estimate $\Vert\mathcal{N}^+_\lambda\zeta\Vert_{C^{k+2,\alpha}(\Omega)}$ in terms of $\Vert\mathcal{N}^+_\lambda\zeta\Vert_{C^{k+2,\alpha}(\RR^3\setminus \overline{B}_R(0))}$ and $\Vert\mathcal{N}^+_\lambda\zeta\Vert_{C^{k+2,\alpha}(\Omega_{2R})}$, where $\Omega_{2R}$ stands for $B_{2R}(0)\setminus \overline{G}$. Set 
$$d_R:=\min\{\vert x-y\vert:\,x\in\RR^3\setminus B_R(0),\,y\in \overline{G}\}>0,$$
and assume that $d_R<1$. 

Equations (\ref{PhiPsi.SolucionFundamental.Descomposicion.Gamma0+Rlambda}),  (\ref{PhiPsi.SolucionFundamental.Cotas.Derivadas.Rlambda.form}) and (\ref{PhiPsi.SolucionFundamental.Cotas.Derivadas.Gamma0.form}) yield
$$\left\vert D^\gamma_x \Gamma_\lambda(x-y)\right\vert\leq \widetilde{C}\frac{1}{d_R^{\vert \gamma\vert+1}},$$
for every multi-index $\gamma$ and every $x\in \RR^3\setminus B_R(0)$ and $y\in G$. One can then take derivatives under the integral sign and obtain the desired estimate for the $C^{k+2,\alpha}$ norm in $\RR^3\setminus \overline{B}_R(0)$. On the other hand, consider $\overline{\zeta}\in C^{k,\alpha}(\RR^3)$ through Proposition \ref{ExtensionHolder.pro}. Then, 
$$
(\mathcal{N}^+_{\lambda}\zeta)(x)=\int_{B_{2R}(0)}\Gamma_\lambda(x-y)\overline{\zeta}(y)\,dy-\int_{\Omega_{2R}}\Gamma_\lambda(x-y)\overline{\zeta}(y)\,dy,
$$
for every $x\in \Omega_{2R}$.
Since both $B_{2R}(0)$ and $\Omega_{2R}$ are $C^{k+1,\alpha}$ bounded domains, then Lemma \ref{Potencial.volumetrico.regularidad.lem1} leads to
\begin{align*}
\Vert \mathcal{N}^+_\lambda\zeta\Vert_{C^{k+2,\alpha}(\overline{\Omega_{2R}})}\leq &\left\Vert\int_{B_{2R}(0)}\Gamma_\lambda(x-y)\overline{\zeta}(y)\,dy\right\Vert_{C^{k+2,\alpha}(\overline{B}_{2R}(0))}+\left\Vert\int_{\Omega_{2R}}\Gamma_\lambda(x-y)\overline{\zeta}(y)\,dy\right\Vert_{C^{k+2,\alpha}(\overline{\Omega_{2R}})}\\
\leq & M\Vert \overline{\zeta}\Vert_{C^{k+2,\alpha}(B_{2R}(0))}\leq MC_\mathcal{P}\Vert \zeta\Vert_{C^{k+2,\alpha}(G)}.
\qedhere
\end{align*}
\end{proof}

Now, we focus on similar estimates for singular and weakly singular kernels in the whole space $\RR^N$. This results are classical in the homogeneous harmonic case, $\Gamma_0(z)$, and can be found in \cite{Giraud,Miranda1,Miranda2}. However, not only we will need harmonic potentials, but we will also deal with general singular and weakly singular kernels. To this end, we remind \cite[Satz 3.4, Satz 5.4]{Heinemann}. 

\begin{theo}[Weakly singular kernels]\label{NucleoDebSingular.Holder.teo}
Let us consider $0\leq \beta\leq N-1,\ 0<\alpha<1$ and $K(x,z),\ x\in\overline{D},z\in\RR^N\setminus\{0\}$ a weakly singular integral kernel of exponent $\beta$ satisfying the following three hypothesis:
\begin{enumerate}
\item For each $x\in \overline{D}$
\begin{equation}\label{NucleoDebSingular.Holder.Hip1.form}
K(x,\cdot)\in C^1(\RR^N\setminus\{0\}).
\end{equation}
\item For each $x\in \overline{D}$ and $z\in\RR^N\setminus\{0\}$
\begin{equation}\label{NucleoDebSingular.Holder.Hip2.form}
\left\vert\nabla_z K(x,z)\right\vert\leq \frac{C}{\vert z\vert^{\beta+1}}.
\end{equation}
\item For all $x_1,x_2\in\overline{D}$ and $z\in \RR^N\setminus\{0\}$ one has
\begin{equation}\label{NucleoDebSingular.Holder.Hip3.form}
\vert K(x_1,z)-K(x_2,z)\vert\leq C\frac{\vert x_1-x_2\vert^\alpha}{\vert z\vert^\beta}.
\end{equation}
\end{enumerate}
Then, for all $R>0$ there exists a nonnegative constant $M=M(N,\alpha,\beta,R)$ so that the generalized volume potential
$$
(\mathcal{N}_K\zeta)(x):=\int_{\RR^N}K(x,x-y)\zeta(y)\,dy,\ x\in \overline{D},
$$
with density $\zeta\in C^{0,\alpha}_c(B_R(0))$ belongs to $C^{0,\alpha}(\overline{D})$ and 
$$\Vert \mathcal{N}_K\zeta\Vert_{C^{0,\alpha}(D)}\leq M\Vert \zeta\Vert_{C^{0,\alpha}(\RR^N)}.$$
\end{theo}

\begin{theo}[Singular kernels]\label{NucleoSingular.1Derivada.Holder.teo}
Consider $0<\alpha<1$ and $K(x,z),\ x\in\overline{D},\,z\in\RR^N\setminus\{0\}$ a kernel satisfying the following hypotheses:
\begin{enumerate}
\item $K(x,z)$ is positively homogeneous of degree $-(N-1)$ with respect to the second variable, i.e., for all $x\in\overline{D},\ z\in\RR^N\setminus\{0\}$ and $\lambda>0$
\begin{equation}\label{NucleoSingular.1Derivada.Holder.Hip1.form}
K(x,\lambda z)=\frac{1}{\lambda^{N-1}} K(x,z).
\end{equation}
\item $K(x,z)$ has the following regularity properties for every $x\in\overline{D}$ and each index $1\leq i,j\leq N$:
\begin{align}\label{NucleoSingular.1Derivada.Holder.Hip2.form}
\begin{split}
K\in C^1(\overline{D}\times(\RR^N\setminus\{0\})),\hspace{0.7cm}&  K(x,\cdot)\in C^2(\RR^N\setminus\{0\}),\\
\frac{\partial K}{\partial x_i}\in C(\overline{D}\times(\RR^N\setminus\{0\})),\hspace{0.7cm}& \frac{\partial K}{\partial x_i}(x,\cdot)\in C^1(\RR^N\setminus\{0\}),\\
\frac{\partial K}{\partial z_i}\in C(\overline{D}\times(\RR^N\setminus\{0\})),\hspace{0.7cm}& \\
\frac{\partial^2 K}{\partial z_i\partial x_j}\in C(\overline{D}\times(\RR^N\setminus\{0\})),\hspace{0.7cm}& \frac{\partial^2 K}{\partial z_i\partial z_j}\in C(\overline{D}\times(\RR^N\setminus\{0\})).
\end{split}
\end{align}
\item The first derivatives of $K(x,z)$ are H\"{o}lder-continuous with exponent $\alpha$ with respect to $x$ in the sense that, for each $x_1,x_2\in\overline{D},\ z\in\RR^N\setminus\{0\}$ and for all index $1\leq i\leq N$,
\begin{align}\label{NucleoSingular.1Derivada.Holder.Hip3.form}
\begin{split}
\left\vert\frac{\partial K}{\partial x_i}(x_1,z)-\frac{\partial K}{\partial x_i}(x_2,z)\right\vert\leq & C\frac{\vert x_1-x_2\vert^\alpha}{\vert z\vert^{N-1}},\\
\left\vert\frac{\partial K}{\partial z_i}(x_1,z)-\frac{\partial K}{\partial z_i}(x_2,z)\right\vert\leq & C\frac{\vert x_1-x_2\vert^\alpha}{\vert z\vert^{N}}.
\end{split}
\end{align}
\end{enumerate}
Then, for every $R>0$ there exists a nonnegative constant $M=M(N,\alpha,R)$ so that the generalized volume potential $\mathcal{N}_K\zeta$ with density $\zeta\in C^{0,\alpha}_c(B_R(0))$ belongs to $C^{1,\alpha}(\overline{D})$ and satisfies the estimate
$$\Vert \mathcal{N}_K\zeta\Vert_{C^{1,\alpha}(D)}\leq M\Vert \zeta\Vert_{C^{0,\alpha}(\RR^N)}.$$
Moreover, its first partial derivatives can be computed as
$$\frac{\partial}{\partial x_i}\left(\mathcal{N}_K\zeta\right)=\mathcal{N}_{\frac{\partial K}{\partial x_i}}\zeta+\mathcal{N}_{\frac{\partial K}{\partial z_i}}\zeta.$$
\end{theo}

Notice that the singular integral kernel $\frac{\partial K}{\partial z_i}$ has an associated singular integral operator $\mathcal{N}_{\frac{\partial K}{\partial z_i}}$, where the integrals are understood in the sense of Cauchy principal values, i.e.,
$$\left(\mathcal{N}_{\frac{\partial K}{\partial z_i}}\zeta\right)(x)=\mbox{PV}\int_{\RR^N}\frac{\partial K}{\partial z_i}(x,x-y)\zeta(y)\,dy,$$
by virtue of the cancelation properties arising from the homogeneity in $z$ of the original kernel $K(x,z)$. Another interesting observation, which explains some differences between volume potentials in the whole $\RR^N$ and volume potentials in a bounded domain, is the change of variables formula
\begin{equation}\label{NucleoDebSingular.CambioVariable.Potencial.form}
(\mathcal{N}_K\zeta)(x)=\int_{\RR^N}K(x,x-y)\zeta(y)\,dy=\int_{\RR^N}K(x,z)\zeta(x-z)\,dz,
\end{equation}
which lets us take derivatives in any of the fwo factors. When the kernel is not sufficiently well behaved, we can put the derivatives on the density, or the other way round. However, for densities on $G$ as in Lemma \ref{Potencial.volumetrico.regularidad.lem1}, the previous change of variable is not allowed and the only way to transfer derivatives to the densities is by means of the integration by parts argument in Lemma \ref{Potencial.volumetrico.regularidad.lem1}. This gives rise to a new boundary term appears that must be studied by estimating single layer potentials as in Theorem \ref{Potencial.capasimple.regularidad.teo}.

As a consequence, one can prove the next two corollaries, where higher order derivatives of these generalized volume potentials can be considered.

\begin{cor}\label{NucleoDebSingular.Derivadas.Holder.cor2}
Les us consider $0\leq \beta\leq N-1$, $0<\alpha<1$, $k,m\in\NN$ so that $\beta+m\leq N-1$ and $K(x,z),\ x\in\overline{D},z\in\RR^N\setminus\{0\}$, a weakly singular integral kernel of exponent $\beta$ which satisfies the following hypotheses:
\begin{enumerate}
\item For each $\gamma_1,\gamma_2$ so that $\vert \gamma_1\vert\leq k$ and $\vert\gamma_2\vert\leq m$, $D^{\gamma_1+\gamma_2}_xK(x,z)$ is weakly singular with exponent $\beta$ and $D^{\gamma_1}_xD^{\gamma_2}_zK(x,z)$ is a finite sum of weakly singular integral kernels with exponent $\beta+n$ where $n=0,\ldots,\vert \gamma_2\vert$, i.e.,
\begin{align}\label{NucleoDebSingular.Derivadas.Holder.Hip0.form2}
\begin{split}
\left\vert D^{\gamma_1+\gamma_2}_x K(x,z)\right\vert & \leq \frac{C}{\vert z\vert^\beta},\\
\left\vert D^{\gamma_1}_x D^{\gamma_2}_z K(x,z)\right\vert & \leq\frac{C}{\vert z\vert^{\beta+\vert \gamma_2\vert}}.
\end{split}
\end{align}
\item For every $x\in \overline{D}$ and $\gamma_1,\gamma_2$ such that $\vert \gamma_1\vert\leq k$ and $\vert\gamma_2\vert\leq m$
\begin{equation}\label{NucleoDebSingular.Derivadas.Holder.Hip1.form2}
(D^{\gamma_1+\gamma_2}_xK)(x,\cdot),\,(D^{\gamma_1}_xD^{\gamma_2}_z K)(x,\cdot)\in C^1(\RR^N\setminus\{0\}).
\end{equation}
\item For all $x\in \overline{D}$, $z\in\RR^N\setminus\{0\}$ and $\gamma_1,\gamma_2$ so that $\vert \gamma_1\vert\leq k$ and $\vert\gamma_2\vert\leq m$ 
\begin{align}\label{NucleoDebSingular.Derivadas.Holder.Hip2.form2}
\begin{split}
\left\vert\nabla_z D^{\gamma_1+\gamma_2}_xK(x,z)\right\vert & \leq \frac{C}{\vert z\vert^{\beta+1}},\\
\left\vert\nabla_z D^{\gamma_1}_xD^{\gamma_2}_z K(x,z)\right\vert & \leq \frac{C}{\vert z\vert^{\beta+\vert\gamma_2\vert+1}}.
\end{split}
\end{align}
\item For any $x_1,x_2\in\overline{D}$, $z\in \RR^N\setminus\{0\}$ and $\gamma_1,\gamma_2$ satisfying $\vert \gamma_1\vert\leq k$ and $\vert\gamma_2\vert\leq m$
\begin{align}\label{NucleoDebSingular.Derivadas.Holder.Hip3.form2}
\begin{split}
\vert D^{\gamma_1+\gamma_2}_xK(x_1,z)-D^{\gamma_1+\gamma_2}_xK(x_2,z)\vert & \leq C\frac{\vert x_1-x_2\vert^\alpha}{\vert z\vert^\beta},\\
\vert D^{\gamma_1}_x D^{\gamma_2}_zK(x_1,z)-D^{\gamma_1}_x D^{\gamma_2}_zK(x_2,z)\vert & \leq C\frac{\vert x_1-x_2\vert^\alpha}{\vert z\vert^{\beta+\vert \gamma_2\vert}}.
\end{split}
\end{align}
\end{enumerate}
Then, for every $R>0$ there exists a nonnegative constant $M=M(N,\alpha,\beta,k,m,R)$ such that the generalized volume potential $\mathcal{N}_K\zeta$ generated by a density $\zeta\in C^{k,\alpha}_c(B_R(0))$ belongs to $C^{k+m,\alpha}(\overline{D})$ and verifies the estimate
$$\Vert \mathcal{N}_K\zeta\Vert_{C^{k+m,\alpha}(D)}\leq M\Vert \zeta\Vert_{C^{k,\alpha}(\RR^N)}.$$
Moreover, for every multi-index $\gamma=\gamma_1+\gamma_2$ so that $\vert \gamma_1\vert\leq k$ and $\vert \gamma_2\vert\leq m$ 
$$D^\gamma(\mathcal{N}_K\zeta)=\sum_{\delta\leq \gamma_1}\binom{\gamma_1}{\delta}\left(\mathcal{N}_{D^{\delta+\gamma_2}_xK}D^{\gamma_1-\delta}\zeta+\mathcal{N}_{D^{\delta}_xD^{\gamma_2}_zK}D^{\gamma_1-\delta}\zeta\right).$$
\end{cor}

\begin{cor}\label{NucleoSingular.kDerivadas.Holder.cor}
Let $0<\alpha<1$, $k\in\NN$, $\ x\in\overline{D},\,z\in\RR^N\setminus\{0\}$ and $K(x,z)$ be a weakly singular kernel, which has the following properties:
\begin{enumerate}
\item $K(x,z)$ is positively homogeneous of degree $-(N-1)$ in the second variable, i.e., 
\begin{equation}\label{NucleoSingular.kDerivadas.Holder.Hip1}
K(x,\lambda z)=\frac{1}{\lambda^{N-1}}K(x,z).
\end{equation}
\item $K(x,z)$ has the regularity properties 
\begin{align}\label{NucleoSingular.kDerivadas.Holder.Hip2}
\begin{split}
D^\gamma_xK\in C^1(\overline{D}\times(\RR^N\setminus\{0\})),\hspace{0.7cm}&  (D^\gamma_xK)(x,\cdot)\in C^2(\RR^N\setminus\{0\}),\\
\frac{\partial}{\partial x_i}D^\gamma_x K\in C(\overline{D}\times(\RR^N\setminus\{0\})),\hspace{0.7cm}& \left(\frac{\partial}{\partial x_i}D^\gamma_x K\right)(x,\cdot)\in C^1(\RR^N\setminus\{0\}),\\
\frac{\partial}{\partial z_i}D^\gamma_x K\in C(\overline{D}\times(\RR^N\setminus\{0\})),\hspace{0.7cm}& \\
\frac{\partial^2}{\partial z_i\partial x_j}D^\gamma_x K\in C(\overline{D}\times(\RR^N\setminus\{0\})),\hspace{0.7cm}& \frac{\partial^2}{\partial z_i\partial z_j}D^\gamma_x K\in C(\overline{D}\times(\RR^N\setminus\{0\})),
\end{split}
\end{align}
for  each index $1\leq i,j\leq N$ and each  $\gamma$ with $\vert \gamma\vert\leq k$.
\item The derivatives of $K(x,z)$ with respect to $x$ up to order $k$ are H\"{o}lder-continuous with exponent $\alpha$ in the sense that
\begin{align}\label{NucleoSingular.kDerivadas.Holder.Hip3}
\begin{split}
\left\vert\left(\frac{\partial}{\partial x_i}D^\gamma_x K\right)(x_1,z)-\left(\frac{\partial}{\partial x_i}D^\gamma_x K\right)(x_2,z)\right\vert\leq & C\frac{\vert x_1-x_2\vert^\alpha}{\vert z\vert^{N-1}},\\
\left\vert\left(\frac{\partial}{\partial z_i}D^\gamma_x K\right)(x_1,z)-\left(\frac{\partial}{\partial z_i}D^\gamma_x K\right)(x_2,z)\right\vert\leq & C\frac{\vert x_1-x_2\vert^\alpha}{\vert z\vert^{N}}
\end{split}
\end{align}
for each $x_1,x_2\in\overline{D}$, $z\in\RR^N\setminus\{0\}$, each index $1\leq i\leq N$ and $\vert \gamma\vert\leq k$.
\end{enumerate}
Then, for every $R>0$ there exists a nonnegative constant $M=M(N,\alpha,k,R)$  such that the generalized volume potential $\mathcal{N}_K\zeta$ generated by any density $\zeta\in C^{k,\alpha}_c(B_R(0))$ belongs to $C^{k+1,\alpha}(\overline{D})$ and 
$$\Vert \mathcal{N}_K\zeta\Vert_{C^{k+1,\alpha}(D)}\leq M\Vert \zeta\Vert_{C^{k,\alpha}(\RR^N)}.$$
Moreover, the partial derivatives of the volume potential are
$$\frac{\partial}{\partial x_i}D^\gamma_x\left(\mathcal{N}_K\zeta\right)=\sum_{\delta\leq \gamma}\binom{\gamma}{\delta}\left(\mathcal{N}_{\frac{\partial}{\partial x_i}D^\gamma_x K}D^{\delta-\gamma}\zeta+\mathcal{N}_{\frac{\partial}{\partial z_i}D^\gamma_x K}D^{\delta-\gamma}\zeta\right),$$
for $\vert \gamma\vert\leq k$ and $1\leq i\leq N$.
\end{cor}

When the constants $C$ appearing in the statements of the above results do not depend on the chosen bounded domain $D$, the above estimates can be extended from H\"{o}lder estimates over $\overline{D}$, to global estimates in $\RR^N$. 
This is the case for the integral kernels which do not depend on the variable $x$ (e.g., $\Gamma_0(z)$, $R_\lambda(z)$ and $\Gamma_\lambda(z)$).
In this way, we get the next result in the spirit of Lemmas \ref{Potencial.volumetrico.regularidad.lem1} and \ref{Potencial.volumetrico.regularidad.lem2}.

\begin{lem}\label{Potencial.volumetrico.regularidad.lem3}
Consider the generalized volume potential with compactly supported density $\zeta:\RR^3\longrightarrow\RR$ associated with the Helmholtz equation
$$(\mathcal{N}_\lambda\zeta)(x):=\int_{\RR^3}\Gamma_\lambda(x-y)\zeta(y)\,d_yS,\ x\in\RR^3.$$
Then, for every $R>0$ there exists a nonnegative constant $K=K(k,\alpha,\lambda,R)$ such that $\mathcal{N}_\lambda\zeta$ is well defined for every $\zeta\in C^{k,\alpha}_c(B_R(0))$, belongs to $C^{k+2,\alpha}(\RR^3)$ and the following estimate is verified
$$\left\Vert\mathcal{N}_\lambda\zeta\right\Vert_{C^{k+2,\alpha}(\RR^3)}\leq K\Vert \zeta\Vert_{C^{k,\alpha}(\RR^3)}.$$
\end{lem}

Combining  the above results, we can estimate generalized volume potentials in $\Omega$ whose densities have compact support in $\overline{\Omega}$ by means of an appropriate splitting. Using Calder\'on's extension theorem (Proposition \ref{ExtensionHolder.pro}),  for every $\zeta\in C^{k,\alpha}_c(\overline{\Omega})$ there exists an extension  $\overline{\zeta}\in C^{k,\alpha}_c(\RR^3)$, so
$$\mathcal{N}^+_\lambda\zeta=\left.\left(\mathcal{N}_\lambda\overline{\zeta}\right)\right\vert_{ \Omega}-\mathcal{N}^+_\lambda\left(\left.\overline{\zeta}\right\vert_{G}\right)\ \mbox{ in }\Omega.$$
Then, Lemmas \ref{Potencial.volumetrico.regularidad.lem2} and \ref{Potencial.volumetrico.regularidad.lem3} lead to the following result:

\begin{theo}[Generalized volume potential]\label{Potencial.volumetrico.regularidad.teo}
Let $G\subseteq \RR^3$ be a bounded domain with regularity $C^{k+1,\alpha}$, $\Omega:=\RR^3\setminus \overline{G}$ its exterior domain and $S=\partial G$ the boundary surface. Define the generalized volume potential in $\Omega$ with density $\zeta:\Omega\longrightarrow\RR$ associated with the Helmholtz equation
$$(\mathcal{N}^+_\lambda\zeta)(x)=\int_{\Omega}\Gamma_\lambda(x-y)\zeta(y)\,dy,\ x\in\Omega.$$
Let  $R>0$ be such  that $\overline{G}\subseteq B_R(0)$, and define $\Omega_R:=B_R(0)\setminus \overline{G}$. Then,  $\mathcal{N}^+_\lambda\zeta$ is well defined in $\Omega$ and  belongs to $C^{k+2,\alpha}(\overline{\Omega})$, for every $\zeta\in C^{k,\alpha}(\overline{\Omega})$ such that $\supp(\zeta)\subseteq \overline{\Omega_R}$. In addition, the bound 
$$\Vert \mathcal{N}_\lambda\zeta\Vert_{C^{k+2,\alpha}(\Omega)}\leq K\Vert \zeta\Vert_{C^{k,\alpha}(\Omega)}$$
holds for some $K>0$ depending on $k,\alpha,\lambda,G,R$ but not on $\zeta$.
\end{theo}

\subsection{Regularity of the boundary integral operator}

The next step is to analyze the regularity properties of the boundary integral operator $T_\lambda$ (\ref{EcIntegral.DatoFrontera.Beltrami.T.form}) arising in the boundary integral equation associated with the boundary data $\eta\times u$ (\ref{EcIntegral.DatoFrontera.Beltrami.form}) in Theorem \ref{BeltramiNoHomog.teo}. Firstly, we split the operator $T_\lambda$ into simpler integral operators. By  inspection, $T_\lambda$ is given by
$$
T_\lambda=\mathcal{M}^T_\lambda+\lambda\mathcal{S}^T_\lambda.
$$
$\mathcal{M}_\lambda^T\zeta$ is known as the \textit{magnetic dipole operator}, which is the tangent component of the electric field generated by a dipole distribution with density $\zeta\in\mathfrak{X}(S)$, i.e.,
$$(\mathcal{M}_\lambda^T\zeta)(x):=\int_S \eta(x)\times\curl_x\left(\Gamma_\lambda(x-y)\zeta(y)\right)\,d_yS,\ \ x\in S.$$
$\mathcal{S}_\lambda^T$ is the tangential component of the generalized single layer potential generated by $\zeta$,
$$
(\mathcal{S}_\lambda^T\zeta)(x)=\int_S \Gamma_\lambda(x-y)\eta(x)\times \zeta(y)\,d_yS,\ \ x\in S.
$$

The integral kernel of $\mathcal{S}_\lambda^T$ is weakly singular over $S$, so this integral is absolutely convergent under suitable hypotheses for $\zeta$. The integral kernel of $\mathcal{M}_\lambda^T$ looks singular over $S$ but, as we shall see below, this integral is just weakly singular when $\zeta$ is a tangent vector field on $S$. Thus, this integral is actually absolutely convergent under minimal assumptions on $\zeta$. In order to see why, notice that, given any tangent field along $S$, $\zeta\in\mathfrak{X}^{k,\alpha}(S)$, one has the decomposition
\begin{align*}
\eta(x)\times(\nabla_x\Gamma_\lambda(x-y)\times \zeta(y))\
&=\eta(x)\cdot \zeta(y)\,\nabla_x \Gamma_\lambda(x-y)-\eta(x)\cdot \nabla_x \Gamma_\lambda(x-y)\,\zeta(y)\\
&=(\eta(x)-\eta(y))\cdot \zeta(y)\,\nabla_x \Gamma_\lambda(x-y)-\eta(x)\cdot \nabla_x \Gamma_\lambda(x-y)\,\zeta(y).
\end{align*}
Consequently, the $j$-th coordinates of the integrands read
\begin{align*}
(\eta(x)\times(\nabla_x\Gamma_\lambda(x-y)\times \zeta(y)))_j
&=(\eta(x)-\eta(y))\cdot \zeta(y)\,\partial_{x_j} \Gamma_\lambda(x-y)-\eta(x)\cdot \nabla_x \Gamma_\lambda(x-y)\,\zeta_j(y)\\
&=\sum_{i=1}^3(\eta_i(x)-\eta_i(y))\zeta_i(y)\,\partial_{x_j}\Gamma_\lambda(x-y)-\eta(x)\cdot \nabla_x \Gamma_\lambda(x-y)\,\zeta_j(y),
\\
(\Gamma_\lambda(x-y)\eta(x)\times \zeta(y))_j &=\Gamma_\lambda(x-y)(\eta(x)\times \zeta(y))\cdot e_j\\
&=\Gamma_\lambda(x-y)(e_j\times \eta(x))\cdot \zeta(y)=\sum_{i=1}^3 \Gamma_\lambda(x-y)(e_i\times e_j)\cdot \eta(x)\zeta_i(y).
\end{align*}
Consider any extension $\widetilde{\eta}\in C^{k+4,\alpha}_c(\RR^3)$ of the outward unit normal vector field $\eta$ to the compact surface $S$ and define the kernels
\begin{align}\label{NucleosMyS.form}
K^{\mathcal{D}}_\lambda(x,z)&=\widetilde{\eta}(x)\cdot \nabla \Gamma_\lambda(z),\nonumber\\
K^{i,j}_\lambda(x,z)&=(\widetilde{\eta}_i(x)-\widetilde{\eta}_i(x-z))\,\partial_{z_j}\Gamma_\lambda(z),\\
\widetilde{K}^{i,j}_\lambda(x,z)&=(e_i\times e_j)\cdot \widetilde{\eta}(x)\,\Gamma_\lambda(z)\nonumber.
\end{align}
Then, we have the associated splitting of the operators $\mathcal{M}_\lambda^T$ and $\mathcal{S}_\lambda^T$
\begin{align}\label{NucleosMyS.Decomposicion.form}
\begin{split}
(\mathcal{M}_\lambda^T\zeta)_j(x)&=\sum_{i=1}^3T_{K^{i,j}_\lambda}\zeta_i-T_{K^\mathcal{D}_\lambda}\zeta_j,\\
(\mathcal{S}_\lambda^T\zeta)_j(x)&=\sum_{i=1}^3T_{\widetilde{K}^{i,j}_\lambda}\zeta_i,
\end{split}
\end{align}
where the integral operators in the above decomposition are
\begin{align}\label{NucleosMyS.Decomposicion.OperadoresIntegrales.form}
(T_{K^\mathcal{D}_\lambda}\zeta_j)(x)&=\int_S K^{\mathcal{D}}_\lambda(x,x-y)\zeta_j(y)\,d_yS,\nonumber\\
(T_{K^{i,j}_\lambda}\zeta_i)(x)&=\int_S K^{i,j}_\lambda(x,x-y)\zeta_i(y)\,d_yS,\\
(T_{\widetilde{K}^{i,j}_\lambda}\zeta_i)(x)&=\int_S \widetilde{K}^{i,j}_\lambda(x,x-y)\zeta_i(y)\,d_yS.\nonumber
\end{align}
Since every $C^2$ compact surface satisfies
\begin{align*}
\vert \eta(x)\cdot (x-y)\vert & \leq L\vert x-y\vert^2,\\
\vert \eta(x)-\eta(y)\vert & \leq L\vert x-y\vert,
\end{align*}
for each $x,y\in S$, then all the preceding integral kernels are weakly singular. In particular, it prevents these integrals from being considered in the Cauchy principal value sense.

The study of H\"{o}lder estimates for all these potentials can be performed along the same lines as in \cite[Satz 4.3, Satz 4.4]{Heinemann}. In that work, the author dealt with the homogeneous harmonic case $\lambda=0$, where the kernels have a simpler form. In our case  $\lambda\neq 0$, we will decompose the $3$-dimensional kernels into a homogeneous part and an inhomogeneous but less singular part as in (\ref{PhiPsi.SolucionFundamental.Descomposicion.Gamma0+Rlambda}). Then, we will consider a coordinate system over $S$ which allows transforming the integrals over $S$ into integrals over planar domains by means of a change of variables. The homogeneous and more singular parts will satisfy the hypothesis in Corollary \ref{NucleoSingular.kDerivadas.Holder.cor} and the terms in the remainder will verify those in Corollary \ref{NucleoDebSingular.Derivadas.Holder.cor2}. Our regularity result then reads as follows:

\begin{theo}\label{Potencial.capasimple.regularidad.frontera.teo}
Let  $G$ be a bounded domain of class $C^{k+5}$, $S=\partial G$ the boundary surface, $\eta\in C^{k+4}(S,\RR^3)$ the outward unit normal vector field along $S$ and any extension $\widetilde{\eta}\in C^{k+4}_c(\RR^3,\RR^3)$ of $\eta$. Let $K^{\mathcal{D}}_\lambda(x,z),\ K^{i,j}_\lambda(x,z)$ and $\widetilde{K}^{i,j}_\lambda(x,z)$ be the kernels given by  (\ref{NucleosMyS.form}) and  $T_{K^{\mathcal{D}}_\lambda},\ T_{K^{i,j}_\lambda}$ and $T_{\widetilde{K}^{i,j}_\lambda}$ the associated boundary integral operators given by (\ref{NucleosMyS.Decomposicion.OperadoresIntegrales.form}). Then, these integral operators are bounded from $C^{k,\alpha}(S)$ into $C^{k+1,\alpha}(S)$, i.e., the following linear operators are continuous:
\begin{equation*}
\begin{array}{cccc}
T_{K^{\mathcal{D}}_\lambda}: & C^{k,\alpha}(S) & \longrightarrow & C^{k+1,\alpha}(S),\\
T_{K^{i,j}_\lambda}: & C^{k,\alpha}(S) & \longrightarrow & C^{k+1,\alpha}(S),\\
T_{\widetilde{K}^{i,j}_\lambda}: & C^{k,\alpha}(S) & \longrightarrow & C^{k+1,\alpha}(S).
\end{array}
\end{equation*}
As a consequence, the linear operators
\begin{equation*}
\begin{array}{cccc}
\mathcal{M}_\lambda^T: & \mathfrak{X}^{k,\alpha}(S) & \longrightarrow & \mathfrak{X}^{k+1,\alpha}(S),\\
\mathcal{S}_\lambda^T: & \mathfrak{X}^{k,\alpha}(S) & \longrightarrow & \mathfrak{X}^{k+1,\alpha}(S)
\end{array}
\end{equation*}
are bounded too.
\end{theo}

\begin{rem}The above regularity assumptions on the boundary surface will be discussed during the proof of the theorem. Roughly speaking, we will need $C^{k+5}$ boundaries for the operators in (\ref{NucleosMyS.Decomposicion.OperadoresIntegrales.form}) of first and second type to be bounded from $C^{k,\alpha}(S)$ to $C^{k+1,\alpha}(S)$ whilst assuming $C^{k+4}$ boundaries suffices to ensure the corresponding result for the third kind of operators in  (\ref{NucleosMyS.Decomposicion.OperadoresIntegrales.form}).  See \cite[Satz 4.3, Satz 4.4]{Heinemann} for the homogeneous harmonic case $\lambda=0$. Let us recall that a similar formalism was introduced in \cite{Nedelec} to deal with weakly singular operators whose homogeneous kernels are of Calder\'on--Zygmund type after a finite amount of derivatives is taken. The \textit{pseudo-homogeneous} kernels are those that can be split into finitely many homogeneous weakly singular kernels of the preceding type and an arbitrarily regular remainder. Its associated integral operators gain $m$ derivatives on any Sobolev space with finite exponent $1<p<\infty$, $-m$ being the \textit{class} of the pseudo-homogeneous kernel (see \cite[Chapter 4, Section 3]{Nedelec} for more details). In particular, \cite[Example 4.3]{Nedelec} shows that $\widetilde{K}^{i,j}_\lambda(x,z)$ is pseudo-homogeneous of class $-1$ and the aforementioned regularity result shows that $\mathcal{S}_\lambda^T$ is bounded from $W^{r,p}(S)$ to $W^{r+1,p}(S)$. This approach involves splitting the exponential function in $\Gamma_\lambda(z)$ on its Taylor series, giving rise to finitely many homogeneous kernels and a regular enough remainder. Since $\widetilde{K}^{i,j}_\lambda(x,z)$ has separated variables, it could be approached within this framework, but one cannot say the same about other of the kernels in (\ref{NucleosMyS.form}). On the contrary, the ideas in \cite{Heinemann} work well for H\"{o}lder regularity using singular and weakly singular kernels like those in Corollaries \ref{NucleoDebSingular.Derivadas.Holder.cor2} and \ref{NucleoSingular.kDerivadas.Holder.cor}.
\end{rem}

\begin{proof}
Since the kernel $\widetilde{K}^{i,j}_\lambda(x,z)$ can be analyzed through a similar reasoning (as shown in \cite{Heinemann} for the case $\lambda=0$), we will restrict out analysis to the kernels $K^{i,j}_{\lambda}(x,z)$ and $K^\mathcal{D}_\lambda(x,z)$, which were not studied in \cite{Heinemann}. Let us then split these inhomogeneous kernels into a homogeneous part and some less singular part (see the decomposition (\ref{PhiPsi.SolucionFundamental.Descomposicion.Gamma0+Rlambda}) and the functions $\phi_\lambda$ and $\psi_\lambda$ in (\ref{PhiPsi.SolucionFundamental.form})). To this end, notice that
\begin{align}\label{NucleosMyS.ReescrituraPhiyPsi.form}
\begin{split}
K^{\mathcal{D}}_\lambda(x,z)&=\frac{\phi_\lambda'(\vert z\vert)}{\vert z\vert}\widetilde{\eta}(x)\cdot z,\\
K^{i,j}_\lambda(x,z)&=(\widetilde{\eta}_i(x)-\widetilde{\eta}_i(x-z))\,\frac{\phi_\lambda'(\vert z\vert)}{\vert z\vert}z_j,
\end{split}
\end{align}
Consequently, 
\begin{align}\label{NucleosMyS.Factorizacion.form}
\begin{split}
K^{i,j}_{\lambda}(x,z)&=K^{i,j}_{\lambda,0}+K^{i,j}_{\lambda,1},\\
K^{\mathcal{D}}_{\lambda}(x,z)&=K^\mathcal{D}_{\lambda,0}+K^\mathcal{D}_{\lambda,1},
\end{split}
\end{align}
where,
\begin{equation}\label{NucleosMyS.Factorizacion.Partes.form}
\begin{array}{llcll}
\displaystyle K^{i,j}_{\lambda,0}(x,z)&\displaystyle :=-\frac{1}{4\pi}(\widetilde{\eta}_i(x)-\widetilde{\eta}_i(x-z))\,\frac{z_j}{\vert z\vert^3},&\hspace{1cm} & \displaystyle K^{i,j}_{\lambda,1}(x,z)&\displaystyle:=(\widetilde{\eta}_i(x)-\widetilde{\eta}_i(x-z))\frac{\psi_\lambda'(\vert z\vert)}{\vert z\vert}z_j,\\
\displaystyle K^\mathcal{D}_{\lambda,0}(x,z)&\displaystyle:=-\frac{1}{4\pi} \widetilde{\eta}(x)\cdot \frac{z}{\vert z\vert^3}, & \hspace{1cm} &\displaystyle K^\mathcal{D}_{\lambda,1}(x,z)& \displaystyle :=\widetilde{\eta}(x)\cdot\frac{\psi_\lambda'(\vert z\vert)z}{\vert z\vert}.
\end{array}
\end{equation}
Notice that the associated integral operators only involve values $x,y\in S$.  Define
$$d_S:=2\max_{x,y\in S}\vert x-y\vert$$
and let us take  $x\in S$ and $z\in B_{d_S}(0)$, so in this case we have
\begin{align}\label{NucleosMyS.CotasLocalesEnz.form}
\begin{split}
\frac{1}{\vert z\vert^{\beta_1}}+\frac{1}{\vert z\vert^{\beta_2}} & \leq \left(1+d_S^{M-m}\right)\frac{1}{\vert z\vert^{M}},\\
\vert z\vert^{\beta_1}+\vert z\vert^{\beta_2} & \leq \left(1+d_S^{M-m}\right)\vert z\vert^{m}
\end{split}
\end{align}
for any couple of exponents $\beta_1,\beta_2\geq 0$ and any $z\in B_{d_S}(0)$. Here $m$ and $M$ stand for the minimum and maximum values i.e.,
$$
m:=\min\{\beta_1,\beta_2\},\ \ M:=\max\{\beta_1,\beta_2\}.
$$

Consider the function arising in (\ref{NucleosMyS.Factorizacion.Partes.form}),
$$f_\lambda(r):=\frac{\psi_\lambda'(r)}{r},\ \ r>0,$$
and note that (\ref{PhiPsi.SolucionFundamental.Cotas.Derivadas.form}) leads to
\begin{equation}\label{PhiPsi.SolucionFundamental.Cotas.Derivadas.flambda.form}
\begin{array}{ll}
\displaystyle\vert f_\lambda^{(m)}(r)\vert \leq C\left(\frac{1}{r}+\frac{1}{r^{m+2}}\right), &r>0,\\
\displaystyle\vert f_\lambda^{(m)}(r)\vert \leq \widetilde{C}\frac{1}{r^{m+2}}, &  r\in \left(0,d_S\right).
\end{array}
\end{equation}
for some $C>0$ that does not depend on $m$  and some $\widetilde{C}$ depending on $m$ and $d_S$.

Let us study the boundedness of the integral operators associated with the integral kernels $K^{i,j}_{\lambda,n}$ and $K^{\mathcal{D}}_{\lambda,n}$ for $n=0,1$. To this end, let us consider a finite covering of $S$ by $M$ coordinate neighborhoods $\Sigma_1,\ldots,\Sigma_M\subseteq S$ endowed with the associated local charts $\mu_m:D_m\longrightarrow \Sigma_m$ belonging to $ C^{k+5}(\overline{D}_m,\RR^3)$ and enjoying homeomorphic extensions up to the boundary of the planar disks $D_m\subseteq \RR^2$. Also consider the associated partition of unity of class $C^{k+5}$, i.e., $\{\varphi_m\}_{m=1}^M\subseteq C^{k+5}(S)$ such that $\mbox{supp}\,\varphi_m\subseteq \Sigma_m$ for each index $\displaystyle m=1,\ldots,M$ and
$$\displaystyle 1=\sum_{m=1}^M\varphi_m(x),$$
for any $x\in S.$ We will denote the Jacobian of each local chart $\mu_m$ by
$$J_m(s):=\left\vert\left(\frac{\partial \mu_m}{\partial s_1}\times \frac{\partial \mu_m}{\partial s_2}\right)(s)\right\vert=\sqrt{\det\left(g^{ij}_m (s)\right)},\ \ s\in D_m.$$
Here $g^{ij}_m(s)$ stands for the $(i,j)$ component of the induced Euclidean metric on $S$ with respect to the local chart $\mu_m$, i.e.,
$$\displaystyle\left(g^{ij}_m(s)\right)=\left(\begin{array}{cc} E_m(s) & F_m(s)\\ F_m(s) & G_m(s)\end{array}\right),$$
where $E_m,F_m,G_m$ stands for the coefficients of the first fundamental form, namely,
$$
E_m(s):=\frac{\partial \mu_m}{\partial s_1}(s)\cdot \frac{\partial \mu_m}{\partial s_1}(s),\hspace{0.5cm} F_m(s):=\frac{\partial \mu_m}{\partial s_1}(s)\cdot \frac{\partial \mu_m}{\partial s_2}(s),\hspace{0.5cm} G_m(s):=\frac{\partial \mu_m}{\partial s_2}(s)\cdot \frac{\partial \mu_m}{\partial s_2}(s).
$$
Consequently,
\begin{align}
(T_{K^{i,j}_{\lambda,n}}\zeta)(\mu_m(s))&=\sum_{m'=1}^M\int_{D_{m'}}K^{i,j}_{\lambda,n}(\mu_m(s),\mu_m(s)-\mu_{m'}(t))\varphi_{m'}(\mu_{m'}(t))\zeta(\mu_{m'}(t))J_{m'}(t)\,dt,\label{NucleosMyS.Homogeneos.Coordenadas.form1}\\
(T_{K^\mathcal{D}_{\lambda,n}}\zeta)(\mu_m(s))&=\sum_{m'=1}^M\int_{D_{m'}}K^\mathcal{D}_{\lambda,n}(\mu_m(s),\mu_m(s)-\mu_{m'}(t))\varphi_{m'}(\mu_{m'}(t))\zeta(\mu_{m'}(t))J_{m'}(t)\,dt.\label{NucleosMyS.Homogeneos.Coordenadas.form2}
\end{align}

We will study the most singular case $m'=m$ and then show how the case $m'\neq m$ follows from it. An important fact is that we will extract the most singular homogeneous parts of $K^{i,j}_{\lambda,0}(x,z)$ and $K^\mathcal{D}_{\lambda,0}(x,z)$ by virtue of the splitting (\ref{NucleosMyS.Factorizacion.form}). However, the change of variables in the coordinate neighborhoods $\Sigma_m$ gives rise to new inhomogeneous planar kernels, 
$$K^{i,j}_{\lambda,0}(\mu_m(s),\mu_m(s)-\mu_{m}(t))	\ \mbox{ and }\ K^\mathcal{D}_{\lambda,0}(\mu_m(s),\mu_m(s)-\mu_{m}(t)).$$
To solve this difficulty, we will decompose them again into the more sigular homogeneous part, which stands for a planar homogeneous kernel of degree $-1$, and some inhomogeneous but less singular term. Then, we will prove the corresponding regularity results for each term through Corollaries \ref{NucleoDebSingular.Derivadas.Holder.cor2} and \ref{NucleoSingular.kDerivadas.Holder.cor}.

Since both  $K^{i,j}_{\lambda,0}(x,z)$ and $K^\mathcal{D}_{\lambda,0}(x,z)$ can be studied by means of a similar reasoning, we will just analyze one of them, e.g. $K^{i,j}_{\lambda,0}(x,z)$. In fact, $K^\mathcal{D}_{\lambda,0}(x,z)$ stands for the integral kernel of the adjoint operator of the harmonic Neumann--Poincar\'e operator and was studied in \cite[Satz 4.4]{Heinemann}. Inspired by \cite[Lemma 4.2]{Heinemann}, let us expand $\mu_m(s)-\mu_m(t)$ though the integral form of Taylor's theorem up to second order,
\begin{equation}\label{DesarrolloTaylor.Parametrizacion.form}
\vert \mu_m(s)-\mu_m(t)\vert=\left(P_m(s,s-t)+Q_m(s,s-t)\right)^{1/2},
\end{equation}
where
\begin{align}
P_m(s,u)&:=\sum_{p,q=1}^2\frac{\partial \mu_m}{\partial s_p}(s)\cdot \frac{\partial \mu_m}{\partial s_q}(s)u_p u_q=\sum_{p,q=1}^2g^{pq}_m(s)u_pu_q=((g^{pq}_m(s))u)\cdot u,\label{DesarrolloTaylor.Parametrizacion.P.form}\\
Q_m(s,u)&:=-2\sum_{p,q,r=1}^2\frac{\partial \mu_m}{\partial s_p}(s)\cdot\left(\int_0^1(1-\theta)\frac{\partial^2\mu_m}{\partial s_q\partial s_r}(s-\theta u)\,d\theta\right)u_pu_qu_r\nonumber\\
&+\sum_{p,q,r,l=1}^2\left(\int_0^1(1-\theta)\frac{\partial^2\mu_m}{\partial s_p\partial s_q}(s-\theta u)\,d\theta\right)\cdot\left(\int_0^1(1-\theta)\frac{\partial^2\mu_m}{\partial s_r\partial s_l}(s-\theta u)\,d\theta\right)u_pu_qu_ru_l.\label{DesarrolloTaylor.Parametrizacion.Q.form}
\end{align}
Straightforward computations shows that $P_m(s,u)$ is positively homogeneous on $u$ of degree $2$, i.e.,
\begin{equation}\label{DesarrolloTaylor.Parametrizacion.P.homogeneo.form}
P_m(s,\rho u)=\rho^2P_m(s,u),
\end{equation}
for all $s\in D_m$, $u\in\RR^2\setminus\{0\}$ and $\rho>0$. Moreover, the estimates 
\begin{equation}\label{DesarrolloTaylor.Parametrizacion.PyQ.cotas.form}
\left.\begin{array}{rcl}
\displaystyle\frac{1}{C}\vert u\vert^2\leq & \displaystyle\vert P_m(s,u)\vert & \displaystyle\leq C\vert u\vert^2,\\
 & \displaystyle\vert Q_m(s,u)\vert & \displaystyle\leq C\vert u\vert^3,\\
\displaystyle\frac{1}{C}\vert u\vert^2\leq &\displaystyle \vert P_m(s,u)+Q_m(s,u)\vert & \displaystyle\leq C\vert u\vert^2.
\end{array}\right\}\left.\begin{array}{rl}
\displaystyle\left\vert D^\gamma_s P_m(s,u)\right\vert & \leq C\vert u\vert^2,\\
\displaystyle\left\vert D^\gamma_s Q_m(s,u)\right\vert & \leq C\vert u\vert^3,\\
\displaystyle\left\vert D^\gamma_s (P(s,u)+Q(s,u))\right\vert & \leq C\vert u\vert^2,\\
\displaystyle\left\vert \frac{\partial}{\partial u_i}D^\gamma_s P_m(s,u)\right\vert & \leq C\vert u\vert,\\
\displaystyle\left\vert \frac{\partial}{\partial u_i}D^\gamma_s Q_m(s,u)\right\vert & \leq C\vert u\vert^2,\\
\displaystyle\left\vert \frac{\partial}{\partial u_i}D^\gamma_s (P(s,u)+Q(s,u))\right\vert & \leq C\vert u\vert,\\
\displaystyle\left\vert \frac{\partial^2}{\partial u_i\partial u_j}D^\gamma_s P_m(s,u)\right\vert & \leq C\vert u\vert^0,\\
\displaystyle\left\vert \frac{\partial^2}{\partial u_i\partial u_j}D^\gamma_s Q_m(s,u)\right\vert & \leq C\vert u\vert,\\
\displaystyle\left\vert \frac{\partial^2}{\partial u_i\partial u_j}D^\gamma_s (P(s,u)+Q(s,u))\right\vert & \leq C\vert u\vert^0.
\end{array}\right\}
\end{equation}
hold for each $s\in D_m$, $u\in\RR^2$ such that $s-u\in D_m$ and every multi-index  with $\vert \gamma\vert\leq k$. See \cite[Satz 4.2]{Heinemann} for the details, which are starightforward. 

Our homogenization procedure follows from the next splitting, where $(\widetilde{\eta}\circ\mu_m)_i$ and $(\mu_m)_j$ are expanded again by means of the integral form of Taylor's theorem up to second order
$$
K^{i,j}_{\lambda,0}(\mu_m(s),\mu_m(s)-\mu_m(t))=H^{i,j}_{\lambda,0}(s,s-t)+R^{i,j}_{\lambda,0}(s,s-t),
$$
where the homogeneous part $H^{i,j}_{\lambda,0}(s,u)$ and the remainder $R^{i,j}_{\lambda,0}(s,u)$ take the form
\begin{align*}
H^{i,j}_{\lambda,0}(s,u)&:=-\frac{1}{4\pi}P_m(s,u)^{-3/2}\sum_{p,q=1}^2\frac{\partial (\widetilde{\eta}\circ\mu_m)_i}{\partial s_p}(s)\frac{\partial (\mu_m)_j}{\partial s_q}(s)u_p u_q,\\
R^{i,j}_{\lambda,0}(s,u)&:=\widetilde{R}^{i,j}_{\lambda,0}(s,u)+\widehat{R}^{i,j}_{\lambda,0}(s,u),
\end{align*}
and the remainder is split into
\begin{align*}
\widetilde{R}^{i,j}_{\lambda,0}(s,u):=&-\frac{1}{4\pi}\left((P_m(s,u)+Q_m(s,u))^{-3/2}-P_m(s,u)^{-3/2}\right)\\
&\hspace{2cm}\times\left(\sum_{p,q=1}^2 \frac{\partial(\widetilde{\eta}\circ\mu_m)_i}{\partial s_p}(s)\frac{\partial (\mu_m)_j}{\partial s_q}(s)u_p u_q\right),\\
\widehat{R}^{i,j}_{\lambda,0}(s,u):= &-\frac{1}{4\pi}\left(P_m(s,u)+Q_m(s,u)\right)^{-3/2}\\
&\hspace{2cm}\times\left\{-\sum_{p,q,r=1}^2\left(\int_0^1(1-\theta)\frac{\partial^2 (\widetilde{\eta}\circ\mu_m)_i}{\partial s_p\partial s_q}(s-\theta u)\,d\theta\right)\frac{\partial (\mu_m)_j}{\partial s_r}(s)u_pu_qu_r\right.\\
&\hspace{2.7cm}-\sum_{p,q,r=1}^2\frac{\partial (\widetilde{\eta}\circ\mu_m)_i}{\partial s_p}(s)\left(\int_0^1(1-\theta)\frac{\partial (\mu_m)_j}{\partial s_q\partial s_r}(s-\theta u)\,d\theta\right)u_pu_qu_r\\
&\hspace{2.7cm}+\sum_{p,q,r,l=1}^2\left(\int_0^1(1-\theta)\frac{\partial ^2(\widetilde{\eta}\circ \mu_m)_i}{\partial s_p\partial s_q}(s-\theta u)\,d\theta\right)\times\\
&\left.\hspace{4cm}\times\left(\int_0^1 (1-\theta)\frac{\partial^2(\mu_m)_j}{\partial s_r\partial s_l}(s-\theta u)\,d\theta\right)u_pu_qu_ru_l\right\}.
\end{align*}
Note again that only small values of $u=s-t$ are involved here; specifically,  $s\in D_m$ and $u\in D_{d_S^m}(0)$ for
$$d_S^m:=2\max_{s,t\in D_m}\vert s-t\vert.$$
Hence, one enjoy similar bounds to those in (\ref{NucleosMyS.CotasLocalesEnz.form}) with $z$ replaced with $u$. 

Let us next analyze each term in the above decomposition for $K^{i,j}_{\lambda,0}(\mu_m(s),\mu_m(s)-\mu_m(t))$. Firstly, since $P_m(s,u)$ is positively homogeneous on $u$ with degree $2$, then $H^{i,j}_{\lambda,0}(s,u)$ is positively homogeneous on $u$ with degree $-1$. The regularity properties in the second part in Corollary \ref{NucleoSingular.kDerivadas.Holder.cor} can be straighforwardly checked. Let us then concentrate on the regularity properties in the third part of such corollary and, to this end, let us compute the next partial derivative
\begin{align*}
D^\gamma_s H^{i,j}_{\lambda,0}&(s,u)=\\
&-\frac{1}{4\pi}\sum_{\sigma\leq \gamma}\binom{\gamma}{\sigma}D^\sigma_s\left(P_m(s,u)^{-3/2}\right)\left[\sum_{p,q=1}^2D^{\gamma-\sigma}_s\left(\frac{\partial (\widetilde{\eta}\circ\mu_m)_i}{\partial s_p}(s)\frac{\partial (\mu_m)_j}{\partial s_q}(s)\right)u_pu_q\right].
\end{align*}

Define the homogeneous function $h(t):=t^{-3/2},\ t>0$ and use chain rule for high order derivatives to arrive at
$$D^{\sigma}_s\left(P_m(s,u)^{-3/2}\right)=\sum_{(l,\beta,\delta)\in\mathcal{D}(\sigma)}(D^\delta h)(P_m(s,u))\prod_{r=1}^l \frac{1}{\delta_r!}\left(\frac{1}{\beta_r!}D^{\beta_r}_s P_m(s,u)\right)^{\delta_r}.$$
Recall that $(l,\beta,\delta)\in\mathcal{D}(\sigma)$ stands for the decompositions
$$\sigma=\sum_{r=1}^l\vert \delta_r\vert\,\beta_r,$$
where $\beta=(\beta_1,\ldots,\beta_l)$, $\delta=\sum_{r=1}^l\delta_r$ and for each $r\in \{1,\ldots,l-1\}$ there exists some $i_r\in\{1,2\}$ such that $(\beta_r)^{i_r}<(\beta_{r+1})^{i_r}$ and $(\beta_r)^i=(\beta_{r+1})^i$ for every $i\neq i_r$.

By virtue of (\ref{DesarrolloTaylor.Parametrizacion.PyQ.cotas.form}), the derivatives with respect to $s$ behave as 
$$\left\vert D^\gamma_s H^{i,j}_{\lambda,0}(s,u)\right\vert\leq C\frac{1}{\vert u\vert}.$$
Let us take derivatives with respect to $u$ and arrive at 
\begin{align*}
\nabla_u&D^\gamma_s H^{i,j}_{\lambda,0}(s,u)\\
=&-\frac{1}{4\pi}\sum_{\sigma\leq \gamma}\binom{\gamma}{\sigma}\nabla_u D^\eta_s\left(P_m(s,u)^{-3/2}\right)\left[\sum_{p,q=1}^2D^{\sigma-\gamma}_s\left(\frac{\partial (\widetilde{\eta}\circ\mu_m)_i}{\partial s_p}(s)\frac{\partial (\mu_m)_j}{\partial s_q}(s)\right)u_pu_q\right]\\
&-\frac{1}{4\pi}\sum_{\sigma\leq \gamma}\binom{\gamma}{\sigma}D^\sigma_s\left(P_m(s,u)^{-3/2}\right)\left[\sum_{p,q=1}^2D^{\sigma-\gamma}_s\left(\frac{\partial (\widetilde{\eta}\circ\mu_m)_i}{\partial s_p}(s)\frac{\partial (\mu_m)_j}{\partial s_q}(s)\right)\nabla_u(u_pu_q)\right],
\end{align*}
They can be similarly estimated by means of (\ref{DesarrolloTaylor.Parametrizacion.PyQ.cotas.form}):
$$\left\vert D^\gamma_s\nabla_u H^{i,j}_{\lambda,0}(s,u)\right\vert\leq C\frac{1}{\vert u\vert^2}.$$
Thus, $H^{i,j}_{\lambda,0}$ has the regularity properties required in Corollary  \ref{NucleoSingular.kDerivadas.Holder.cor}, so
$$\left\Vert\int_{D_m} H^{i,j}_{\lambda,0}(s,s-t)\varphi_m(\mu_m(t))\zeta(\mu_m(t))J_m(t)\,dt\right\Vert_{C^{k+1,\alpha}(D_m)}\leq M\Vert \zeta\Vert_{C^{k,\alpha}(\Sigma_m)}.$$

Let us now move to the remainder $R^{i,j}_{\lambda,0}(s,u)$ and show that the hypoteses in Corollary \ref{NucleoDebSingular.Derivadas.Holder.cor2} are satisfied too. On the one hand, in  the first term $\widetilde{R}^{i,j}_{\lambda,0}(s,u)$ in $R^{i,j}_{\lambda,0}(s,u)$ one can rearranged terms as 
\begin{align*}
(P_m(s,u)+Q_m(s,u))^{-3/2}-P_m(s,u)^{-3/2}&=\int_0^1\frac{d}{d\theta}(P_m(s,u)+\theta Q_m(s,u))^{-3/2}\,d\theta\\
&=-\frac{3}{2}Q_m(s,u)\int_0^1(P_m(s,u)+\theta Q_m(s,u))^{-5/2}\,d\theta.
\end{align*}
Therefore, a $D^\gamma_s$ derivative of $\widetilde{R}^{i,j}_{\lambda,0}(s,u)$ takes the form
\begin{align*}
D^\gamma_s \widetilde{R}^{i,j}_{\lambda,0}(s,u)=\frac{1}{4\pi}\frac{3}{2}\sum_{\sigma\leq \gamma}\binom{\gamma}{\sigma}D^{\sigma}_s&\left(Q_m(s,u)\int_0^1(P_m(s,u)+\theta Q_m(s,u))^{-5/2}\,d\theta\right)\\
\times & \sum_{p,q=1}^2D^{\gamma-\sigma}\left(\frac{\partial(\widetilde{\eta}\circ\mu_m)_i}{\partial s_p}(s)\frac{\partial (\mu_m)_j}{\partial s_q}(s)\right)u_pu_q.
\end{align*}
If we consider $\widetilde{h}(t)=t^{-5/2}$, a similar argument shows that
\begin{align*}
D^\sigma_s&\left(Q_m(s,u)\int_0^1(P_m(s,u)+\theta Q_m(s,u))^{-5/2}\,d\theta\right)\\
&\hspace{0.5cm}=\sum_{\rho\leq \sigma}\binom{\sigma}{\rho}D^\rho_s(Q_m(s,u))\int_0^1D^{\sigma-\rho}_s\left(\left(P_m(s,u)+\theta Q_m(s,u)\right)^{-5/2}\right)\,d\theta\\
&\hspace{0.5cm}=\sum_{\rho\leq \sigma}\binom{\sigma}{\rho}D^\rho_s(Q_m(s,u)) \int_0^1\sum_{(l,\beta,\delta)\in \mathcal{D}(\sigma-\rho)}(D^\delta\widetilde{h})(P_m(s,u)+\theta Q_m(s,u))\\
&\hspace{5cm}\times\prod_{r=1}^l\frac{1}{\delta_r!}\left(\frac{1}{\beta_r!}D^{\beta_r}_s(P_m(s,u)+\theta Q_m(s,u))\right)^{\delta_r}\,d\theta.
\end{align*}

Now, the estimates in (\ref{DesarrolloTaylor.Parametrizacion.PyQ.cotas.form}) yields
\begin{align*}
\left\vert D^\gamma_s \widetilde{R}^{i,j}_{\lambda,0}(s,u)\right\vert&\leq C\frac{1}{\vert u\vert^0},\\
\left\vert\frac{\partial}{\partial u_l}D^\gamma_s \widetilde{R}^{i,j}_{\lambda,0}(s,u)\right\vert&\leq C\frac{1}{\vert u\vert},\\
\left\vert\frac{\partial^2}{\partial u_{l_1}\partial u_{l_2}} D^\gamma_s\widetilde{R}^{i,j}_{\lambda,0}(s,u)\right\vert&\leq C\frac{1}{\vert u\vert^2}.
\end{align*}
These estimates ensure that all the hypotheses in Corollary \ref{NucleoDebSingular.Derivadas.Holder.cor2} are satisfied, so
$$\left\Vert\int_{D_m}\widetilde{R}^{i,j}_{\lambda,0}(s,s-t)\varphi_m(\mu_m(t))\zeta(\mu_m(t))J_m(t)\,dt\right\Vert_{C^{k+1,\alpha}(D_m)}\leq M\Vert \zeta\Vert_{C^{k,\alpha}(\Sigma_m)}.$$

Regarding the second term $\widehat{R}^{i,j}_{\lambda,0}(s,u)$ of $R^{i,j}_{\lambda,0}(s,u)$ we can use a similar argument. First, the formula for the $D^\gamma_s$ derivative 
\begin{align*}
D^\gamma_s &\widehat{R}^{i,j}_{\lambda,0}(s,u)=\frac{1}{4\pi}\sum_{\sigma\leq \gamma}\binom{\gamma}{\sigma}D^\sigma_s\left((P_m(s,u)+Q_m(s,u))^{-3/2}\right)\\
&\hspace{1.5cm}\times\left\{\sum_{p,q,r=1}^2D^{\gamma-\sigma}_s\left(\left(\int_0^1(1-\theta)\frac{\partial^2 (\widetilde{\eta}\circ\mu_m)_i}{\partial s_p\partial s_q}(s-\theta u)\,d\theta\right)\frac{\partial (\mu_m)_j}{\partial s_r}(s)\right)u_pu_qu_r\right.\\
&\hspace{1.87cm}+\sum_{p,q,r=1}^2D^{\gamma-\sigma}_s\left(\frac{\partial (\widetilde{\eta}\circ\mu_m)_i}{\partial s_p}(s)\left(\int_0^1(1-\theta)\frac{\partial (\mu_m)_j}{\partial s_q\partial s_r}(s-\theta u)\,d\theta\right)\right)u_pu_qu_r\\
&\hspace{1.87cm}-\sum_{p,q,r,l=1}^2D^{\gamma-\sigma}_s\left(\left(\int_0^1(1-\theta)\frac{\partial ^2(\widetilde{\eta}\circ \mu_m)_i}{\partial s_p\partial s_q}(s-\theta u)\,d\theta\right)\times\right.\\
&\hspace{3.5cm}\left.\left.\times\left(\int_0^1 (1-\theta)\frac{\partial^2(\mu_m)_j}{\partial s_r\partial s_l}(s-\theta u)\,d\theta\right)\right)u_pu_qu_ru_l\right\},
\end{align*}
and the chain rule for high order derivatives lead to
\begin{align*}
&D^\sigma_s\left((P_m(s,u)+Q_m(s,u))^{-3/2}\right)\\
&\hspace{0.8cm}=\sum_{(l,\beta,\delta)\in\mathcal{D}(\sigma)}(D^\delta h)(P_m(s,u)+Q_m(s,u))\prod_{r=1}^l\frac{1}{\delta_r!}\left(\frac{1}{\beta_r!}D^{\beta_r}_s(P_m(s,u)+Q_m(s,u))\right)^{\delta_r}.
\end{align*}
Consequently, the estimates in (\ref{DesarrolloTaylor.Parametrizacion.PyQ.cotas.form}) show that
\begin{align*}
\left\vert D^{\gamma}_s\widehat{R}^{i,j}_{\lambda,0}(s,u)\right\vert&\leq C\frac{1}{\vert u\vert^0},\\
\left\vert \frac{\partial}{\partial u_l}D^{\gamma}_s\widehat{R}^{i,j}_{\lambda,0}(s,u)\right\vert&\leq C\frac{1}{\vert u\vert},\\
\left\vert \frac{\partial^2}{\partial u_{l_1}\partial u_{l_2}}D^{\gamma}_s\widehat{R}^{i,j}_{\lambda,0}(s,u)\right\vert&\leq C\frac{1}{\vert u\vert^2},
\end{align*}
and Corollary \ref{NucleoDebSingular.Derivadas.Holder.cor2} yields
$$\left\Vert\int_{D_m}\widehat{R}^{i,j}_{\lambda,0}(s,s-t)\varphi_m(\mu_m(t))\zeta(\mu_m(t))J_m(t)\,dt\right\Vert_{C^{k+1,\alpha}(D_m)}\leq M\Vert \zeta\Vert_{C^{k,\alpha}(\Sigma_m)}$$

Now we move to $K^{i,j}_{\lambda,1}(x,z)$ and expand $\widetilde{\eta}_i\circ\mu_m$ and $(\mu_m)_j$ through Taylor's theorem in integral form up to first order
\begin{align*}
K^{i,j}_{\lambda,1}&(\mu_m(s),\mu_m(s)-\mu_m(s-u))\\
&=f_\lambda(\vert P_m(s,u)+Q_m(s,u)\vert^{1/2})\sum_{p,q=1}^2\left(\int_0^1\frac{\partial(\widetilde{\eta}_i\circ\mu_m)}{\partial s_q}(s-\theta u)\,d\theta\right)\left(\int_0^1\frac{\partial(\mu_m)_j}{\partial s_q}(s-\theta u)\,d\theta\right)u_pu_q.
\end{align*}
Then, the $D^\gamma_s$ derivative of $K^{i,j}_{\lambda,1}(\mu_m(s),\mu_m(s)-\mu_m(t))$ takes the form
\begin{align*}
D^\gamma_s K^{i,j}_{\lambda,1}&(\mu_m(s),\mu_m(s)-\mu_m(s-u))\\
&=\sum_{\sigma\leq \gamma}\binom{\gamma}{\sigma}D^\sigma_s\left(f_\lambda(\vert P_m(s,u)+Q_m(s,u)\vert^{1/2})\right)\\
&\hspace{1cm}\times \left[\sum_{p,q=1}^2D^{\gamma-\sigma}_s\left(\left(\int_0^1\frac{\partial(\widetilde{\eta}_i\circ\mu_m)}{\partial s_q}(s-\theta u)\,d\theta\right)\left(\int_0^1\frac{\partial(\mu_m)_j}{\partial s_q}(s-\theta u)\,d\theta\right)\right)u_pu_q\right].
\end{align*}
Again, by the chain derivative formula we arrive at
\begin{multline*}
D^\sigma_s\left(f_\lambda(( P_m(s,u)+Q_m(s,u))^{1/2})\right)\\
=\sum_{(l,\beta,\delta)\in\mathcal{D}(\sigma)}\left.D^\delta(f_\lambda(\cdot^{1/2}))\right\vert_{P_m(s,u)+Q_m(s,u)}\prod_{r=1}^l\frac{1}{\delta_r!}\left(\frac{1}{\beta_r!}D^{\beta_r}_s(P_m(s,u)+Q_m(s,u))\right)^{\delta_r}.
\end{multline*}
Notice that (\ref{PhiPsi.SolucionFundamental.Cotas.Derivadas.flambda.form}) leads to 
$$\left\vert\frac{d^k}{dr^k}\left(f_\lambda(r^{1/2})\right)\right\vert\leq \widetilde{C}\frac{1}{r^{k+1}},\ \forall\,r\in \left(0,d_S^m\right).$$
Consequently, (\ref{DesarrolloTaylor.Parametrizacion.PyQ.cotas.form}) proves the upper bounds
\begin{align*}
\left\vert D^\gamma_s K^{i,j}_{\lambda,1}(\mu_m(s),\mu_m(s)-\mu_m(s-u))\right\vert & \leq C\frac{1}{\vert u\vert^0},\\
\left\vert \frac{\partial}{\partial u_{l_1}}D^\gamma_s K^{i,j}_{\lambda,1}(\mu_m(s),\mu_m(s)-\mu_m(s-u))\right\vert & \leq C\frac{1}{\vert u\vert},\\
\left\vert \frac{\partial^2}{\partial u_{l_1}\partial u_{l_2}}D^\gamma_s K^{i,j}_{\lambda,1}(\mu_m(s),\mu_m(s)-\mu_m(s-u))\right\vert & \leq C\frac{1}{\vert u\vert^2},
\end{align*}
so the hypotheses in Corollary \ref{NucleoDebSingular.Derivadas.Holder.cor2} are satisfied and
$$\left\Vert\int_{D_m}K^{i,j}_{\lambda,1}(\mu_m(s),\mu_m(s)-\mu_m(t))\varphi_m(\mu_m(t))\zeta(\mu_m(t))J_m(t)\,dt\right\Vert_{C^{k+1,\alpha}(D_m)}\leq M\Vert \zeta\Vert_{C^{k,\alpha}(\Sigma_m)}.$$

In order to complete the proof of the theorem, let us show how to deal with the terms $m'\neq m$ in (\ref{NucleosMyS.Homogeneos.Coordenadas.form1}) and (\ref{NucleosMyS.Homogeneos.Coordenadas.form2}). The idea is to obtain estimates over $\Sigma_m\cap \Sigma_{m'}$ and $\Sigma_m\setminus\overline{\Sigma_{m'}}$ separately. First,
\begin{align*}
&\left\Vert \int_{D_{m'}} K^{i,j}_{\lambda}(\mu_m(s),\mu_m(s)-\mu_{m'}(t))\varphi_{m'}(\mu_{m'}(t))\zeta(\mu_{m'}(t))J_{m'}(t)\,dt\right\Vert_{C^{k+1,\alpha}(\mu_m^{-1}(\Sigma_m\cap \Sigma_{m'}))}\\
&\hspace{1cm}\leq C\left\Vert \int_{D_{m'}} K^{i,j}_{\lambda}(\mu_{m'}(s),\mu_{m'}(s)-\mu_{m'}(t))\varphi_{m'}(\mu_{m'}(t))\zeta(\mu_{m'}(t))J_{m'}(t)\,dt\right\Vert_{C^{k+1,\alpha}(D_{m'})}\\
&\hspace{1cm}\leq CM\Vert \zeta\Vert_{C^{k,\alpha}(\Sigma_{m'})}.
\end{align*}
Second, define 
$$C_{m'}:=\mu_{m'}^{-1}(\supp\varphi_{m'}),\ K_{m'}:=\mu_{m'}(C_{m'})\mbox{ and }d_{m,m'}:=\mbox{dist}\left(\Sigma_m\setminus\overline{\Sigma_{m'}},K_{m'}\right)>0,$$
as showed in Figure \ref{fig:Fig2}.
\begin{figure}[t]
\centering
\includegraphics[scale=0.85]{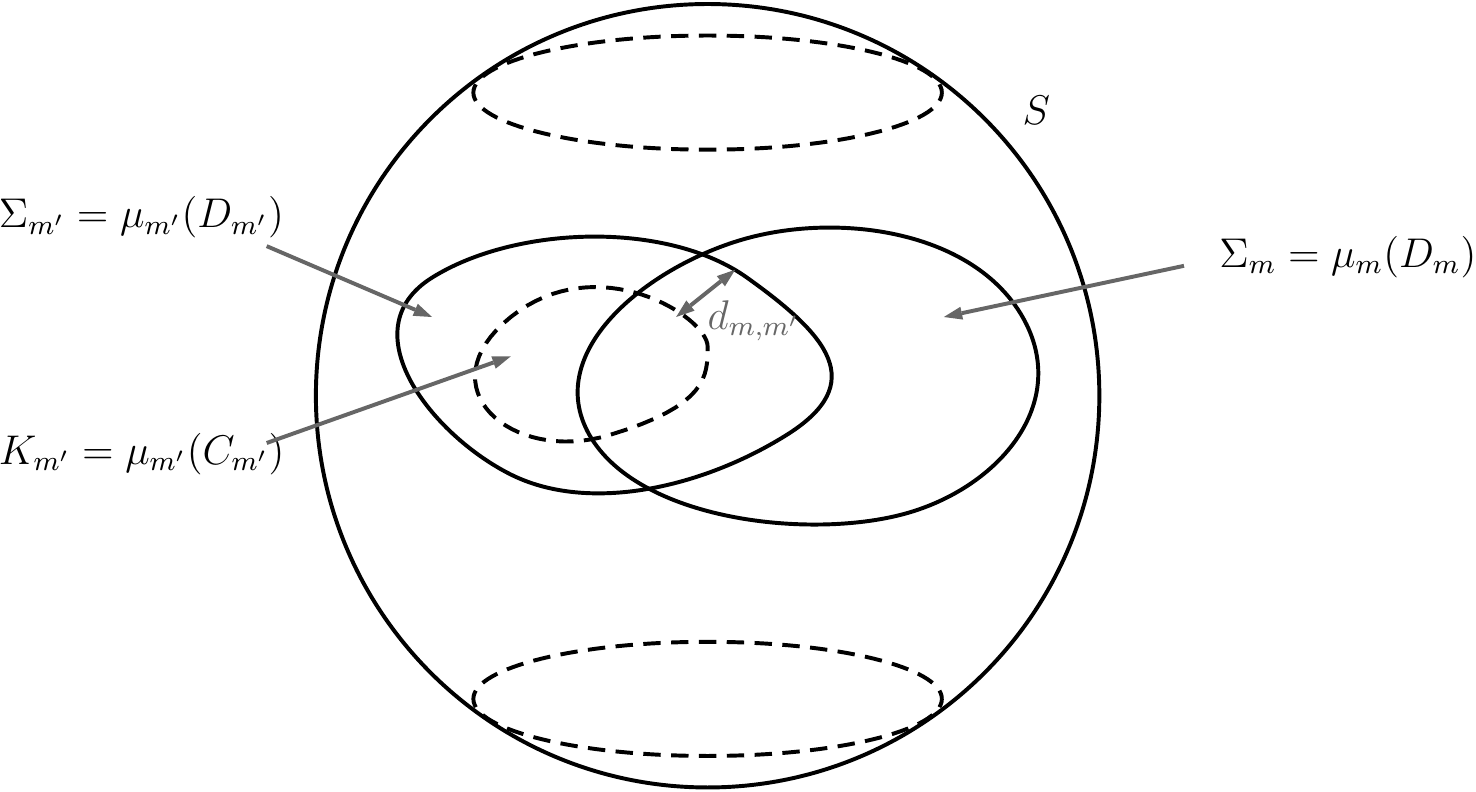}
\caption{Overlapping coordinate neighborhoods $\Sigma_m$ and $\Sigma_{m'}$.}
\label{fig:Fig2}
\end{figure}
This avoids the singularity near $z=0$ in the preceding kernels. Hence, we can take derivatives under the integral sign, obtaining the formula
\begin{align*}
D^\gamma_s \int_{D_{m'}} K^{i,j}_{\lambda}(\mu_m(s),&\mu_m(s)-\mu_{m'}(t))\varphi_{m'}(\mu_{m'}(t))\zeta(\mu_{m'}(t))J_{m'}(t)\,dt\\
=&\sum_{(l,(\delta^1,\delta^2),\beta)\in\mathcal{D}(\gamma)}\int_{C_{m'}}\left(D^{\delta^1}_xD^{\delta^2}_zK^{i,j}_\lambda\right)(\mu_m(s),\mu_m(s)-\mu_{m'}(t))\\
&\hspace{2.5cm}\times\left[\prod_{r=1}^l\frac{1}{\delta^1_r!\delta^2_r!}\left(\frac{1}{\beta_r!}D^{\beta_r}\mu_{m}(s)\right)^{\delta^1_r}\left(\frac{1}{\beta_r!}D^{\beta_r}\mu_{m}(s)\right)^{\delta^2_r}\right]\\
&\hspace{2.5cm}\times\varphi_{m'}(\mu_{m'}(t))\zeta(\mu_{m'}(t))J_{m'}(t)\,dt.
\end{align*}
for each $s\in \mu_{m}^{-1}(\Sigma_m\setminus\overline{\Sigma_{m'}})$. Since
$$\vert D^{\delta^1}_xD^{\delta^2}_zK^{i,j}_\lambda(x,z)\vert\leq \widetilde{C}\frac{1}{\vert z\vert^{\vert \delta^2\vert}},$$
for every $z\in B_{d_{m,m'}}(0)$, then
$$
\left\vert D^\gamma_s \int_{D_{m'}} K^{i,j}_{\lambda}(\mu_m(s),\mu_m(s)-\mu_{m'}(t))\varphi_{m'}(\mu_{m'}(t))\zeta(\mu_{m'}(t))J_{m'}(t)\,dt\right\vert\leq \frac{C}{d^{\vert \gamma\vert}}\vert C_{m'}\vert \Vert \zeta\Vert_{ C^0(\Sigma_{m'})}.
$$
Here, $0<d<1$ is such that $d<d_{m,m'}$ for every $m'\neq m$. Since one can take any $|\gamma|\leq k+2$ by the regularity of $S$, then we obtain the desired estimate for $m'\neq m$ and the result follows.
\end{proof}

\appendix
\section{Gradient, curl and divergence on surfaces}\label{Appendix.A}

In this Appendix we record some well known formulas for the gradient, curl and divergence operators on a compact surface $S\subseteq\RR^3$. These formulas have been useful in several sections to analyze boundary integrals. This is particularly true in the case of Lemma \ref{UnicidadBeltramiNoAcotados.lem}.

Let us consider the vector spaces of smooth tangent vector fields along $S$ and smooth $1$-forms, i.e., $\mathfrak{X}(S)$ and $\Omega^1(S)$ respectively. It is well known that these vector spaces can be identified using the Riemannian metric on $S$ by virtue of the musical isomorphisms
$$\begin{array}{cccccccc}
\phantom{o}^\flat: & \mathfrak{X}(S) & \longrightarrow & \Omega^1(S),\hspace{1cm}&\phantom{o}^\sharp: & \Omega^1(S) & \longrightarrow & \mathfrak{X}(S)\\
 & X & \longmapsto & X^\flat,\hspace{1cm}& &\alpha & \longmapsto & \alpha^\sharp.
\end{array}$$
These are defined as
$$
X^\flat(Y)=X\cdot Y, \hspace{1cm}\alpha^\sharp\cdot X=\alpha(X).
$$
for any given $X,Y\in \mathfrak{X}(S)$ and $\alpha\in \Omega^1(S)$.

The \textit{gradient vector field} over $S$ of any function $f\in C^1(S)$ can be identified with the exterior differential $1$-form over $S$ through the musical isomorphisms:$$\nabla_S f:=(d_S f)^\sharp.$$
If $\overline{f}\in C^1(\RR^3)$ is any extension of  $f$, it turns out that $\nabla_S f$ is the tangential component to the surface of the $\RR^3$ gradient field $\nabla \overline{f}$, that is,
$$\nabla_S f=-\eta\times(\eta\times \nabla \overline{f})\ \mbox{ on }S.$$

Now, we focus on the \textit{divergence} and \textit{curl} of a tangent vector field $X\in \mathfrak{X}(S)$. They can be distributionally defined by the identity
$$\begin{array}{lll}
\displaystyle\int_S \divop_S(X)\varphi\,dS=-\int_S X\cdot \nabla_S \varphi\,dS,&\forall\,\varphi\in C^\infty(S),\\
\displaystyle\int_S \curl_S(X)\varphi\,dS=-\int_S X\cdot(\eta\times \nabla_S \varphi)\,dS,&\forall\,\varphi\in C^\infty(S).
\end{array}$$
Another way to provide a coordinate-free expression for $\divop_S$ and $\curl_S$ is through the Hodge star operator $*$ and the codifferential $\delta_S$. Recall that $*$ acts on each $k$-forms space $\Omega^k(S)$ as the bijection
$$*:\,\Omega^k(S)\longrightarrow \Omega^{2-k}(S),$$
given by
$$\alpha\wedge *\beta=\alpha\cdot \beta \, \text{area}_S,$$
where $\text{area}_S$ stands for the Riemannian area 2-form on $S$ and $\alpha,\beta\in\Omega^1(S)$. The dot symbol here is the pointwise inner product of $k$-forms induced by the musical isomorphisms. Its inverse can be computed thought the next classical formula
$$**=(-1)^{k(2-k)}I\ \ \mbox{ in }\ \Omega^k(S).$$
Analogously, 
$$\delta_S:\,\Omega^k(S) \longrightarrow  \Omega^{k-1}(S),$$
acts on each $k$-forms space $\Omega^k(S)$ as
$$\delta_S \alpha:=(-1)^{2k-1}(*d_S*) \alpha.$$
Recall that $\delta_S$ is the adjoint of $d_S$. Specifically, for any $\alpha\in \Omega^1(S)$ and $\varphi\in C^1(S)$ one has
$$\int_S \varphi \, \delta_S \alpha\,dS=\int_S d_S\varphi\cdot \alpha\,dS,$$
where the above pointwise inner product is the one induced by the Riemannian metric in $S$ through the musical isomorphisms, i.e.,
$$\int_S \varphi \, \delta_S \alpha\,dS=\int_S (d_S\varphi)^\sharp\cdot \alpha^\sharp\,dS=\int_S \nabla_S \varphi\cdot \alpha^\sharp\,dS.$$
As a consequence, take any couple $X\in \mathfrak{X}(S)$ and $\varphi\in C^1(S)$ and note that
$$\int_S \divop_S (X)\,\varphi\,dS=-\int_S \nabla_S \varphi\cdot X,dS=-\int_S \delta_S(X^\flat)\varphi\,dS.$$
Consequently,
$$\divop_SX=-\delta_S(X^\flat)=(*d_S*)(X^\flat)\ \mbox{ on }S.$$
With $\curl_S$ we can also argue as above to arrive at the analogous formula
$$\curl_S X=(*d_S)(X^\flat)\ \mbox{ on }S.$$

To conclude, let us list a few useful identities that follow from the definition of $\nabla_S$, $\divop_S$ and $\curl_S$:

\begin{pro}\label{GradienteDivergenciaRotacional.S.Propiedades.pro}
$\,$
\begin{enumerate}
\item $\curl_S(X)=-\divop_S(\eta\times X)\ \ \forall\,X\in \mathfrak{X}(S).$
\item $\curl_S(-\eta\times(\eta\times F))=\eta\cdot \curl F\ \ \forall\,F\in C^1(\overline{\Omega})$.
\item $\curl_S(\nabla_S f)=0\ \ \forall\,f\in C^2(S).$
\item $\divop_S(\eta\times \nabla_S f)=0\ \ \forall\,f\in C^2(S).$
\item (Poincar\'e's lemma) Assume that $S$ is simply connected and consider any tangent vector field  $X\in \mathfrak{X}(S)$ such that $\curl_S(X)=0$. Then, there exists some $f\in C^2(S)$ such that $X=\nabla_S f$.
\end{enumerate}
\end{pro}

\section{Obstructions to the existence of generalized Beltrami fields}\label{Appendix.B}

In this Appendix we will review the main results on the non-existence of Beltrami fields with a non-constant factor proved in~\cite{Enciso1}, as they are of direct interest for the theorems that we have presented in this paper.

Hence, let us consider in this Appendix a solution to the Beltrami field equation with a factor~$f$:
\begin{equation}\label{B}
\curl u= fu\,,\qquad \divop u=0\,.
\end{equation}
We will not specify the domain of the solution as the results that we will review are mostly local. The key observation is that, as the divergence of $u$ is zero, $f$ is a first integral of $u$:
\[
u\cdot \nabla f=0\,.
\]
Since this first integral condition is
very restrictive, it stands to reason that Equation~\eqref{B}
should not admit any nontrivial solutions for most functions~$f$. Before we make this idea precise in the next paragraphs, let us point out that the (well established) idea of constructing the iterations starting by dragging a function along the integral curves of a field, as we have done in the main body of this work, is fully consistent with the intuition that the first integral condition is the heart of the matter.

The first obstruction to the existence of solutions to the Beltrami equation~\eqref{B} is the following:

\begin{theo}\label{T.main}
  Let $D\subseteq\RR^3$ be a domain and assume that the function $f$
  is nonconstant and of class~$C^{6,\alpha}$. Suppose that the vector
  field~$u$ satisfies the Eq.~\eqref{B} in~$D$. Then there is a
  nonlinear partial differential operator $P\neq0$, which can be
  computed explicitly and involves derivatives of order at most~$6$,
  such that $u\equiv 0$ unless $P[f]$ is identically zero in $D$. In
  particular, $u\equiv0$ for all $f$ in a set of infinite codimension of
  $C^{k,\alpha}(U)$ with any $k\geq 6$.
\end{theo}

It should be noticed that Theorem~\ref{T.main} is of a purely {local} nature, as it provides obstructions for the existence of
nontrivial Beltrami fields in any open set and most proportionality
factors.

A less powerful but more easily visualized obstruction is that if~$f$ has
a regular level set homeomorphic to the sphere, then 
Equation~\eqref{B} does not have any nontrivial solutions. In
particular, there are no Beltrami fields whenever $f$ has local
extrema or is a radial function. This is related to the
classical theorem of Cowling ensuring that
there are no poloidal Beltrami fields with nonconstant
factor and axial symmetry~\cite{Chandra}:

\begin{theo}\label{T.spheres}
Suppose that the function $f$ is
of class $C^{2,\alpha}$ in a domain $D\subseteq\RR^3$. If a regular level set $f^{-1}(c)$ has a connected component
 in $D$ homeomorphic to the sphere, then any solution to 
  Equation~\eqref{B} in $D$ is identically zero. 
\end{theo}

Although we will not repeat here the proof of these results, which can be found in~\cite{Enciso1}, let us give a few words on the main idea. The proof of these theorems is based on formulating the Beltrami
equation~\eqref{B} as a constrained evolution problem. Indeed, one can
show that Equation~\eqref{B} is locally equivalent, in a sense
that can be made precise, to the assertion that there is a
time-dependent 1-form $\beta(t)$ on a surface $\Sigma$ that satisfies the
evolution equation
\begin{equation}\label{evolution}
\partial_t\beta=T(t)\, \beta
\end{equation}
together with the differential constraint
\begin{equation}\label{constraint}
d\beta=0\,.
\end{equation}
Here $T(t)$ is a time-dependent tensor field that depends on $f$ and the exterior differential
$d$ is computed with respect to the coordinates on the surface $\Sigma$,
which, in turn, is a regular level set of $f$. It should be stressed
that this formulation depends strongly on the choice of coordinates.

This formulation lays bare the reason for which the Beltrami equation does not
generally admit nonzero solutions: the evolution~\eqref{evolution} is
not generally compatible with the constraint~\eqref{constraint}, and
the resulting compatibility conditions translate into equations that
$f$ and its derivatives must satisfy. In Theorems~\ref{T.main}
and~\ref{T.spheres} we have presented the first two of these
compatibility conditions, but in fact the method of proof yields a whole hierarchy of
explicitly computable obstructions (with increasingly
cumbersome expressions) to the existence of solutions. To ascertain
how many of these obstructions are actually independent remains an interesting
open problem.

Furthermore, the above formulation provides an appealing explanation,
without even resorting to the statement of the previous theorems, of
the reason for which the attempts at constructing solutions
to~\eqref{B} using variational techniques have failed: while
the regularity of the equation is indeed determined by an elliptic
system, its existence is in fact controlled by a constrained
evolution problem for which the existence theory is ill posed.

\end{document}